\documentclass[12pt,letterpaper]{report} 

\def\GlobalMargin{1.0in}                    
\def\PrintingOffset{0.5in}                  
\def\MainTextSpacing{\doublespacing}        

\def\FigurePath{figures}                    
\def\BibFileName{thesis.bib}                

\def\FontPackage{lmodern}                   

\def\HyperlinkColor{blue}                   

\def\TitleFont{\Large\bfseries\singlespacing\MakeUppercase}  

\def\NoSectionLevel{3}                      
\def\ChapterFont{\Large\bfseries\singlespacing} 
\def\SectionFont{\large\bfseries\singlespacing} 
\def\SubsectionFont{\normalsize\bfseries\singlespacing}  
\def\SubsubsectionFont{\normalsize\itshape\singlespacing}

\def\ParagraphSpacing{\baselineskip}        
\def\ParagraphIndent{0 pt}                  

\def\CaptionFontSize{small}                 
\def\CaptionLabelFontType{bf}               
\def\CaptionSeparator{colon}                
\def\CaptionSpacing{1.0}                    
\def\FigureToCaption{0 pt}                  

\def\HeaderHeight{30 pt}                    
\def\HeaderSpace{12 pt}                     

\def\FootnoteSpacing{\baselineskip}         

\def\BibTextSpacing{\singlespacing}         
\def\BibItemSpacing{\baselineskip}          

\def\ChapQuoteFontSize{\small}              
\def\ChapQuoteLocation{flushright}          
\def\ChapQuoteTextShape{\itshape}           
\def\ChapQuoteAuthorTextShape{\scshape}     
\def\MaxQuoteWidth{0.65\textwidth}          

\def\NoTOCLevel{2}                          
\def\TOCIndent{0 pt}                        

\def\TOCTextSpacing{\singlespacing}         
\def\ChapTOCSpacing{\baselineskip}          
\def\SecTOCSpacing{0.5\baselineskip}        
\def\SubsecTOCSpacing{0.3\baselineskip}     
\def\SubsubsecTOCSpacing{0.3\baselineskip}  
\def\LOTItemSpacing{\baselineskip}          

\def\TitleTopSpacing{-\HeaderHeight-\HeaderSpace}   
\def\BeforeTOCTitleSpacing{-42 pt}          
\def\AfterTOCTitleSpacing{34 pt}            
\def\AfterLOTTitleSpacing{34 pt}            

\def\NumChapterTopMargin{-64 pt}            
\def\UnNumChapterTopMargin{-85 pt}          
\def\ChapLabelToTitle{-27 pt}               
\def\ChapTitleToText{24 pt}                 
\def\SpaceBeforeQuote{-20 pt}               

\usepackage[utf8]{inputenc}                 
\usepackage[american]{babel}                
\usepackage[T1]{fontenc}                    

\usepackage{\FontPackage}
\usepackage{newunicodechar}
\newunicodechar{̈}{\"{}}

\usepackage{amsfonts,amssymb,amsmath,amsthm,autobreak,
    cancel,dsfont,mathtools,mathbbol,mathrsfs,siunitx,upgreek}

\usepackage{booktabs,longtable,dcolumn,makecell,
    multicol,multirow,tabularx,xltabular,rotating}

\usepackage[activate={true,nocompatibility}]{microtype}

\usepackage[backend=biber, defernumbers=true, style=numeric, maxnames=99,
     date=year, isbn=false, url=false, doi=true]{biblatex}

\usepackage{tikz-cd}   
\usepackage{nicefrac}  
\usepackage{algorithm}

\newcommand{\R}{\mathbb{R}}

\newcommand{\E}{\mathbb{E}}
\newcommand{\cD}{{\cal D}}
\newcommand{\cO}{{\cal O}}
\newcommand{\la}{{\langle}}
\newcommand{\ra}{{\rangle}}
\newcommand{\B}{\mathbf{B}}
\newcommand{\A}{\mathbf{A}}
\newcommand{\C}{\mathbf{C}}
\newcommand{\M}{\mathbf{M}}

\usepackage{colortbl}
\usepackage{pifont}
\definecolor{PineGreen}{RGB}{0,110,51}
\definecolor{BrickRed}{RGB}{143,20,2}
\newcommand{\cmark}{{\color{PineGreen}\ding{51}}}%
\newcommand{\xmark}{{\color{BrickRed}\ding{55}}}%

\usepackage{xspace}
\newcommand{\algname}[1]{{\color{BrickRed}\sf  #1}\xspace}

\newcommand{\Prob}[1]{\mathbb{P} \left[ #1\right]}
\newcommand{\hx}{\hat{x}}
\newcommand{\g}{\gamma}

\newcommand{\vep}{\varepsilon}

\newcommand{\om}{\omega}
\newcommand{\eps}{{\varepsilon}}
\newcommand{\eqdef}{:=}
\newcommand{\argmin}{\mathop{\rm argmin}}

\newcommand{\prox}{\mathbf{prox}}

\newcommand{\EE}{\operatorname{\mathbb{E}}}
\newcommand{\Exp}[1]{\mathbb{E} \left[ #1\right]}
\newcommand{\Expk}[1]{\mathbb{E}_k \left[ #1\right]}
\newcommand{\norm}[1]{\left\|#1\right\|}
\newcommand{\sqnorm}[1]{\left\| #1 \right\|^2}
\newcommand{\inprod}[2]{\left\langle #1,#2 \right\rangle}
\newcommand{\colorword}[2]{{\color{#2}#1}}

\newcommand{\reals}{\mathbb{R}}
\newcommand{\autopar}[1]{{\left(#1\right)}}
\newcommand{\automedpar}[1]{{\left[#1\right]}}
\newcommand{\autobigpar}[1]{{\left\{#1\right\}}}

\newcommand{\autonorm}[1]{{\left\|#1\right\|}}
\newcommand{\autoprod}[1]{{\left\langle#1\right\rangle}}

\usepackage[flushleft]{threeparttable} 
\usepackage{colortbl}
\definecolor{bgcolor}{rgb}{0.8,1,1}
\definecolor{bgcolor2}{rgb}{0.8,1,0.8}
\definecolor{niceblue}{rgb}{0.0,0.19,0.56}

\usepackage[pagewise,mathlines]{lineno}     
\usepackage{algorithmic}    

\usepackage[titletoc]{appendix}             
\usepackage{blindtext}                      
\usepackage{calc}                           
\usepackage{caption}                        
\usepackage{comment}                        
\usepackage{epigraph,varwidth}              
\usepackage{enumitem}                       
\usepackage{float}                          
\usepackage[bottom,multiple,hang,flushmargin]{footmisc}      
\usepackage{graphicx,wrapfig}               
\usepackage{geometry}                       
\usepackage{glossaries}                     
\usepackage{fancyhdr}                       
\usepackage[dvipsnames]{xcolor}             
\usepackage[a-1b]{pdfx}                     
\usepackage[pdfa]{hyperref}                 
\usepackage[all]{hypcap}                    
\usepackage{ifthen}                         
\usepackage{lscape}                         
\usepackage{listings,minted}                
\usepackage{csquotes}                       
\usepackage{setspace}                       
\usepackage{seqsplit}                       
\usepackage[rightcaption]{sidecap}          
\usepackage{tocloft}                        
\usepackage{textcomp}                       
\usepackage[absolute]{textpos}              
\usepackage{titlesec}                       
\usepackage[final]{pdfpages}                
\usepackage{parskip}                        
\usepackage{tikz}                           
\usepackage{subcaption}                     

\graphicspath{{\FigurePath/}}

\hypersetup{linktocpage, unicode, linktoc=all, colorlinks=true, 
    citecolor=\HyperlinkColor, filecolor=\HyperlinkColor, 
    linkcolor=\HyperlinkColor, urlcolor=\HyperlinkColor}
\AtBeginBibliography{\urlstyle{rm}}     
\AtBeginBibliography{\vspace*{8pt}}     

\DeclareFieldFormat{sentencecase}{\MakeSentenceCase{#1}}
\renewbibmacro*{title}{%
  \ifthenelse{\iffieldundef{title}\AND\iffieldundef{subtitle}}
    {}
    {\ifthenelse{\ifentrytype{article}\OR\ifentrytype{inbook}%
      \OR\ifentrytype{incollection}\OR\ifentrytype{inproceedings}%
      \OR\ifentrytype{inreference}}
      {\printtext[title]{%
        \printfield[sentencecase]{title}%
        \setunit{\subtitlepunct}%
        \printfield[sentencecase]{subtitle}}}%
      {\printtext[title]{%
        \printfield[titlecase]{title}%
        \setunit{\subtitlepunct}%
        \printfield[titlecase]{subtitle}}}%
     \newunit}%
  \printfield{titleaddon}}

\DeclareBibliographyCategory{mypapers}

\renewcommand{\cftchappresnum}{\chaptername\space}
\newcommand*{\noaddvspace}{\renewcommand*{\addvspace}[1]{}}

\addtocontents{lot}{\protect\noaddvspace}

\addtocontents{lof}{\protect\noaddvspace}

\titleformat{\chapter}[display]{\ChapterFont}
    {\chaptertitlename\ \thechapter}{\ChapLabelToTitle}{\ChapterFont}

\titlespacing*{\chapter}{0pt}{\NumChapterTopMargin}{\ChapTitleToText} 

\titlespacing*{name=\chapter,numberless}{0pt}{\UnNumChapterTopMargin}
    {\ChapTitleToText}

\titleformat*{\section}{\SectionFont}
\titleformat*{\subsection}{\SubsectionFont}
\titleformat*{\subsubsection}{\SubsubsectionFont}

\renewcommand{\arraystretch}{\GlobalTableSpacing}   
\newcommand{\ThesisTitle}[1]{           
    \vspace*{\TitleTopSpacing}
    {\TitleFont {#1} \par}
}

\newcommand{\ThesisAuthor}[1]{          
    by \\ #1 \\
}      

\newcommand{\ThesisStatement}[2]{       
    A {#1} submitted to The Johns Hopkins University in conformity \\
    with the requirements for the degree of {#2}
}

\newcommand{\Location}{Baltimore, Maryland}     

\newcommand{\ThesisDate}[2]{#1 #2}              

\newcommand{\ThesisCopyright}[2]{               
\begin{textblock*}{\textwidth}(\GlobalMargin +\PrintingOffset,9in)
    \copyright\ {#1} {#2} \\ All rights reserved
\end{textblock*}
\null
}

\newcommand\chap[1]{%
    \chapter*{#1}%
    \markboth{#1}{}
    \addcontentsline{toc}{chapter}{#1}}
  
\newcommand{\mytableofcontents}{
    \clearpage
    \renewcommand{\contentsname}{Table of Contents}
    \tableofcontents
    \clearpage
}
\newcommand{\mylistoffigures}{
    \clearpage \phantomsection
    \addcontentsline{toc}{chapter}{List of Figures}
    \listoffigures
    \clearpage
}
\newcommand{\mylistoftables}{
    \clearpage \phantomsection
    \addcontentsline{toc}{chapter}{List of Tables}
    \listoftables
    \clearpage
}

\renewcommand{\epigraphflush}{\ChapQuoteLocation}       
\renewcommand{\epigraphsize}{\ChapQuoteFontSize}        
\renewcommand{\textflush}{\ChapQuoteLocation}
\renewcommand{\sourceflush}{\ChapQuoteLocation}
\newcommand{\epitextfont}{\ChapQuoteTextShape}          
\newcommand{\episourcefont}{\ChapQuoteAuthorTextShape}  

\makeatletter
\newsavebox{\epi@textbox}
\newsavebox{\epi@sourcebox}
\newlength\epi@finalwidth
\renewcommand{\epigraph}[2]{%
    \vspace{\beforeepigraphskip}
    {\epigraphsize\begin{\epigraphflush}
    \epi@finalwidth=\z@
    \sbox\epi@textbox{%
        \varwidth{\epigraphwidth}
        \begin{\textflush}\epitextfont#1\end{\textflush}
        \endvarwidth
   }%
    \epi@finalwidth=\wd\epi@textbox
    \sbox\epi@sourcebox{%
        \varwidth{\epigraphwidth}
        \begin{\sourceflush}\episourcefont#2\end{\sourceflush}%
        \endvarwidth
   }%
    \ifdim\wd\epi@sourcebox>\epi@finalwidth 
        \epi@finalwidth=\wd\epi@sourcebox
    \fi
   \leavevmode\vbox{
        \hb@xt@\epi@finalwidth{\hfil\box\epi@textbox}
        \vskip 1ex         
        \hrule height \epigraphrule
        \vskip 1ex         
        \hb@xt@\epi@finalwidth{\hfil\box\epi@sourcebox}
   }%
   \end{\epigraphflush}
   \vspace{\afterepigraphskip}}}
\makeatother

\allowdisplaybreaks[1]                  
\numberwithin{equation}{chapter}        
\theoremstyle{definition}
\newtheorem{definition}{Definition}

\newtheorem{assumption}{Assumption}
\newtheorem{theorem}{Theorem}
\newtheorem{corollary}{Corollary}
\newtheorem{lemma}{Lemma}
\newtheorem{proposition}{Proposition}
\counterwithin{definition}{chapter}
\counterwithin{remark}{chapter}
\counterwithin{theorem}{chapter}
\counterwithin{corollary}{chapter}
\counterwithin{assumption}{chapter}
\counterwithin{lemma}{chapter}
\counterwithin{proposition}{chapter}

\begin{document}


\begin{titlepage}

\begin{center} \singlespacing \thispagestyle{empty}

\ThesisTitle{Next-Generation Iterative Algorithms for Large-scale Min-max Optimization: Design and Analysis}

\vspace{1in}                    

\ThesisAuthor{Sayantan Choudhury}         

\vspace{1.5in}                  

\ThesisStatement{dissertation}{Doctor of Philosophy}  
\vspace{0.5in} \\               

\Location \\                    
\ThesisDate{October}{2025}        

\ThesisCopyright{2025}{Sayantan Choudhury}

\end{center}

\end{titlepage}


\MainTextSpacing                            
\pagenumbering{roman}                       
\setcounter{page}{2}                        
\pagestyle{fancy}
\renewcommand{\chaptermark}[1]{\markboth{#1}{#1}}
\fancyhead[R]{}               
\fancyhead[L]{\nouppercase \leftmark}       

\chap{Abstract} 

This thesis investigates the design of algorithms for solving min-max optimization problems, which form the mathematical foundation of many modern applications in machine learning, game theory, and optimization. This work offers new theoretical insights and practical algorithms that address the limitations of existing methods in various problem settings.

In chapter \ref{chap:chap-1}, we introduce the min-max optimization problem, highlight its relevance in machine learning applications, and present its reformulation as a root-finding problem. We also review existing algorithms for solving these problems and establish the necessary preliminary concepts required to understand the rest of the thesis.

Chapter \ref{chap:chap-2} focuses on the single-call stochastic extragradient method called \algname{SPEG}. Our analysis introduces the Expected Residual condition, leading to improved convergence guarantees for structured non-monotone root-finding problems under broad sampling schemes. Moreover, our work addresses questions related to convergence under practical conditions, step size selection, and sampling strategies.

In chapter \ref{chap:chap-3}, we propose adaptive min-max optimization algorithms based on Polyak-type update step size rules. We develop both deterministic and stochastic extragradient methods, such as \algname{PolyakSEG} and \algname{DecPolyakSEG}, that eliminate the need for manual step size tuning and ensure convergence in monotone settings. Furthermore, we provide experimental results to support our theoretical claims.

Chapter \ref{chap:chap-4} relaxes the standard Lipschitz assumption by employing the $\alpha$-symmetric $(L_0, L_1)$-Lipschitz condition, enabling a more refined characterization of problem structure. We develop novel step size strategies for the extragradient method, achieving sublinear or linear convergence rates depending on whether the problem is monotone or strongly monotone, and establish local convergence for weak Minty problems.

Finally, in chapter \ref{chap:chap-5}, we extend our study to distributed min-max optimization, proposing a unified framework for communication-efficient algorithms. Our approach recovers several state-of-the-art methods as special cases. The new algorithms achieve improved communication complexity and demonstrate superior empirical performance in distributed tasks.

Overall, this thesis contributes both novel theoretical frameworks and practical algorithms for solving min-max problems.



\begin{singlespace}

    \section*{Primary reader and thesis advisor}
    
    Dr. Nicolas Loizou \\
    Assistant Professor\\
    Department of Applied Mathematics and Statistics\\
    Johns Hopkins University, Baltimore MD

    \section*{Secondary readers}
    
    Dr. Enrique Mallada\\
    Associate Professor\\
    Department of Electrical and Computer Engineering \\
    Johns Hopkins University, Baltimore, MD


\end{singlespace}

\chapter*{~}
\addcontentsline{toc}{chapter}{Dedication}

\begin{center}                  
\vspace*{3in}                   
    \begin{onehalfspacing}      
    \textit{
    To my brother, Sandipan.     
    }
\end{onehalfspacing}
\end{center}
\chap{Acknowledgement}


As this dissertation marks the culmination of my PhD journey, I would like to take this opportunity to thank the individuals without whom this work would not have been possible.

First and foremost, I owe my deepest thanks to my advisor, \textbf{Nicolas Loizou}. His willingness to brainstorm ideas, his thoughtful feedback, and his steady encouragement shaped every stage of this thesis. I benefited immensely from his scientific rigor and his constant confidence in my abilities. I am also grateful to my thesis committee, comprising Prof. \textbf{Amitabh Basu}, Prof. \textbf{Enrique Mallada}, and Prof. \textbf{Yanxun Xu}, for their valuable suggestions and for their time.

I am thankful to \textbf{Eduard Gorbunov} for teaching me various optimization tricks and how to craft a clear and compelling research paper. I also thank my collaborators, \textbf{TaeHo Yoon}, \textbf{Siqi Zhang}, \textbf{Nazarii Tupitsa}, \textbf{Martin Takáč}, \textbf{Sebastian U Stich}, and \textbf{Samuel Horváth}, for their stimulating discussions and decisive contributions to our joint work. In addition, I wish to acknowledge \textbf{Debjit Ash}, \textbf{Rajat Subhra Hazra}, \textbf{Soumendu Sundar Mukherjee}, \textbf{Ben Grimmer}, \textbf{Donniell Fishkind}, and \textbf{Soledad Villar}, whose teaching deepened my appreciation for the beauty of mathematics.

My years in Baltimore were made memorable by a remarkable group of friends. \textbf{Debangan, Sayan, Arunima, Prosenjit, Neha, Anik, Anagh, Eva, Tamasa,} and \textbf{Ayan}: thank you for the board-game nights, weekend outings, and countless moments that turned Baltimore into a home away from home. Even after long days of work, online gaming sessions with my undergraduate friends \textbf{Anirban, Rohan, Sayan, Alokesh,} and \textbf{Sayantan} brought laughter and balance to my routine.

A special thank you goes to \textbf{Subhangi} for standing beside me despite the thousands of miles between us. Your unwavering support and motivation carried me through the toughest stretches of this journey.

Finally, I would not be who I am without the love of my family. I will be forever grateful to my parents, \textbf{Sarmistha} and \textbf{Ajoy Choudhury}, and my elder brother, \textbf{Sandipan}. I owe all my good qualities to them. 
My mother has constantly reminded me that being a good human being comes before everything else. I strive every day to live up to that lesson. My father has taught me how to stay calm in challenging circumstances. I admire how he evaluates difficult situations and makes the correct decisions. My brother has always been my greatest cheerleader. He made me fall in love with maths from my childhood, and he celebrates each step of my career with boundless pride.

To all of you, my sincere thanks. This dissertation is as much yours as it is mine.        

\TOCTextSpacing                             
\microtypesetup{protrusion=false}           
\hypersetup{linkcolor=black}                
\mytableofcontents                          
\mylistoftables                             
\mylistoffigures                            
%
%
\hypersetup{linkcolor=blue}                 
\microtypesetup{protrusion=true}            



\MainTextSpacing                        
\pagenumbering{arabic}                  
\fancyhead[L]{\chaptername\ \thechapter. \nouppercase \leftmark}

\chapter{Introduction} \label{chap:chap-1}



We begin this section by introducing standard minimization problems, followed by min-max optimization problems and their growing importance in modern machine learning applications. We then present a reformulation of the min-max problem that serves as the foundation for the remainder of this thesis. Subsequently, we discuss various classes of min-max problems and the algorithms commonly employed to solve them. We also provide the necessary theoretical background to understand the analysis of these optimization methods. Finally, we outline the thesis structure and provide a brief overview of the chapters that follow.

\section{Minimization Problems}

Most of the modern machine learning applications rely on the empirical risk minimization, where one seeks a parameter $x \in \R^d$ that minimizes a scalar objective function $f: \R^d \to \R$, i.e.
\begin{eqnarray}\label{eq:min}
    \min_{x \in \R^d} f(x).
\end{eqnarray}
Here, the function $f$ aggregates the model's loss on the observed data. For example, the binary classification problem on dataset $\{a_i, b_i \}_{i = 1}^n$, where $a_i \in \R^d$, $b_i \in \{ +1, -1\}$ can be solved through logistic regression~\cite{cox2018analysis}, which boils down to solving \eqref{eq:min} with
\begin{eqnarray*}
    f(x) & = & \frac{1}{n} \sum_{i = 1}^n \log \left( 1 + \exp{\left( - b_i \la a_i, x \ra \right)}\right).
\end{eqnarray*}
The theoretical analysis of \eqref{eq:min} for any arbitrary function $f$ often depends on the assumption that $\nabla f$ is Lipschitz~\cite{nesterov2018lectures}, i.e.
\begin{eqnarray}\label{eq:min_lipschitz}
    \left\| \nabla f(x) - \nabla f(y) \right\| & \leq & L \|x - y\|
\end{eqnarray}
for some $L >0$ (such $f$ is also called smooth). Under the assumption of \eqref{eq:min_lipschitz}, the first-order algorithm, Gradient Descent (\algname{GD}), given by,
\begin{eqnarray*}
    x_{k+1} & = & x_k - \gamma_k \nabla f(x_k)
\end{eqnarray*}
solves \eqref{eq:min} with $\gamma_k = \frac{1}{L} \quad \forall k$. Apart from \eqref{eq:min_lipschitz}, the analysis of \algname{GD} also needs more structure of the $f$, like convexity~\cite{nesterov2003introductory}:
\begin{eqnarray}\label{eq:convex}
    \la \nabla f(x) - \nabla f(y), x - y \ra & \geq & 0 \qquad \forall x, y \in \R^d.
\end{eqnarray}
or $\mu$-strongly convexity:
\begin{eqnarray*}
    \la \nabla f(x) - \nabla f(y), x - y \ra & \geq & \mu \|x - y\|^2 \qquad \forall x, y \in \R^d.
\end{eqnarray*}
to prove convergence guarantees. The problem \eqref{eq:min} has been heavily studied over the years, and researchers have extended the \algname{GD} algorithm to propose its stochastic~\cite{gower2019sgd}, adaptive~\cite{ward2020adagrad}, or distributed~\cite{khaled2020tighter} counterparts. 

However, many modern machine learning applications can be framed as games or, in a simplified form, as a min-max optimization problem. In such a setting, several parameterized models/ players compete to minimize their respective objective functions. Some landmark advances in machine learning that fall under min-max optimization include Generative Adversarial Networks (GANs) \cite{goodfellow2014generative, arjovsky2017wasserstein}, adversarial training of neural networks \cite{madry2018towards, wang2021adversarial}, reinforcement learning \cite{brown2020combining, sokota2022unified}, and distributionally robust learning \cite{namkoong2016stochastic,yu2022fast}. Although the literature on minimization problems has been extensively studied, numerous challenges remain in solving min-max optimization problems. In this thesis, we address several of those challenges.

\section{Min-Max Optimization Problems}

The mathematical formulation of the min-max optimization~\cite{facchinei2007games, loizou2021stochastic, gorbunov2022stochastic, beznosikov2022smooth} problem (also called a two-player zero-sum game) is:
\begin{equation}
    \label{eq: MinMax}
    \min_{w_1 \in \mathcal{X}_1} \max_{w_2 \in \mathcal{X}_2} \mathcal{L} \left( w_1,w_2 \right) \, ,
\end{equation}
where the function $\mathcal{L} : \mathcal{X}_1 \times \mathcal{X}_2 \rightarrow \R$ is assumed to be smooth (i.e. $\nabla_{w_2} \mathcal{L} (w_1, \cdot), \nabla_{w_1} \mathcal{L} (\cdot, w_2)$ are Lipschitz continuous for any $w_1 \in \mathcal{X}_1, w_2 \in \mathcal{X}_2$, respectively). Below, we provide a few instances of machine learning applications that can be reformulated as a min-max optimization problem.
\paragraph{Constrained Minimization Problems.} Consider the constrained minimization problem of the form 
\begin{eqnarray}\label{eq:constrained_min}
    \min_{w_1 \in \R^{d_1}} f \left( w_1 \right) \notag \\
    \text{such that } \mathbf{A} w_1 = b
\end{eqnarray}
for some matrix $\mathbf{A} \in \R^{d_2 \times d_1}$ and vector $b \in \R^{d_2}$. Then the associated Lagrangian of this problem is given by 
\begin{equation}\label{eq:lagrangian}
    \mathcal{L} (w_1, w_2) = f \left( w_1 \right) + \left\langle w_2, \mathbf{A} w_1 - b \right\rangle
\end{equation}
and solving the constrained minimization problem \eqref{eq:constrained_min} is equivalent to solving for  $\min_{w_1} \max_{w_2} \mathcal{L} (w_1, w_2)$ in \eqref{eq:lagrangian}~\cite{boyd2004convex}. This kind of constrained problem finds its application in compressed sensing~\cite{candes2006robust}, distributed optimization \cite{yang2019survey}. 

In distributed optimization, the data is distributed across a network of $N$ nodes in a graph $\mathcal{G} = \left\{ \mathcal{V}, \mathcal{E} \right\}$, where $|\mathcal{V}| = N$. Such an optimization problem can be written in the form of \eqref{eq:constrained_min}, i.e.
\begin{eqnarray}\label{eq:distributed_min}
    \min_{w_1 \in \R^N} f(w_1) \eqdef \sum_{i = 1}^N f_i \left( w_1 [i] \right) \quad \text{such that} \quad \mathbf{A} w_1 = 0.
\end{eqnarray}
Here $f_i: \R \to \R$ are functions only known by the node $i \in [N]$, $\mathbf{A} \in \R^{|\mathcal{E}| \times N} $ is the incidence matrix for graph $\mathcal{G}$ and $w_1 = \left( w_1[1], \cdots, w_1[N] \right)$ is the concatenation of all copies of the decision variables~\cite{razaviyayn2020nonconvex}. Therefore, the distributed optimization problem \eqref{eq:distributed_min} can be formulated as a min-max optimization problem similar to \eqref{eq:lagrangian}.

\paragraph{Generative Adversarial Network Training}

The goal of Generative Adversarial Network (GAN)~\cite{goodfellow2014generative} is to generate samples from a distribution $\mathbf{q}_{w_1}$ that best matches the distribution $\mathbf{p}$ of the data. To achieve this, GAN utilizes two neural networks called a generator and a discriminator. The generator produces a sample that the discriminator must identify as either fake or real. Therefore, the objective of training is to build a generator that produces realistic samples to fool the discriminator.

In the seminal paper of GAN~\cite{goodfellow2014generative}, the objective function is formulated as a min-max optimization problem where the cost function of the discriminator $\mathbf{D}_{w_2}$ is given by 
\begin{eqnarray*}
    \min_{w_1} \max_{w_2} \mathcal{L}  \left( w_1, w_2 \right) \eqdef - \mathbf{E}_{\mathbf{a} \sim \mathbf{p}} \left[ \log \mathbf{D}_{w_2} (\mathbf{a}) \right] - \mathbf{E}_{\mathbf{a}' \sim \mathbf{q}_{w_1}} \left[ \log \left( 1 - \mathbf{D}_{w_2}(\mathbf{a}') \right) \right].
\end{eqnarray*}
This highlights another important application of min-max optimization in modern machine learning.

\paragraph{Distributionally Robust Optimization}
Many machine learning models encounter different data-generating distributions during training versus deployment \cite{wiles2021fine}, such as those in driverless cars \cite{janai2020computer} and medical imaging \cite{erickson2017machine}. For example, \cite{koh2021wilds, perone2019unsupervised} find that a model trained on one set of hospitals may not generalise to the imaging conditions of another; \cite{alcorn2019strike} find that a model for driverless cars may not generalise to new lighting conditions or object poses; and \cite{buolamwini2018gender} find that a model may perform worse on subsets of the distribution,
such as different ethnicities, if the training set has an imbalanced distribution. This motivates the researchers to reformulate the contemporary empirical risk minimization problem $\min_{w_1} \frac{1}{n} \sum_{i = 1}^n \ell_i \left( w_1\right)$ as
\begin{equation*}
    \min_{w_1 \R^d} \max_{w_2 \in \mathcal{Q}} \frac{1}{n} \sum_{i = 1}^n w_2[i] \ell_i \left( w_1 \right)
\end{equation*}
where $\mathcal{Q} = \left\{ w_2 \in \R^n \mid \sum_{i = 1}^n w_2[i] = 1 \right\}$ and $\ell_i : \R^d \to \R$ are loss functions on each data point $i \in [n]$.

\section{Variational Inequality Problem Formulation}

In this thesis, we consider a more abstract formulation of the problem \eqref{eq: MinMax} and focus on solving the following variational inequality problem (VIP) \cite{hsieh2019convergence, stampacchia1964formes, zhao2024learning, antonakopoulos2020adaptive, antonakopoulos2021sifting, diakonikolas2021efficient, cai2022tight}: 
\begin{eqnarray}\label{eq:constrained_vip}
    \text{find } x_* \in \mathcal{X} \text{ such that } \left\langle F(x_*), x - x_* \right\rangle \geq 0 \quad \forall x \in \mathcal{X}
\end{eqnarray}
where $F: \mathbb{R}^d \to \mathbb{R}^d $ is an operator. This problem captures the constrained min-max problem of \eqref{eq: MinMax} as a special case~\cite{hsieh2019convergence} with $x \eqdef \left( w_1, w_2 \right)$, $\mathcal{X} \eqdef \mathcal{X}_1 \times \mathcal{X}_2$ and operator 
\begin{equation}\label{eq:vip_operator}
    F(x) := \begin{bmatrix}
                \nabla_{w_1} \mathcal{L} \left( w_1,w_2 \right) \\
                -\nabla_{w_2} \mathcal{L} \left( w_1,w_2 \right)
            \end{bmatrix}.
\end{equation}
In particular, throughout this manuscript, we concentrate on the unconstrained version of the problem \eqref{eq: MinMax}, i.e., 
\begin{equation}\label{eq:unconstrained_MinMax}
    \min_{w_1 \in \R^{d_1}} \max_{w_2 \in \R^{d_2}} \mathcal{L} \left( w_1,w_2 \right).
\end{equation}
In that case, $\mathcal{X} = \R^{d = d_1 + d_2}$ and \eqref{eq:constrained_vip} boils down to the unconstrained variational inequality problem~\cite{gorbunov2022extragradient, hsieh2020explore, gorbunov2022last, gorbunov2022stochastic} or also known as the root-finding problem~\cite{luo2022extragradient, tran2025randomized, tran2025vfog, tran2024variance}, expressed as follows,
\begin{equation}\label{eq: Variational Inequality Definition}
    \text{find } x_* \in \mathbb{R}^d \text{ such that } F(x_*) = 0.
\end{equation}
Therefore, solving~\eqref{eq: Variational Inequality Definition} amounts to finding a stationary point (where the gradients are zero) $x_* = \left( w_{1*}; w_{2*} \right)$ for~\eqref{eq: MinMax} with $\mathcal{X}_1 = \R^{d_1}$, $\mathcal{X}_2 = \R^{d_2}$ and $F$ defined in \eqref{eq:vip_operator}. Under a convex-concavity assumption for $\mathcal{L} $ (i.e. $\mathcal{L}  \left(\cdot, w_2 \right)$ is convex $\forall w_2 \in \R^d$ and $\mathcal{L} \left( w_1, \cdot \right)$ is concave $\forall w_1 \in \R^d$), the stationary point $\left( w_{1*}, w_{2*} \right)$ is a global solution for the min-max problem, and is known as a Nash Equilibrium, i.e.
\begin{equation*}
    \mathcal{L}  \left( w_{1*}, w_2 \right) \leq \mathcal{L}  \left( w_{1*}, w_{2*} \right) \leq \mathcal{L} \left(w_1, w_{2*} \right) \qquad \forall w_1 \in \R^{d_1}, w_2 \in \R^{d_2}.
\end{equation*}
Apart from min-max optimization problems, \eqref{eq: Variational Inequality Definition} also finds its applications in multiplayer games~\cite{balduzzi2018mechanics} and Reinforcement Learning~\cite{bertsekas2008neuro}. Therefore, our analysis of algorithms for solving \eqref{eq: Variational Inequality Definition} also encompasses these applications, in addition to min-max problems.
\paragraph{Multiplayer Game}
In a multiplayer game, $N$ players compete with each other to reach an equilibrium, i.e., to get a set of actions where no player can unilaterally deviate from their action to achieve a better payoff, given the best actions chosen by all other players \cite{yoon2025multiplayer}. Let $x^i \in \R^{d_i}$ denote the action of player $i \in [N]$ and let $x = \left( w_1,\dots,w_N \right) \in \R^{D} = \R^{d_1 + \cdots + d_n}$ be the joint action vector of all players. 
Let $\mathcal{L}_i (w_1,\dots, w_N) \colon \R^{d_1+\dots+d_N} \to \R$ denotes the objective function of the player $i \in [n]$ (which player $i$ prefers to minimize in $w_i$) and let $w_{-i} = (w_1, \cdots, w_{i-1}, w_{i+1}, \cdots, w_N) \in \R^{D - d_i}$ be the vector containing all players' actions except that of player $i$. With this notation, the goal in a $N$ player game is to find a joint action $x_* = \left(w_{1*}, \dots, w_{N*} \right) \in \R^D$, formally expressed as:
\begin{align}\label{eq:equilibrium}
    \begin{split}
        \underset{x_* = (w_{1*},\dots, w_{N*}) \in \R^D}{\text{find}} \quad \mathcal{L}_i(w_{i*}; w_{-i*}) \le \mathcal{L}_i (w_i; w_{-i*}), \quad
        \forall x^i \in \R^{d_i}, \quad \forall i \in [N],
    \end{split}
\end{align}
where $\mathcal{L}_i (w_i ; w_{-i}) = \mathcal{L}_i (w_1,\dots, w_N)$. As discussed in~\cite{yoon2025multiplayer}, when $\mathcal{L}_i(\cdot, w_{-i})$ is convex for all $w_{-i}$ and $i \in [N]$, the equilibrium point $x_* = \left(w_{1*}, \dots, w_{N*}\right)$ defined in~\eqref{eq:equilibrium} corresponds to a stationary point of the operator
\begin{equation*}
    F(x) = \left( 
        \nabla \mathcal{L}_1(w_1; w_{-1}), \,
        \dots, \,
        \nabla \mathcal{L}_N(w_N; w_{-N})
    \right),
\end{equation*}
that is, $F(x_*) = 0$.

\paragraph{Reinforcement Learning} 
Reinforcement Learning (RL) is concerned with sequential decision making in unknown, stochastic, and dynamically evolving environments~\cite{bertsekas2008neuro, sutton1998reinforcement}.
The interaction between an agent and its environment can be modeled as a Markov decision process (MDP) $\mathcal{M}\;=\;\bigl(\mathcal{S},\mathcal{A}, P, R,\gamma\bigr)$, where $\mathcal{S}$ and $\mathcal{A}$ denote the (finite) state and action spaces, $P \left(s' \mid s, a \right)$ is the transition kernel giving the probability of moving from state $s \in\mathcal{S}$ to state $s' \in\mathcal{S}$ after taking action $a\in\mathcal{A}$, $R:\mathcal{S}\times\mathcal{A} \to \mathbb{R}$ is a bounded reward function and $\gamma\in [0,1)$ is the discount factor that down–weights future rewards. The main objective of a Reinforcement Learning algorithm is to find an optimal policy, and a key step in many of these algorithms is to compute the value function of a given policy $\pi$, i.e.,
\begin{eqnarray}\label{eq:RL}
    V^{\pi} & = & R^{\pi} + \gamma P^{\pi} V^{\pi}
\end{eqnarray}
where $R^{\pi}$ and $P^{\pi}$ are the reward function and transition kernel of the Markov chain induced by the policy $\pi$~\cite{liu2016proximal}. Then problem \eqref{eq:RL} is a special case of \eqref{eq: Variational Inequality Definition} with operator $F(x) = \left(I - \gamma P^{\pi} \right)x - R^{\pi}$~\cite{sutton1998reinforcement}. 

\section{Structure of Operator $F$}

For designing and analyzing algorithms for solving \eqref{eq: Variational Inequality Definition}, it is pivotal to make assumptions on the structure of operator $F$. In this thesis, we are primarily interested in three classes of operators: monotone, strongly monotone, and weak Minty.

\begin{assumption}[Monotone Operators]
    The operator $F$ is monotone if it satisfies
    \begin{eqnarray}\label{eq:monotone}
        \la F(x) - F(y), x - y \ra & \geq & 0 \qquad \forall x, y \in \R^d.
    \end{eqnarray}     
\end{assumption}
This assumption~\eqref{eq:monotone} is a generalization of convexity for functions $f$ \eqref{eq:convex}, as given a convexity of $f$ is equivalent to the assumption $\la \nabla f(x) - \nabla f(y), x - y \ra \geq 0$ for all $x, y \in \R^d$~\cite{nesterov2003introductory}. Therefore, the analysis of monotone operators captures the analysis of convex functions as a special case. 

The $\mu$-strong convexity assumption on a function $f$ is equivalent to $\la \nabla f(x) - \nabla f(y), x - y \ra \geq \mu \| x - y\|^2$ \cite{nesterov2003introductory}. A simple generalization of this assumption for the operators is given as follows.
\begin{assumption}[Strongly Monotone Operators]
    The operator $F$ is strongly monotone if it satisfies
    \begin{eqnarray}\label{eq:strong_monotone}
        \la F(x) - F(y), x - y \ra & \geq & \mu \| x - y\|^2 \qquad \forall x, y \in \R^d.
    \end{eqnarray}     
\end{assumption}
for $\mu > 0$. We refer to such operators as $\mu$-strongly monotone. Similar to monotone operators, this assumption captures the strongly convex-strongly concave min-max optimization problems as a special case. Here is a lemma to establish the equivalence of the strong monotonicity of $F$ and the strong convexity, strong concavity of $g$.

\begin{lemma}\label{lemma:equiv_monotone_convex}
    Functions $\mathcal{L} \left(\cdot, w_2 \right), \mathcal{L} \left( w_1, \cdot \right)$ are strongly convex and strongly concave, respectively, for any $w_1, w_2$ if and only if the operator $F$ from \eqref{eq:vip_operator} is strongly monotone.
\end{lemma}

\begin{proof}
    Suppose the operator $F$ from \eqref{eq:vip_operator} is strongly monotone, i.e., we have
    \begin{eqnarray}\label{eq:equiv_monotone_convex_eq1}
        \left\langle F(x) - F(y), x - y \right\rangle & \geq & \mu \|x - y\|^2
    \end{eqnarray}
    for any $x, y$. Let us choose $x = (w_1, w_2)$ and $y = (w_1', w_2)$. Then the left-hand side of the above inequality gives us
    \begin{eqnarray}\label{eq:equiv_monotone_convex_eq2}
        \left\langle F(x) - F(y), x - y \right\rangle & = & \left\langle \nabla_{w_1} \mathcal{L} (w_1, w_2) - \nabla_{w_1'} \mathcal{L} (w_1', w_2), w_1 - w_1' \right\rangle
    \end{eqnarray}
    while the right-hand side boils down to 
    \begin{eqnarray}\label{eq:equiv_monotone_convex_eq3}
        \mu \|x - y\|^2 & = & \mu \|w_1 - w_1'\|^2.
    \end{eqnarray}
    Therefore, combining the above equations \eqref{eq:equiv_monotone_convex_eq1}, \eqref{eq:equiv_monotone_convex_eq2} and \eqref{eq:equiv_monotone_convex_eq3} we get
    \begin{eqnarray*}
        \left\langle \nabla_{w_1} \mathcal{L} (w_1, w_2) - \nabla_{w_1'} \mathcal{L} (w_1', w_2), w_1 - w_1' \right\rangle & \geq & \mu \|w_1 - w_1'\|^2
    \end{eqnarray*}
    for any $w_1, w_1', w_2$. Therefore, $\mathcal{L}  \left(\cdot, w_2 \right)$ is strongly convex for any $w_2$. Following similar arguments as above, we can show $-\mathcal{L}  \left( w_1, \cdot \right)$ is strongly convex or $\mathcal{L}  \left( w_1, \cdot \right)$ is strongly concave for any $w_1$. Hence, we have shown that whenever the operator $F$ is strongly monotone, the function $\mathcal{L}  \left(\cdot, w_2 \right)$ is strongly convex and $\mathcal{L}  \left(w_1, \cdot \right)$ is strongly concave. For the other implication, refer to Lemma 1 in \cite{grimmer2023landscape}. 

\end{proof}

A similar proof technique (by replacing $\mu = 0$ in previous proof) shows the equivalence of monotonicity and convexity-concavity as well. Some work also considers the $\mu$-quasi strongly monotone ~\cite{loizou2021stochastic} given by the following assumption. 
\begin{assumption}[Quasi Strongly Monotone Operators] The operator $F$ is quasi strongly monotone if it satisfies
    \begin{eqnarray}\label{eq: Strong Monotonicity}
        \left\langle F(x), x - x_* \right\rangle & \geq & \mu \| x - x_*\|^2 \qquad \forall x \in \R^d.
    \end{eqnarray}     
\end{assumption}
This assumption relaxes the \eqref{eq:strong_monotone} as any $\mu$-strongly monotone operator $F$ satisfy \eqref{eq: Strong Monotonicity} with $y = x_*$ (as $F(x_*) = 0$). However, \eqref{eq: Strong Monotonicity} does not guarantee that the operator $F$ is necessarily monotone and captures some non-monotone problems as well. 

Beyond monotone problems, we are also interested in weak Minty operators that satisfy the following assumption.
\begin{assumption}[Weak Minty Operators] The operator $F$ is weak minty if it satisfies
    \begin{eqnarray}\label{eq: weak MVI}
        \la F(x), x - x_* \ra & \geq & - \rho \| F(x)\|^2 \qquad \forall x \in \R^d
    \end{eqnarray}     
\end{assumption}
for some $\rho \geq 0$. \cite{diakonikolas2021efficient} introduced this class of problems, which are not necessarily monotone. For $\rho = 0$, the problem reduces to the Minty root-finding problems~\cite{mertikopoulos2018optimistic}. A variation of this assumption holds with $\rho = 0$~\cite{lan2023policy} for policy optimization in reinforcement learning. \cite{pethick2023solving} discusses more applications of \eqref{eq: weak MVI} for $\rho > 0$.

\section{Methods for Solving Min-Max Optimization Problems}

Over the years, a wide range of algorithms have been proposed to solve \eqref{eq:unconstrained_MinMax}. Among the earliest and most fundamental approaches is the Gradient Descent Ascent (\algname{GDA}) method, which simultaneously performs gradient descent for the minimization player and gradient ascent for the maximization player. Despite its simplicity, \algname{GDA} suffers from critical limitations. To address these shortcomings, researchers have proposed more sophisticated algorithms that incorporate extrapolation steps, such as the Extragradient (\algname{EG}) method~\cite{korpelevich1976extragradient} and the Past Extragradient (\algname{PEG}) method~\cite{popov1980modification}. These modifications enable the algorithm to stabilize its iterates. In this section, we provide a detailed discussion of these algorithms, including their update rules and convergence properties.

\subsection{Deterministic Methods}
\paragraph{Gradient Descent Ascent Method.}
A natural extension of the classical Gradient Descent (\algname{GD}) algorithm for unconstrained minimization problems $ \min_{x \in \R^{d}} f(x)$ is obtained by replacing the objective’s gradient with the operator $F$.
Starting from the \algname{GD} update $x_{k+1}=x_k-\gamma_k \nabla f(x_k),$ the \algname{GDA} iteration for solving root-finding problems is given by
\begin{eqnarray}
    \label{eq:GDA}
    x_{k+1} & = & x_k - \gamma_k F(x_k),
\end{eqnarray}
where $\gamma_k > 0$ is the step size. Choosing the step size is critical, as a $\gamma_k$ that is too small slows progress, whereas a $\gamma_k$ that is too large causes divergence. When the operator $F$ is $\mu$-strongly monotone, \algname{GDA} with a constant step size achieves a linear convergence rate, given by $\mathcal{O} \left( \frac{L^2}{\mu^2} \log \frac{1}{\vep} \right)$. However, \algname{GDA} often fails to solve monotone root-finding problems \cite{noor2003new}. 

The failure of \algname{GDA} under monotonicity has inspired the development of extrapolation–based methods such as the classic \algname{EG} algorithm and its variant \algname{PEG}.

\paragraph{Extragradient Method.} 
\cite{korpelevich1976extragradient} augments the \algname{GDA} step with an additional look-ahead or extrapolation step, yielding the update rule
\begin{eqnarray}\label{eq:EG}
    \hx_k & = & x_k - \gamma_k F(x_k) \notag \\
    x_{k+1} & = & x_k - \omega_k F(\hx_k)
\end{eqnarray}
where $\gamma_k>0$ and $\omega_k>0$ are the extrapolation and update step sizes, respectively. The intermediate point $\hx_k$ serves as a proxy for the unavailable quantity $F(x_{k+1})$ and improves the stability of the algorithm.

When $F$ is $\mu$-strongly monotone, \cite{mokhtari2020unified} show that \algname{EG} enjoys a faster convergence rate compared to the \algname{GDA} with $\mathcal{O} \left( \frac{L}{\mu} \log \frac{1}{\vep} \right)$. As the iteration complexity scales with the condition number $\frac{L}{\mu}$, rather than $\frac{L^2}{\mu^2}$ as in \algname{GDA}, \algname{EG} is remarkably faster on ill-conditioned $\left( \frac{L}{\mu} >> 1 \right)$, strongly monotone
problems. Even for monotone problem, \algname{EG} guarantees convergence with sublinear $\mathcal{O} \left(\frac{1}{\vep} \right)$ rate~\cite{solodov1999hybrid, gorbunov2022extragradient}, a guarantee unattainable by \algname{GDA}.

However, in contrast to \algname{GDA}, \algname{EG} requires two oracle calls in each iteration. This motivates researchers to design single-call variants of the \algname{EG}, which require only one oracle call per iteration.

\paragraph{Past Extragradient Method.}

\cite{popov1980modification} modified the \algname{EG} algorithm by eliminating one oracle call per iteration. They reused the most recent look-ahead operator evaluation for their update rule which is given by
\begin{eqnarray}
    \hx_k & = & x_k - \gamma_k F(\hx_{k-1}) \notag \\
    x_{k+1} & = & x_k - \omega_k F(\hx_k) \label{eq:PEG}.
\end{eqnarray}
Unlike the classical \algname{EG} algorithm~\eqref{eq:EG}, the extrapolation step in~\eqref{eq:PEG} employs the stored value $F(\hx_{k-1})$ (computed at $k-1$ th iteration) rather than $F(x_k)$.  Consequently, the iteration $k$ requires a single fresh oracle evaluation, namely $F(\hx_k)$. For constant step sizes $\gamma_k = \gamma$ and $\omega_k = \alpha$, \algname{PEG} reduces to the Optimistic Gradient Descent Ascent (\algname{OGDA}) algorithm in the unconstrained setting, with the recursion
\begin{eqnarray*}
    \hx_{k+1} = \hx_k - (\alpha + \gamma) F(\hx_k) + \gamma F(\hx_{k-1}).
\end{eqnarray*}
Instead of \algname{PEG}, many works analyze \algname{OGDA} for solving min-max problems~\cite{mokhtari2019convergence}. Under $\mu$-strong monotonicity assumption on $F$, \algname{PEG} attains the same linear rate as
\algname{EG}, namely $\mathcal{O} \left( \frac{L}{\mu} \log \frac{1}{\vep}\right)$ iterations to achieve an $\varepsilon$-accurate solution \cite{mokhtari2020unified}. Because it needs only one oracle call per iteration, \algname{PEG} reduces the per-iteration cost by roughly a factor of two relative to \algname{EG}. Moreover, it can solve monotone problems at a sublinear rate $\mathcal{O}\left(\frac{1}{\varepsilon}\right)$ \cite{gorbunov2022last}, similar to \algname{EG}. 

Apart from these algorithms, there are various other algorithms used for solving min-max optimization problems like Extra Anchored Gradient \cite{yoon2021accelerated}, Halpern iteration \cite{diakonikolas2020halpern, halpern1967fixed}, Mirror-Prox \cite{nemirovski2004prox}, Proximal Point method \cite{rockafellar1976monotone}, Forward-Backward-Forward method \cite{tseng2000modified}.

\subsection{Stochastic Methods}
In modern machine–learning applications, data sets are often so large that evaluating the full operator $F$ at every iteration is prohibitively expensive. Therefore, practitioners use stochastic algorithms that use a stochastic estimate of the operator $F$. For instance, many machine learning applications have a finite-sum structure, i.e., the operators $F$ in \eqref{eq: Variational Inequality Definition} can be written as
\begin{eqnarray}\label{eq:finite_sum}
    F(x) = \frac{1}{n} \sum_{i = 1}^n F_i (x),
\end{eqnarray}
where $n$ is the total number of data and each $F_i$ is an operator based on data point $i \in [n]$. In such a setting, a stochastic estimate of $F$ is given by $F_{\mathcal{S}} (x) \eqdef \frac{1}{B} \sum_{i \in \mathcal{S}} F_i (x)$ where $\mathcal{S} \subseteq [n]$ is a random subset of size $B$. We utilize a stochastic estimate of operators to replace $F$ in the update rules of \algname{GDA}, \algname{EG}, and \algname{PEG}, thereby obtaining their stochastic versions. Thus, the stochastic Gradient Descent Ascent (\algname{SGDA}) \cite{loizou2021stochastic} update is
\begin{equation*}
    x_{k+1} = x_k - \gamma_k F_{\mathcal{S}_k}(x_k)
\end{equation*}
where $\mathcal{S}_k$ is a independent subset drawn on the $k$th iteration. Similarly the update rule of the same-sample stochastic \algname{EG}, which we call \algname{SSEG} \cite{gorbunov2022stochastic, mishchenko2020revisiting}, is given by
\begin{eqnarray*}
    \hx_k & = & x_k - \gamma_k F_{\mathcal{S}_k} (x_k) \\
    x_{k+1} & = & x_k - \omega_k F_{\mathcal{S}_k}(\hx_k).
\end{eqnarray*}
Note that this algorithm is called the same-sample as the same subset $\mathcal{S}_k$ is used both in the extrapolation and update step. However, \cite{gorbunov2022stochastic, juditsky2011solving} also investigates the use of independent samples for extrapolation and update step, which is known as independent sample stochastic \algname{EG} or in short \algname{ISEG}. In case of \algname{PEG}, the stochastic version has the following update rule
\begin{eqnarray*}
    \hx_k & = & x_k - \gamma_k F_{\mathcal{S}_{k-1}} (\hx_{k-1}) \\
    x_{k+1} & = & x_k - \omega_k F_{\mathcal{S}_k} (\hx_k).
\end{eqnarray*}
Unlike their deterministic counterparts, stochastic methods require additional conditions on the mini-batch estimator to guarantee convergence.  Common assumptions include bounded operator \cite{abernethy2021last,nemirovski2009robust}, bounded variance \cite{lin2020gradient,tran2020hybrid,juditsky2011solving}, and growth conditions \cite{lin2020finite}.  These assumptions and how they relate to one another are reviewed in detail in Chapter~\ref{chap:chap-2}.

Moreover, implementing the stochastic algorithms with step sizes proposed for deterministic counterparts may lead to divergence. Even though the deterministic algorithm converges with a constant step size (i.e., step sizes are the same for every iteration), the stochastic algorithms only converge to a neighbourhood of the solution. 
To achieve exact convergence, stochastic algorithms typically require a diminishing step size strategy.

\section{Convergence Analysis Overview for Min-Max Optimization} 

So far, we have presented the problem we are interested in for the thesis and the various algorithms used to solve it. In the rest of the thesis, we will focus on the convergence rates of the algorithms. The convergence rate of an algorithm is defined as the number of iterations or oracle calls (say, number of gradient computations) required to find the $\vep$-accurate solution.

We now present these notions with an example for better understanding. Here, we present the existing convergence analysis of \algname{EG} and discuss its convergence rate. This will help us understand the technical details in the chapters to come.

 \textbf{Sublinear Convergence:} Consider the \algname{EG} method for solving the monotone root-finding problem. Following \cite{gorbunov2022extragradient}, we can derive the following result.
\begin{theorem}
    Suppose $F$ is $L$-Lipschitz and monotone. Then \algname{EG} with step size $\omega_k = \gamma_k = \frac{1}{\sqrt{2}L}$ satisfy 
    \begin{eqnarray*}
        \min_{0 \leq k \leq K} \| F(x_k)\|^2 & \leq & \frac{4 L^2 \|x_0 - x_*\|^2}{K+1}.
    \end{eqnarray*}
\end{theorem}

\begin{proof}
    Expanding $\| x_{k+1} - x_* \|^2$ using the update rule of \algname{EG} we have
    \begin{eqnarray*}
        \| x_{k+1} - x_* \|^2 & = & \| x_k - \omega_k F(\hx_k) - x_* \|^2 \\
        & = & \| x_k - x_*\|^2 - 2 \omega_k \left\langle F(\hx_k), x_k - x_* \right\rangle + \omega_k^2 \| F(\hx_k)\|^2 \\
        & = & \| x_k - x_*\|^2 - 2 \omega_k \left\langle F(\hx_k), x_k - \hx_k + \hx_k - x_* \right\rangle + \omega_k^2 \| F(\hx_k)\|^2 \\
        & = & \| x_k - x_*\|^2 - 2 \omega_k \left\langle F(\hx_k), x_k - \hx_k \right\rangle - 2 \omega_k \left\langle F(\hx_k), \hx_k - x_* \right\rangle \\
        && + \omega_k^2 \| F(\hx_k)\|^2 \\
        & = & \| x_k - x_*\|^2 - 2 \omega_k \left\langle F(\hx_k), x_k - \hx_k \right\rangle \\
        &&- 2 \omega_k \left\langle F(\hx_k) - F(x_*), \hx_k - x_* \right\rangle  + \omega_k^2 \| F(\hx_k)\|^2 \\
        & \overset{\eqref{eq:monotone}}{\leq} & \| x_k - x_*\|^2 - 2 \omega_k \left\langle F(\hx_k), x_k - \hx_k \right\rangle + \omega_k^2 \| F(\hx_k)\|^2 \\
        & = & \| x_k - x_*\|^2 - 2 \omega_k \gamma_k \left\langle F(\hx_k), F(x_k) \right\rangle + \omega_k^2 \| F(\hx_k)\|^2 
    \end{eqnarray*}
    where the last line follows from the extrapolation step of \algname{EG}. Note that we can expand $\la F(\hx_k), F(x_k) \ra = \frac{1}{2} \left(\| F(\hx_k)\|^2 + \| F(x_k)\|^2 - \| F(x_k) - F(\hx_k)\|^2 \right)$. Therefore we obtain 
    \begin{eqnarray*}
        \| x_{k+1} - x_* \|^2 & \leq & \| x_k - x_*\|^2 - \omega_k \gamma_k \| F(\hx_k)\|^2 - \omega_k \gamma_k \| F(x_k)\|^2  \\
        && + \omega_k \gamma_k \| F(x_k) - F(\hx_k)\|^2 + \omega_k^2 \| F(\hx_k)\|^2 \\
        & = & \| x_k - x_*\|^2 - \omega_k (\gamma_k - \omega_k) \| F(\hx_k)\|^2 - \omega_k \gamma_k \| F(x_k)\|^2  \\
        && + \omega_k \gamma_k \| F(x_k) - F(\hx_k)\|^2 \\
        & \leq & \| x_k - x_*\|^2 - \omega_k (\gamma_k - \omega_k) \| F(\hx_k)\|^2 - \omega_k \gamma_k \| F(x_k)\|^2  \\
        && + \omega_k \gamma_k L^2 \| x_k - \hx_k\|^2 \\
        & = & \| x_k - x_*\|^2 - \omega_k (\gamma_k - \omega_k) \| F(\hx_k)\|^2 - \omega_k \gamma_k \left(1 - \gamma_k^2 L^2 \right) \| F(x_k)\|^2 
    \end{eqnarray*}
    where the last inequality follows from the $L$ Lipschitz property of $F$ and the last line follows from the update rule of \algname{EG}. Now we use $\omega_k = \gamma_k = \frac{1}{\sqrt{2} L}$ to get
    \begin{eqnarray*}
        \|x_{k+1} - x_*\|^2 & \leq & \| x_k - x_*\|^2 - \frac{1}{4L^2} \| F(x_k) \|^2.
    \end{eqnarray*}
    Then, rearranging the terms, we have
    \begin{eqnarray*}
        \frac{1}{4L^2} \| F(x_k) \|^2 & \leq & \| x_k - x_*\|^2 - \| x_{k+1} - x_*\|^2.
    \end{eqnarray*}
    Now we sum up this inequality for $k = 0, 1, \cdots K$ to get
    \begin{eqnarray*}
        \frac{1}{4L^2} \sum_{k = 0}^K \| F(x_k) \|^2 & \leq & \| x_0 - x_*\|^2 - \| x_{K+1} - x_*\|^2.
    \end{eqnarray*}
    As $\| x_{K+1} - x_*\|^2 \geq 0$, we have 
    \begin{eqnarray*}
        \frac{1}{4L^2} \sum_{k = 0}^K \| F(x_k) \|^2 & \leq & \| x_0 - x_*\|^2.
    \end{eqnarray*}
    Therefore, dividing both sides by $K+1$ to get
    \begin{eqnarray*}
        \frac{1}{K+1} \sum_{k = 0}^K \| F(x_k)\|^2 & \leq & \frac{4 L^2 \| x_0 - x_*\|^2}{K+1}.
    \end{eqnarray*}
    Now note that, $\min_{0 \leq k \leq K} \| F(x_k)\|^2 \leq  \frac{1}{K+1} \sum_{k = 0}^K \| F(x_k)\|^2$. Hence we obtain
    \begin{eqnarray*}
        \min_{0 \leq k \leq K} \| F(x_k)\|^2 & \leq & \frac{4 L^2 \| x_0 - x_*\|^2}{K+1}.
    \end{eqnarray*}
    This completes the proof.
\end{proof}

Therefore, in order to find a point $x_K$ such that $\| F(x_K)\|^2 \leq \vep$ we need to run \algname{EG} for atleast $K \geq \frac{4L^2 \|x_0 - x_*\|^2}{\vep}$ iterations and require atleast $\frac{8L^2 \|x_0 - x_*\|^2}{\vep}$ oracle calls (as each iteration of \algname{EG} requires two operator evaluation). This rate of convergence $\mathcal{O} \left( \frac{1}{\vep} \right)$ is also known as a sublinear rate.

\textbf{Linear Convergence:} Similarly, \cite{mokhtari2020unified, gidel2018variational} proves that \algname{EG} with step size $\gamma_k = \omega_k = \frac{1}{4L}$ satisfies
\begin{eqnarray}\label{eq:strong_monotone_exp_decay}
    \|x_{K} - x_*\|^2 \leq \left( 1- \frac{\mu}{4L} \right)^{K+1} \| x_0 - x_*\|^2
\end{eqnarray}
for solving the $\mu$-strongly monotone root-finding problems. Therefore, to ensure $\|x_K - x_*\|^2 \leq \vep$ we require at least $K \geq \frac{4L}{\mu} \log \frac{\|x_0 - x_*\|^2}{\vep}$ iteratios or $\frac{8L}{\mu} \log \frac{\|x_0 - x_*\|^2}{\vep}$ oracle calls using the following lemma.
\begin{lemma}
    Suppose $\left\{ A_k \right\}$ is a non-negative sequence satisfying $A_k \leq \varrho^k A_0$ for all $k$ and some $\varrho \in (0,1)$. Then, for any $\varepsilon > 0$, we have $A_K \leq \epsilon$ when $K \geq \frac{1}{1 - \varrho} \log \left( \frac{A_0}{\epsilon} \right)$.    
\end{lemma}

\begin{proof}
    Note that for any $\varrho \in (0, 1)$ we have $\frac{1}{1 - \varrho} \log \left( \frac{1}{\varrho} \right) >  1$. Thus, for $K \geq \frac{1}{1 - \varrho} \log \left( \frac{A_0}{\epsilon} \right)$ we obtain
    \begin{eqnarray*}
        \log \left( \frac{A_0}{A_K} \right) & \geq & \log \left( \frac{1}{\varrho^K} \right) \\
         & = & K \log \left( \frac{1}{\varrho} \right) \\
         & \geq &  \frac{1}{1 - \varrho} \log \left( \frac{1}{\varrho} \right) \log \left( \frac{A_0}{\epsilon} \right) \\
         & \geq & \log \left( \frac{A_0}{\epsilon} \right).
    \end{eqnarray*}
    Then applying exponentials to the above inequality completes the proof.
\end{proof}

This convergence rate, expressed as $\mathcal{O} \left( \frac{L}{\mu} \log \frac{1}{\vep} \right)$, is referred to as a linear convergence rate, as $\|x_K - x_0\|^2$ in \eqref{eq:strong_monotone_exp_decay} decreases exponentially with the number of iterations. In contrast, a rate of the form $\mathcal{O} \left( \frac{1}{\vep} \right)$ is called sublinear, indicating a slower, polynomial decay. These notions of sublinear and linear convergence will be used throughout the subsequent chapters to characterize the efficiency of the algorithms under consideration.

\section{Structure of the Thesis}

In this section, we provide a brief overview of the four main chapters of the thesis, followed by a roadmap to help readers understand the organization of the thesis chapters.

\subsection{Short Summary of the Chapters}

Here we summarize each chapter of the thesis. The detailed proofs of any claims made here are left to the chapters.

\paragraph{Chapter 2: Single-Call Stochastic Extragradient Methods.} Single-call stochastic extragradient methods, like stochastic past extragradient
(\algname{SPEG}) have gained considerable interest in recent years and are among the most efficient algorithms for solving large-scale min-max optimization and root-finding problems that appear in various machine learning tasks. However, despite their undoubted popularity, current convergence analyses of \algname{SPEG} and \algname{SOG} require strong assumptions like bounded
variance or growth conditions. In addition, several important questions regarding the convergence properties of these methods are still open, including mini-batching, efficient step-size selection, and convergence guarantees under different sampling strategies. In this chapter, we address these questions and provide convergence guarantees for two large classes of structured non-monotone root-finding problems: (i) quasi-strongly monotone problems (a generalization of strongly monotone problems) and (ii) weak Minty root-finding problems (a generalization of monotone and Minty root-finding problems). We introduce the expected residual condition, explain its benefits, and show how it allows us to obtain a strictly weaker bound than previously used growth conditions, expected co-coercivity, or bounded variance assumptions. Finally, our convergence analysis holds under the arbitrary sampling paradigm, which includes importance sampling and various mini-batching strategies as special cases.

\paragraph{Chapter 3: Polyak-type (Stochastic) Extragradient.} The \algname{EG} method and its stochastic counterpart \algname{SEG} are foundational algorithms for solving min-max optimization and root-finding problems. A significant practical challenge in deploying these methods is the careful tuning of the two main parameters of the algorithms: the extrapolation and the update step sizes. In this chapter, inspired by the successful use of Polyak step size in solving deterministic minimization problems, as well as its modern stochastic variants used primarily in machine learning applications, we propose a novel class of Polyak-type extragradient algorithms. As a first step, in the deterministic setting (\algname{EG}), we propose a Polyak-type choice for the update step size, while we use a line search strategy for selecting the extrapolation step size.  Extending this framework to stochastic environments, we present and analyze the first stochastic Polyak-type extragradient methods, namely \algname{PolyakSEG} and \algname{DecPolyakSEG}. Our methods, which incorporate adaptive update and line search extrapolation, stand out by eliminating the need for manual parameter tuning while still guaranteeing exact convergence. We provide rigorous convergence guarantees for both monotone and strongly monotone scenarios. We supplement our analysis with experiments that validate the theory and demonstrate the effectiveness of the new algorithms.

\paragraph{Chapter 4: Extragradient Method for $(L_0, L_1)$-Lipschitz Root-finding Problems.} Despite its longstanding presence and significant attention within the optimization community for the extragradient method, most works focusing on understanding its convergence guarantees assume the strong $L$-Lipschitz condition. In this chapter, building on the proposed assumptions by \cite{zhang2019gradient} for minimization and \cite{pmlr-v235-vankov24a} for VIPs, we focus on the more relaxed $\alpha$-symmetric $(L_0, L_1)$-Lipschitz condition. This condition generalizes the standard Lipschitz assumption by allowing the Lipschitz constant to scale with the operator norm, providing a more refined characterization of problem structures in modern machine learning. Under the $\alpha$-symmetric $(L_0, L_1)$-Lipschitz condition, we propose a novel step size strategy for \algname{EG} and establish sublinear convergence rates for monotone operators and linear convergence rates for strongly monotone operators. Additionally, we prove local convergence guarantees for weak Minty root-finding problems. We supplement our analysis with experiments that validate our theory and demonstrate the effectiveness and robustness of the proposed step sizes for \algname{EG}.

\paragraph{Chapter 5: Communication-Efficient Gradient Descent Ascent Methods for Distributed Min-Max Optimization.} Distributed and federated learning algorithms and techniques are associated primarily
with minimization problems. However, with the increase of minimax optimization and variational inequality problems in machine learning, the necessity of
designing efficient distributed/federated learning approaches for these problems is
becoming more apparent. In this paper, we provide a unified convergence analysis of communication-efficient local training methods for distributed variational
inequality problems (VIPs). Our approach is based on a general key assumption on
the stochastic estimates that allows us to propose and analyze several novel local
training algorithms under a single framework for solving a class of structured nonmonotone VIPs. We present the first local gradient descent-accent algorithms with
provably improved communication complexity for solving distributed variational
inequalities on heterogeneous data. The general algorithmic framework recovers
state-of-the-art algorithms and their sharp convergence guarantees when the setting is specialized to minimization or minimax optimization problems. Finally,
we demonstrate the strong performance of the proposed algorithms compared to
state-of-the-art methods when solving federated minimax optimization problems.

\begin{figure}
    \centering
    \includegraphics[width=1\linewidth]{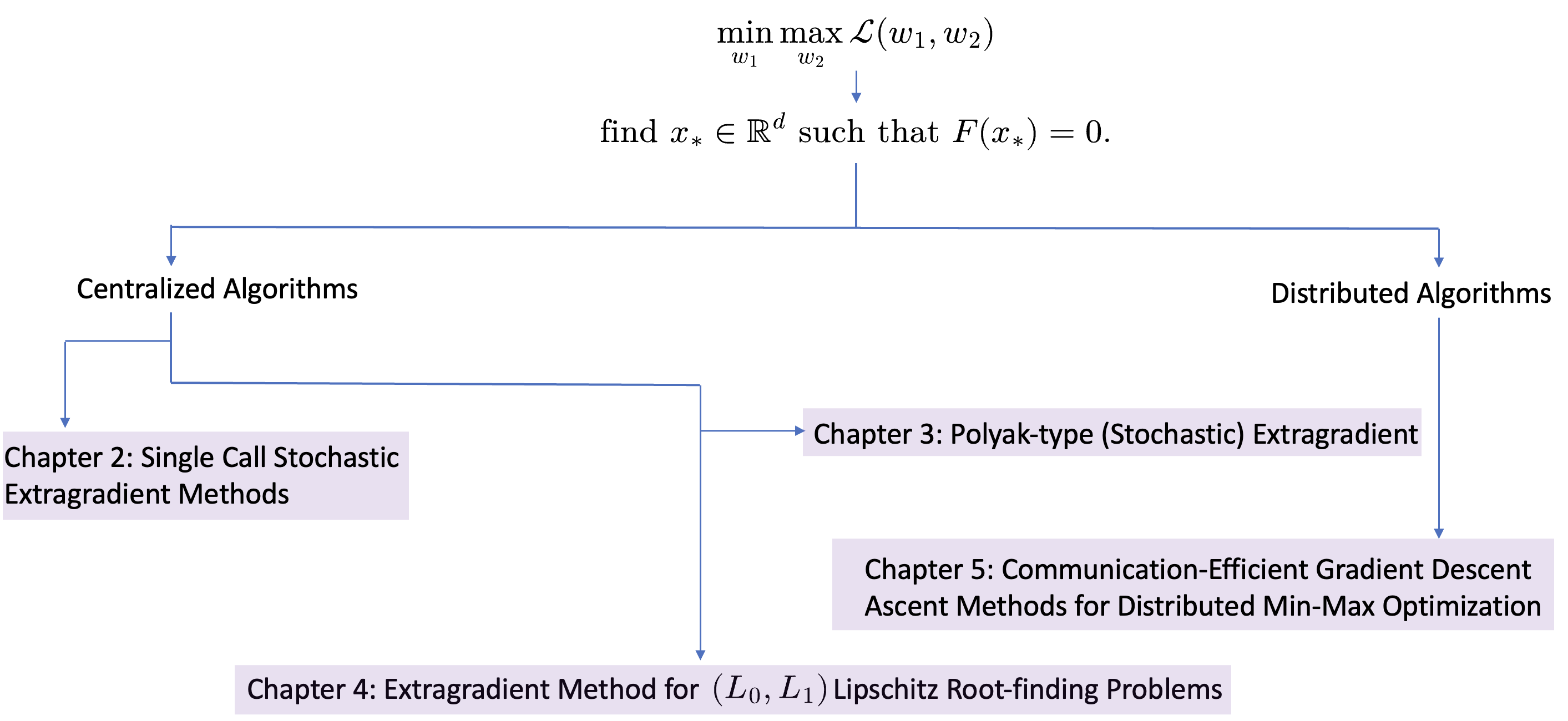}
    \caption{Roadmap of the thesis. This figure illustrates the organization of the thesis chapters and their relationship to the problem under study.}
    \label{fig:roadmap}
\end{figure}

\subsection{Roadmap and Connection between Chapters} For chapters \ref{chap:chap-2}, \ref{chap:chap-3}, \ref{chap:chap-4}, our focus is on the unconstrained min-max optimization problem \eqref{eq:unconstrained_MinMax}, which can be reformulated as the root-finding problem of~\eqref{eq: Variational Inequality Definition}. In chapter \ref{chap:chap-5}, we deal with the variational inequality problem of \eqref{eq:constrained_vip}.

\paragraph{Algorithms.} Chapters \ref{chap:chap-2}, \ref{chap:chap-3}, and \ref{chap:chap-4} consider the centralized setting, while Chapter \ref{chap:chap-5} addresses the distributed min-max optimization problem. In Chapter \ref{chap:chap-2}, we analyze the stochastic single-call algorithm \algname{SPEG}. Chapter \ref{chap:chap-3} focuses on the \algname{PolyakEG} algorithm and introduces two novel stochastic adaptive algorithms: \algname{PolyakSEG} and \algname{DecPolyakSEG}. In Chapter \ref{chap:chap-4}, we study the \algname{EG} algorithm and propose new step-size strategies that ensure convergence under relaxed conditions.

The algorithms presented in chapters \ref{chap:chap-2}, \ref{chap:chap-3}, \ref{chap:chap-4} are centralized, assuming a single node (or computational unit) with access to all the data. In contrast, the distributed setting involves data that is partitioned across multiple nodes, each with limited local information and computational resources (we provide more details on this setting in chapter \ref{chap:chap-5}). In chapter \ref{chap:chap-5}, we shift focus to this distributed paradigm and introduce the \algname{ProxSkip-VIP-FL} and \algname{ProxSkip-L-SVRGDA-FL} algorithms, which demonstrate improved communication efficiency compared to existing methods. Note that, in the centralized setting, we employ the Extragradient type method, whereas in the distributed setting, we utilize the Gradient Descent Ascent type method with local steps. 

\paragraph{Classes of Problems.} In this thesis, we investigate different classes of problems to establish convergence guarantees across various chapters. Strongly monotone problems are studied in Chapters \ref{chap:chap-3}, \ref{chap:chap-4}, and \ref{chap:chap-5}, while monotone problems are considered in Chapters \ref{chap:chap-3} and \ref{chap:chap-4}. Chapter \ref{chap:chap-2} focuses on structured non-monotone problems, including quasi-strongly monotone and weak Minty root-finding problems. In all chapters, we attain sublinear convergence rates for monotone and weak Minty problems, whereas linear convergence rates are achieved for strongly monotone and quasi-strongly monotone problems.

\section{Research Contributions and Related Publications}

The content of the thesis is based on the following publications and preprints:

\textbf{Chapter 2}

$\bullet$ Sayantan Choudhury, Eduard Gorbunov, and Nicolas Loizou. "Single-Call Stochastic Extragradient Methods for Structured Non-monotone Variational Inequalities: Improved Analysis under Weaker Conditions", Neural Information Processing Systems (NeurIPS), 2023. \cite{choudhury2023single}

\textbf{Chapter 3}

$\bullet$ Sayantan Choudhury, TaeHo Yoon, and Nicolas Loizou. "Polyak Meets (Stochastic) Extragradient: Adaptive Update with Line Search Extrapolation", Preprint. \cite{choudhury2025PolyakEG}

\textbf{Chapter 4} 

$\bullet$ Sayantan Choudhury and Nicolas Loizou. "Extragradient Method for $(L_0, L_1)$-Lipschitz Root-finding Problems", Neural Information Processing Systems (NeurIPS), 2025. \cite{choudhury2025extragradient}

\textbf{Chapter 5}

$\bullet$ Siqi Zhang, Sayantan Choudhury, and Nicolas Loizou. "Communication-Efficient Gradient Descent Ascent Methods for Distributed Variational Inequalities: Unified Analysis and Local Updates", International Conference on Learning Representations (ICLR), 2024. \cite{zhang2023communication}

\textbf{Additional Research Contributions:} Apart from the above set of works, I also co-authored the following works during my course of study, which were not used in the formation of this thesis:

$\bullet$ Sayantan Choudhury, Nazarii Tupitsa, Nicolas Loizou, Samuel Horváth, Martin Takac, and Eduard Gorbunov. "Remove that Square Root: A New Efficient Scale-Invariant Version of AdaGrad" Neural Information Processing Systems (NeurIPS), 2024. \cite{choudhury2024remove}

$\bullet$ Eduard Gorbunov, Nazarii Tupitsa, Sayantan Choudhury, Alen Aliev, Peter Richtárik, Samuel Horváth, and Martin Takáč, "Methods for Convex $(L_0, L_1)$-Smooth Optimization: Clipping, Acceleration, and Adaptivity" International Conference on Learning Representations (ICLR), 2025. \cite{gorbunov2024methods}

$\bullet$ TaeHo Yoon, Sayantan Choudhury, and Nicolas Loizou, "Multiplayer Federated Learning:
Reaching Equilibrium with Less Communication", Neural Information Processing Systems (NeurIPS), 2025. \cite{yoon2025multiplayer}

\chapter{Single Call Stochastic Extragradient Methods} \label{chap:chap-2}




\section{Introduction}\label{Introduction}

In this chapter, we consider \eqref{eq: Variational Inequality Definition} where operator $F$ has finite-sum structure 
\begin{equation}\label{eq:chap2_finitesum}
    F(x) \eqdef \frac{1}{n} \sum_{i = 1}^n F_i(x)
\end{equation}
where each $F_i: \mathbb{R}^d \to \mathbb{R}^d $ is a Lipschitz continuous operator.

In modern machine learning applications, game-theoretical formulations that are special cases of problem~\eqref{eq: Variational Inequality Definition} are rarely monotone. That is, the min-max optimization problem~\eqref{eq: MinMax} does not satisfy the popular and well-studied convex-concave setting. For this reason, the ML community started focusing on non-monotone problems with extra structural properties.\footnote{The computation of approximate first-order locally optimal solutions for general non-monotone problems (without extra structure) is intractable. See \cite{daskalakis2021complexity} and \cite{diakonikolas2021efficient} for more details.} In this chapter, we focus on such settings (structured non-monotone operators) for which we are able to provide tight convergence guarantees and avoid the standard issues (like cycling and divergence of the methods) appearing in the more general non-monotone regime. In particular, we focus on understanding and efficiently analyze the performance of single-call extragradient methods for solving (i) $\mu$-quasi-strongly monotone VIPs~\cite{loizou2021stochastic,beznosikov2022stochastic} and (ii) weak Minty variational inequalities~\cite{diakonikolas2021efficient, lee2021fast}.

\subsection{Main Contribution}

\begin{itemize}
    \item \textbf{Expected Residual.} We propose the expected residual \eqref{eq: ER Condition} condition for stochastic variational inequality problems \eqref{eq: Variational Inequality Definition}. We explain the
    benefits of \ref{eq: ER Condition} and show how it can be used to derive an upper bound on $\E \|g(x)\|^2$ (see Lemma~\ref{Lemma: variance bound}) that it is strictly weaker than the bounded variance assumption and “growth conditions” previously used for the analysis of stochastic algorithms for solving \eqref{eq: Variational Inequality Definition}. We prove that \ref{eq: ER Condition} holds for a large class of operators, i.e.,  whenever  $F_i$ of \eqref{eq: Variational Inequality Definition} are Lipschitz continuous.

\begin{table*}[htbp]
    \centering
    \caption{
    \small
    Summary of known and new convergence results for versions of \algname{SEG} and \algname{SPEG} with constant step-sizes applied to solve quasi-strongly monotone variational inequalities and variational inequalities with operators satisfying Weak Minty condition. Columns: ``Setup'' = quasi-strongly monotone or Weak MVI; ``No UBV?'' = is the result derived without bounded variance assumption?; ``Single-call'' = does the method require one oracle call per iteration?; ``Convergence rate'' = rate of convergence neglecting numerical factors. Notation: $K$ = number of iterations; $L_{\max} = \max_{i\in [n]} L_i$, where $L_i$ is a Lipschitz constant of $F_i$; $\overline{\mu} = \frac{1}{n}\sum_{i=1}^n\mu_i$, where $\mu_i$ is quasi-strong monotonicity constant of $F_i$ (see details in \cite{gorbunov2022stochastic}); $\sigma_{\text{US}*}^2 = \frac{1}{n}\sum_{i=1}^n \|F_i(x_*)\|^2$; $\overline{L} = \frac{1}{n}\sum_{i=1}^n L_i$; $\sigma_{\text{IS}*}^2 = \frac{1}{n}\sum_{i=1}^n \frac{\overline{L}}{L_i}\|F_i(x_*)\|^2$; $L$ = Lipschitz constant of $F$; $\mu$ = quasi-strong monotonicity constant of $F$; $\delta, \sigma_*^2$ = parameters from \eqref{eq: variance bound}; $\rho$ = parameter from Weak Minty condition; $\tau$ = batchsize.}
    \label{tab:comparison_of_rates_21}
    \begin{threeparttable}
    \resizebox{\columnwidth}{!}{%
        \begin{tabular}{|c|c|c c c|}
        \hline
        Setup & Method & No UBV? & Single-call? & Convergence rate 
        \\
        \hline\hline
        \multirow{7}{2cm}{\centering Quasi-strong mon.} & \begin{tabular}{c}
            \algname{S-SEG-US}\\
            \cite{gorbunov2022stochastic}
        \end{tabular} & \cmark\tnote{{\color{blue}(1)}} & \xmark & $\frac{L_{\max}}{\overline{\mu}}\exp\left(- \frac{\overline{\mu}}{L_{\max}}K\right) + \frac{\sigma_{\text{US}*}^2}{\overline{\mu}^2 K}$\\
        & \begin{tabular}{c}
            \algname{S-SEG-IS}\\
            \cite{gorbunov2022stochastic}
        \end{tabular} & \cmark\tnote{{\color{blue}(1)}} & \xmark & $\frac{\overline{L}}{\overline{\mu}}\exp\left(- \frac{\overline{\mu}}{\overline{L}}K\right) + \frac{\sigma_{\text{IS}*}^2}{\overline{\mu}^2 K}$\\
        & \begin{tabular}{c}
            \algname{SPEG}\\
            \cite{hsieh2019convergence}
        \end{tabular} & \xmark\tnote{{\color{blue}(2)}} & \cmark & $\frac{L}{\mu}\exp\left(- \frac{\mu}{L}K\right) + \frac{\sigma_*^2}{\mu^2 K}$\tnote{{\color{blue}(3)}}\\
        &\cellcolor{bgcolor2}\begin{tabular}{c}
            \algname{SPEG}\\
            (This chapter)
        \end{tabular} & \cellcolor{bgcolor2}\cmark & \cellcolor{bgcolor2}\cmark & \cellcolor{bgcolor2} $\max\left\{\frac{L}{\mu}, \frac{\delta}{\mu^2}\right\}\exp\left(- \min\left\{\frac{\mu}{L}, \frac{\mu^2}{\delta}\right\}K\right) + \frac{\sigma_*^2}{\mu^2 K}$ \\

        \hline\hline
        \multirow{6}{2cm}{\centering Weak MVI\tnote{{\color{blue}(4)}}} & \begin{tabular}{c}
            \algname{SEG+}\\
            \cite{diakonikolas2021efficient}
        \end{tabular} & \xmark\tnote{{\color{blue}(2)}} & \xmark & {\large $\frac{L^2\|x_0 - x_*\|^2}{K(1 - 8\sqrt{2}L\rho)} + \frac{\sigma_*^2}{\tau (1 - 8\sqrt{2}L\rho)}$} \tnote{{\color{blue}(5)}}\\
        & \begin{tabular}{c}
            \algname{OGDA+}\\
            \cite{bohm2022solving}
        \end{tabular} & \xmark\tnote{{\color{blue}(2)}} & \cmark & {\large $\frac{\|x_0 - x_*\|^2}{Kac(a - \rho)} + \frac{\sigma_*^2}{\tau L^2 ac(a - \rho)}$} \tnote{{\color{blue}(6)}}\\
        & \cellcolor{bgcolor2}\begin{tabular}{c}
            \algname{SPEG}\\
            (This chapter)
        \end{tabular} & \cellcolor{bgcolor2}\cmark & \cellcolor{bgcolor2}\cmark &\cellcolor{bgcolor2} {\large $\frac{\left(1 + \frac{48\omega\gamma\delta}{\tau(1-L\gamma)^2}\right)^K\|x_0 - x_*\|^2}{K \omega\gamma(1-L(\gamma+4\omega))} + \frac{\left(1 + \frac{1-L\gamma}{K}\left(1 + \frac{48\omega\gamma\delta}{\tau(1-L\gamma)^2}\right)^K\right)\sigma_*^2}{\tau(1-L\gamma)(1-L(\gamma+4\omega))}$} \tnote{{\color{blue}(7)}} \; \\
        \hline
    \end{tabular}%
    }
    \begin{tablenotes}
        {\small
        \item [{\color{blue}(1)}] Quasi-strong monotonicity of all $F_i$ is assumed.
        \item [{\color{blue}(2)}] It is assumed that \eqref{eq: variance bound} holds with $\delta = 0$.
        \item [{\color{blue}(3)}] \cite{hsieh2019convergence} do not derive this result but it can be obtained from their proof using standard \\
        choice of step-sizes.
        \item [{\color{blue}(4)}] All mentioned results in this case require large batchsizes $\tau = \cO(K)$ to get $\cO(\nicefrac{1}{K})$ rate.
        \item [{\color{blue}(5)}] The result is derived for $\rho < \nicefrac{1}{8\sqrt{2}L}$.
        \item [{\color{blue}(6)}] The result is derived for $\rho < \nicefrac{3}{8L}$. Here $a$ and $c$ are assumed to satisfy \\
        $aL \leq \frac{7 - \sqrt{1 + 48c^2}}{8(1+c)}$, $c > 0$ and $a > \rho$.
        \item [{\color{blue}(7)}] The result is derived for $\rho < \nicefrac{1}{2L}$. Here we assume that $\max\{2\rho, \nicefrac{1}{(2L)}\} < \gamma < \nicefrac{1}{L}$ and \\
        $0 < \omega < \min\{\gamma - 2\rho, \nicefrac{(4-\gamma L)}{4L}\}$.
        }
        \vspace{-4mm}
    \end{tablenotes}
    \end{threeparttable}
\end{table*}

    \item \textbf{Novel Convergence Guarantees.} We prove the first convergence guarantees for \algname{SPEG}\eqref{SPEG_UpdateRule} in the quasi-strongly monotone~\eqref{eq: Strong Monotonicity} and weak MVI~\eqref{eq: weak MVI} cases \emph{without using the bounded variance assumption}. We achieve that by using the proposed \eqref{eq: ER Condition} condition. In particular, for the class of quasi-strongly monotone VIPs, we show a linear
    convergence rate to a neighborhood of $x_*$ when constant step sizes are used. We also provide theoretically
    motivated step size switching rules that guarantee exact convergence of \algname{SPEG} to $x_*$. In the weak MVI case, we prove the convergence of \algname{SPEG} for $\rho < \frac{1}{2L}$, improving the existing restrictions on $\rho$. We compare our results with the existing literature in Table~\ref{tab:comparison_of_rates_21}.
    
    \item \textbf{Arbitrary Sampling.} Via a stochastic reformulation of the variational inequality problem~\eqref{eq: Variational Inequality Definition} we explain how our convergence guarantees of \algname{SPEG} hold under the arbitrary sampling paradigm. This allows us to cover a wide range of samplings for \algname{SPEG} that were never considered in the literature before, including mini-batching, uniform sampling, and importance sampling as special cases. In this sense, our analysis of \algname{SPEG} is unified for different sampling strategies. Finally, to highlight the tightness of our analysis, we show that the best-known convergence guarantees of deterministic \algname{PEG} for strongly monotone and weak MVI can be obtained as special cases of our main theorems.
\end{itemize}


\section{Stochastic Reformulation of VIPs \& Single-Call Extragradient Methods}

In this chapter, we provide a theoretical analysis of single-call stochastic extragradient methods that allows us to obtain convergence guarantees of any minibatch and reasonable sampling selection.  We achieve that by using the recently proposed ``stochastic reformulation”
of the variational inequality problem \eqref{eq: Variational Inequality Definition} from \cite{loizou2021stochastic}. That is, to allow for any form of minibatching, we use the \emph{arbitrary sampling} notation 
\begin{equation}
\label{gEstimator}
g(x) = F_v(x) \eqdef \frac{1}{n} \sum _{i=1}^n v_i F_i(x),
\end{equation}
where $v \in \R^n_+$ is a random \emph{sampling vector} drawn from a user-defined distribution $\cD$ such that $\E_{\cD}[v_i]  = 1, \,\mbox{for }i=1,\ldots, n$. In this setting, the original problem \eqref{eq: Variational Inequality Definition} can be equivalently written as,
\begin{equation}\label{Reformulation}
    \text{Find } x_* \in \R^d: \E_\cD \left[F_v(x_*)\eqdef\frac{1}{n}\sum_{i=1}^n v_i F_i(x_*) \right]=0,
\end{equation}
where the equivalence trivially holds since $\E_{\cD}[F_v(x)] =\frac{1}{n} \sum _{i=1}^n \E_{\cD}[v_i] F_i(x) = F(x).$

In this chapter, we consider \emph{Stochastic Past Extragradient Method} (\algname{SPEG}) applied to~\eqref{Reformulation}:
\begin{equation}\label{SPEG_UpdateRule}
    \begin{split}
    \hat{x}_k & = x_k - \gamma_k F_{v_{k - 1}}(\hat{x}_{k - 1}) \\ 
    x_{k + 1} & = x_k - \om_k F_{v_k}(\hat{x}_k) 
    \end{split}
    \end{equation} 
where $\hat x_{-1} = x_{0}$ and $v^k \sim \cD$ is sampled i.i.d at each iteration and $\gamma_k >0$ and $\om_k >0$ are the extrapolation step size and update step size respectively. We note that in our convergence analysis, we allow selecting \emph{any} distribution $\cD$ that satisfies $\E_{\cD}[v_i]  = 1$ $\forall i$. This means that for a different selection of $\cD$, \eqref{SPEG_UpdateRule} yields different interpretations of \algname{SPEG} for solving the original problem \eqref{eq: Variational Inequality Definition}.

One example of distribution $\cD$ is $\tau$--minibatch sampling, which is defined as follows.
\begin{definition}[$\tau$-Minibatch sampling]\label{def:minibatch}
Let $\tau \in [n]$. We say that $v \in \R^n$ is a $\tau$--minibatch sampling if
for every subset $S \in [n]$ with $|S| =\tau$, we have that $\Prob{v=\frac{n}{\tau}\sum_{i \in S} e_i} \eqdef  \frac{1}{\binom{n}{\tau}} = \frac{\tau!(n-\tau)!}{n!}.$
\end{definition}

By using a double-counting argument, one can show that if $v$ is a $\tau$--minibatch sampling, it is also a valid sampling vector ($\E_{\cD}[v_i]  = 1$)~\cite{gower2019sgd}. We highlight that our analysis holds for every form of minibatching and for several choices of sampling vectors $v$. Later in Section~\ref{section: Arbitrary Sampling}, we provide more details related to non-uniform sampling. In addition, by Definition~\ref{def:minibatch}, it is clear that if $\tau=n$, then $v_i=1$ for all $i \in [n]$. Later in Section~\ref{Theory}, we prove how our analysis captures the deterministic Past Extragradient Method as a special case. 

In \cite{loizou2021stochastic}, an analysis of stochastic gradient descent-ascent ($x_{k+1} = x_k - \om_k F_{v_k}(x_k)$) under the arbitrary sampling paradigm was proposed for solving star-co-coercive VIPs. Later \cite{gorbunov2022stochastic}, extended this approach and provided general convergence guarantees for the stochastic extragradient method (\algname{SEG}) (a stochastic variant of the popular extragradient method \cite{korpelevich1976extragradient, juditsky2011solving}) for solving quasi-strongly monotone and monotone VIPs. Despite its popularity, \algname{SEG} requires two oracle calls per iteration, which makes it prohibitively expensive in many large-scale applications and not easily applicable to the online learning problems \cite{golowich2020tight}. This motivates us to explore in detail the convergence guarantees of single-call variants of extragradient methods (extragradient methods that require only a single oracle call per iteration). 

\paragraph{On Single-Call Extragradient Methods.} The seminal work of \cite{popov1980modification} is the first paper that proposes the deterministic Past Extagradient method. In the stochastic setting, \cite{hsieh2019convergence} provides an analysis of several stochastic single-call extragradient methods for solving strongly monotone VIPs. In \cite{hsieh2019convergence}, it was also shown that in the unconstrained setting, the update rules of Past Extragradient and Optimistic Gradient are exactly equivalent (see also Proposition \ref{proposition: equivalence of SPEG and SOG} in appendix). Through this connection, and via our stochastic reformulation~\eqref{Reformulation} our theoretical results hold also for the \emph{Stochastic Optimistic Gradient Method} (\algname{SOG}): $ x_{k+1} = x_k - \om_k F_{v_k}(x_k) - \gamma_k (F_{v_k}(x_k) - F_{v_{k-1}}(x_{k-1}))$.

\cite{bohm2022solving} provides the convergence guarantees of \algname{SOG} for weak MVI. To the best of our knowledge, our work is the first that provides convergence guarantees for \algname{SOG} under the arbitrary sampling paradigm (captures sampling beyond uniform sampling) and also without using the bounded variance assumption.


\section{Expected Residual}

In our theoretical results, we rely on the Expected Residual (ER) condition. In this chapter, we define ER and explain how it is connected with similar conditions used in optimization literature. We further provide sufficient conditions for ER to hold and prove how it can be used to obtain a strictly weaker upper bound of $\E \|g(x)\|^2$ than previously used growth conditions, expected co-coercivity, or bounded variance assumptions.
\begin{assumption}\label{as:expected_residual}
We say the Expected Residual (ER) condition holds if there is a parameter $\delta>0$ such that for an unbiased estimator $g(x)$ of the operator $F$, we have
\begin{align*}
\label{eq: ER Condition}
    \E \left[ \| (g(x) - g(x_*)) - (F(x) - F(x_*))\|^2 \right] \leq \frac{\delta}{2} \|x - x_*\|^2. \tag{ER}
\end{align*}
\end{assumption}

The \ref{eq: ER Condition} condition bounds how far the stochastic estimator $g(x)= F_v(x)$ \eqref{gEstimator} used in \algname{SPEG} is from the true operator $F(x)$. \ref{eq: ER Condition} depends on both the properties of the operator $F(x)$ and of the selection of sampling (via $g(x)$). 
Conditions similar to \ref{eq: ER Condition} appeared before in optimization literature but they have never been used in operator theory and the analysis of \algname{SPEG}. In particular, \cite{gower2021sgd} used a similar condition for analyzing \algname{SGD} in stochastic optimization problems but with the right-hand side of \ref{eq: ER Condition} to be the function suboptimality $f(x)-f(x_*)$ (such concept is not available in VIPs). In \cite{szlendak2021permutation} and \cite{gorbunov2022variance}, similar conditions appear under the name ``Hessian variance'' assumption for distributed minimization problems. In the context of distributed VIPs, a similar but stronger condition to \ref{eq: ER Condition} is used by \cite{beznosikov2023compression}. 

\textbf{Bound on Operator Noise.}
A common approach for proving the convergence of stochastic algorithms for solving the VIPs is assuming uniform boundedness of the stochastic operator or uniform boundedness of the variance. However, as we explain below, these assumptions either do not hold or are true only for a restrictive set of problems. In this chapter, we do not assume such bounds. Instead, we use the following direct consequence of \ref{eq: ER Condition}.

\begin{lemma}\label{Lemma: variance bound}
Let $\sigma_*^2 \eqdef \E\|g(x_*)\|^2 <\infty$ (operator noise at the optimum is finite). If \ref{eq: ER Condition} holds, then
\begin{equation}\label{eq: variance bound}
    \begin{split}
        \E \|g(x)\|^2 \leq \delta \|x - x_*\|^2 + \|F(x)\|^2 + 2\sigma_*^2.
    \end{split}
\end{equation}
\end{lemma}
\paragraph{Sufficient Conditions for \ref{eq: ER Condition}.}
Let us now provide sufficient conditions which guarantee that the \ref{eq: ER Condition} condition holds and give a closed-form expression for the expected residual parameter $\delta$ and $\sigma_*^2 = \E\|g(x_*)\|^2$ for the case of $\tau$-minibatch sampling (Def.~\ref{def:minibatch}).
\begin{proposition}\label{Prop_SufficientCondition}
Let $F_i$ of problem~\eqref{eq: Variational Inequality Definition} be $L_i$-Lipschitz operators, then \ref{eq: ER Condition} holds.  If, in addition, 
vector $v \in \R^n$ is a $\tau$--minibatch sampling (Def.~\ref{def:minibatch}) then: $\delta = \frac{2}{n \tau} \frac{n- \tau}{n -1} \sum_{i = 1}^n L_i^2, \text{and } \sigma_*^2 = \frac{1}{n \tau} \frac{n- \tau}{n -1} \sum_{i = 1}^n \|F_i(x_*)\|^2.$
\end{proposition}
Similar results to Prop.~\ref{Prop_SufficientCondition} but under different sufficient conditions have been obtained for $\tau$--minibatch sampling under expected smoothness and a variant of expected residual for solving minimization problems in \cite{gower2019sgd} and \cite{gower2021sgd}, respectively. In \cite{loizou2021stochastic}, a similar proposition was derived but for the much more restrictive class of co-coercive operators.

\paragraph{Connection to Other Assumptions.}  In the proofs of our convergence results, we use the bound \eqref{eq: variance bound}, which, as we explained above, is a direct consequence of \ref{eq: ER Condition}. In this paragraph, we place this bound in a hierarchy of common assumptions used for the analysis of stochastic algorithms for solving VIPs. 
In the literature on stochastic algorithms for solving the VIPs and min-max optimization problems, previous works assume either bounded operator ($\E\|g(x)\|^2 \leq c$) \cite{abernethy2021last, nemirovski2009robust}, bounded variance ($\E \|g(x) - F(x)\|^2 \leq c$) \cite{lin2020gradient, tran2020hybrid, juditsky2011solving} (in Appendix \ref{sec:BoundedVarianceCounterExample} we provide a simple example where bounded variance assumption does not hold) or growth condition ($\E \|g(x)\|^2 \leq c_1 \|F(x)\|^2 + c_2$) \cite{lin2020finite}.  In all of these conditions, the parameters $c$, $c_1$, and $c_2$ are usually constants that do not have a closed-form expression. The closer works to our results are \cite{loizou2021stochastic,beznosikov2022stochastic} which assumes existence of $l_{F}>0$ such that the expected co-coercivity condition ($\E \|g(x) - g(x_*)\|^2 \leq l_{F} \la F(x), x- x_*\ra$) holds. Their convergence guarantees provide an efficient analysis for several variants of \algname{SGDA} for solving co-coercive VIPs. In the proposition below, we prove how these conditions are related to the bound \eqref{eq: variance bound} obtained using \ref{eq: ER Condition}.

\begin{proposition}\label{Proposition connecting assumptions} Suppose $F$ is a $L$-Lipschitz operator. Then we have the following hierarchy of assumptions:
\begin{center}
\begin{tikzcd}[column sep=.4em, row sep = .4em]
    \boxed{\text{Bounded Operator}} \arrow[r] & \boxed{\text{Bounded Variance}} \arrow[r] & \boxed{\text{Growth Condition}} \arrow[r] & \boxed{\eqref{eq: variance bound}}\\
    & \boxed{\text{$F_i$ are $L_i$-Lipschitz}} \arrow[r] & \boxed{\eqref{eq: ER Condition}}  \arrow[ur] & \\
    & \boxed{\text{Expected Cocoercivity}}\arrow[ur] & & &
\end{tikzcd}
\end{center}
\end{proposition}
Let us also mention that \cite{hsieh2020explore} provided convergence guarantee of double-oracle stochastic extragradient (\algname{SEG}) method under the variance control condition $\E \|g(x) - F(x)\|^2 \leq (a \|x - x_*\| + b)^2$ where $a, b \geq 0$. In their work, they focus on solving VIPs satisfying the error-bound condition, and they did not provide closed-form expressions of parameters $a$ and $b$. Although the analysis of \cite{hsieh2020explore} can be conducted with $a > 0$, the authors only provide rates for the case $a = 0$. The main difference between their results (for \algname{SEG}) and our results (for \algname{SPEG}) is that our bound \eqref{eq: variance bound} is not really an assumption, but it holds for free when $F_i$ are $L_i$-Lipschitz. In addition, the values of parameters $\delta$ and $\sigma_*^2$ in \eqref{eq: variance bound} could have different values based on the sampling used in the update rule of \algname{SPEG}.


\section{Convergence Analysis}\label{Theory}
In this section, we present and discuss the main convergence results of this chapter. In the first part, we focus on the ones derived for $\mu$-quasi strongly monotone problems \eqref{eq: Strong Monotonicity} (both for constant and decreasing step sizes), and in the second part on the Weak Minty VIP \eqref{eq: weak MVI}. 
\subsection{Quasi-Strongly Monotone Problems}
\paragraph{Constant step size:}We start with the case of $\mu$-quasi strongly monotone problems and consider the convergence of \algname{SPEG} with constant step size.

\begin{theorem}\label{Theorem: constant stepsize theorem}
Let $F$ be $L$-Lipschitz, $\mu$-quasi strongly monotone, and let \ref{eq: ER Condition} hold.  Choose  step sizes $\gamma_k = \omega_k = \omega$ such that 
\begin{equation}\label{eq:constant_stepsize}
0 < \om \leq \min \left\{ \frac{\mu}{18 \delta}, \frac{1}{4L}\right\}    
\end{equation}
for all $k$. Then the iterates produced by \algname{SPEG}, given by \eqref{SPEG_UpdateRule} satisfy
\begin{equation}
    R_{k}^2 \leq \left(1 - \frac{\omega\mu}{2}\right)^{k} R_0^2 + \frac{24\om \sigma_*^2}{\mu}, \label{eq:SPEG_const_steps_neighborhood}
\end{equation}
where $R_{k}^2 \coloneqq \E \left[\|x_{k} - x_*\|^2 + \|x_{k} - \hat{x}_{k-1}\|^2 \right]$. Hence, given any $\varepsilon > 0$, and choosing $\om = \min \left\{\frac{\mu}{18 \delta}, \frac{1}{4L}, \frac{\varepsilon \mu}{48 \sigma_*^2} \right\}$,
\algname{SPEG} achieves $\E \|x_K - x_*\|^2 \leq \varepsilon$ after $$K \geq \max \bigg\{\frac{8L}{\mu}, \frac{36 \delta}{\mu^2}, \frac{96 \sigma_*^2}{\varepsilon \mu^2} \bigg\} \log \bigg( \frac{2 R_0^2}{\varepsilon} \bigg)$$
iterations.
\end{theorem}
To the best of our knowledge, the above theorem is the first result on the convergence of \algname{SPEG} that does not rely on the bounded variance assumption. Theorem~\ref{Theorem: constant stepsize theorem} recovers the same rate of convergence with the Independent-Samples \algname{SEG} (\algname{I-SEG}) under assumption \eqref{eq: variance bound} \cite{gorbunov2022stochastic}, although \cite{gorbunov2022stochastic} simply assume \eqref{eq: variance bound}, while we show that it follows from Assumption~\ref{as:expected_residual} holding whenever all summands $F_i$ are Lipschitz. However, in the case when all $F_i$ are $\mu$-quasi strongly monotone and $L_i$-Lipschitz (on average), one can use Same-Sample \algname{SEG} (\algname{S-SEG}). The existing results for \algname{S-SEG} have better exponentially decaying term \cite{mishchenko2020revisiting, gorbunov2022stochastic} then Theorem~\ref{Theorem: constant stepsize theorem}, e.g., in the case when $L_i = L $ for all $i\in [n]$, we have $\delta = \cO(L^2)$ meaning that the exponentially decaying term in \eqref{eq:SPEG_const_steps_neighborhood} is $\cO(R_0^2\exp(-\nicefrac{\mu^2k}{L^2})))$, while \algname{S-SEG} has much better exponentially decaying term $\cO(R_0^2\exp(-\nicefrac{\mu k}{L})))$. 

Such a discrepancy can be partially explained by the following fact: \algname{S-SEG} can be seen as one step of deterministic Extragradient for stochastic operator $F_{v_k}$ allowing to use one-iteration analysis of Extragradient without controlling the variance. In contrast, there is no version of \algname{SPEG} that uses the same sample for extrapolation and update steps. This forces to use different samples for these steps and this is a key reason why \algname{SPEG} cannot be seen as one iteration of deterministic Past-Extragradient for some operator. Due to this, we need to rely on some bound on the variance to handle the stochasticity in the updates; see also \cite[Appendix F.1]{gorbunov2022stochastic}. Therefore, in our analysis, we use Assumption~\ref{as:expected_residual}, implying \eqref{eq: variance bound}. Nevertheless, it is still an open question whether it is possible to improve the rate of \algname{SPEG} in the case of $\mu$-quasi strongly monotone and Lipschitz operators $F_i$.

To highlight the generality of Theorem~\ref{Theorem: constant stepsize theorem}, we note that for the
deterministic \algname{PEG}, $\delta = 0$ and $\sigma_*^2 = 0$ (by selecting $\tau=n$ in the definition~\ref{def:minibatch} of minibatch sampling). In this case, Theorem~\ref{Theorem: constant stepsize theorem} recovers the well-known result (up to $\nicefrac{1}{2}$ factor in the rate) for deterministic \algname{PEG} proposed in \cite{gidel2018variational} as shown in the following corollary.
\begin{corollary}
\label{dawna}
    Let the assumptions of Theorem~\ref{Theorem: constant stepsize theorem} hold and a deterministic version of \algname{SPEG} is considered, i.e., $\delta = 0$, $\sigma_*^2 = 0$. Then, Theorem~\ref{Theorem: constant stepsize theorem} implies that for all $k\geq 0$ the iterates produced by \algname{SPEG} with step sizes $\gamma_k = \omega_k = \omega$ such that $ 0 < \om \leq \frac{1}{4L} $ satisfy $R_{k}^2 \leq \left(1 - \frac{\omega\mu}{2}\right)^{k} R_0^2$,
where $R_{k}^2 \coloneqq \|x_{k} - x_*\|^2 + \|x_{k} - \hat{x}_{k-1}\|^2$.
\end{corollary}
\paragraph{Decreasing step size:}In this section, we consider two different decreasing step sizes policies for \algname{SPEG} applied to solve quasi-strongly monotone problems.
\begin{theorem}\label{SPEG switching rule}
Let $F$ be $L$-Lipschitz, $\mu$-quasi strongly monotone, and Assumption~\ref{as:expected_residual} hold. Let
\begin{equation}
    \gamma_k = \om_k \coloneqq 
\begin{cases}
\Bar{\om}, &\text{if } k \leq k^*, \\
\frac{2k+1}{(k+1)^2}\frac{2}{\mu}, &\text{if } k > k^*,
\end{cases}\label{eq:stepsize_switching_1}
\end{equation}
where $\Bar{\om} \coloneqq \min \left \{\nicefrac{1}{(4L)},\nicefrac{\mu}{(18 \delta)} \right\}$ and $k^* = \lceil \nicefrac{4}{(\mu \Bar{\om})} \rceil$. Then for all $K \geq k^*$ the iterates produced by \algname{SPEG} with step sizes \eqref{eq:stepsize_switching_1} satisfy 
\begin{equation}
    R_{K}^2  \leq \left(\frac{k^*}{K}\right)^2 \frac{R_0^2}{\exp(2)} + \frac{192 \sigma_*^2}{\mu^2 K}, \label{eq:SPEG_convergence_decr_steps_1}
\end{equation}
where $R_{K}^2 \coloneqq \E \left[\|x_{K} - x_*\|^2 + \|x_{K} - \hat{x}_{K-1}\|^2 \right]$.
\end{theorem}
\algname{SPEG} with step size policy\footnote{Similar step size policy is used for \algname{SGD} \cite{gower2019sgd} and \algname{SGDA} \cite{loizou2021stochastic}.} \eqref{eq:stepsize_switching_1} has two stages of convergence: during first $k^*$ iterations it uses constant step size to reach some neighborhood of the solution and then the method switches to the decreasing $\cO(\nicefrac{1}{k})$ step size allowing to reduce the size of the neighborhood.

For the case of strongly monotone problems (a special case of our quasi-strongly monotone setting) \cite{hsieh2019convergence} also analyze \algname{SPEG} with decreasing $\cO(\nicefrac{1}{k})$ step size\footnote{We point out the proof by \cite{hsieh2019convergence} can be generalized to the case of constant step size, though the authors do not consider this step size schedule explicitly.} under bounded variance assumption, i.e., when \eqref{eq: variance bound} holds with $\delta = 0$ and some $\sigma_*^2 \geq 0$, which is equivalent to the uniformly bounded variance assumption. In particular, Theorem 5 \cite{hsieh2019convergence} states
$\E\left[\|x_{K} - x_*\|^2\right] \leq \frac{C\sigma_*^2}{\mu^2 K} + o\left(\frac{1}{K}\right)$ where $C$ is some numerical constant. If the problem is strongly monotone, the result of \cite{hsieh2019convergence} is closely related to what is obtained in Theorem~\ref{SPEG switching rule}: the main difference in the upper-bound is that we provide an explicit form of $o\left(\nicefrac{1}{K}\right)$ term. Moreover, in contrast to the result from \cite{hsieh2019convergence}, Theorem~\ref{SPEG switching rule} holds even when $\delta > 0$ in \eqref{eq: variance bound}, which covers a larger class of problems. 

Following \cite{stich2019unified, gorbunov2022stochastic, beznosikov2022stochastic}, we also consider another decreasing step size policy.
\begin{theorem} \label{Theorem: Total number of iteration knowledge}
Let $F$ be $L$-Lipschitz, $\mu$-quasi strongly monotone, and Assumption~\ref{as:expected_residual} hold. Let $\Bar{\om} \coloneqq \min \left\{ \nicefrac{1}{(4L)}, \nicefrac{\mu}{(18\delta)} \right\}$. If for $K \geq 0$ step sizes $\{\gamma_k\}_{k \geq 0}$, $\{\om_k\}_{k \geq 0}$ satisfy $\gamma_k = \om_k$ and
\begin{equation}\label{eq:stepsize_switching_2}
    \om_k \coloneqq 
\begin{cases}
\Bar{\om}, &\text{if $K \leq \frac{2}{\mu \Bar{\om}}$}, \\
\Bar{\om}, &\text{if $K > \frac{2}{\mu \Bar{\om}}$ and $k \leq k_0$,}\\
\frac{2}{\frac{2}{\Bar{\om}} + \frac{\mu}{2}(k - k_0)}, &\text{if $K > \frac{2}{\mu \Bar{\om}}$ and $k > k_0$}
\end{cases}
\end{equation}
where $k_0 = \lceil \nicefrac{K}{2} \rceil$, then the iterates produced by \algname{SPEG} with the step sizes defined above satisfy
\begin{equation}\label{rate for Total number of iteration knowledge}
\begin{split}
R_K^2\leq \frac{64R_0^2}{\Bar{\om} \mu} \exp \left\{-\min \left\{ \frac{\mu}{16L}, \frac{\mu^2}{72\delta} \right\}K \right\} + \frac{ 1728\sigma_*^2}{\mu^2 K},
\end{split}    
\end{equation}
where $R_{K}^2 \coloneqq \E \left[\|x_{K} - x_*\|^2 + \|x_{K} - \hat{x}_{K-1}\|^2 \right]$.
\end{theorem}
In contrast to \eqref{eq:SPEG_convergence_decr_steps_1}, the rate from \eqref{rate for Total number of iteration knowledge} has much better (exponentially decaying) $o\left(\nicefrac{1}{K}\right)$ term. When $\sigma_*^2$ is large and one needs to achieve very good accuracy of the solution, this difference is negligible, since the dominating $\cO(\nicefrac{1}{K})$ term is the same for both bounds (up to numerical factors). However, when $\sigma_*^2$ is small enough, e.g., the model is close to over-parameterized, or it is sufficient to achieve low accuracy of the solution, the dominating term in \eqref{rate for Total number of iteration knowledge} is typically much smaller than the one from \eqref{eq:SPEG_convergence_decr_steps_1}. Finally, it is worth mentioning, that the improvement of $o\left(\nicefrac{1}{K}\right)$ is not achieved for free: unlike the policy from \eqref{eq:stepsize_switching_1}, step size rule \eqref{eq:stepsize_switching_2} relies on the knowledge of the total number of steps $K$, which can be inconvenient for the practical use in some cases.
\subsection{Weak Minty Variational Inequality Problems}
In this subsection we will discuss convergence of Stochastic Past Extragradient method for weak Minty Variational Inequality problem.
To solve the weak Minty variational inequality problem we use different step sizes for \algname{SPEG} iterates (\ref{SPEG_UpdateRule}).  
\begin{theorem}\label{cor:weak_MVI_convergence}
    Let $F$ be $L$-Lipschitz and satisfy Weak Minty condition with parameter $\rho < \nicefrac{1}{(2L)}$. Let Assumption~\ref{as:expected_residual} hold. Assume that $\gamma_k = \gamma$, $\omega_k = \omega$ such that $\max\left\{2\rho, \frac{1}{2L}\right\} < \gamma < \frac{1}{L}$ and $0 < \omega < \min\left\{\gamma - 2\rho, \frac{1}{4L} - \frac{\gamma}{4}\right\}.$
    Then, for all $K \geq 2$ the iterates produced by mini-batched \algname{SPEG} with batch-size 
    \begin{eqnarray}
        \tau \geq \max\Bigg\{1, \frac{32\delta}{(1-L\gamma)L^3\omega}, \frac{48\omega\gamma \delta(K-1)}{(1 - L\gamma)^2},\frac{2\omega\gamma\sigma_*^2(K-1)}{(1-L\gamma)\|x_0 - x_*\|^2}\Bigg\} \label{eq:SPEG_weak_MVI_batchsize}
    \end{eqnarray}
    satisfy $\min\limits_{0\leq k \leq K-1}\E\left[\|F(\hat x_k)\|^2\right] \leq \frac{C\|x_0 - x_*\|^2}{K-1},$
where $C = \frac{48}{\omega\gamma (1 - L(\gamma + 4\omega))}$.
\end{theorem}
The above result establishes $\cO(\nicefrac{1}{K})$ convergence with $\cO(K)$ batchsizes for \algname{SPEG} applied to problems satisfying Weak Minty condition.\footnote{See also Appendix~\ref{AppendixE5} for a discussion related to the oracle complexity of Theorem \ref{cor:weak_MVI_convergence}.} The closest result is obtained by \cite{bohm2022solving}, for the same method under bounded variance assumption, i.e., when $\delta = 0$. In particular, the result of \cite{bohm2022solving} also gives $\cO(\nicefrac{1}{K})$ rate and requires $\cO(K)$ batchsizes at each step. We extend this result to the case of non-zeroth $\delta$ and we also improve the assumption on $\rho$: \cite{bohm2022solving} assumes that $\rho < \nicefrac{3}{8L}$, while Theorem~\ref{cor:weak_MVI_convergence} holds for $\rho < \nicefrac{1}{2L}$. The bound on $\rho$ cannot be improved even in the deterministic case \cite{gorbunov2022convergence}. Moreover, it is worth mentioning that the proof of Theorem~\ref{cor:weak_MVI_convergence} noticeably differs from the one obtained by \cite{bohm2022solving}.

In the case of a deterministic oracle, we recover the best-known result for Optimistic Gradient in the Weak Minty setup \cite{bohm2022solving, gorbunov2022convergence}.

\begin{corollary}
    Let the assumptions of Theorem~\ref{cor:weak_MVI_convergence} hold and deterministic version of \algname{SPEG} is considered, i.e., $\delta = 0$, $\sigma_*^2 = 0$. Then, Theorem~\ref{cor:weak_MVI_convergence} implies that for all $k\geq 0$ the iterates produced by \algname{SPEG} with step sizes $\max\left\{2\rho, \frac{1}{2L}\right\} < \gamma < \frac{1}{L}$ and $0 < \omega < \min\left\{\gamma - 2\rho, \frac{1}{4L} - \frac{\gamma}{4}\right\}$
    satisfy $\min\limits_{0\leq k \leq K-1}\|F(\hat x_k)\|^2 \leq \frac{C\|x_0 - x_*\|^2}{K-1},$
    where $C = \frac{48}{\omega\gamma (1 - L(\gamma + 4\omega))}$.
\end{corollary}


\section{Beyond Uniform Sampling}\label{section: Arbitrary Sampling}
In this section, we illustrate the generality of our analysis by focusing on the non-uniform sampling. In particular, we focus on \emph{single-element sampling} in which only the singleton sets $\{i\}$ for $i=\{1,\ldots, n\}$ have a non-zero probability of being sampled; that is,
$\Prob{|S|=1} = 1$. We have $\Prob{v = e_i/p_i} = p_i$. \cite{gower2019sgd} proved that if $v$ is a single-element sampling, it is also a valid sampling vector ($\E_{\cD}[v_i]  = 1$). With the following proposition, we provide closed-form expressions for the \ref{eq: ER Condition} parameter $\delta$ and $\sigma_*^2 = \E\|g(x_*)\|^2$ for the case of (non-uniform) single-element sampling.

\begin{proposition}\label{Prop_SingleElement}
Let $F_i$ of problem~\eqref{eq: Variational Inequality Definition} be $L_i$-Lipschitz operators.  If, vector $v \in \R^n$ is a single element sampling then $\delta = \frac{2}{n^2} \sum_{i = 1}^n \frac{L_i^2}{p_i}$ and $\sigma_*^2 = \frac{1}{n^2} \sum_{i=1}^n \frac{1}{p_i} \|F_i(x_*)\|^2.$
\end{proposition}
\textbf{Importance Sampling.} In importance sampling we aim to choose the probabilities $p_i$ that optimize the iteration complexity. \cite{gower2019sgd} and \cite{gorbunov2022stochastic} analyze importance sampling for \algname{SGD} and \algname{SEG} respectively. In this chapter, we provide the first convergence guarantees of \algname{SPEG} with importance sampling. In particular, we optimize the expected residual parameter $\delta$ with respect to $p_i$, which in turn affects the iteration complexity. Note that, by using Cauchy-Schwarz inequality~\eqref{eq: Cauchy Schwarz Inequality}, we have 
$\sum_{i = 1}^n \frac{L_i^2}{p_i} \geq \left(\sum_{i = 1}^n L_i \right)^2$,
and this lower bound can be achieved for $p_i^{\delta} = \nicefrac{L_i}{\sum_{j = 1}^n L_j}$. In case of importance sampling, we will use these probabilities $p_i^{\delta}$ which optimizes $\delta$ and define the corresponding $\delta$ as $\delta_{\text{IS}} \coloneqq \frac{2}{n^2}\left( \sum_{i = 1}^n L_i\right)^2$.
For uniform sampling \big(i.e. $p_i = \frac{1}{n}$\big), the value of the parameter is $\delta_{\text{US}} = \frac{2}{n} \sum_{i = 1}^n L_i^2$. Note that, $\delta_{\text{IS}}$ equals $\delta_{\text{US}}$ when all $L_i$ are equal, however $\delta_{\text{IS}}$ can be much smaller than $\delta_{\text{US}}$ when $L_i$ are very different from each other, e.g., when all $L_i$ are relatively small (close to zero) and one $L_i$ is large, $\delta_{\text{IS}}$ is almost $n$ times smaller than $\delta_{\text{US}}$. In this latter scenario (when $\delta_{\text{IS}}$ is much smaller than $\delta_{\text{US}}$), importance sampling could be useful and can significantly improve the performance of \algname{SPEG}. For example, note that the exponentially decaying term in \eqref{rate for Total number of iteration knowledge} decreases with $\delta$. Hence, this term will decrease much faster with importance sampling than with uniform sampling.

\section{Numerical Experiments}
\label{Numerical Experiments}
To verify our theoretical results, we run several experiments on two classes of problems, i.e., strongly monotone problems (a special case of the quasi-strongly monotone problems) and weak Minty problems. The code to reproduce our results can be found at \href{https://github.com/isayantan/Single-Call-Stochastic-Extragradient-Methods}{https://github.com/isayantan/Single-Call-Stochastic-Extragradient-Methods}.
\subsection{Strongly Monotone Problems}\label{sec: experiments on quasi monotone}
Our experiments consider the quadratic strongly-convex strongly-concave min-max problem from \cite{gorbunov2022stochastic}. That is, we implement \algname{SPEG} on quadratic games of the form $\min_{x \in \mathbb{R}^d} \max_{y \in \mathbb{R}^d} \frac{1}{n}\sum_{i = 1}^n \mathcal{L}_i(w_1, w_2)$ where 

\begin{equation}
    \mathcal{L}_i(x,y) \coloneqq \frac{1}{2} w_1^{\intercal}A_i w_1 + w_1^{\intercal} B_i w_2 - \frac{1}{2} w_2^{\intercal} C_i w_2 + a_i^{\intercal}w_1 - c_i^{\intercal}w_2.
\end{equation}
Here $A_i, B_i, C_i$ are generated such that the quadratic game is strongly monotone and smooth. In all our experiments, we take $n = 100$ and $d = 30$. We generate positive semi-definite matrices $A_i, B_i, C_i$ such that their eigenvalues lie in the interval $[\mu_A, L_A], [\mu_B, L_B]$ and $[\mu_C, L_C]$ respectively. In all our experiments, we consider $L_A = L_B = L_C = 1$ and $\mu_A = \mu_C = 0.1, \mu_B = 0$ unless otherwise mentioned. The vectors $a_i$ and $c_i$ are generated from $\mathcal{N}_d (0,I_d)$. 
Here, the $i$th operator is given by

\begin{equation*}\label{operator for quad game}
    \begin{split}
        F_i \begin{pmatrix} w_1\\
        w_2 \end{pmatrix} = \begin{pmatrix} \nabla_{w_1} \mathcal{L}_i(w_1, w_2) \\ - \nabla_{w_2} \mathcal{L}_i(w_1, w_2) \end{pmatrix} =\begin{pmatrix} A_iw_1 + B_i w_2 + a_i \\
        C_i w_2 - B_i^{\intercal}w_1 + c_i
        \end{pmatrix}
    \end{split}
\end{equation*}
In Figures~\ref{fig: Synthetic Dataset 1}, \ref{fig:hsieh vs our steps}, and \ref{fig:us_vs_is}, 
we plot the relative error on the $y$-axis i.e. $\frac{\|x_k - x_*\|^2}{\|x_0 - x_*\|^2}$. 

\begin{figure}
    \centering
    \includegraphics[width=0.5\linewidth]{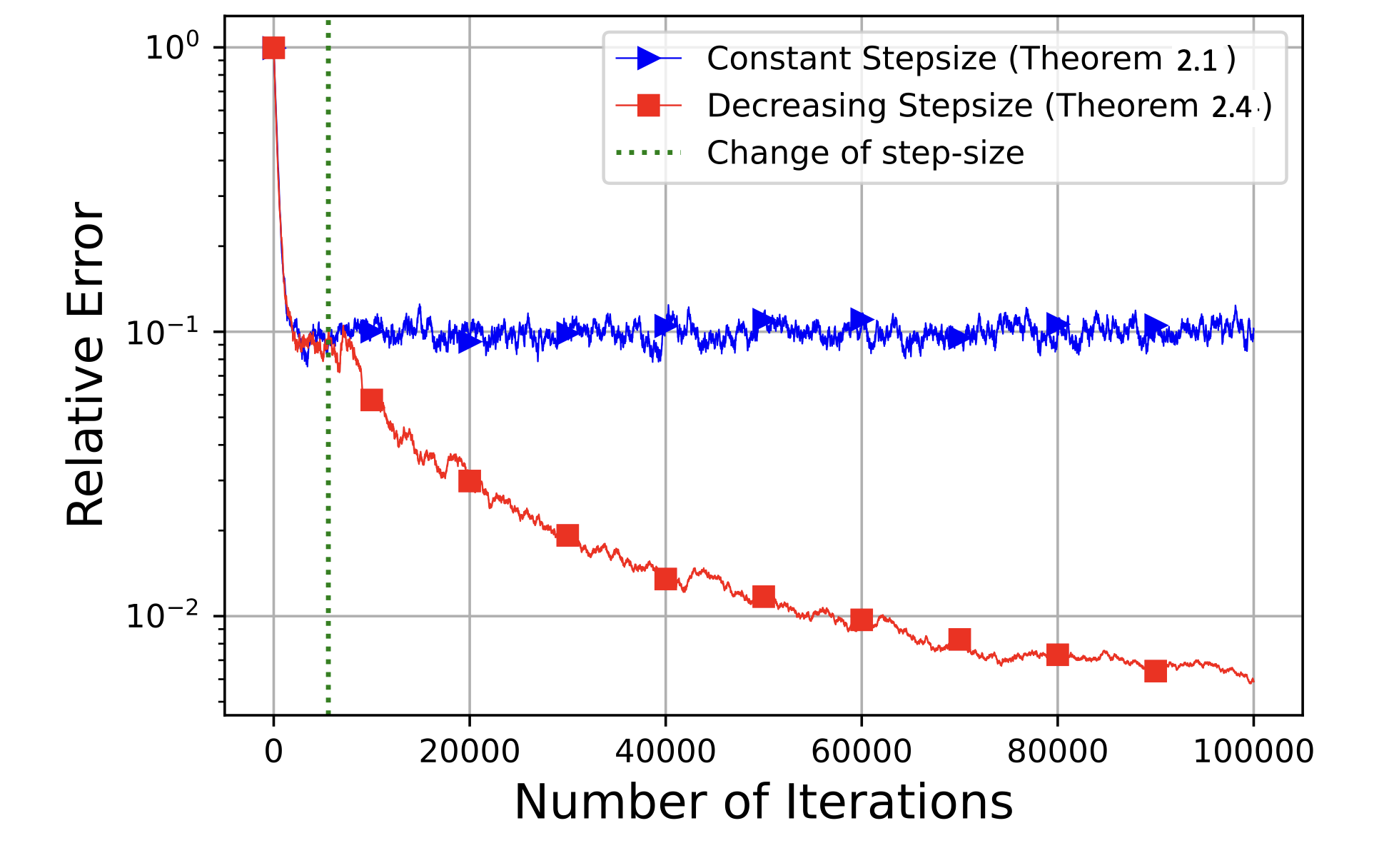}
    \caption{Constant vs Switching}
    \label{fig: Synthetic Dataset 1}
\end{figure}

\paragraph{Constant vs Switching Step-size Rule.}
In Fig.~\ref{fig: Synthetic Dataset 1}, we illustrate the step-size switching rule of Theorem \ref{SPEG switching rule}. We place the dotted line to mark when we switch from constant step-size to decreasing step-size. In Fig.~\ref{fig: Synthetic Dataset 1}, the trajectory of switching step-size rule \eqref{eq:stepsize_switching_1} matches that of constant step-size \eqref{eq:constant_stepsize} in the first phase \big(where \algname{SPEG} runs with constant step-size following \eqref{eq:stepsize_switching_1}\big). However, it becomes stagnant when the constant step-size \algname{SPEG} reaches a neighbourhood of optimality. In contrast, the step-size of Theorem \ref{SPEG switching rule} helps the method to converge to better accuracy.

\paragraph{Comparison to \cite{hsieh2019convergence}.}
In this experiment, we compare \algname{SPEG} step-sizes proposed in Theorems \ref{Theorem: constant stepsize theorem} and \ref{SPEG switching rule} with step-sizes from \cite{hsieh2019convergence}.
To implement \algname{SPEG} with the step-sizes from \cite{hsieh2019convergence}, we choose $\gamma$ and $b$ such that $\frac{1}{\mu} < \gamma \leq \frac{b}{4L}$ and set $\om_k = \gamma_k = \frac{\gamma}{k + b}$. For Fig.~\ref{fig: Interpolated Model}, we generate $A_i, B_i, C_i$ as before. First, we sample optimal points $w_{1*}, w_{2*}$ from $\mathcal{N}_d(0, I_d)$ and then generate $a_i, c_i$ such that $F(w_{1*}, w_{2*}) = 0$.   
\begin{equation*}
    \begin{split}
        \begin{pmatrix} a_i \\c_i \end{pmatrix} = \begin{pmatrix} A_i & B_i \\ -B_i^{\intercal} & C_i \end{pmatrix}^{-1} \begin{pmatrix} w_{1*} \\ w_{2*} \end{pmatrix}.
    \end{split}
\end{equation*}
In Fig.~\ref{fig: Interpolated Model}, we run the algorithms on interpolated model $\big( F_i(x_*) = 0$ for all $i \in [n]\big)$. Since the model is interpolated, we have $\sigma_*^2 = 0$ in Theorem \ref{Theorem: constant stepsize theorem} and linear convergence to the exact optimum asymptotically. In this setting, as shown in Fig.~\ref{fig: Interpolated Model}, our proposed step-size results in major improvement compared to the decreasing step-size selection analyzed in~\cite{hsieh2019convergence}. In Fig.~\ref{fig: Non-Interpolated Model}, we compare the switching step-size rule with step-size from \cite{hsieh2019convergence}. In Fig.~\ref{fig: Non-Interpolated Model}, we generate $a_i, c_i$ from the normal distribution. In this plot, we manually switch the step-size from constant to decreasing after $305$ steps. We observe that such a semi-empirical rule has comparable performance to the step-size selection of \cite{hsieh2019convergence}. 

\begin{figure}[t]
\centering
\begin{subfigure}[b]{.4\textwidth}
    \centering
    \includegraphics[width=\textwidth]{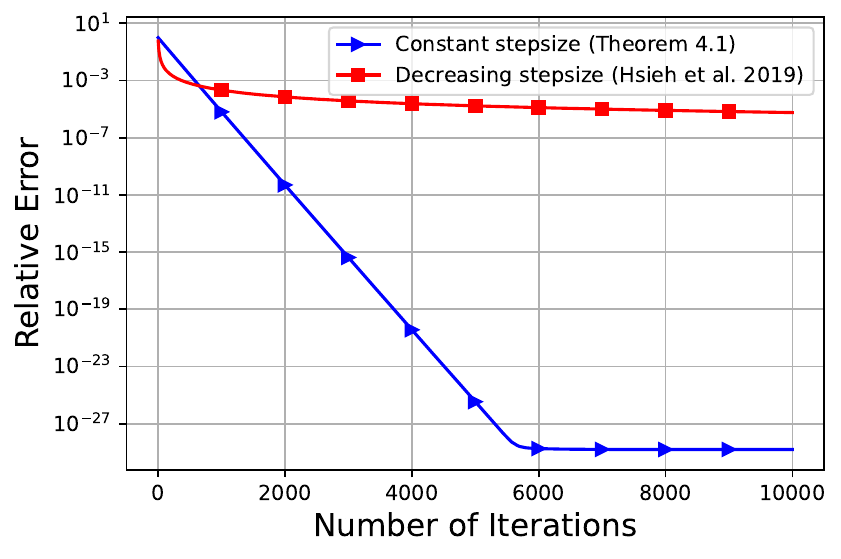}
    \caption{Interpolated Model}\label{fig: Interpolated Model}
\end{subfigure}
\begin{subfigure}[b]{0.4\textwidth}
    \centering
    \includegraphics[width=\textwidth]{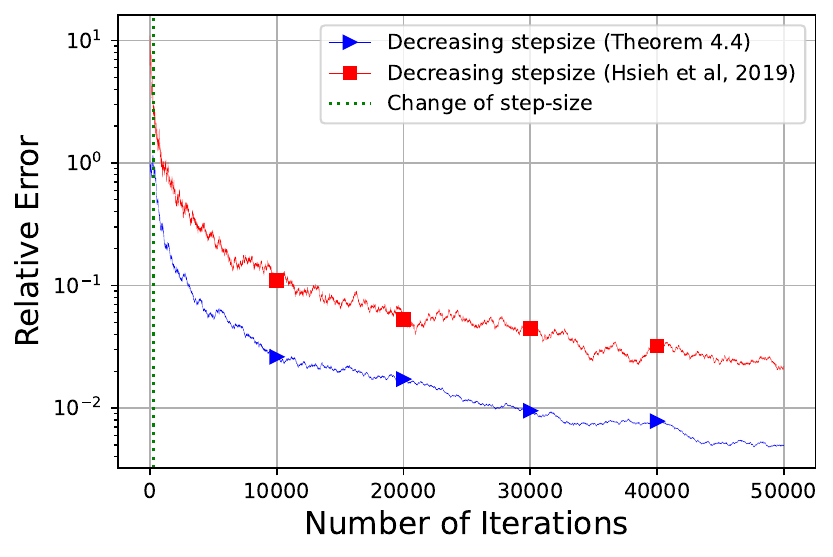}
    \caption{Non-Interpolated Model}\label{fig: Non-Interpolated Model}
\end{subfigure}
\caption{\emph{Comparison of our \algname{SPEG} using our step-size against decreasing step-size of \cite{hsieh2019convergence}. In plot (a), for constant step-size of \algname{SPEG} we use the upper bound of \eqref{eq:constant_stepsize}. In plot (b), we run our switching step-size \algname{SPEG}~\eqref{eq:stepsize_switching_1}.}}
\label{fig:hsieh vs our steps}
\end{figure}

\begin{figure}[t]
\centering
\begin{subfigure}[b]{0.4\linewidth}
    \centering
    \includegraphics[width=\textwidth]{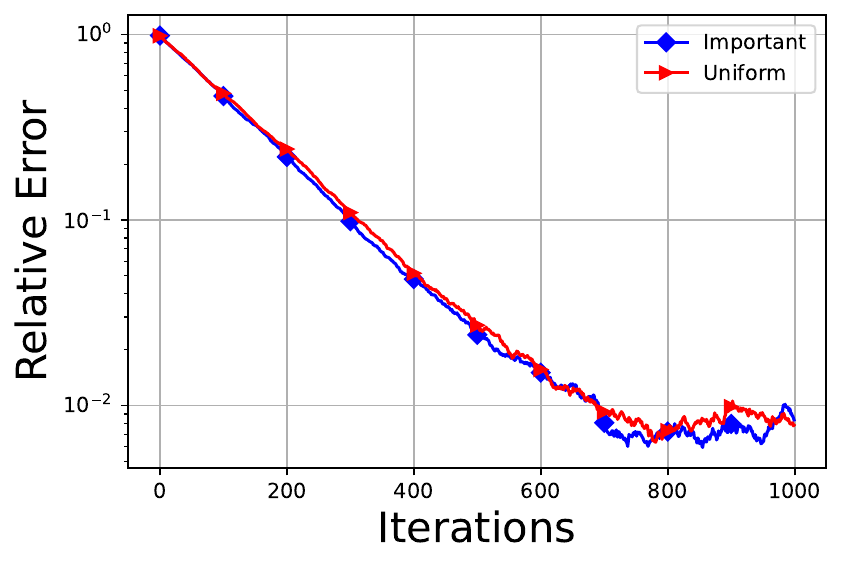}
    \caption{$\Lambda = 2 $}
\end{subfigure}
 \begin{subfigure}[b]{0.4\linewidth}
      \centering
      \includegraphics[width=\textwidth]{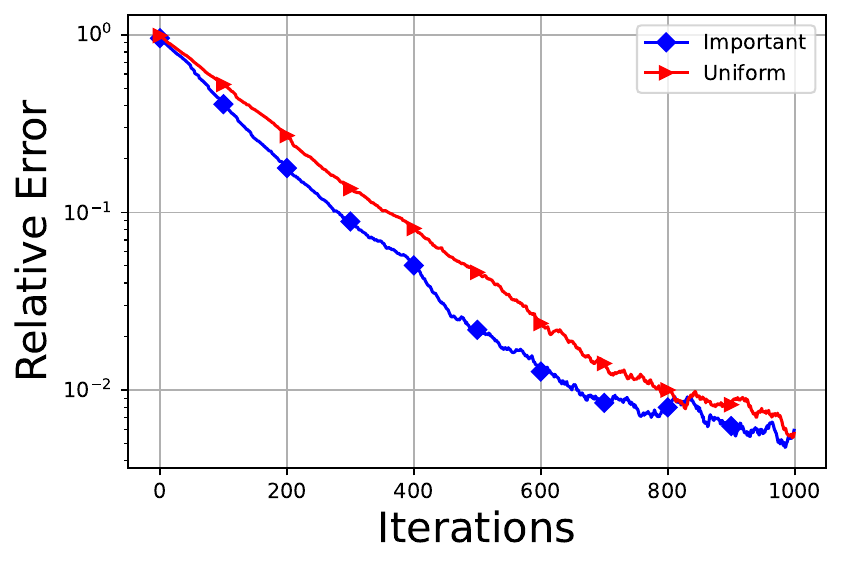}
     \caption{$\Lambda = 5 $}
  \end{subfigure}
\begin{subfigure}[b]{0.4\textwidth}
    \centering
    \includegraphics[width=\textwidth]{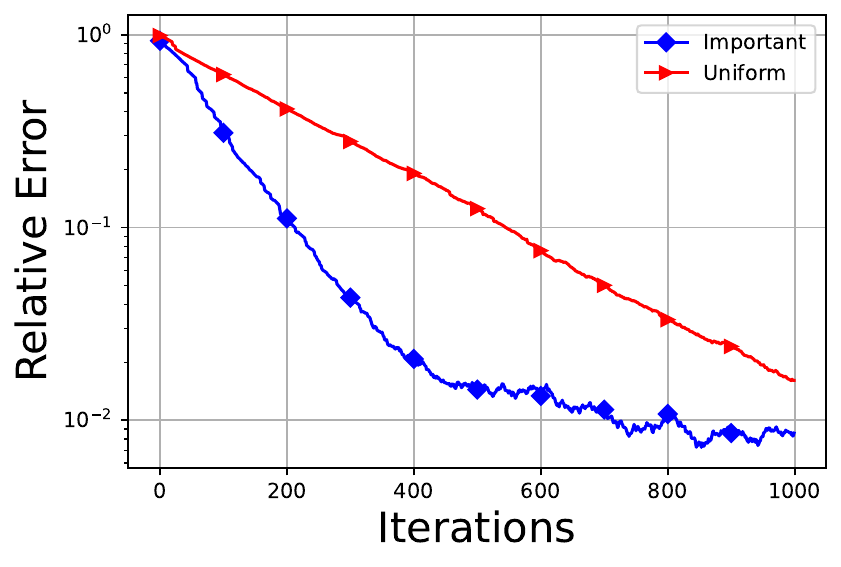}
    \caption{$\Lambda = 10$}
\end{subfigure}
\begin{subfigure}[b]{0.4\textwidth}
    \centering
    \includegraphics[width=\textwidth]{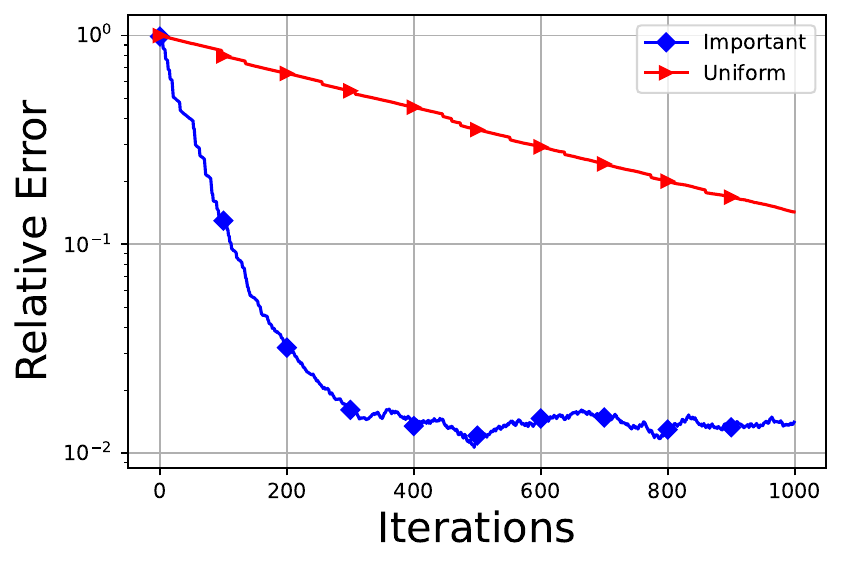}
    \caption{$\Lambda = 20$}
\end{subfigure}  
\caption{\emph{Comparison of \algname{SPEG} with Uniform and Importance Sampling for different $\Lambda \in\{2,5,10,20\}$, where the eigenvalues of matrices $A_1, C_1$ are uniformly generated from the interval $[0.1, \Lambda].$}}
\label{fig:us_vs_is}
\end{figure}

\paragraph{Uniform vs. Importance Sampling.}  In this experiment, we highlight the advantage of using importance sampling over uniform sampling. The eigenvalues of $A_1, C_1$ are uniformly generated from the interval $[0.1, \Lambda]$ while the rest of the matrices are generated as mentioned before. We vary the value of $\Lambda \in \{2,5,10,20\}$ and run and compare \algname{SPEG} with both uniform and importance sampling (see~Fig.~\ref{fig:us_vs_is}).  For importance sampling, we use the probabilities $p_i = \nicefrac{L_i}{\sum_{j = 1}^n L_j}$. In Fig.~\ref{fig:us_vs_is}, it is clear that as the value of $\Lambda$ increases, the trajectories under uniform sampling get worse, while the trajectory under importance sampling remains almost identical. This behaviour aligns well with our discussion in Section~\ref{section: Arbitrary Sampling}.

\subsection{Weak Minty Problems}\label{sec: Experiment on WMVI}

\begin{figure}
    \centering
    \includegraphics[width=0.5\linewidth]{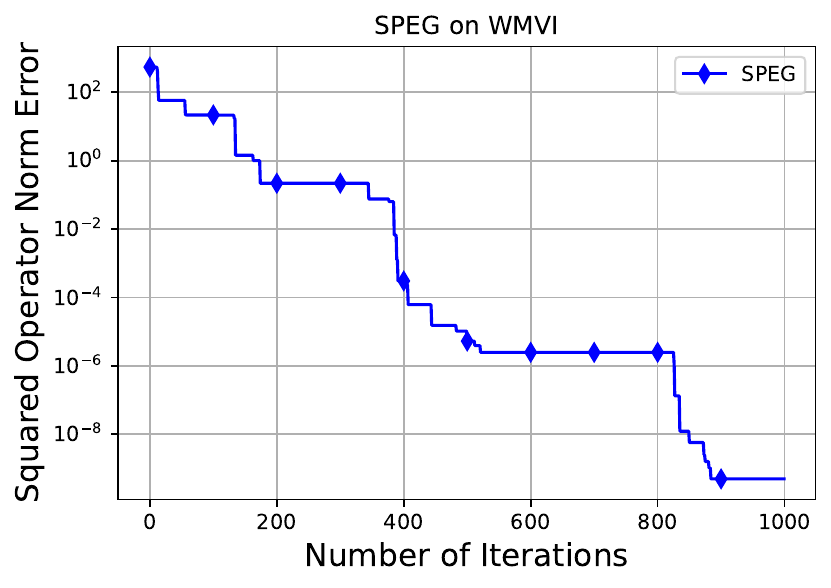}
    \caption{Trajectory of \algname{SPEG} for solving weak Minty problems. "Squared Operator Norm Error'' in vertical axis denotes the $\min\limits_{0\leq k \leq K-1}\E\left[\|F(\hat x_k)\|^2\right]$ of Theorem~\ref{cor:weak_MVI_convergence}.}
    \label{fig:performance of SPEG on WMVI_a}
\end{figure}

This experiment verifies the convergence guarantees of \algname{SPEG} in Theorem \ref{cor:weak_MVI_convergence}. Following the min-max problem mentioned in \cite{bohm2022solving}, we consider the objective function 
\begin{equation}\label{asoxasl}
\begin{split}
       \min_{w_1 \in \mathbb{R}} \max_{w_2 \in \mathbb{R}} \frac{1}{n} \sum_{i = 1}^n \xi_{i} w_1 w_2 + \frac{\zeta_i}{2}(w_1^2 - w_2^2). 
\end{split}
\end{equation}
In this experiment, we generate $\xi_i, \zeta_i$ such that $L = 8$ and $\rho = \nicefrac{1}{32}$ for the above min-max problem \cite{bohm2022solving}. We implement \algname{SPEG} with extrapolation step $\gamma_k = 0.08$ and update step $\omega_k = 0.01$ which satisfies the conditions on step-size in Theorem \ref{cor:weak_MVI_convergence}. In Fig.~\ref{fig:performance of SPEG on WMVI_a}, we use a batchsize of $6$. This plot illustrates that for some weak MVI problems the requirement on the step-size from Theorem~\ref{cor:weak_MVI_convergence} can be too pessimistic and \algname{SPEG} with relatively small batchsize achieves reasonable accuracy of the solution. The choice of batchsize ensures that bound \eqref{eq:SPEG_weak_MVI_batchsize} holds and $\delta$ is small enough to guarantee convergence of \algname{SPEG}. We also tried to compare \algname{SPEG} with \algname{SEG+} from \cite{pethick2022escaping}, however, the authors do not mention their choice of update step-size. We examined several decreasing update step-size for which \algname{SEG+} failed to converge. Further details on experiments can be found in Appendix \ref{subsec:FurtherDetails}.

\chapter{Polyak-type (Stochastic) Extragradient} \label{chap:chap-3}

\section{Introduction}

Among the various algorithms designed for solving the root-fidning problems of the form~\eqref{eq: Variational Inequality Definition} and the unconstrained min-max problems of the form~\eqref{eq:unconstrained_MinMax}, the Extragradient method (\hyperref[eq:EG]{\algname{EG}})~\cite{korpelevich1977extragradient,gorbunov2022extragradient} and its stochastic counterpart \algname{SEG}~\cite{juditsky2011solving, mishchenko2020revisiting,gorbunov2022stochastic} have been particularly successful due to their superior convergence properties and simplicity of analysis. 
The stochastic variant, \algname{SEG}, replaces $F(x)$ with the unbiased estimator $F_{\mathcal{S}_k}(x) \eqdef \frac{1}{B} \sum_{i \in \mathcal{S}_k} F_i(x)$, where $\mathcal{S}_k \subseteq [n]$ is a random subset of datapoints (minibatch) of size $B$ sampled independently at each iteration $k$. In the literature, \hyperref[eq:EG]{\algname{EG}} and \algname{SEG} have been often analyzed with step sizes $\omega_k, \gamma_k$ that depend on the \emph{Lipschitz} parameter $L$ of the operator $F$~\cite{gorbunov2022extragradient, diakonikolas2021efficient, juditsky2011solving, gidel2018variational, mishchenko2020revisiting, gorbunov2022stochastic}. 

However, since in most machine learning applications the parameter $L$ is rarely known a priori, the practicality of such theoretical analyses is limited, motivating the community to design adaptive extragradient methods, that is, methods that adjust the step sizes on the fly without, requiring the knowledge of problem parameters.

However, existing adaptive \algname{EG} variants are typically limited to deterministic settings~\cite{antonakopoulos2020adaptive, pethick2023escaping}. 
In the stochastic setting, a major limitation of some adaptive \algname{SEG} methods is their impracticality due to the requirement of increasing batch sizes~\cite{wang2022self}. Others rely on AdaGrad-type step sizes~\cite{antonakopoulos2021sifting}, which, based on insights from the minimization literature~\cite{zeiler2012adadelta}, are known to be sensitive to hyperparameter tuning. \emph{In this chapter, we aim to design efficient adaptive variants of \algname{EG} and \algname{SEG} so that the proposed step size selection is precisely what one should run in practice without any further expensive tuning.} 

To achieve that, we take inspiration from the popular Polyak step size \cite{polyak1987introduction,hazan2019revisiting}, originally developed for gradient descent and subgradient methods, and more recently used extensively for stochastic optimization and training of deep neural networks ~\cite{loizou2021stochastic, orvieto2022dynamics, dstochastic, gower2021stochastic, oikonomou2025stochastic}. It stands out as a particularly interesting adaptive strategy in optimization, as Polyak-type step sizes have simple update rules and solid theoretical and empirical performance without requiring sensitive hyperparameter tuning.


\begin{table*}[h]
\vspace{-3mm}
        \centering
        \caption{\small
        Correspondence between Polyak step sizes in minimization and min-max optimization.
        }
        \label{tab:polyak_step_stochastic}
        \begin{threeparttable}
        \resizebox{1\textwidth}{!}{%
            \begin{tabular}{|c|c|c|}
            \hline
            Setup & \begin{tabular}{c}
                \large Minimization\\
                $\min_x f(x) \coloneqq \frac{1}{n} \sum_{i = 1}^n f_i(x)$ \\
            \end{tabular} & \begin{tabular}{c}
                 \large Min-Max Optimization / VIPs\\
                 $\text{find } x_* \text{ such that } F(x_*) \coloneqq \frac{1}{n} \sum_{i = 1}^n F_i(x_*) = 0$ \\
            \end{tabular} 
            \\
            \hline\hline
            Deterministic & \begin{tabular}{c}
                {\large $\nicefrac{f(x_k) - \min_x f(x)}{ \left\| \nabla f\left(x_k \right) \right\|^2}$}~\cite{polyak1987introduction}
            \end{tabular} & \cellcolor{bgcolor2} \begin{tabular}{c}
                 \large{$ \nicefrac{\la F(\hx_k), x_k - \hx_k \ra}{\left\|F(\hx_k) \right\|^2}$}
            \end{tabular}  \\[10pt]
            \hline
            Stochastic & \begin{tabular}{c}
                {\large $ \frac{f_{\mathcal{S}_k}(x_k) - \min_x f_{\mathcal{S}_k}(x)}{ \left\| \nabla f_{\mathcal{S}_k} \left(x_k \right) \right\|^2}$}~\cite{loizou2021stochastic}
            \end{tabular} & \cellcolor{bgcolor2} \begin{tabular}{c}
                 {\large $\nicefrac{\la F_{\mathcal{S}_k}(\hx_k), x_k - \hx_k \ra}{\left\|F_{\mathcal{S}_k}(\hx_k) \right\|^2}$}
            \end{tabular}  \\[10pt]
            \hline
             Stochastic & \begin{tabular}{c}
                {\large $\nicefrac{1}{c_{k+1}} \min \left\{ \nicefrac{f_{\mathcal{S}_k}(x_k) - \min_x f_{\mathcal{S}_k}(x)}{c \left\| \nabla f_{\mathcal{S}_k}(x_k) \right\|^2}, c_k \eta_{k-1}\right\}$} \cite{orvieto2022dynamics}
            \end{tabular} &  \cellcolor{bgcolor2} {\large $\min \left\{ \nicefrac{\la F_{\mathcal{S}_k}(\hx_k), x_k - \hx_k \ra}{ \left\|F_{\mathcal{S}_k}(\hx_k) \right\|^2}, \omega_{k-1}\right\}$} \\[10pt]
            \hline
        \end{tabular}%
        }
        \end{threeparttable}
\end{table*}

In this chapter, we establish a connection between \algname{EG} and Polyak step size from the minimization literature, and leverage this insight to design adaptive variants of \algname{SEG}.
In Section~\ref{sec:motivation}, we illustrate how Polyak-like reasoning in the min-max setting for \algname{EG} leads to the update step size identical to that proposed in \cite{solodov1996modified}. Motivated by the developments in the stochastic minimization setting, we then propose Polyak step sizes for \algname{SEG}. Table~\ref{tab:polyak_step_stochastic} summarizes the classical Polyak step sizes $\eta_k$ used in gradient descent and stochastic gradient descent for minimization problems, alongside our proposed counterparts $\omega_k$ for \algname{EG} and \algname{SEG} in the min-max optimization setting. In contrast to the step sizes in minimization, our min-max Polyak updates do not rely on quantities such as functional optimality (e.g., $\min_x f(x)$ or $\min_x f_{\mathcal{S}_k}(x)$). 

Furthermore, we integrate the Polyak update step size with a line-search scheme for the extrapolation step, resulting in \algname{SEG} methods which do not require tuning. A comparative overview of our proposed algorithms against existing \algname{EG} and \algname{SEG} variants is provided in Table~\ref{tab:comparison_of_rates_3}. While many existing \algname{EG} or \algname{SEG} methods require knowledge of the Lipschitz constant~$L$, others rely on either AdaGrad-type step sizes~\cite{antonakopoulos2020adaptive, antonakopoulos2021sifting} or line-search techniques~\cite{pethick2023escaping, solodov1996modified, iusem2019variance}. However, AdaGrad-type step sizes often suffer from hyperparameter sensitivity~\cite{zeiler2012adadelta} and perform poorly without tuning. Among the line-search-based approaches, the method of \cite{pethick2023escaping} utilizes second-order information, and the algorithm of ~\cite{iusem2019variance} requires an increasing batch size, i.e., the size of $\mathcal{S}_k$ increases with $k$. In contrast, our line-search scheme avoids these limitations. Additionally, unlike many existing line-search strategies, our method is guaranteed to terminate, and we provide an explicit upper bound on the number of oracle calls required by the line-search procedure in Theorem~\ref{lemma:LS-step-size-lower-bound}. We summarize the main contributions below.


\begin{table}[H]
    \centering
    \caption{\small Variants of the extragradient algorithm for solving min-max optimization problems. Columns: "Adaptive?" = Does the algorithm require the knowledge of $L$?; "Extra comments for adaptive algorithms" = we provide additional comments about the adaptive algorithms under this column. 
    }
    \label{tab:comparison_of_rates_3}
    \begin{threeparttable}
    \resizebox{\textwidth}{!}{%
        \begin{tabular}{|c|c|c c c c c|}
        \hline
        Setup & Algorithms & Problem &\begin{tabular}{c} Strategy for $\gamma_k$ \end{tabular} & \begin{tabular}{c} Strategy for $\omega_k$ \end{tabular} & Adaptive? & Comments  
        \\
        \hline\hline
        \multirow{12}{2cm}{\centering Determini-\\stic} & \begin{tabular}{c}
            \algname{EG}\\
            \cite{mokhtari2020unified}
        \end{tabular} & \begin{tabular}{c}
             \eqref{eq:strong_monotone}
        \end{tabular} & $\frac{1}{4L}$ & $\frac{1}{4L}$ & \xmark & - \\[10pt]
        
        & \begin{tabular}{c}
            \algname{EG} \\
            \cite{gorbunov2022extragradient} 
        \end{tabular} & \eqref{eq:monotone} & $\frac{1}{L}$ & $\frac{1}{2L}$ & \xmark & - \\[10pt]

        & \begin{tabular}{c}
            \algname{EG} \\
            \cite{antonakopoulos2020adaptive} 
        \end{tabular} & \begin{tabular}{c}
             \eqref{eq:monotone}
        \end{tabular} & $\gamma_k = \omega_k$ & {\large $\frac{1}{\sqrt{1 + \sum_{i = 1}^{k-1} \| F(\hx_k) - F(x_k)\|^2}}$} & \cmark & \tnote{{\color{blue}(4)}}\\[10pt]

        & \begin{tabular}{c}
            \algname{CurvatureEG+} \\
            \cite{pethick2023escaping} 
        \end{tabular} & \eqref{eq:monotone} & line-search & {\large $\frac{\la F(\hx_k), x_k - \hx_k \ra}{\| F(\hx_k)\|^2}$} & \cmark & \tnote{{\color{blue}(1)}} \\[10pt]

        & \begin{tabular}{c}
            \algname{PolyakEG} \\
            \cite{solodov1996modified} 
        \end{tabular} & \begin{tabular}{c}
             \eqref{eq:strong_monotone}
        \end{tabular} & line-search & {\large $\frac{\la F(\hx_k), x_k - \hx_k \ra}{\| F(\hx_k)\|^2}$} & \cmark & \tnote{{\color{blue}(2)}} \\[10pt]

        &\cellcolor{bgcolor2}\begin{tabular}{c}
            \algname{PolyakEG} 
        \end{tabular} & \cellcolor{bgcolor2} \begin{tabular}{c}
             \eqref{eq:strong_monotone}
        \end{tabular} & \cellcolor{bgcolor2} \begin{tabular}{c}
             $\frac{1}{3L}  \text{ for \eqref{eq:strong_monotone} }$
        \end{tabular} & \cellcolor{bgcolor2} {\large $\frac{\la F(\hx_k), x_k - \hx_k \ra}{\| F(\hx_k)\|^2}$} & \cellcolor{bgcolor2} \xmark & \cellcolor{bgcolor2} - \\[10pt]

        &\cellcolor{bgcolor2}\begin{tabular}{c}
            \algname{PolyakEG} 
        \end{tabular} & \cellcolor{bgcolor2} \begin{tabular}{c}
             \eqref{eq:monotone}
        \end{tabular} & \cellcolor{bgcolor2} \begin{tabular}{c}
             $\frac{1}{L} \text{ for \eqref{eq:monotone} }$
        \end{tabular} & \cellcolor{bgcolor2} {\large $\frac{\la F(\hx_k), x_k - \hx_k \ra}{\| F(\hx_k)\|^2}$} & \cellcolor{bgcolor2} \xmark & \cellcolor{bgcolor2} - \\[10pt]

        &\cellcolor{bgcolor2}\begin{tabular}{c}
            \algname{PolyakEG} 
        \end{tabular} & \cellcolor{bgcolor2} \begin{tabular}{c}
             \eqref{eq:monotone}, \eqref{eq:strong_monotone}
        \end{tabular} & \cellcolor{bgcolor2} \begin{tabular}{c}
             line-search  \\
             (Algorithm \ref{alg:PolyakEG_linesearch})
        \end{tabular} & \cellcolor{bgcolor2} {\large $\frac{\la F(\hx_k), x_k - \hx_k \ra}{\| F(\hx_k)\|^2}$} & \cellcolor{bgcolor2} \cmark & \cellcolor{bgcolor2}\tnote{{\color{blue}(3)}} \\[10pt]

        \hline\hline
        
        
        \multirow{18}{2cm}{\centering Stochastic} & \begin{tabular}{c}
            \algname{SEG}\\
            \cite{gorbunov2022stochastic}
        \end{tabular} & \eqref{eq:monotone} & \begin{tabular}{c}
             $\frac{1}{4\mu + \sqrt{2}L}$
        \end{tabular}   & $\frac{1}{16\mu + 4\sqrt{2}L}$ & \xmark & - \\[10pt]

        & \begin{tabular}{c}
            \algname{SEG}\\
            \cite{mishchenko2020revisiting}
        \end{tabular} & \eqref{eq:monotone} & $\mathcal{O}\left( \frac{1}{L\sqrt{k}}\right)$ & $\mathcal{O}\left( \frac{1}{L\sqrt{k}}\right)$ & \xmark & - \\[10pt]

        & \begin{tabular}{c}
            \algname{GEG}\\
            \cite{antonakopoulos2021sifting}
        \end{tabular} & \eqref{eq:monotone} & $\gamma_k = \omega_k$ & {\large $\frac{1}{\sqrt{1 + \sum_{i = 1}^{k-1} \| F_{\mathcal{S}_k}(\hx_k) - F_{\mathcal{S}_k}(x_k)\|^2}}$} & \cmark & \tnote{{\color{blue}(4)}} \\[10pt]


         &\cellcolor{bgcolor2}\begin{tabular}{c}
            \algname{PolyakSEG} 
        \end{tabular} & \cellcolor{bgcolor2} \begin{tabular}{c}
             \eqref{eq:strong_monotone}
        \end{tabular} & \cellcolor{bgcolor2} \begin{tabular}{c}
             $\frac{1}{3L}  \text{ for \eqref{eq:strong_monotone} } $
        \end{tabular} & \cellcolor{bgcolor2} {\large $\frac{\la F_{\mathcal{S}_k}(\hx_k), x_k - \hx_k \ra}{\| F_{\mathcal{S}_k}(\hx_k)\|^2}$} & \cellcolor{bgcolor2} \xmark &\cellcolor{bgcolor2} - \\[10pt]

        &\cellcolor{bgcolor2}\begin{tabular}{c}
            \algname{PolyakSEG} 
        \end{tabular} & \cellcolor{bgcolor2} \begin{tabular}{c}
             \eqref{eq:monotone}
        \end{tabular} & \cellcolor{bgcolor2} \begin{tabular}{c}
             $\frac{1}{L} \text{ for \eqref{eq:monotone} }$
        \end{tabular} & \cellcolor{bgcolor2} {\large $\frac{\la F_{\mathcal{S}_k}(\hx_k), x_k - \hx_k \ra}{\| F_{\mathcal{S}_k}(\hx_k)\|^2}$} & \cellcolor{bgcolor2} \xmark &\cellcolor{bgcolor2} - \\[10pt]

        &\cellcolor{bgcolor2}\begin{tabular}{c}
            \algname{PolyakSEG} 
        \end{tabular} & \cellcolor{bgcolor2} \begin{tabular}{c}
             \eqref{eq:monotone}, \eqref{eq:strong_monotone}
        \end{tabular} & \cellcolor{bgcolor2} \begin{tabular}{c}
             line-search  \\
             (Algorithm \ref{alg:PolyakSEG_linesearch})
        \end{tabular} & \cellcolor{bgcolor2} {\large $\frac{\la F_{\mathcal{S}_k}(\hx_k), x_k - \hx_k \ra}{\| F_{\mathcal{S}_k}(\hx_k)\|^2}$} & \cellcolor{bgcolor2} \cmark & \cellcolor{bgcolor2} \tnote{{\color{blue}(3)}} \\[10pt]

         &\cellcolor{bgcolor2}\begin{tabular}{c}
            \algname{DecPolyakSEG} 
        \end{tabular} & \cellcolor{bgcolor2} \begin{tabular}{c}
             \eqref{eq:monotone}
        \end{tabular} & \cellcolor{bgcolor2} \begin{tabular}{c}
             $\frac{2}{L\sqrt{k+2}} $
        \end{tabular} & \cellcolor{bgcolor2} $ \min \left\{ \frac{\la F_{\mathcal{S}_k}(\hx_k), x_k - \hx_k \ra}{\| F_{\mathcal{S}_k}(\hx_k)\|^2}, \omega_{k-1} \right\}$ & \cellcolor{bgcolor2} \xmark & \cellcolor{bgcolor2} - \\[10pt]

        &\cellcolor{bgcolor2}\begin{tabular}{c}
            \algname{DecPolyakSEG} 
        \end{tabular} & \cellcolor{bgcolor2} \begin{tabular}{c}
             \eqref{eq:monotone}, \eqref{eq:strong_monotone}
        \end{tabular} & \cellcolor{bgcolor2} \begin{tabular}{c}
             line-search  \\
             (Algorithm \ref{alg:decPolyakSEG_linesearch})
        \end{tabular} & \cellcolor{bgcolor2} $ \min \left\{ \frac{\la F_{\mathcal{S}_k}(\hx_k), x_k - \hx_k \ra}{\| F_{\mathcal{S}_k}(\hx_k)\|^2}, \omega_{k-1} \right\}$ & \cellcolor{bgcolor2} \cmark & \cellcolor{bgcolor2} 
        \tnote{{\color{blue}(3)}} \\[10pt]

        \hline
    \end{tabular}%
    }
    \begin{tablenotes}
        {\scriptsize 
        \item \tnote{{\color{blue}(1)}} This algorithm needs second-order information for line-search, while our algorithms only use first-order \\
        information.
        \item \tnote{{\color{blue}(2)}} No explicit bound on the number of oracle calls required for the line-search to terminate.
        
        \item \tnote{{\color{blue}(3)}} We have explicit bound on the number of oracle calls required for the line-search to terminate.
        \item \tnote{{\color{blue}(4)}} AdaGrad type step size.
        }
        \vspace{-5mm}
    \end{tablenotes}
    \end{threeparttable}
\end{table}

\subsection{Main Contributions}
\begin{itemize}[leftmargin=*]
   \setlength{\itemsep}{0pt}
    \item \textbf{Polyak-Inspired \hyperref[eq:EG]{\algname{EG}} Updates.} In Section \ref{sec:motivation}, we explain how the principles from the Polyak step size literature can be applied to derive the adaptive update steps of \hyperref[eq:EG]{\algname{EG}}. Drawing an analogy to the Polyak step size for Gradient Descent, we derive the update step size for the \hyperref[eq:EG]{\algname{EG}} method by optimizing the upper bound on $\|x_{k+1} - x_*\|^2$. Based on this connection, we design a new algorithm for stochastic min-max optimization, importing the proof techniques from stochastic Polyak step size literature for convex minimization.
    
    To the best of our knowledge, we are the first to integrate these two distinct areas of research: stochastic Polyak step sizes and \hyperref[eq:EG]{\algname{EG}} methods. Connecting these concepts lays the groundwork for further advancements in adaptive stochastic \hyperref[eq:EG]{\algname{EG}} methods.

    \item \textbf{Analysis for Deterministic Setting.} We present a refined, unified study of \hyperref[alg:PolyakEG]{\algname{PolyakEG}} using extrapolation step size $\gamma_k$ satisfying a critical condition~\eqref{eqn:gamma-critical-condition}. This condition enables the treatment of both constant and line-search step size strategies as special cases. For each strategy, we establish sublinear (\textit{resp}.\ linear) convergence rates under Lipschitz and monotone (\textit{resp}.\ strongly monotone) settings. 
    
\begin{itemize}[leftmargin=*]
   \setlength{\itemsep}{0pt}
    \item  \textbf{Constant $\gamma_k$:} We show that our proof technique allows us to use $\gamma_k = \nicefrac{1}{3L}$ and $\omega_k \geq \nicefrac{1}{4L}$ to guarantee linear convergence. To the best of our knowledge, this is the largest step size known to ensure linear convergence for \hyperref[eq:EG]{\algname{EG}} with an explicit quantitative bound on convergence rate. 
    
   \item  \textbf{Line-Search Based $\gamma_k$:} We derive a rigorous quantitative bound on the number of iterations required for the line search to terminate (Theorem ~\ref{lemma:LS-step-size-lower-bound}). Furthermore, we highlight that our proof technique ensures linear convergence without requiring prior knowledge of the strong monotonicity parameter $\mu$.
\end{itemize}

    \item \textbf{Analysis for Stochastic Setting.}
    Building upon the analysis of deterministic \algname{PolyakEG}, we extend it to the stochastic setting.
    First, we analyze its stochastic extension \algname{PolyakSEG} and show that it converges provided that an interpolating solution exists (see Section~\ref{subsec:PolyakSEG}). 
    Then, to remove the interpolation assumption, we propose stochastic \algname{EG} with \textit{Decreasing} Polyak step sizes, the algorithm we call \algname{DecPolyakSEG}.

    We prove that \hyperref[alg:decPolyakSEG]{\algname{DecPolyakSEG}} converges at $\mathcal{O}(\nicefrac{1}{\sqrt{K}})$ rate when each stochastic operator is Lipschitz and monotone (Theorem \ref{theorem:DecPolyakSEGLS_monotone}). 
    This result assumes that the algorithm's iterates remain bounded throughout the runtime, a common assumption in the literature of stochastic adaptive algorithms. 
    Nevertheless, we show that the bounded iterates assumption can be removed when stochastic operators are additionally strongly monotone (Proposition \ref{proposition:DecPolyakSEGLS_strong_monotone}).
    \item \textbf{Numerical Experiments.} In Section \ref{sec:numerical_experiment}, we conduct multiple numerical experiments to verify our theoretical findings and demonstrate the superior performance of \hyperref[alg:decPolyakSEG]{\algname{DecPolyakSEG}} for solving a variety of problems.
\end{itemize}

\section{Polyak Update Step Size for \algname{EG}: Algorithmic Design}\label{sec:motivation}
\vspace{-3mm}
We explore the motivation behind our algorithm's design. To gain a comprehensive understanding, we revisit the concept of Polyak steps in the context of minimization problems and explain how the same ideas can be extended in the \algname{EG} update rule for solving min-max problems. \\ 
\newline
\textbf{Polyak Step Size for Gradient Descent.} 
One of the most widely recognized methods for solving smooth convex minimization problems of the form $\min_{x \in \R^d} f(x)$ is Gradient Descent (\algname{GD}). The \algname{GD} algorithm updates the iterates according to $x_{k+1} = x_k - \eta_k \nabla f(x_k)$ where $\eta_k$ represents the step size. Typically, the analysis of \algname{GD} assumes a constant step size $\eta_k = \nicefrac{1}{L}$, where $L$ is the Lipschitz constant of the gradient $\nabla f$. However, \cite{polyak1987introduction} introduced the notion of the Polyak step size, which is an adaptive strategy for choosing the $\eta_k$ (\cite{polyak1987introduction} originally introduced Polyak steps for the subgradient method). Instead of $\nicefrac{1}{L}$, \cite{polyak1987introduction} proposed to use the step size $\eta_k$ minimizing the following upper bound on $\|x_{k+1}-x_*\|^2$: $\|x_{k+1} - x_*\|^2 \leq \|x_k - x_*\|^2 - 2 \eta_k \left( f(x_k) - f(x_*) \right) + \eta_k^2 \| \nabla f(x_k)\|^2,$ where the inequality follows from the convexity of $f$. The upper bound on the right-hand side is minimized for step size $\eta_k = \frac{f(x_k) - f(x_*)}{\| \nabla f(x_k)\|^2}.$

This choice of $\eta_k$ is known as the Polyak step size for \algname{GD}. However, implementing this algorithm is impractical since \textit{it requires prior knowledge of the optimal function value} $f(x_*)$, which is typically unknown. 

\begin{figure}
    \centering
    \includegraphics[width=.6\textwidth]{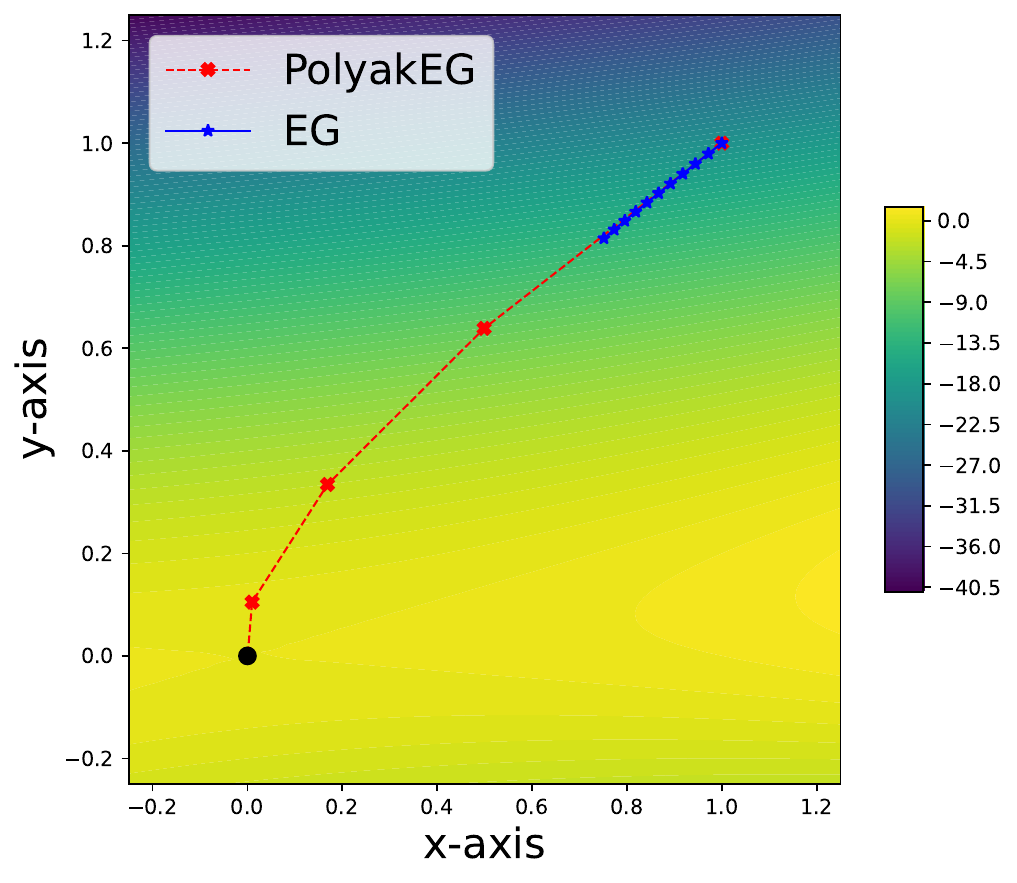} 
    \caption{Comparison of \hyperref[eq:EG]{\algname{EG}} and \hyperref[alg:PolyakEG]{\algname{PolyakEG}} on a two-dimensional \textit{min-max} problem with $\gamma_k = \frac{1}{L}$. Here we plot the trajectories for $10$ iterations. We mark the equilibrium point $(0, 0)$ with $(\bullet)$.}\label{fig:2Dcontour}
\end{figure}

\textbf{Polyak Step Size for Extragradient Method.}
        Here, we want to follow an analogous strategy for choosing the update step size $\omega_k$ of \hyperref[eq:EG]{\algname{EG}}. Consider solving the problem \eqref{eq: Variational Inequality Definition} using \hyperref[eq:EG]{\algname{EG}}. 
        Following the approach of \cite{polyak1987introduction}, we can derive adaptive step size for $\omega_k$, i.e. by minimizing the upper bound on
        \vspace{-2mm}
        \begin{align*}
        \textstyle
            \|x_{k+1} - x_*\|^2 
            = & \left\| x_k - x_* \right\|^2 - 2 \omega_k \la F(\hx_k), x_k - \hx_k\ra - 2 \omega_k \la F(\hx_k), \hx_k - x_* \ra + \omega_k^2 \|F(\hx_k)\|^2. \notag 
        \end{align*}
        Note that for monotone $F$ we have $\la F(\hx_k), \hx_k - x_* \ra \geq 0$. Therefore for $\omega_k > 0$ the above bound reduces to
        $$ \|x_{k+1} - x_*\|^2 \leq  \left\| x_k - x_* \right\|^2 - 2 \omega_k \la F(\hx_k), x_k - \hx_k\ra + \omega_k^2 \|F(\hx_k)\|^2,$$
        where the right-hand side of the inequality is minimised for $\omega_k = \frac{\la F(\hx_k), x_k - \hx_k \ra}{\| F(\hx_k)\|^2}$. This idea of choosing a step size which minimizes a particular upper bound is similar to the Polyak step sizes for minimization problems. However, the Polyak step size in minimization requires the knowledge of 
        $f(x_*)$, which is not known apriori. In contrast, we highlight that even though we use Polyak-like step sizes, \textit{this update step does not run into issues of knowing functional optimality or other quantities that cannot be computed.} We refer to this algorithm as \hyperref[alg:PolyakEG]{\algname{PolyakEG}} (Algorithm \ref{alg:PolyakEG}). Here, we emphasise that neither this step size selection $\omega_k$ nor the algorithm is a contribution of our work. \cite{solodov1996modified} first introduced this step size, and several subsequent works analysed it. However, to our knowledge, the specific connection between this step size $\omega_k$ and the Polyak step size had not been established. This Polyak-type interpretation of $\omega_k$ allows us to extend these ideas to stochastic settings and design new algorithms with rigorous convergence guarantees. In Figure \ref{fig:2Dcontour}, we illustrate the advantages of using \hyperref[alg:PolyakEG]{\algname{PolyakEG}} over the standard \hyperref[eq:EG]{\algname{EG}} method on a two-dimensional min-max problem. We add more details related to the figure in Appendix \ref{sec:compare_polyakeg_eg}. 

\begin{algorithm}[H]
\small
    \caption{\algname{PolyakEG}}
    \label{alg:PolyakEG}
    \begin{algorithmic}[1]
        \REQUIRE Initial point $x_0 \in \R^d$ and positive sequence $\{\gamma_k\}_{k=0}^\infty$.
        \FOR{$k = 0, 1,...,K$}
        \STATE $\hx_k = x_k - \gamma_k F(x_k)$.
        \vspace{2mm}
        \STATE Set \colorbox{green!20}{$\omega_k = \frac{\la F(\hx_k), x_k - \hx_k \ra}{\|F(\hx_k)\|^2}$.}
        \vspace{2mm}
        \STATE $x_{k+1} = x_k - \omega_k F(\hx_k)$.
        \ENDFOR
    \end{algorithmic}
\end{algorithm}

\section{Deterministic Setting}\label{sec:conv_ana1}
In this section, we provide an analysis of \hyperref[alg:PolyakEG]{\algname{PolyakEG}} in the deterministic case. The ideas we present here are not conceptually new.
However, we provide a refined understanding of existing results under a unified perspective.
This will allow us to extend these results to the stochastic setup later.

\subsection{Critical Condition for $\gamma_k$}

While \hyperref[alg:PolyakEG]{\algname{PolyakEG}} specifies the update step size $\omega_k$, it does not specify the extrapolation step size $\gamma_k$. The step size $\gamma_k$ can be a constant depending on $L$ or can be chosen with a line-search technique. Here, we propose a critical condition that captures the theory of both such choices. 
\begin{definition} 
\label{def:critical-condition-deterministic}
We say $\gamma_k$ satisfies the critical condition for some $A \in (0, 1]$ if
\begin{equation}\label{eqn:gamma-critical-condition}
    \norm{F(\hat{x}_k) - F(x_k)} \le A \norm{F(x_k)} \quad \text{for } \hx_k = x_k - \gamma_k F(x_k).
\end{equation}
\end{definition}

This condition captures the cases when $\bullet$ $\gamma_k = \nicefrac{A}{L}$ for $A \in (0, 1]$ and $\bullet$ $\gamma_k$ is selected using line-search (check update rule for \algname{PolyakEG-LS}). The first significance of~\eqref{eqn:gamma-critical-condition} is that it provides a lower bound on $\omega_k$, which guarantees a sufficient progress every iteration. 

\begin{lemma}
\label{lemma:alpha-lower-bound}
If \eqref{eqn:gamma-critical-condition} is satisfied with $A \in (0,1]$, then we have $\textstyle \inprod{F(\hat{x}_k)}{F(x_k)} \ge \frac{1}{2} \left( \sqnorm{F(\hat{x}_k)} + (1 - A^2) \sqnorm{F(x_k)} \right)$ which implies $\omega_k \ge \frac{\gamma_k}{1+A}$.
\end{lemma}

Another importance of \eqref{eqn:gamma-critical-condition} is that it only relies on the local geometry of $F$ around $x_k$. While $L$-Lipschitz assumption assures that \eqref{eqn:gamma-critical-condition} is satisfied with any $\gamma_k \le \frac{A}{L}$, if $F$ has a larger local Lipschitz constant near $x_k$, then larger values of $\gamma_k$ can be accepted, allowing more progress at once.

\subsection{Convergence of \hyperref[alg:PolyakEG]{\algname{PolyakEG}} Under Critical Condition}
First, we present the unified study of \hyperref[alg:PolyakEG]{\algname{PolyakEG}}, which captures the analysis for constant and line-search extrapolation step size $\gamma_k$. Then, we explain how this unified analysis yields the results for specific settings. Under the critical condition~\eqref{eqn:gamma-critical-condition}, we derive the following Theorem.
\begin{theorem}
\label{theorem:deterministic-master-theorem}
Let $F$ is $L$-Lipschitz and (strongly) monotone. Suppose that we choose $\gamma_k$ so that \eqref{eqn:gamma-critical-condition} holds for all $k\ge 0$. Then, \hyperref[alg:PolyakEG]{\algname{PolyakEG}} satisfies, when $F$ is

\begin{itemize}
    \item monotone and $ A \in (0, 1]$,
    \begin{equation}\label{eq:deterministic-master-theorem_eq1}
        \min_{k=0,\dots,K} \gamma_k^2 \sqnorm{F(\hat{x}_k)} \le \frac{(1+A)^2 \sqnorm{x_0 - x_*}}{K+1}.
    \end{equation}
    If $A \in (0, 1)$, we additionally have,
    \begin{equation}\label{eq:deterministic-master-theorem_eq2}
        \min_{k=0,\dots,K} \gamma_k^2 \sqnorm{F(x_k)} \le \frac{(1+A) \sqnorm{x_0 - x_*}}{(1-A) (K+1)}.
    \end{equation}
    \item strongly monotone and $A \in (0, 1)$,
    \begin{equation}\label{eq:deterministic-master-theorem_eq3}
        \sqnorm{x_{k+1} - x_*} \le \prod_{j=0}^{k} \left( 1 - \frac{2(1-A) \gamma_j \mu}{(1+A)^2} \right) \sqnorm{x_0 - x_*}.
    \end{equation}
\end{itemize}
\end{theorem}

\textbf{Constant $\gamma_k$ Case.}
With the choice $A=1$, $\gamma_k = \nicefrac{1}{L}$, \eqref{eq:deterministic-master-theorem_eq1} becomes $\textstyle \min_{0 \leq k \leq K} \sqnorm{F(\hx_k)} \leq \frac{4 L^2\|x_0 - x_*\|^2}{K+1},$ which recovers the rate of Corollary~3.2 in \cite{pethick2023escaping} for the unconstrained monotone setup. Moreover, for $A \in (0, 1]$ and $\gamma_k = \nicefrac{A}{L}$, \eqref{eq:deterministic-master-theorem_eq2} gives us $\min_{0 \leq k \leq K} \| F(x_k)\|^2 \leq \frac{(1 + A) L^2 \|x_0 - x_*\|^2}{A^2(1 - A) (K+1)}.$

This result ensures convergence of the iterates $x_k$, indicating that it suffices to track only the sequence $x_k$ of \algname{PolyakEG} regardless of whether the problem is monotone or strongly monotone. 
\newline
The linear convergence result for \hyperref[alg:PolyakEG]{\algname{PolyakEG}} was first shown in \cite{solodov1996modified}, while we refine it to an explicit quantitative form with tight analysis. 
In \eqref{eq:deterministic-master-theorem_eq3}, for any $A \in (0,1)$ we can take $\gamma_k \equiv \nicefrac{A}{L}$ and obtain the $\sqnorm{x_{k+1} - x_*} \le \left(1 - \frac{2A(1-A)}{(1+A)^2} \frac{\mu}{L}\right)^{k+1} \sqnorm{x_0 - x_*}$ rate.
With $A=\frac{1}{3}$ and $\gamma_k \equiv \frac{1}{3L}$, the linear convergence factor is $1 - \nicefrac{\mu}{4L}$, matching the state-of-the-art rate for \hyperref[eq:EG]{\algname{EG}} \cite{mokhtari2020unified, azizian2020tight} that uses $\omega_k = \gamma_k \le \nicefrac{1}{4L}$. 
We highlight that our proof technique is new compared to these works, as we allow for $\gamma_k$ and $\omega_k$ exceeding this range.
Even though our theoretical linear convergence rate is the same as in standard \hyperref[eq:EG]{\algname{EG}} with constant step size $\nicefrac{1}{4L}$, experiments show that \hyperref[alg:PolyakEG]{\algname{PolyakEG}} often makes more progress per iteration and converges much faster than standard \hyperref[eq:EG]{\algname{EG}} (Figure \ref{fig:PolyakEG_theoretical_gradients}). 

\subsection{\hyperref[alg:PolyakEG_linesearch]{\hyperref[alg:PolyakEG_linesearch]{\algname{PolyakEG-LS}}}: Parameter Free Version of \hyperref[alg:PolyakEG]{\algname{PolyakEG}} Using Line Search}
When one does not have knowledge of problem parameters $L$ and $\mu$, we can use line search to find a $\gamma_k$ satisfying \eqref{eqn:gamma-critical-condition}.
The algorithm, \hyperref[alg:PolyakEG_linesearch]{\hyperref[alg:PolyakEG_linesearch]{\algname{PolyakEG-LS}}} (Algorithm \ref{alg:PolyakEG_linesearch}), combines a specific line search strategy with \hyperref[alg:PolyakEG]{\algname{PolyakEG}}. While \hyperref[alg:PolyakEG_linesearch]{\hyperref[alg:PolyakEG_linesearch]{\algname{PolyakEG-LS}}} involves a line search loop every iteration, the total number of line search iterations throughout the algorithm's execution is bounded.

\begin{theorem}
    \label{lemma:LS-step-size-lower-bound}
    Suppose $F$ is $L$-Lipschitz. Then the total number of while loop calls in \hyperref[alg:PolyakEG_linesearch]{\hyperref[alg:PolyakEG_linesearch]{\algname{PolyakEG-LS}}} is at most $\left\lfloor \nicefrac{\log (L\gamma_{-1} / A)}{\log 1/\beta} \right\rfloor + 1$, and $\gamma_k \ge \min \left\{ \nicefrac{\beta A}{L} , \gamma_{-1} \right\}$ for all $k\ge 0$.
\end{theorem}

Combining Theorems~\ref{lemma:LS-step-size-lower-bound} and \ref{theorem:deterministic-master-theorem}, we observe the complexity of \hyperref[alg:PolyakEG_linesearch]{\hyperref[alg:PolyakEG_linesearch]{\algname{PolyakEG-LS}}} is comparable to that of the constant $\gamma_k$ case, despite being parameter-free. 

\textbf{Choosing $\gamma_{-1}$.} When the initial guess $\gamma_{-1}$ is conservative (small), the convergence may be slow.
To avoid this and exploit the local geometry of $F$ further, one can start with any value of $\gamma_{-1}$ and repetitively \emph{increase} it by a factor of $\beta$ as long as \eqref{eqn:gamma-critical-condition} holds, and start the line search from there.
This procedure guarantees $\gamma_{-1} \ge \frac{A}{L}$ (without knowing $L$).

\begin{algorithm}[H]
    \caption{\algname{PolyakEG-LS}}
    \label{alg:PolyakEG_linesearch}
    \begin{algorithmic}[1]
        \REQUIRE Initial point $x_0 \in \R^d$, Initial step size $\gamma_{-1} > 0$, Line search factor $\beta \in (0, 1)$ and $0 < A \le 1$
        \FOR{$k = 0, 1,...,K$}
        \STATE $\gamma_k = \gamma_{k-1}$ 
        \STATE $\hx_k = x_k - \gamma_k F(x_k)$.
        \WHILE{$\| F(x_k) - F(\hx_k) \| > A \|F(x_k)\|$}
            \STATE $\gamma_k = \beta \gamma_k$
            \STATE $\hx_k = x_k - \gamma_k F(x_k)$
        \ENDWHILE
        \STATE Set $\omega_k = \frac{\la F(\hx_k), x_k - \hx_k \ra}{\|F(\hx_k)\|^2}$.
        \vspace{2mm}
        \STATE $x_{k+1} = x_k - \omega_k F(\hx_k)$.
        \ENDFOR
    \end{algorithmic}
\end{algorithm}
\textbf{\color{black}Comparison with Prior Work.} \hyperref[alg:PolyakEG_linesearch]{\hyperref[alg:PolyakEG_linesearch]{\algname{PolyakEG-LS}}} uses the same idea of reusing the step size from the previous iteration as in \cite{solodov1996modified}. \cite{pethick2023escaping} also proposes a line search variant of \hyperref[alg:PolyakEG]{\algname{PolyakEG}}, but they re-initialize the step size based on the Jacobian of $F$ at every iteration, requiring more computations.

\section{Stochastic Setting}\label{sec:stochastic}

So far, we have established a concise, unified analysis of \algname{PolyakEG} in the deterministic setup.
In this section, we extend it to the more general stochastic setup where we use $F_{\mathcal{S}_k}(x_k), F_{\mathcal{S}_k}(\hx_k)$ to estimate $F(x_k)$ and $F(\hx_k)$.
The naive extension \algname{PolyakSEG} 
which simply replaces the occurrences of $F$ with $F_{\mathcal{S}_k}$ in \algname{PolyakEG}, requires us to assume the existence of an interpolating solution $x_*$ (for which $F_{\mathcal{S}_k}(x_\star) = 0$ a.s.) for the convergence analysis.
To remove this stringent assumption, we propose and analyze \textit{Decreasing Polyak} \algname{SEG} (\algname{DecPolyakSEG}; Algorithm \ref{alg:decPolyakSEG}), the main contribution of this chapter.
This sequence of ideas parallels the development led by the prior work \cite{loizou2021stochastic, orvieto2022dynamics} on Polyak step size for stochastic convex minimization, outlined below.

\textbf{Connection to the Minimization Setting.} In the stochastic convex minimization setting $\min_x f(x) = \frac{1}{n} \sum_{i = 1}^n f_i(x)$, \cite{loizou2021stochastic} and \cite{orvieto2022dynamics} respectively studied the Stochastic Polyak step size (\algname{SPS}) $$\eta_k = \frac{f_{\mathcal{S}_k}(x_k) - \min_x f_{\mathcal{S}_k}(x)}{\| \nabla f_{\mathcal{S}_k}(x_k)\|^2}$$
\big(here $f_{\mathcal{S}_k}(x) = \frac{1}{B} \sum_{i \in \mathcal{S}_k} f_i(x)$ and $\nabla f_{\mathcal{S}_k}(x) = \frac{1}{B} \sum_{i \in \mathcal{S}_k} \nabla f_i(x)$ \big) and its decreasing variant \algname{DecSPS} $\eta_k = \frac{1}{c_{k+1}} \min \left\{ \frac{f_{\mathcal{S}_k}(x_k) - \min_x f_{\mathcal{S}_k}(x)}{c \| \nabla f_{\mathcal{S}_k}(x_k)\|^2}, c_k \eta_{k-1}\right\}$ for \algname{SGD}: $x_{k+1} - x_k - \eta_k \nabla f_{\mathcal{S}_k}(x_k)$. 
Specifically, \cite{loizou2021stochastic} showed that \algname{SPS} converges linearly given that each $f_i$ is smooth and strongly convex and the solution $x_*$ is interpolated, but without interpolation, one can only guarantee convergence to a neighborhood.
\cite{orvieto2022dynamics} introduced \algname{DecSPS} to resolve this issue and proved its (sublinear) convergence without assuming an interpolated solution.

Moreover, similar to Definition \ref{def:critical-condition-deterministic} in the deterministic setting, we consider the following critical condition for the stochastic algorithms, which captures constant and line-search step size schemes.
\begin{definition} 
\label{def:critical-condition-stochastic}
We say $\gamma_k$ satisfies the critical condition with respect to $F_{\mathcal{S}_k}$ for some $A \in (0, 1]$ if
\begin{equation}\label{eqn:gamma-critical-condition-stochastic}
    \norm{F_{\mathcal{S}_k}(\hx_k) - F_{\mathcal{S}_k}(x_k)} \le A \norm{F_{\mathcal{S}_k}(x_k)}
\end{equation}
for $\hx_k = x_k - \gamma_k F_{\mathcal{S}_k}(x_k)$.
\end{definition}

\subsection{Analysis of \algname{PolyakSEG}}\label{subsec:PolyakSEG}

\begin{algorithm}[H]
   \small
    \caption{\algname{PolyakSEG}}
    \label{alg:PolyakSEG}
    \begin{algorithmic}[1]
        \REQUIRE Initial point $x_0 \in \R^d$ and positive sequence $\{\gamma_k\}_{k=0}^\infty$.
        \FOR{$k = 0, 1,...,K$}
        \STATE Sample $\mathcal{S}_k \subseteq [n]$.
        \STATE $\hx_k = x_k - \gamma_k F_{\mathcal{S}_k}(x_k)$.
        \STATE Set $\omega_k = \frac{\la F_{\mathcal{S}_k}(\hx_k), x_k - \hx_k \ra}{\|F_{\mathcal{S}_k}(\hx_k)\|^2}$.
        \vspace{2mm}
        \STATE $x_{k+1} = x_k - \omega_k F_{\mathcal{S}_k}(\hx_k)$.
        \ENDFOR
    \end{algorithmic}
\end{algorithm}

Consider \algname{PolyakSEG} (Algorithm \ref{alg:PolyakSEG}), where at each iteration $k$, a random sample $\mathcal{S}_k \subseteq [n]$ is drawn and a \algname{PolyakEG} step is performed with respect to $F_{\mathcal{S}_k}$. Under the interpolation assumption, i.e. $F_i(x_*) = 0, \forall i \in [n]$, essentially the same proof techniques as in the deterministic setting apply. 

\begin{theorem}\label{theorem:PolyakSEG-convergence}
    Let $F_i$ are $L$-Lipschitz, (strongly) monotone and there exists an interpolating solution $x_*$ for which $F_i(x_*) = 0$ for all $i \in [n]$ almost surely.
    Then, provided that $\gamma_k$ satisfy the critical condition ~\eqref{eqn:gamma-critical-condition-stochastic} with $0 < A < 1$, \algname{PolyakSEG} satisfies 

    $\bullet$ When $F_i$ are monotone~\eqref{eq:monotone}, $\min_{0 \leq k \leq K} \Exp{\gamma_k^2 \|F(x_k)\|^2} \leq \frac{(1+A) \|x_0 - x_*\|^2}{(1-A)(K+1)}.$ \\
    $\bullet$ When  $F_i$ are strongly monotone~\eqref{eq:strong_monotone}, $\Exp{\|x_{k+1} - x_*\|^2} \le \Exp{ \left(1 - \frac{2 (1-A) \gamma_k \mu}{(1+A)^2} \right)\sqnorm{x_k - x_*}}.$
\end{theorem}

In particular, this implies that \algname{PolyakSEG} converges linearly to the interpolating solution $x_*$ when $\mu > 0$, which is similar to the result of \cite[Appendix~E.4]{vaswani2019painless} on stochastic \algname{EG} (without Polyak step sizes) using the line search scheme identical to that of \algname{PolyakEG-LS}. However, while the linear convergence results in \cite{vaswani2019painless} requires $\gamma_{-1} \le \frac{1}{4\mu}$, we do not need to choose $\gamma_{-1}$ based on $\mu$ (or $L$). We provide the pseudocode for \algname{PolyakSEG-LS} in Algorithm \ref{alg:PolyakSEG_linesearch}.

Beyond the cases where we have an interpolating solution, \algname{PolyakSEG} seems to fail to converge. The following section is devoted to an algorithmic fix resolving this issue. 

\begin{algorithm}[h]
\small
\caption{\algname{PolyakSEG-LS}}
\label{alg:PolyakSEG_linesearch}
\begin{algorithmic}[1]
    \REQUIRE Initial point $x_0 \in \R^d$, initial step size $\gamma_{-1} > 0$, line search factor $\beta \in (0, 1)$ and $A \in (0, 1)$.
    \FOR{$k = 0, 1,...,K$}
    \STATE $\gamma_k = \gamma_{k-1}$ 
    \STATE Sample $\mathcal{S}_k \subseteq [n]$.
    \STATE $\hx_k = x_k - \gamma_k F_{\mathcal{S}_k}(x_k)$.
    \WHILE{$\| F_{\mathcal{S}_k}(x_k) - F_{\mathcal{S}_k}(\hx_k) \| > A \|F_{\mathcal{S}_k}(x_k)\|$}
        \STATE $\gamma_k = \beta \gamma_k$
        \STATE $\hx_k = x_k - \gamma_k F_{\mathcal{S}_k}(x_k)$
    \ENDWHILE
    \vspace{2mm}
    \STATE Set $\omega_k = \frac{\la F_{\mathcal{S}_k}(\hx_k), x_k - \hx_k \ra}{\|F_{\mathcal{S}_k}(\hx_k)\|^2}$.
    \vspace{2mm}
    \STATE $x_{k+1} = x_k - \omega_k F_{\mathcal{S}_k}(\hx_k)$.
    \ENDFOR
\end{algorithmic}
\end{algorithm}
    
\subsection{Analysis of \algname{DecPolyakSEG}}
Here, we propose the novel algorithm \algname{DecPolyakSEG} (Algorithm \ref{alg:decPolyakSEG}) whose convergence analysis does not require the stringent interpolation condition. The main distinction from \algname{PolyakSEG} is that the step sizes are decreasing (highlighted in green): we ensure 
\textbf{(1)} $\gamma_k \le \frac{c_{k-1}}{c_k} \gamma_{k-1}$ and \textbf{(2)} $\omega_k \le \omega_{k-1}$.
While condition~\textbf{(1)} is similar to \cite{orvieto2022dynamics}, condition \textbf{(2)} is a new component.

\textbf{Intuition of Decreasing Step Sizes.}
When analyzing \algname{SEG} with non-adaptive $\omega_k$, one can use $\Expk{\omega_k \la F_{\mathcal{S}_k}(u), \hx_k - u\ra} = 0$ for $u=x_*$, implied by $\Expk{F_{\mathcal{S}_k}(x_*)} = F(x_*) = 0$~\cite{gorbunov2022stochastic}.
This is not possible when $\omega_k$ is chosen adaptively and thus is random. Instead, we show $\frac{\gamma_k}{2} \le \omega_k \le \frac{\gamma_k}{1-A}$ (Lemma \ref{lemma:PolyakSEG-LS-key-properties}) and control $\gamma_k, \omega_k$ by selecting $c_k$ appropriately.
Having $\omega_k$ within a desirable range, we can manipulate the non-vanishing $\Expk{\omega_k \la F_{\mathcal{S}_k}(u), \hx_k - u\ra}$ term to prove convergence result.

For Theorem \ref{theorem:DecPolyakSEGLS_monotone}, we do not use $\sqnorm{F(\cdot)}$ as a convergence measure, unlike in deterministic setups.
Generally, convergence guarantees for stochastic algorithms are given as a sum of (sub)linearly converging terms and error terms proportional to the step size, 
and one can obtain asymptotic rates $\mathcal{O}(\nicefrac{1}{T^\beta})$ by tuning the step size.
However, this is not the case for \algname{SEG}; using $\sqnorm{F(\cdot)}$ as a convergence measure for monotone problems yields irreducible error terms independent of the step size.
Multiple prior works \cite{diakonikolas2021efficient, lee2021fast, bohm2022solving, choudhury2024single} observed this and they used increasing batch sizes or variance reduction techniques to ensure a meaningful rate on $\sqnorm{F(\cdot)}$, which we do not pursue here. Instead, we provide convergence using a distinct measure, under the following assumption.

\begin{assumption}
\label{assumption:bounded-iterates}
    There exists a compact subset $\mathcal{C} \subseteq \reals^d$ containing a zero $x_*$ of $F$ such that $x_k, \hx_k \in \mathcal{C}$ for all $k \geq 1$ during the runtime of \algname{DecPolyakSEG}.
\end{assumption}

\begin{theorem}\label{theorem:DecPolyakSEGLS_monotone}
    Let $F_i$ are $L$-Lipschitz, monotone, and Assumption~\ref{assumption:bounded-iterates} holds. Then $\overline{x}_K = \frac{1}{K+1} \sum_{k=0}^K x_k$ generated by \algname{DecPolyakEG} with $c_k = \sqrt{k+1}$ satisfies
    \begin{align*}
        \Exp{\inprod{F(u)}{\overline{x}_K - u}}
        \le \frac{1}{K+1} \Exp{\frac{D^2}{\gamma_{K+1}}} + \frac{\sigma^2 \gamma_0}{(1-A)(K+1)} \sum_{k=0}^K \frac{1}{\sqrt{k+1}} 
    \end{align*}
    \vspace{-1.5mm}
    for any $u \in \mathcal{C}$. Here, $\sigma^2 \eqdef 2L^2 \max_{u_1, u_2 \in \mathcal{C}} \|u_1 - u_2\|^2 + 2 \Exp{\|F_{\mathcal{S}}(x_*)\|^2}$.
\end{theorem}

The choice $c_k = \sqrt{k+1}$ in Theorem \ref{theorem:DecPolyakSEGLS_monotone} follows \cite{orvieto2022dynamics}, where similar ideas were used to analyze \algname{DecSPS}. However, we need some control on the extrapolation step size $\gamma_K$ to obtain a quantitative convergence rate from Theorem \ref{theorem:DecPolyakSEGLS_monotone}.

\begin{corollary}
\label{corollary:DecPolyakSEG-1/sqrtK-convergence}
    Under the assumptions of Theorem \ref{theorem:DecPolyakSEGLS_monotone}, if $\gamma_K = \Omega\left(\nicefrac{1}{\sqrt{K}}\right)$, then for any $u \in \mathcal{C}$, $\Exp{\inprod{F(u)}{\overline{x}_K - u}} = \cO\left(\nicefrac{1}{\sqrt{K}}\right)$.
\end{corollary}

Assuming that $L$ is known, one can use $\gamma_0 = \nicefrac{1}{L}$ and $\gamma_k = \nicefrac{c_k}{c_{k+1}} \gamma_{k-1}$ for $k \ge 1$ (which gives $\gamma_k = \nicefrac{2}{L\sqrt{k+2}} = \Omega\left(\nicefrac{1}{\sqrt{k}}\right)$).
However, as this is rarely the case in practice, we propose to use the efficient line search scheme as before, which gives \algname{DecPolyakSEG-LS} (Algorithm \ref{alg:decPolyakSEG_linesearch}). In Lemma \ref{lemma:DecPolyakSEGLS-gamma-lower-bound}, we show that \algname{DecPolyakSEG-LS} ensures $\gamma_K = \Omega\left(\nicefrac{1}{\sqrt{K}}\right)$ and thus $\max_{u \in \mathcal{C}}\Exp{\inprod{F(u)}{\overline{x}_K - u}} = \cO\left(\nicefrac{1}{\sqrt{K}}\right)$ in Corollary \ref{corollary:DecPolyakSEG-1/sqrtK-convergence}. 

\begin{algorithm}[h]
        \small
            \caption{\algname{DecPolyakSEG}}
            \label{alg:decPolyakSEG}
            \begin{algorithmic}[1]
                \REQUIRE Initial point $x_0 \in \R^d$, positive sequence $\{\gamma_k\}_{k=0}^\infty$, $\omega_{-1} = \infty$ and a non-decreasing sequence $\{ c_k\}_{k=0}^{\infty}$.
                \FOR{$k = 0, 1,...,K$}
                \STATE Sample $\mathcal{S}_k \subseteq [n]$.
                \STATE Ensure \colorbox{green!20}{$\gamma_k \le \frac{c_{k-1}}{c_k}\gamma_{k-1}$ and \eqref{eqn:gamma-critical-condition-stochastic}}.
                \STATE $\hx_k = x_k - \gamma_k F_{\mathcal{S}_k}(x_k)$.
                \STATE Set \colorbox{green!20}{$\omega_k = \min \left\{ \frac{\la F_{\mathcal{S}_k}(\hx_k), x_k - \hx_k \ra}{\|F_{\mathcal{S}_k}(\hx_k)\|^2}, \omega_{k-1}\right\}$.}
                \vspace{2mm}
                \STATE $x_{k+1} = x_k - \omega_k F_{\mathcal{S}_k}(\hx_k)$.
                \ENDFOR
            \end{algorithmic}
        \end{algorithm}

\textbf{Remark on the Measure of Convergence.}
Theorem~\ref{theorem:DecPolyakSEGLS_monotone} indicates $\max_{u \in \mathcal{C}} \langle F(u), \mathbb{E}\left[\overline{x}_K\right] - u \rangle = \mathcal{O}(\nicefrac{1}{\sqrt{K}})$, i.e., the restricted merit function $\mathrm{Err}(\cdot) = \max_{u \in \mathcal{C}} \inprod{F(u)}{\cdot - u}$ \cite{nesterov2007dual}, a standard performance metric for stochastic variational inequality and minimax optimization, converges to zero for the expected iterate $\mathbb{E}[\overline{x}_K]$.
That is, $\mathbb{E}[\overline{x}_K]$ is a good approximate solution for large $K$. 
Note that, however, this does not guarantee that $\overline{x}_K$ obtained from a single execution of the algorithm behaves nicely, as highlighted in \cite{alacaoglu2022complexity} using an example where $\mathbb{E}[x_k] = x_*$ for all $k$ but $\norm{x_k - x_*} \to \infty$.
In principle, to obtain a well-behaved solution given the guarantee of the form of Theorem \ref{theorem:DecPolyakSEGLS_monotone}, one needs to run the algorithm multiple times and take the mean of $\overline{x}_K$'s over the trials.
Therefore, as \cite{alacaoglu2022complexity} points out, a theoretically more meaningful result would be a bound on $\mathbb{E}\left[\mathrm{Err}(\overline{x}_K)\right]$ (which is generally larger than $\mathrm{Err}(\mathbb{E}\left[\overline{x}_K\right])$). 

Interestingly, on the other hand, we are not aware of any convergence result on same-sample \algname{SEG} (using the same minibatches for extrapolation and update steps) achieving that in the monotone setup.
Prior works either provide a guarantee on $\mathrm{Err}(\mathbb{E}\left[\overline{x}_K\right])$ for same-sample \algname{SEG}  \cite{mishchenko2020revisiting, beznosikov2022decentralized} or use independent minibatches for extrapolation and update steps \cite{juditsky2011solving, gidel2018variational, antonakopoulos2021sifting}.
In practice, however, same-sample \algname{SEG} is frequently used and provides consistent results over distinct trials, i.e., $\mathrm{Var}(\overline{x}_K)$ is small.
This indicates that the extreme cases like the example of \cite{alacaoglu2022complexity} are empirically rare with same-sample \algname{SEG}.
We expect $\mathrm{Var}(\overline{x}_K)$ could be bounded for same-sample \algname{SEG} under certain scenarios, with which Theorem~\ref{theorem:DecPolyakSEGLS_monotone} will indicate the in-probability convergence $\mathrm{Err}(\overline{x}_K) \xrightarrow{P} 0$. For instance, this will be the case if minibatch operators $F_\mathcal{S}(\cdot)$ have uniformly small variance.

\textbf{Removing Bounded Iterates Assumption.} We required Assumption~\ref{assumption:bounded-iterates} to prove  Theorem \ref{theorem:DecPolyakSEGLS_monotone}.
Such an assumption is common in the convergence analysis of adaptive stochastic methods:
\cite{orvieto2022dynamics} also assumes bounded iterates, and \cite{antonakopoulos2021sifting} assumes bounded operator (gradient) norm. Nevertheless, we show that the bounded iterate assumption can be removed under strong monotonicity.

\begin{proposition}\label{proposition:DecPolyakSEGLS_strong_monotone}
    Let $F_i$ are $L$-Lipschitz and strongly monotone.
    Then, for $A \in (0,1)$ and $c_k = \sqrt{k+1}$, the iterates $x_k$ of \hyperref[alg:decPolyakSEG_linesearch]{\algname{DecPolyakSEG}} stay bounded almost surely.
\end{proposition}

\begin{algorithm}[H]
\small
    \caption{\algname{DecPolyakSEG-LS}}
    \label{alg:decPolyakSEG_linesearch}
    \begin{algorithmic}[1]
        \REQUIRE Initial point $x_0 \in \R^d$, initial step size $\gamma_{-1} > 0$, $\omega_{-1} = \infty$, line search factor $\beta \in (0, 1)$, $A \in (0, 1)$ and a non-decreasing sequence $\{ c_k\}_{k=-1}^{\infty}$.
        \FOR{$k = 0, 1,...,K$}
        \STATE \colorbox{green!20}{$\gamma_k = \frac{c_{k-1}}{c_{k}}\gamma_{k-1}$.}
        \STATE Sample $\mathcal{S}_k \subseteq [n]$.
        \STATE $\hx_k = x_k - \gamma_k F_{\mathcal{S}_k}(x_k)$.
        \WHILE{$\| F_{\mathcal{S}_k}(x_k) - F_{\mathcal{S}_k}(\hx_k) \| > A \|F_{\mathcal{S}_k}(x_k)\|$}
            \STATE $\gamma_k = \beta \gamma_k$
            \STATE $\hx_k = x_k - \gamma_k F_{\mathcal{S}_k}(x_k)$
        \ENDWHILE
        \vspace{2mm}
        \STATE Set \colorbox{green!20}{$\omega_k = \min \left\{ \frac{\la F_{\mathcal{S}_k}(\hx_k), x_k - \hx_k \ra}{\|F_{\mathcal{S}_k}(\hx_k)\|^2}, \omega_{k-1}\right\}$.}
        \vspace{2mm}
        \STATE $x_{k+1} = x_k - \omega_k F_{\mathcal{S}_k}(\hx_k)$.
        \ENDFOR
    \end{algorithmic}
\end{algorithm}

\section{Numerical Experiments}\label{sec:numerical_experiment}

In this section, we conduct several experiments to verify our theoretical findings and demonstrate our proposed algorithm's efficiency. All experiments in this chapter were conducted using a personal MacBook with an Apple M3 chip and 16GB of RAM. 

\subsection{Deterministic Setting}
\textbf{Comparison of \hyperref[eq:EG]{\algname{EG}}, \hyperref[alg:PolyakEG]{\algname{PolyakEG}} and \hyperref[alg:PolyakEG_linesearch]{\algname{PolyakEG-LS}}.}\label{subsec:compare_theoreticalstep_SM}
In Figure \ref{fig:PolyakEG_theoretical_gradients} and \ref{fig:PolyakEG_tuned_gradients}, we compare the performance of the \hyperref[eq:EG]{\algname{EG}}, \hyperref[alg:PolyakEG]{\algname{PolyakEG}} and \hyperref[alg:PolyakEG_linesearch]{\algname{PolyakEG-LS}}. For this experiment, we focus on implementing our methods on the following strongly convex strongly concave min-max optimization problem $\min_{w_1 \in \mathbb{R}^{d_1}} \max_{w_2 \in \mathbb{R}^{d_2}} \mathcal{L}(w_1, w_2) \eqdef \frac{1}{n} \sum_{i = 1}^n \mathcal{L}_i(w_1, w_2),$~\cite{gorbunov2022stochastic} where 
\begin{equation}\label{eq:quad_minmax}
\textstyle
    \mathcal{L}_i(w_1, w_2) = \frac{1}{2}w_1^{\top}\mathbf{A}_i w_1 + w_1^{\top}\mathbf{B}_i w_2 - \frac{1}{2}w_2^{\top}\mathbf{C}_iw_2 + a_i^{\top}w_1 - c_i^{\top}w_2.
\end{equation}
For \hyperref[alg:PolyakEG_linesearch]{\algname{PolyakEG-LS}}, we start with different initialization of step size $\gamma_{-1} \in \left\{ 10^{-3}, 10^{-1}, 10, 10^{3}\right\}$. While for \algname{EG} and \algname{PolyakEG} we use the theoretical step sizes in Figure \ref{fig:PolyakEG_theoretical_gradients} and tuned step sizes in \ref{fig:PolyakEG_tuned_gradients}. In both the figures, we have the number of gradient computations on $x$-axis against relative error $\nicefrac{\|x_k - x_*\|^2}{\|x_0 - x_*\|^2}$ on $y$-axis. We find that, with theoretical steps \algname{PolyakEG} performs better than \algname{EG} and has similar performance with tuned ones. Moreover, \algname{PolyakEG-LS} performs better than other two for theoretical steps while having a comparable performance with tuned step sizes. 

\textbf{Comparison with \cite{pethick2023escaping}, \cite{solodov1996modified}.} In Figure \ref{fig:PolyakEGLS_CurvatureEG}, we compare the performance of \algname{PolyakEG-LS} with other deterministic tuning free \algname{EG} algorithms: \algname{CurvatureEG+}~\cite{pethick2023escaping} and the algorithm from ~\cite{solodov1996modified}, which we call \algname{Projection-EG}. We run our experiments to solve \eqref{eq:quad_minmax} and find that \algname{CurvatureEG+} performs better than \algname{PolyakEG-LS}. The superior performance of \algname{CurvatureEG+} can be attributed to its dependence on second-order information for initializing the $\gamma_k$ in each iteration. \algname{CurvatureEG+} starts its line-search with $\gamma_k = \nicefrac{A}{\|\mathbf{J}(x_k)\|}$ where $\mathbf{J}(x_k)$ represents the Jacobian of \eqref{eq:quad_minmax}. However, we emphasize that, \algname{PolyakEG-LS} only uses first-order information.

\textbf{Verification of Theorem \ref{lemma:LS-step-size-lower-bound}.} In this experiment, we verify the tightness of the upper bound on the number of  while loop calls in Theorem \ref{lemma:LS-step-size-lower-bound}. 
We implement \hyperref[alg:PolyakEG_linesearch]{\hyperref[alg:PolyakEG_linesearch]{\algname{PolyakEG-LS}}} to solve~\eqref{eq:quad_minmax} for $100$ different choices of step size initialization $\gamma_{-1}$. The purple line in Figure \ref{fig:PolyakEGLS_upperbound} denotes the upper bound $\left\lfloor \nicefrac{\log (L\gamma_{-1} / A)}{\log 1/\beta} \right\rfloor + 1$ for different choices of $\gamma_{-1}$, while the yellow line represents the number of while loop calls required by {\hyperref[alg:PolyakEG_linesearch]{\algname{PolyakEG-LS}}} for each $\gamma_{-1}$ in practice. We observe that the while-loop calls never exceed the theoretical bound, validating our prediction. Furthermore, the close alignment between the violet and yellow lines demonstrates the tightness of our proposed bound.
\begin{figure}[h]
\centering
\begin{subfigure}[b]{0.3\textwidth}
    \centering
    \includegraphics[width=\textwidth]{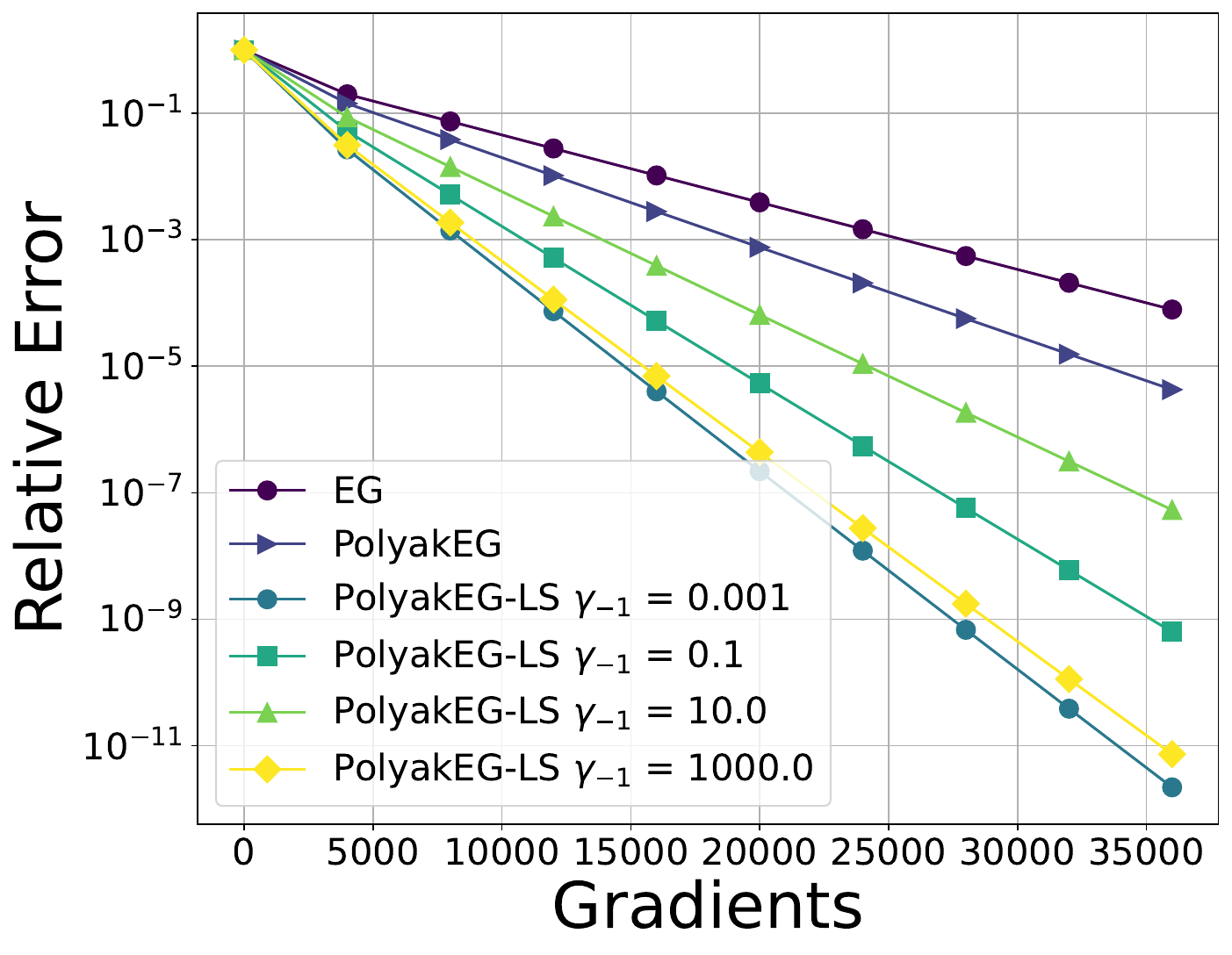}
    \caption{}\label{fig:PolyakEG_theoretical_gradients}
\end{subfigure}
\begin{subfigure}[b]{0.3\textwidth}
    \centering
    \includegraphics[width=\textwidth]{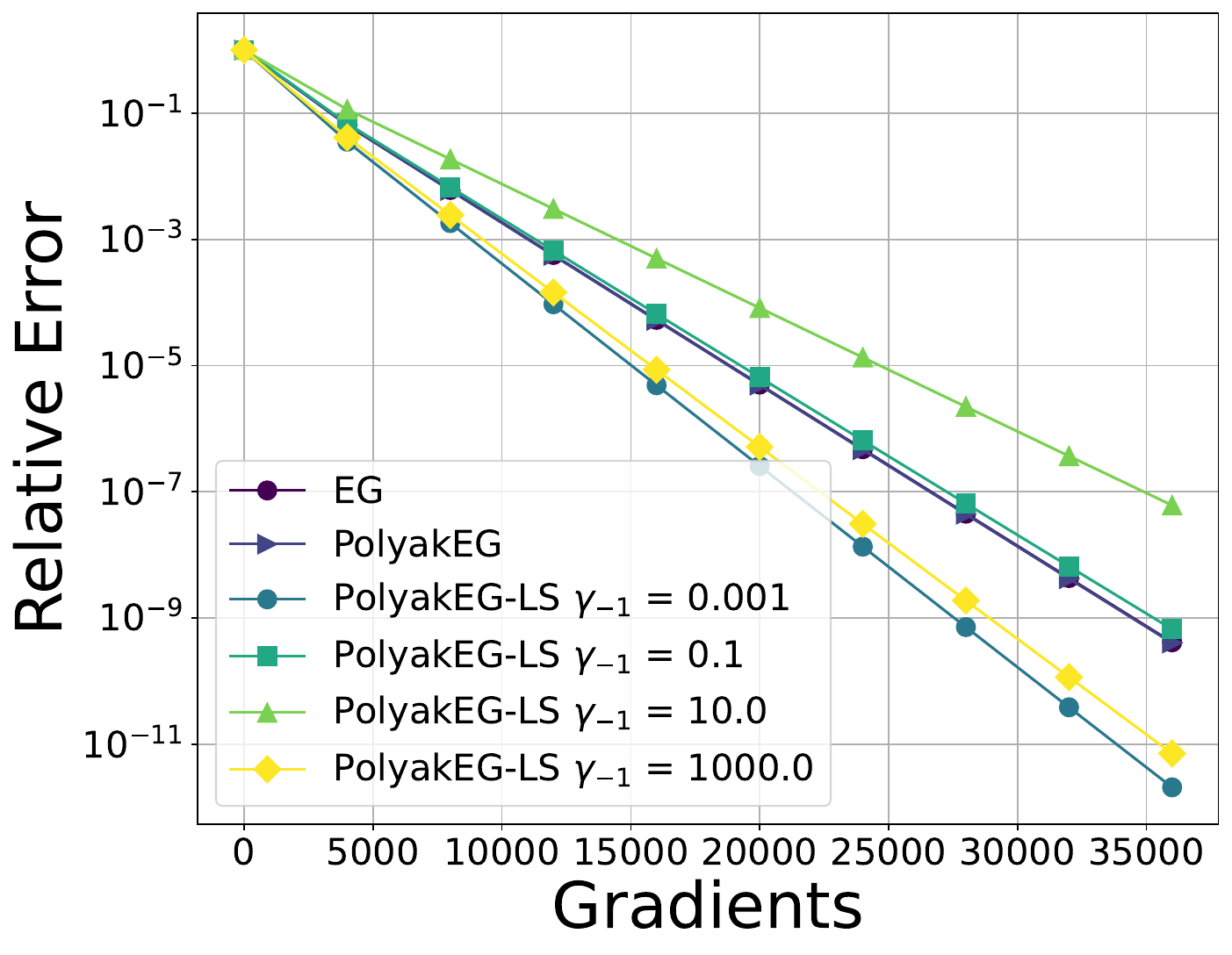}
    \caption{}\label{fig:PolyakEG_tuned_gradients}
\end{subfigure}
\begin{subfigure}[b]{0.3\textwidth}
    \centering
    \includegraphics[width=\textwidth]{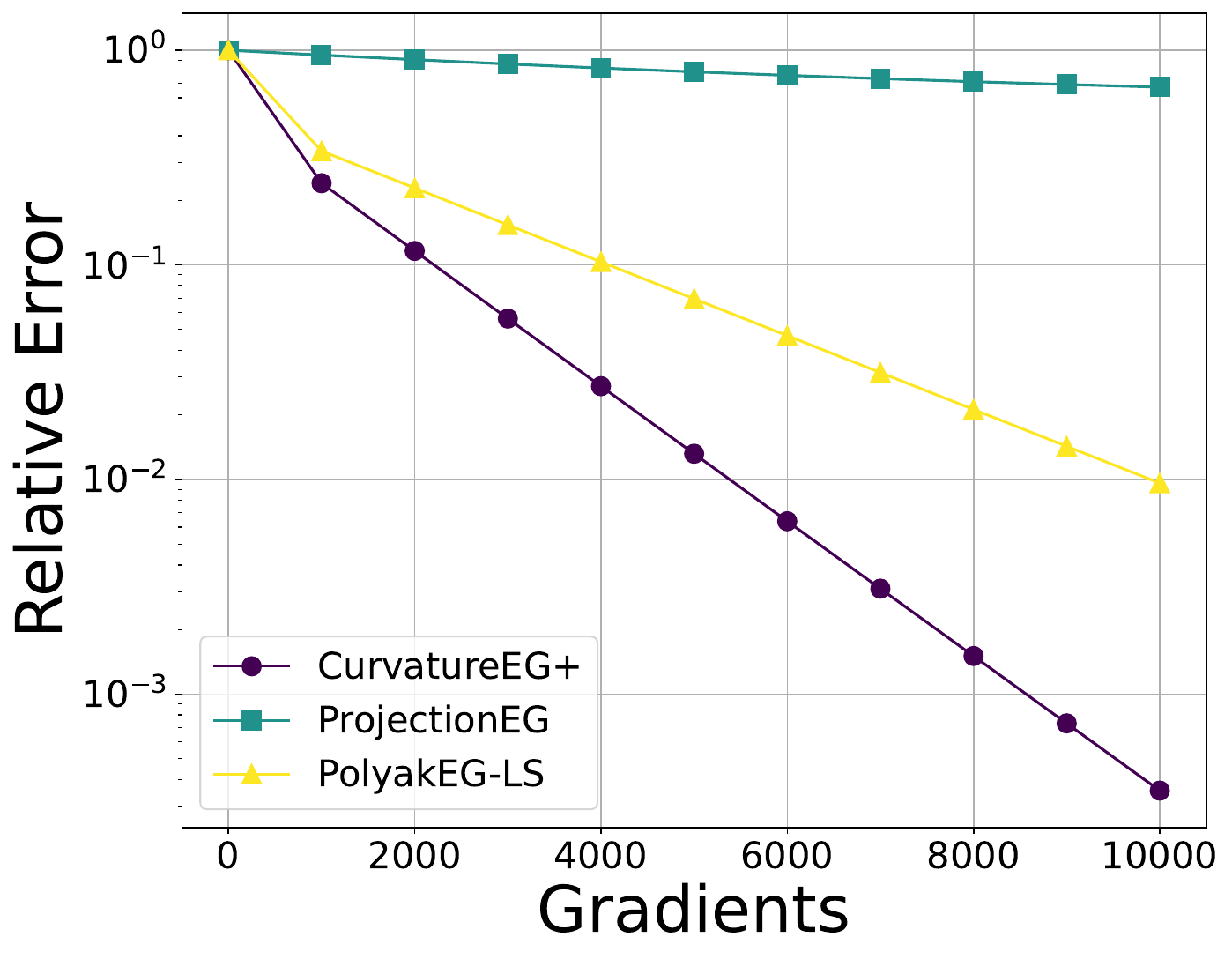}
    \caption{}\label{fig:PolyakEGLS_CurvatureEG}
\end{subfigure}

\begin{subfigure}[b]{0.3\textwidth}
    \centering
    \includegraphics[width=\textwidth]{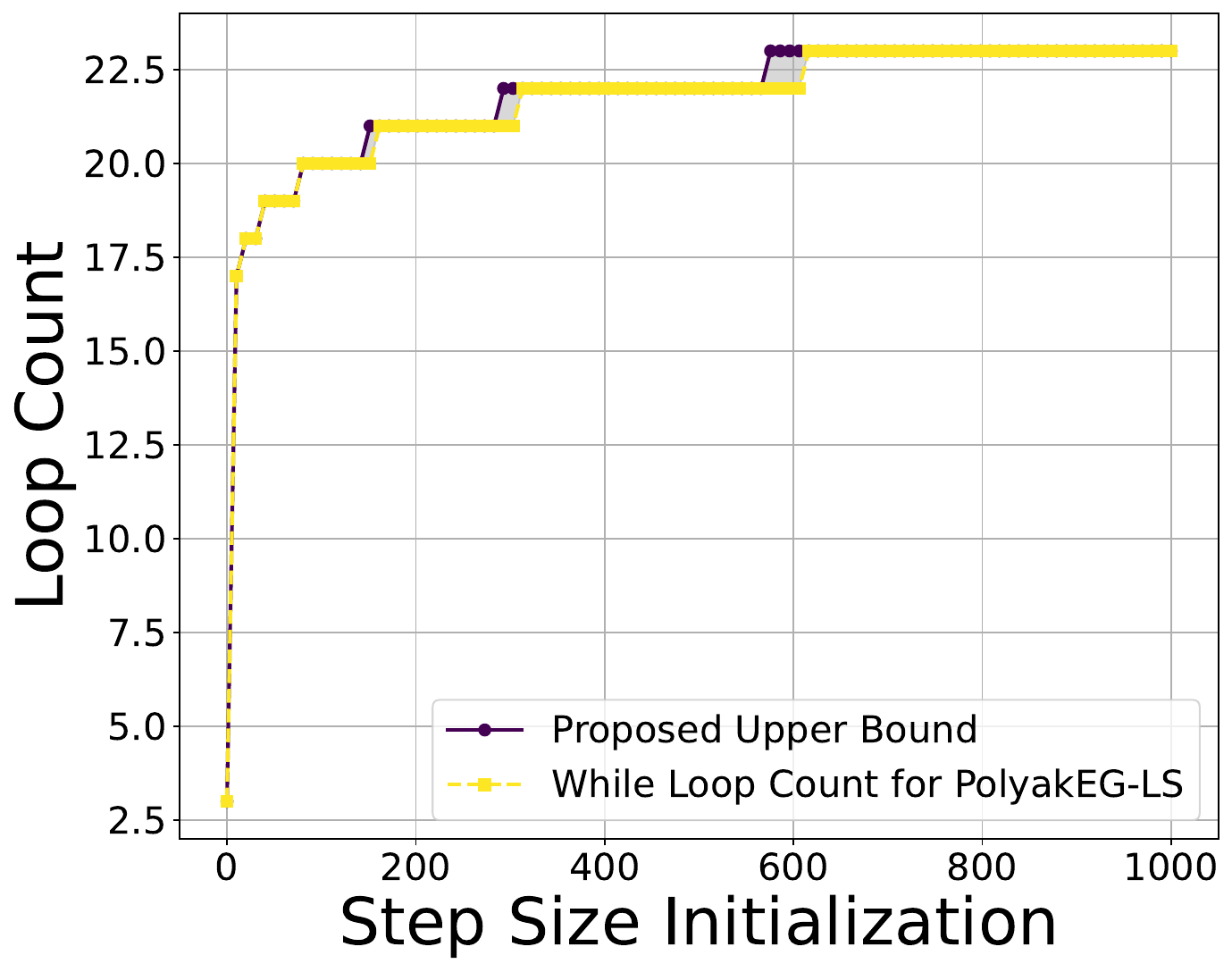}
    \caption{}\label{fig:PolyakEGLS_upperbound}
\end{subfigure}
\begin{subfigure}[b]{0.3\textwidth}
    \centering
    \includegraphics[width=\textwidth]{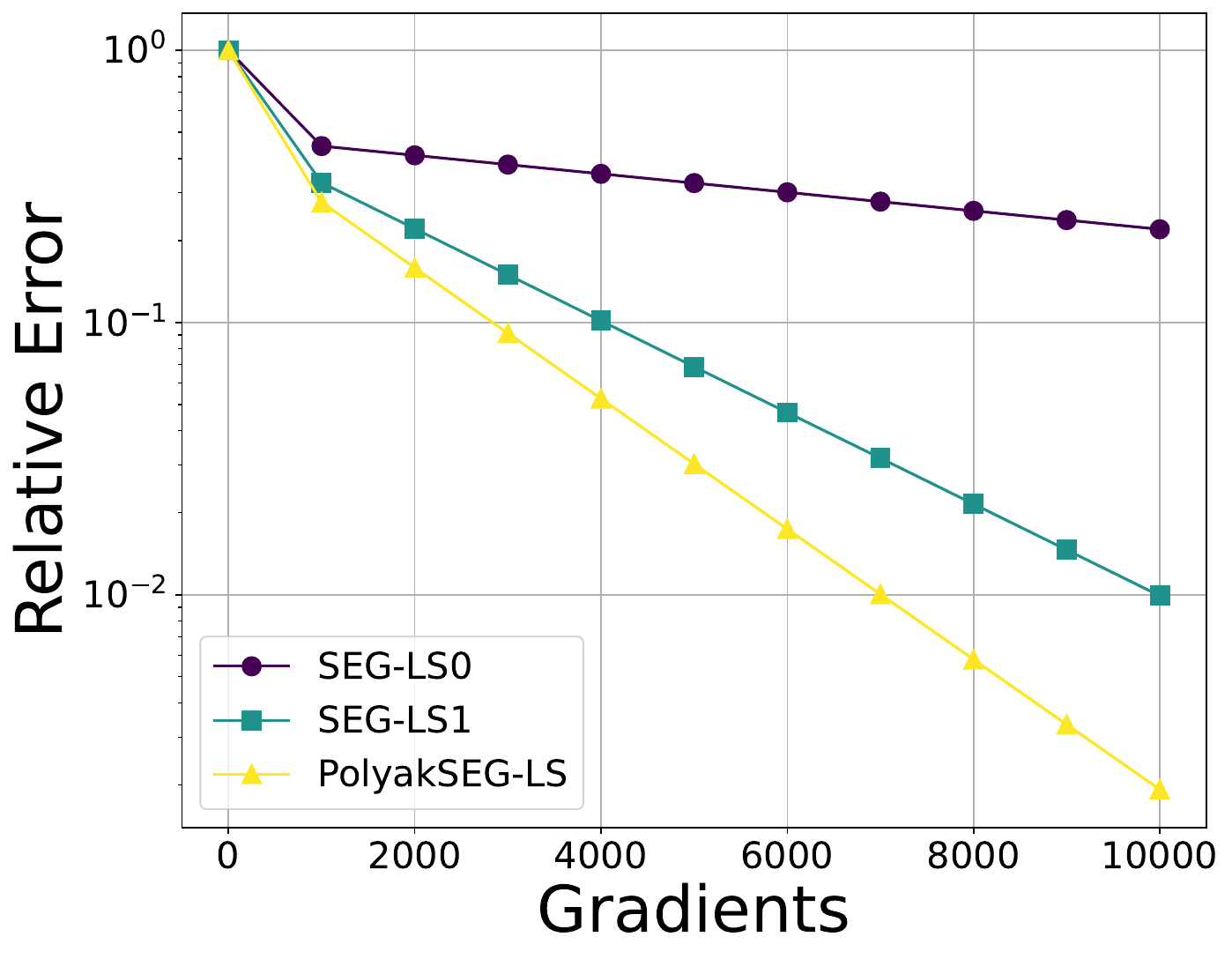}
    \caption{}\label{fig:PolyakSEGLS_interpolation_v1}
\end{subfigure}
\begin{subfigure}[b]{0.3\textwidth}
    \centering
    \includegraphics[width=\textwidth]{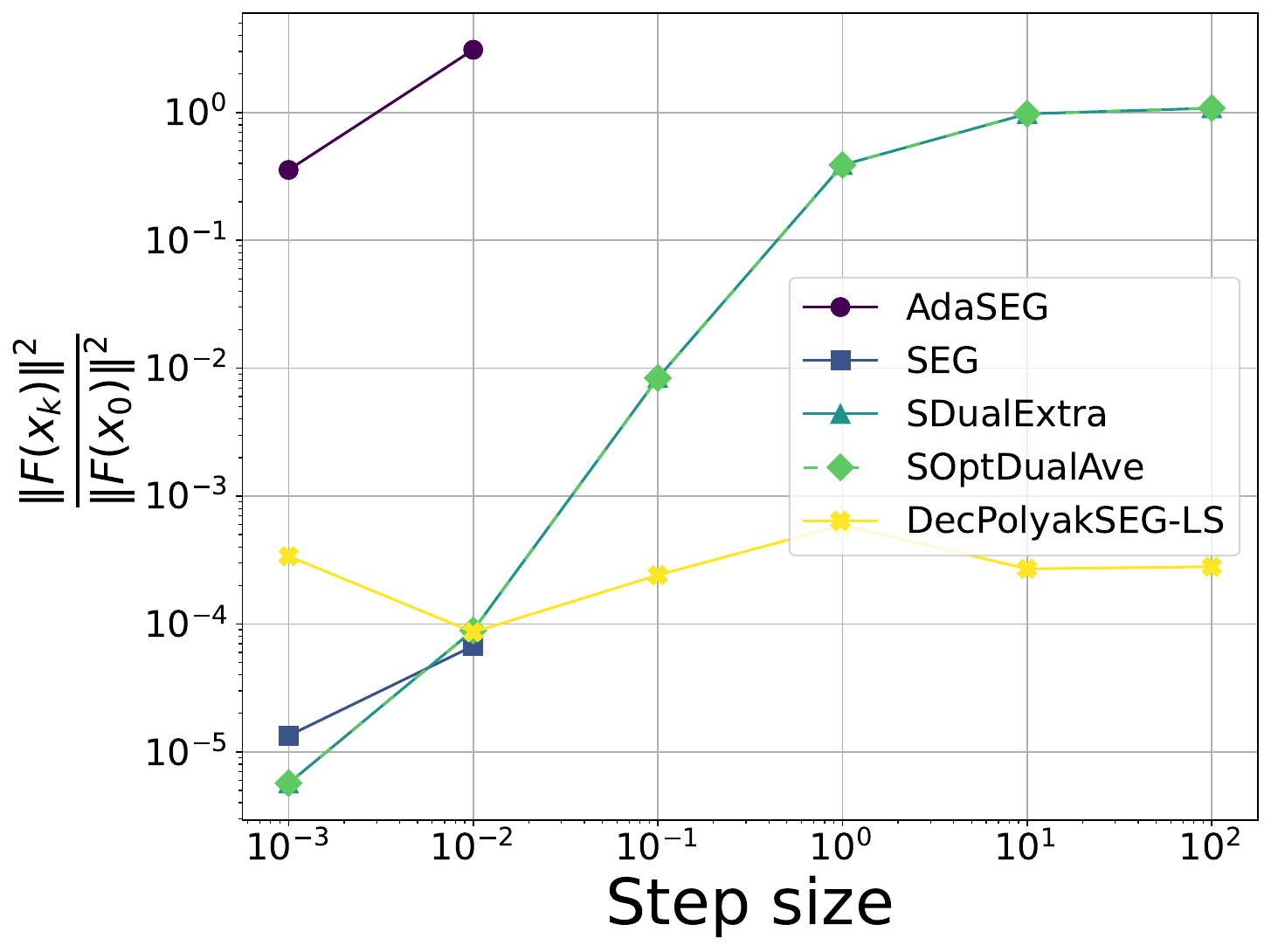}
    \caption{}\label{fig:DecPolyakSEGLS_RL}
\end{subfigure}
\begin{subfigure}[b]{0.3\textwidth}
    \centering
    \includegraphics[width=\textwidth]{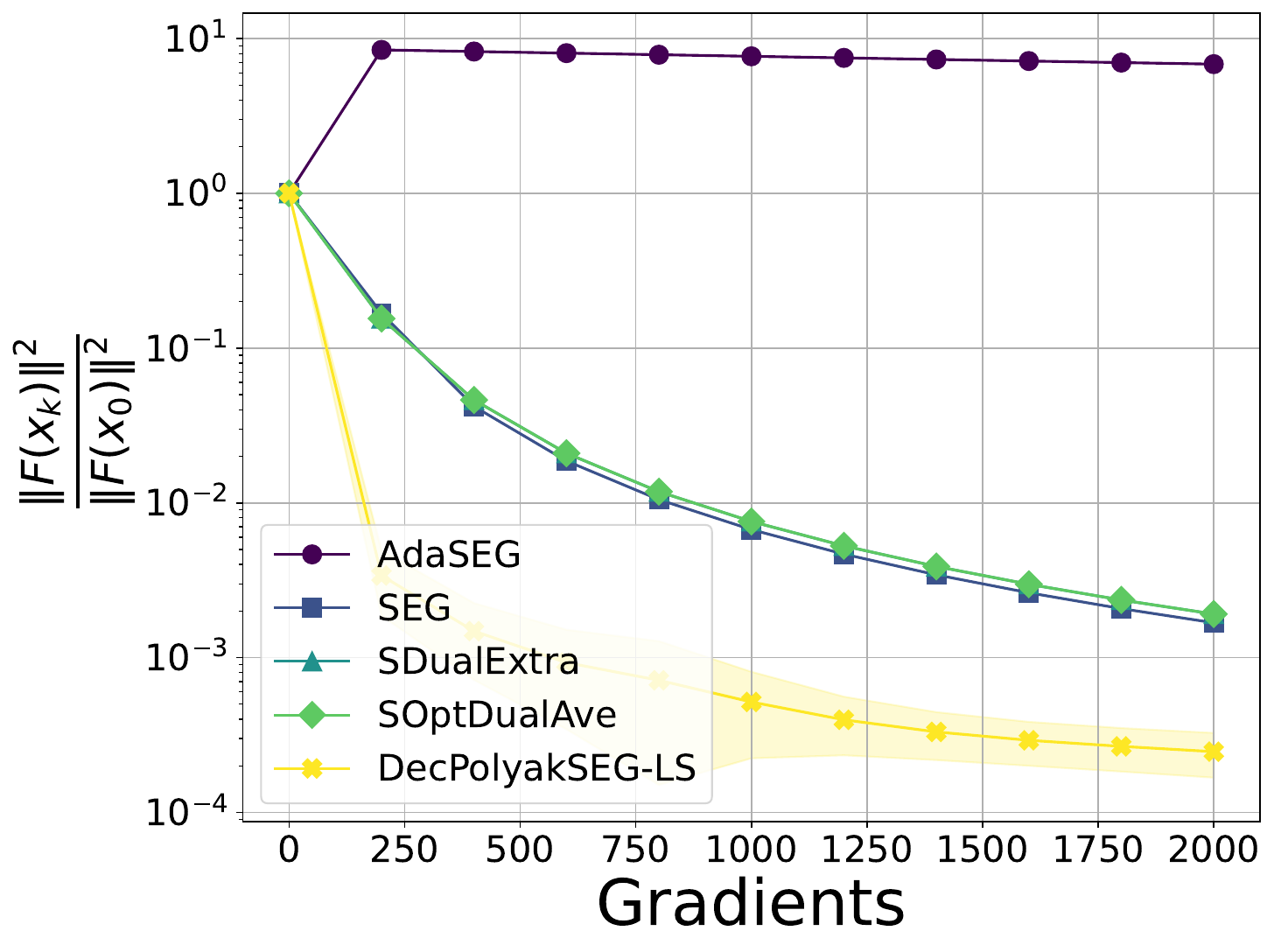}
    \caption{}\label{fig:DecPolyakSEGLS_RL_v3}
\end{subfigure}
\begin{subfigure}[b]{0.3\textwidth}
    \centering
    \includegraphics[width=\textwidth]{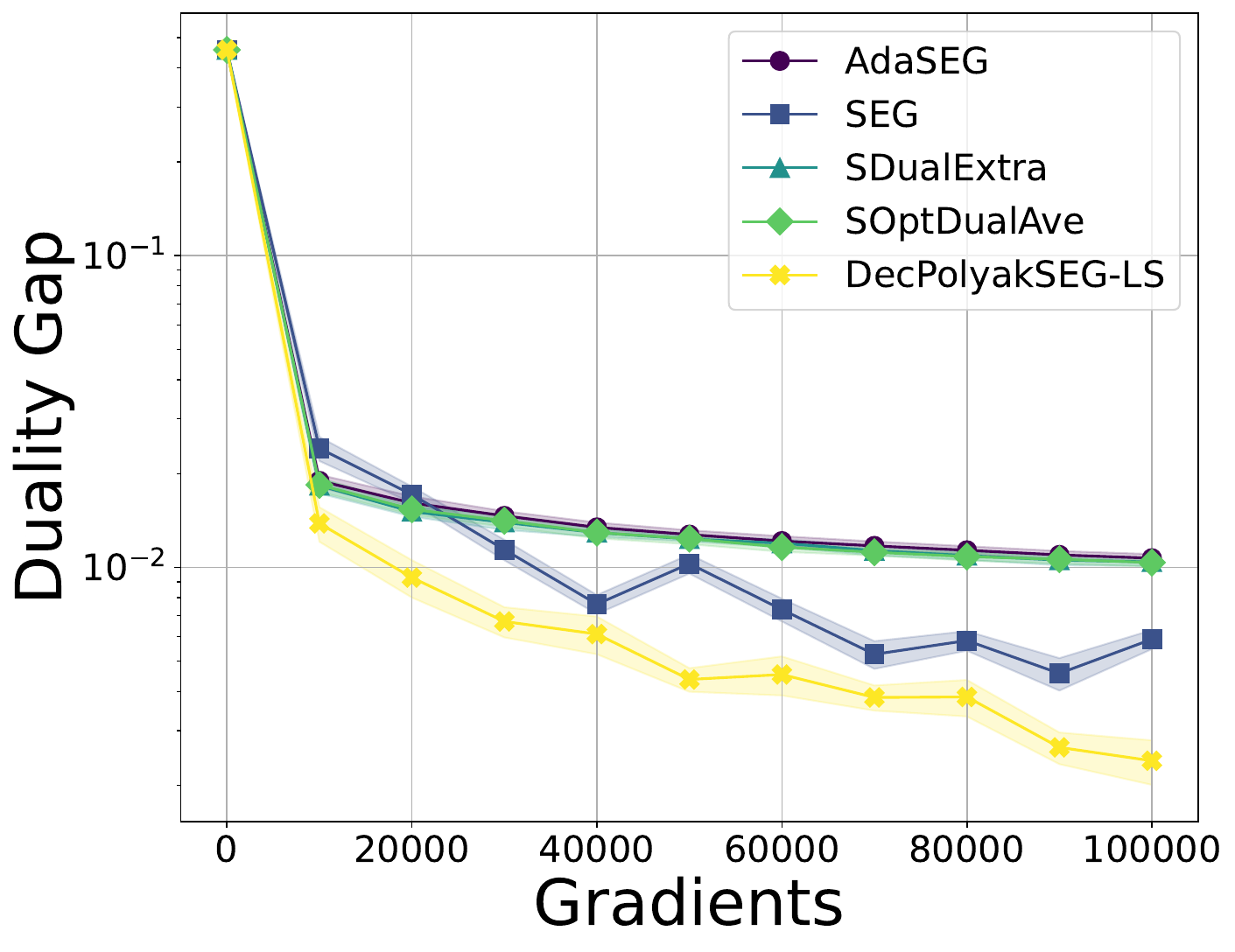}
    \caption{}\label{fig:DecPolyakSEGLS_matrix}
\end{subfigure}
    \caption{\small In Figure \ref{fig:PolyakEG_theoretical_gradients} and \ref{fig:PolyakEG_tuned_gradients}, we compare the performance of \algname{EG}, \algname{PolyakEG} and \algname{PolyakEG-LS} with theoretical and tuned step sizes respectively. In Figure \ref{fig:PolyakEGLS_CurvatureEG}, we compare the performance of \algname{PolyakEG-LS} with \algname{CurvatureEG+} and \algname{ProjectionEG}. In Figure \ref{fig:PolyakEGLS_upperbound}, we compare the number of while loop calls required by \hyperref[alg:PolyakEG_linesearch]{\hyperref[alg:PolyakEG_linesearch]{\algname{PolyakEG-LS}}} and the proposed theoretical upper bound $\left\lfloor \nicefrac{\log (L\gamma_{-1} / A)}{\log 1/\beta} \right\rfloor + 1$ from Theorem \ref{lemma:LS-step-size-lower-bound}.
    In Figure \ref{fig:PolyakSEGLS_interpolation_v1} we compare the performance of \algname{PolyakSEG-LS} with \algname{SEG} from \cite{vaswani2019painless}. 
    In Figure \ref{fig:DecPolyakSEGLS_RL}, we compare the performance of \algname{DecPolyakSEG-LS} with other stochastic algorithms for different step size initialization ($x$-axis) after $10^4$ gradient computations.
    In Figure \ref{fig:DecPolyakSEGLS_RL_v3} and \ref{fig:DecPolyakSEGLS_matrix}, we compare the performance of \algname{DecPolyakSEG-LS} with other stochastic algorithms on the robust least squares and matrix game problems.} 
    \vspace{-2.5mm}
\end{figure}

\subsection{Stochastic Setting}

\textbf{Comparison of \algname{PolyakSEG-LS} and \algname{SEG-LS}.} In Figure \ref{fig:PolyakSEGLS_interpolation_v1}, we compare the performance of \algname{PolyakSEG-LS} and \algname{SEG} with line-search from \cite{vaswani2019painless} on an interpolated model. For this, we consider the strongly convex- strongly concave quadratic min-max problem of \eqref{eq:quad_minmax} such that $\nabla_{w_1} \mathcal{L}_i(w_{1*}, w_{2*}) = \nabla_{w_2} \mathcal{L}_i(w_{1*}, w_{2*}) = 0$ (i.e. $F_i(x_*) = 0$). We compare \algname{PolyakSEG-LS} with two line-search strategies of \cite{vaswani2019painless}: \textbf{(1)} \algname{SEG-LS0}, which restarts line-search in every iteration from $\gamma_{-1}$ and \textbf{(2)} \algname{SEG-LS1}, which uses $\gamma_{k-1}$ from previous iteration similar to \algname{PolyakSEG-LS}. In Figure \ref{fig:PolyakSEGLS_interpolation_v1}, \algname{PolyakSEG-LS} enjoys linear convergence and outperforms both the methods from \cite{vaswani2019painless}.

\textbf{Evaluation of \algname{DecPolyakSEG-LS}.} In this experiment, we compare the performance of \algname{DecPolyakSEG-LS} with \algname{SEG}, \algname{SDualExtra}~\cite{antonakopoulos2021sifting}, \algname{SOptDualAve}~\cite{antonakopoulos2021sifting} and \algname{AdaSEG} (stochastic version of \cite{antonakopoulos2020adaptive}). In Figure \ref{fig:DecPolyakSEGLS_RL} and \ref{fig:DecPolyakSEGLS_RL_v3} we implement these algorithms for solving a robust least square problem while in Figure \ref{fig:DecPolyakSEGLS_matrix}, we solve a matrix game. In Figure \ref{fig:DecPolyakSEGLS_RL}, we analyze the robustness of these algorithms for different step size initilizations (on $x$-axis) similar to \cite{ward2020adagrad}.  For each of these algorithms, we start with a step size from the grid $\left\{10^{-3}, 10^{-2}, 10^{-1}, 1, 10, 10^{2} \right\}$ and plot $\nicefrac{\|F(\bar{x}_k)\|^2}{\|F(x_0)\|^2}$ after $10^4$ gradient computations. 
The performance of \hyperref[alg:decPolyakSEG_linesearch]{\algname{DecPolyakSEG-LS}} remains consistent over step size initializations, while the performance of other algorithms degrades with larger step size initlizations. 
This highlights that our proposed algorithm \hyperref[alg:decPolyakSEG_linesearch]{\algname{DecPolyakSEG-LS}} \textit{is robust to step size initialization.} In Figure \ref{fig:DecPolyakSEGLS_RL_v3} and \ref{fig:DecPolyakSEGLS_matrix}, we compare the performance of these algorithms in terms of $\nicefrac{\|F(\bar{x}_k)\|^2}{\|F(x_0)\|^2}$ and duality gap on $y$-axis. We observe that \hyperref[alg:decPolyakSEG_linesearch]{\algname{DecPolyakSEG-LS}} outperforms others in solving these problem. We present details related to these experiments in Appendix~\ref{sec:more_numerical_exp}.

\chapter{Extragradient Method for $(L_0, L_1)$-Lipschitz Root-finding Problems} \label{chap:chap-4}

\section{Introduction}

In this chapter, we focus on solving the root-finding problem of the form~\eqref{eq: Variational Inequality Definition} with the \algname{EG} algorithm. Despite a rich literature for analysing \algname{EG} and its variants, most of the existing convergence guarantees heavily rely on the $L$-Lipschitz assumption of the operator $F$~\cite{korpelevich1977extragradient, diakonikolas2021efficient}, i.e.
\begin{equation}\label{eq:L-Lipschitz}
\textstyle
    \|F(x) - F(y)\| \leq L\| x - y\|
\end{equation}
for all $x, y \in \R^d$. However, this assumption can be restrictive; for instance, the operator $F(x) = x^2$ for $x \in \R$ does not satisfy \eqref{eq:L-Lipschitz} for any finite $L$~\cite{zhang2019gradient}. \emph{The primary goal of this section is to relax this assumption and establish convergence guarantees for solving the root-finding problem of the form \eqref{eq: Variational Inequality Definition} under a more general framework.} 

\textbf{Relaxing the $L$-Lipschitz Assumption.} Recently, \cite{zhang2019gradient} introduced the $(L_0, L_1)$-smoothness assumption for the minimization problems. Specifically, for $\min_{x \in \R^d} f(x)$, \cite{zhang2019gradient} assume $ \| \nabla^2 f(x) \| \leq L_0 + L_1 \| \nabla f(x)\|$ (when $f$ is twice differentiable) and later~\cite{chen2023generalized} proved that this is equivalent to:
\begin{equation}\label{eq:(L0,L1)}
\textstyle
    \|\nabla f(x) - \nabla f(y)\| \leq (L_0 + L_1 \|\nabla f(x)\|) \|x - y\|.    
\end{equation}
\cite{zhang2019gradient} demonstrated that modern neural networks, such as LSTMs (Long Short-Term Memorys)~\cite{hochreiter1997long}, 
align with $(L_0, L_1)$-smoothness assumption rather than the traditional $L$-smoothness assumption (i.e. $\|\nabla f(x) - \nabla f(y)\| \leq L \|x - y\|$). Moreover, they used this assumption to justify why gradient clipping speeds up neural network training. Later, \cite{ahn2023linear} showed that similar trends hold for the transformer~\cite{vaswani2017attention} architecture. 

In Figure~\ref{fig:hessian_vs_gradient_plot}, we present an example demonstrating the validity of the $(L_0, L_1)$-smoothness condition~\eqref{eq:(L0,L1)} for the iterates of the \algname{EG}.

Similar plots have been presented for gradient descent methods~\cite{zhang2019gradient,gorbunov2024methods}, but to the best of our knowledge, this linear connection between $\|\nabla^2 f(x_k)\|$ and $\|\nabla f(x_k)\|$ for \algname{EG} iterates was never reported before. Following~\cite{gorbunov2024methods}, we consider the minimization problem $\min_{x \in \R^d} f(x) \eqdef \log \left(1 + \exp(- a^{\top}x) \right)$, and plot the values of $\|\nabla^2 f(x_k)\|$ on the $y$-axis against $\|\nabla f(x_k)\|$ on the $x$-axis. Each point is colored according to the iteration index $k$, as indicated by the accompanying colorbar. 

\begin{figure}
    \centering
    \includegraphics[width=.5\linewidth]{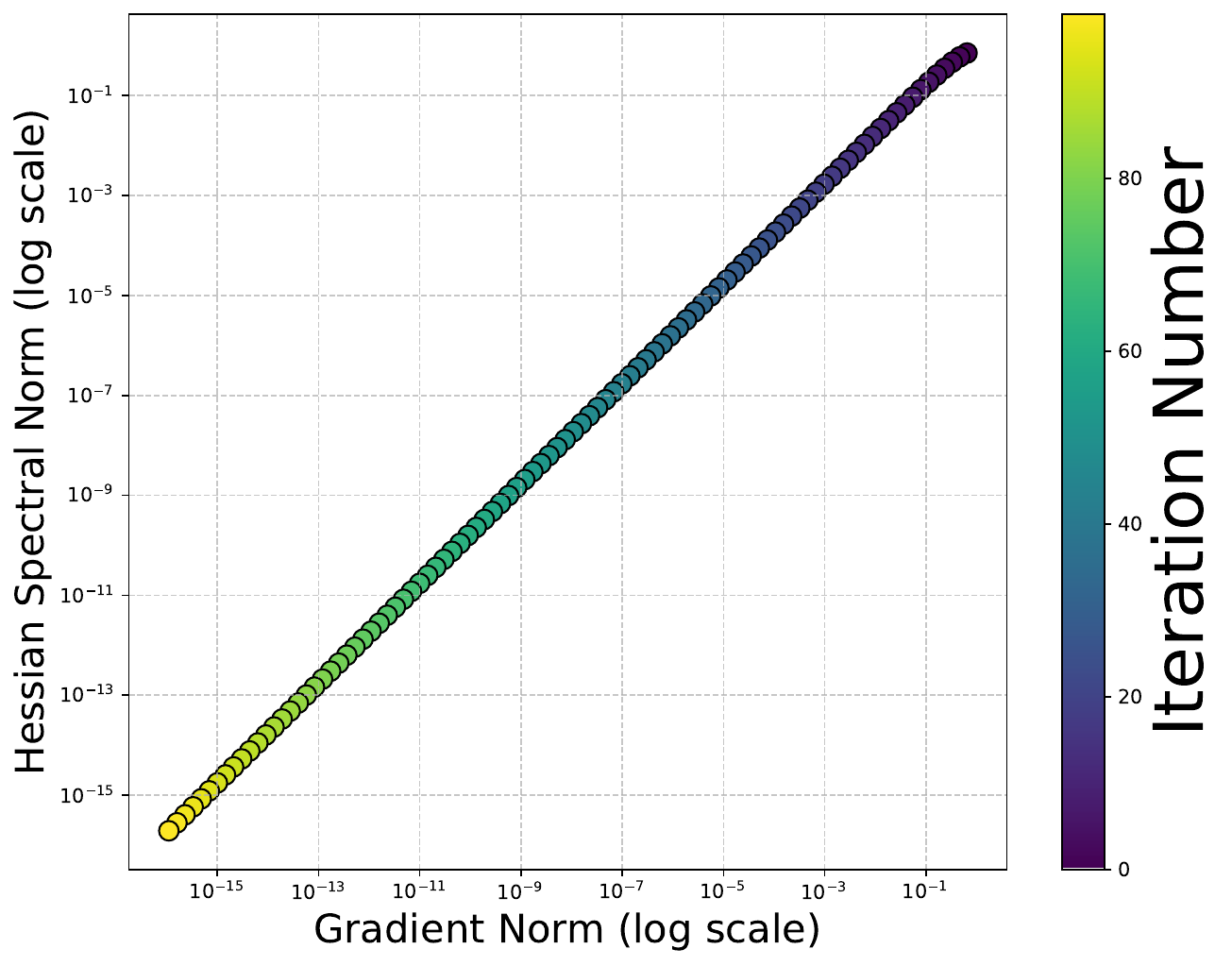}
        \caption{Scatter plot of $\| \nabla^2 f(x_k)\|$ on $y$-axis and $\|\nabla f(x_k)\|$ on $x$-axis.}\label{fig:hessian_vs_gradient_plot}
\end{figure}

The resulting plot reveals an approximately linear relationship between $\|\nabla^2 f(x_k)\|$ and $\|\nabla f(x_k)\|$, thereby supporting the modeling of this function within the $\| \nabla^2 f(x)\| \leq L_0 + L_1 \| \nabla f(x)\|$ or $(L_0, L_1)$-smoothness framework.

For a min-max optimization problem if the associaed operator $F$ is $L$-Lipschitz, then its Jacobian matrix $\mathbf{J}(x)$, defined in \eqref{eq:jacobian}, satisfies $\|\mathbf{J}(x)\| \leq L$ for all $x = (w_1^{\top}, w_2^{\top})^{\top}$(follows from Theorem \ref{thm:equiv_formulation} with $L_1 = 0$). For example, consider the quadratic min-max objective $\min_{w_1} \max_{w_2} \mathcal{L}(w_1, w_2) = \frac{1}{2}w_1^2 + w_1 w_2 - \frac{1}{2}w_2^2$. In this case, implementing the \algname{EG} method and plotting the Jacobian norm $\| \mathbf{J}(x_k) \|$ (on the $y$-axis) against the operator norm $\|F(x_k) \|$ (on the $x$-axis) yields a horizontal line parallel to $x$-axis (check Appendix \ref{appendix:tech_lemma}). 
\begin{figure}
    \centering
    \includegraphics[width=.5\linewidth]{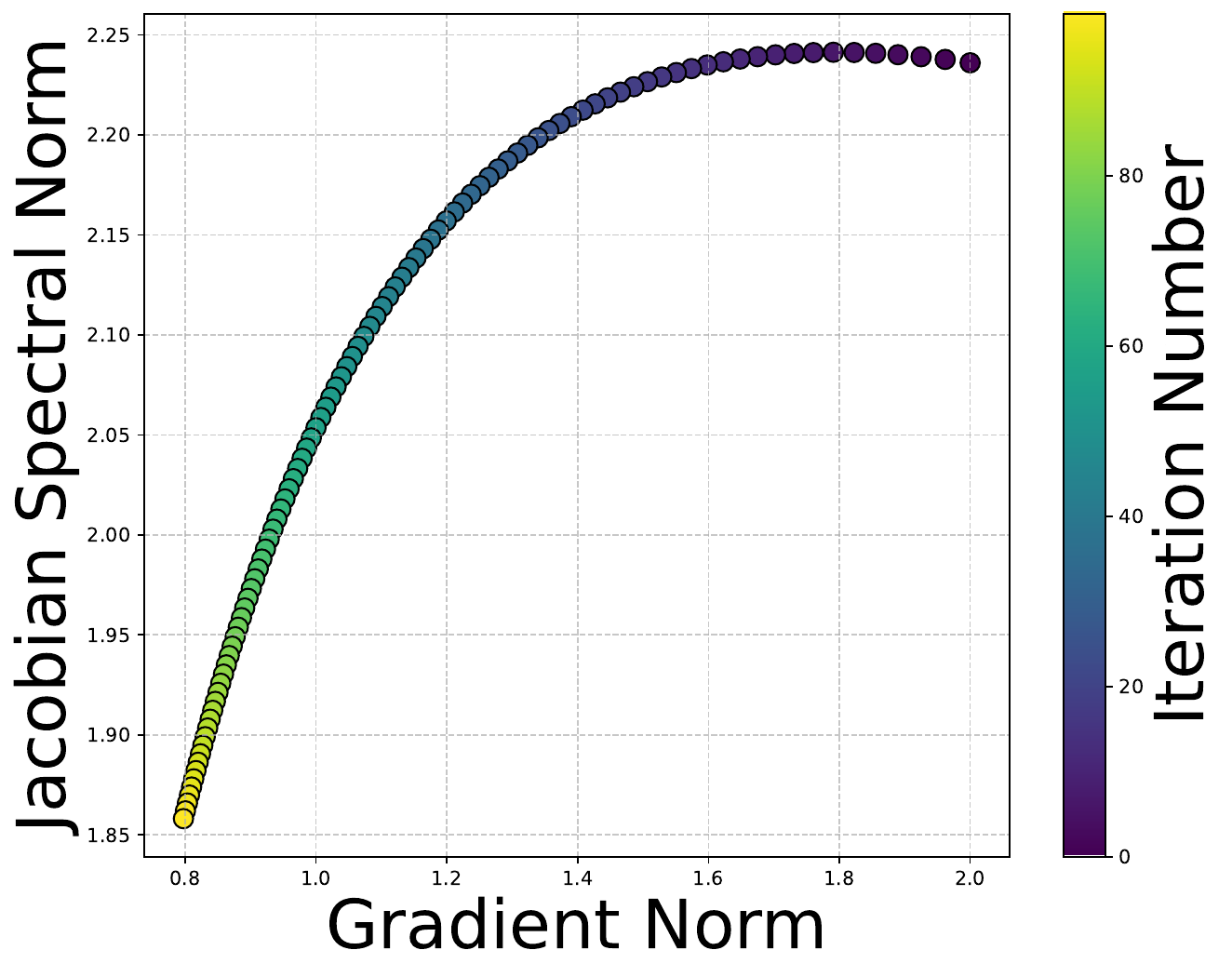}
        \caption{Scatter plot of $\| \mathbf{J}(x_k)\|$ on $y$-axis and $\| F(x_k)\|$ on $x$-axis.}\label{fig:jacobian_vs_operator_plot}
\end{figure}
 However, this behaviour does not persist for more complex problems. For instance, for the cubic objective
\begin{equation}\label{eq:min_max_cubic}
    \min_{w_1} \max_{w_2} \mathcal{L}(w_1, w_2) = \frac{1}{3}w_1^3 + w_1w_2 - \frac{1}{3}w_2^3,
\end{equation}
$\|\mathbf{J}(x_k)\|$ increases with the $\|F(x_k)\|$. This observation suggests that the standard Lipschitz assumption may be overly restrictive for capturing the structure of such problems (check Figure \ref{fig:jacobian_vs_operator_plot}).
        
To better model this relationship, we investigate a relaxed condition of the form $\| \mathbf{J}(x) \| \leq L_0 + L_1 \|F(x)\|^\alpha$ with
$\alpha \in (0, 1]$ which generalizes the standard Lipschitz bound (for $L_1 = 0$, this boils down to $\|\mathbf{J}(x)\| \leq L_0$, which is the Lipschitz property). Note that, instead of $\alpha = 1$, our formulation permits $\alpha$ to lie in the broader interval $(0, 1]$. This condition is motivated by the plot in Figure~\ref{fig:jacobian_vs_operator_plot}, which suggests a sublinear relationship, resembling the form $h(r) = L_0 + L_1 r^{\alpha}$ for some $\alpha \in (0, 1)$, rather than a linear trend. 

As we will prove later (in Theorem~\ref{thm:equiv_formulation}), for any  doubly differentiable min-max optimization problem $\min_{w_1} \max_{w_2} \mathcal{L}(w_1, w_2)$, the condition $\| \mathbf{J}(x) \| \leq L_0 + L_1 \| F(x) \|^{\alpha}$ is equivalent to the operator $F$ satisfying the $\alpha$-symmetric $(L_0, L_1)$-Lipschitz condition (see Assumption~\ref{asdioa}). 
The equivalent $\alpha$-symmetric $(L_0, L_1)$-Lipschitz condition does not rely on second-order information and applies to a broader class of problems (no need for double differentiability). Therefore, in the remainder of this chapter, we focus on analyzing the convergence of \algname{EG} under the $\alpha$-symmetric $(L_0, L_1)$-Lipschitz assumption~\cite{pmlr-v235-vankov24a} on the operator $F$, defined below.

\begin{assumption}
\label{asdioa}
$F$ is called $\alpha$-symmetric $(L_0, L_1)$-Lipschitz operator if for some $L_0, L_1 \geq 0$ and $\alpha \in (0, 1]$,
    \begin{equation}\label{eq:(L0,L1)-Lipschitz}
        \| F(x) - F(y)\| \leq \left( L_0 + L_1 \max_{\theta \in [0, 1]} \left\|F(\theta x + ( 1- \theta) y) \right\|^{\alpha} \right) \| x- y\| \quad \forall x, y \in \R^d.
    \end{equation}
\end{assumption}
Instead of a fixed Lipschitz constant in \eqref{eq:L-Lipschitz}, Assumption~\ref{asdioa} allows the Lipschitz-like quantity to depend on the norm of the operator itself along the path from $x$ to $y$. This assumption generalizes the standard $L$-Lipschitz condition \eqref{eq:L-Lipschitz}, corresponding to the special case where $L_0 = L$ and $L_1 = 0$. 

Moreover, the $\alpha$-symmetric $(L_0, L_1)$-Lipschitz condition provides a more refined characterisation of operators whose Lipschitz constant depends on their norm, offering a tighter bound by balancing $L_0 \ll L$ and $L_1 \ll L$. Additionally, \eqref{eq:(L0,L1)-Lipschitz} provides a more relaxed bound compared to \eqref{eq:(L0,L1)} with $\alpha = 1$.

\begin{table*}[htbp]
    \centering
    \caption{
       Summary of step size selection for \algname{EG} under the $L$-Lipschitz and $\alpha$-symmetric $(L_0, L_1)$-Lipschitz assumptions. Our proposed step size strategy is of the general form $\gamma_k = \frac{1}{c_0 + c_1 \|F(x_k)\|^{\alpha}}$, tailored for solving problems involving $\alpha$-symmetric $(L_0, L_1)$-Lipschitz operators.
    }
    \label{tab:comparison_of_rates_41}
    \begin{threeparttable}
    \resizebox{\textwidth}{!}{%
        \begin{tabular}{|c|c|c c c|}
        \hline
        Setup & Assumption & $\alpha$ & $\gamma_k$ & $\omega_k$ 
        \\
        \hline\hline
        \multirow{7}{2cm}{\centering Strongly \\ Monotone \tnote{{\color{blue}(1)}}} & \begin{tabular}{c}
            $L$-Lipschitz \tnote{{\color{blue}(2)}}\\
            \cite{mokhtari2020unified}
        \end{tabular} & - & $\frac{0.25}{L}$ & $\gamma_k$\\[10pt]
        
        & \begin{tabular}{c}
            $\alpha$-symmetric $(L_0, L_1)$-Lipschitz \\
            \cite{pmlr-v235-vankov24a} 
        \end{tabular} & $1$ & $\min \left\{ \frac{1}{4 \mu}, \frac{1}{2\sqrt{2} e L_0}, \frac{1}{2 \sqrt{2} e L_1 \|F(x_k)\|} \right\}$ & $\gamma_k$ \\[10pt]
      
        &\cellcolor{bgcolor2}\begin{tabular}{c}
            $\alpha$-symmetric $(L_0, L_1)$-Lipschitz \\
            (Theorem \ref{theorem:1symm_strongmonotone})
        \end{tabular} & \cellcolor{bgcolor2} $1$ & \cellcolor{bgcolor2} $\frac{0.21}{L_0 + L_1 \|F(x_k)\|}$ & \cellcolor{bgcolor2} $\gamma_k$ \\[10pt]
        &\cellcolor{bgcolor2}\begin{tabular}{c}
            $\alpha$-symmetric $(L_0, L_1)$-Lipschitz\\
            (Theorem \ref{theorem:1symm_strongmonotone_alhpha01})
        \end{tabular} & \cellcolor{bgcolor2} $(0, 1)$ & \cellcolor{bgcolor2} $\frac{0.61}{2K_0 + \left( 2 K_1 + 2^{1 - \alpha} K_2^{1 - \alpha} \right) \| F(x_k)\|^{\alpha}}$ \tnote{{\color{blue}(2)}} & \cellcolor{bgcolor2} $\gamma_k$ \\[10pt]

        \hline\hline
        \multirow{6}{2cm}{\centering Monotone \tnote{{\color{blue}(1)}}} & \begin{tabular}{c}
            $L$-Lipschitz\\
            \cite{gorbunov2022extragradient}
        \end{tabular} & - & $\frac{1}{L}$  & $\frac{\gamma_k}{2}$ \\[10pt]
    
        &\cellcolor{bgcolor2}\begin{tabular}{c}
            $\alpha$-symmetric $(L_0, L_1)$-Lipschitz\\
            (Theorem \ref{theorem:1symm_monotone})
        \end{tabular} & \cellcolor{bgcolor2} $1$ & \cellcolor{bgcolor2} $\frac{0.45}{L_0 + L_1 \|F(x_k)\|}$ & \cellcolor{bgcolor2} $\gamma_k$  \\[10pt]

        &\cellcolor{bgcolor2}\begin{tabular}{c}
            $\alpha$-symmetric $(L_0, L_1)$-Lipschitz\\
            (Theorem \ref{theorem:alpha01})
        \end{tabular} & \cellcolor{bgcolor2} $(0, 1)$ & \cellcolor{bgcolor2} $\frac{1}{2 \sqrt{2} K_0  + \left( 2\sqrt{2} K_1 + 2^{\nicefrac{3 (1 - \alpha)}{2}} K_2^{1 - \alpha} \right) \|F(x_k)\|^{\alpha}}$ & \cellcolor{bgcolor2} $\gamma_k$ \\[10pt]

        \hline\hline
        \multirow{8}{2cm}{\centering Weak Minty \tnote{{\color{blue}(1)}}} & \begin{tabular}{c}
            $L$-Lipschitz\\
            \cite{diakonikolas2021efficient}
        \end{tabular} & - & $\frac{1}{L}$ & $\frac{\gamma_k}{2}$ \\[10pt]
        & \begin{tabular}{c}
            $L$-Lipschitz\\
            \cite{pethick2023escaping}
        \end{tabular} & - & $\frac{1}{L}$ & $\rho + \frac{\la F(\hx_k), x_k - \hx_k\ra}{\| F(\hx_k)\|^2}$ \\[10pt]

        &\cellcolor{bgcolor2}\begin{tabular}{c}
            $\alpha$-symmetric $(L_0, L_1)$-Lipschitz\\
            (Theorem \ref{theorem:weak_minty_alpha1})
        \end{tabular} & \cellcolor{bgcolor2} $1$ & \cellcolor{bgcolor2} $\frac{0.56}{L_0 + L_1 \|F(x_k)\|}$ & \cellcolor{bgcolor2} $\frac{\gamma_k}{2}$ \\[10pt]
        
        & \cellcolor{bgcolor2}\begin{tabular}{c}
            $\alpha$-symmetric $(L_0, L_1)$-Lipschitz \\
            (Theorem \ref{theorem:weak_minty_alpha01})
        \end{tabular} & \cellcolor{bgcolor2} $(0, 1)$ & \cellcolor{bgcolor2} $\frac{1}{2 \sqrt{2} K_0 + \left(2 \sqrt{2}K_1 + 2^{\nicefrac{3 (1 - \alpha)}{2}} K_2^{1 - \alpha} \right)\| F(x_k)\|^{\alpha}}$ &\cellcolor{bgcolor2} $\frac{\gamma_k}{2}$ \; \\[10pt]
        \hline
    \end{tabular}%
    }
    \begin{tablenotes}
        {\scriptsize 
        \item [{\color{blue}(1)}] Convergence measure $\bullet$ strongly monotone: $\|x_K - x_*\|^2$, \\
        $\bullet$ monotone: $ \min_{0 \leq k \leq K}\|F(x_k)\|^2$, \\
        $\bullet$ weak minty: $ \min_{0 \leq k \leq K}\|F(\hx_k)\|^2$.
        \item [{\color{blue}(2)}] For $K_0, K_1, K_2$, check Proposition \ref{prop:equiv_formulation}. Note that, for $L_1 = 0$ we have $K_1 = K_2 = 0$.
        }
        \vspace{-4mm}
    \end{tablenotes}
    \end{threeparttable}
\end{table*}

\subsection{Main Contributions}
We summarize the main contributions of this chapter below.
\begin{itemize}[itemsep=0pt, leftmargin=*]

    \item \textbf{Tighter analysis for strongly monotone:} We establish linear convergence guarantees for strongly monotone~\eqref{eq:strong_monotone} $\alpha$-symmetric $(L_0, L_1)$-Lipschitz problems (see Theorem \ref{theorem:1symm_strongmonotone}, \ref{theorem:1symm_strongmonotone_alhpha01}). In contrast to the results in~\cite{pmlr-v235-vankov24a} for $\alpha = 1$, our analysis shows that linear convergence can be achieved without incurring exponential dependence on the initial distance to the solution $\|x_0 - x_*\|$ (see Corollary~\ref{corollary:1symm_strongmonotone}).

    \item \textbf{First analysis for monotone and weak Minty:} We provide the first convergence analysis of \algname{EG} for solving monotone~\eqref{eq:monotone} and weak Minty \eqref{eq: weak MVI} problems under $\alpha$-symmetric $(L_0, L_1)$-Lipschitz assumption. We establish global sublinear convergence for monotone problems (see Theorem \ref{theorem:1symm_monotone}, \ref{theorem:alpha01}) and local sublinear convergence for weak Minty problems (see Theorem \ref{theorem:weak_minty_alpha1}, \ref{theorem:weak_minty_alpha01}).

    \item \textbf{Novel step size for \algname{EG}:} We propose a novel adaptive step-size strategy for the \algname{EG} method designed to handle $\alpha$-symmetric $(L_0, L_1)$-Lipschitz operators. Specifically, all our step-size schemes adopt the general form $\gamma_k = \frac{1}{c_0 + c_1 \|F(x_k)\|^{\alpha}},$ where $c_0, c_1 > 0$ are constants determined by the problem-dependent parameters $L_0$, $L_1$, and $\alpha$. In Table \ref{tab:comparison_of_rates_41}, we included a detailed summary of our proposed step size selection for different classes of problems and compared it with closely related works. 
    
    \item \textbf{Numerical experiments:} 
    Finally, in Section~\ref{sec:experiments}, we present experiments validating different aspects of our theoretical results. We compare our proposed step size selections with existing alternatives, demonstrating the effectiveness and robustness of our approach.

\end{itemize}

\section{On the $\alpha$-Symmetric $(L_0, L_1)$-Lipschitz Assumption}
We divide this section into two parts. In the first subsection, we present an equivalent reformulation of the $\alpha$-symmetric $(L_0, L_1)$-Lipschitz condition~\eqref{eq:(L0,L1)-Lipschitz} in the context of min-max optimization. In the second subsection, we provide some examples of operators that satisfy \eqref{eq:(L0,L1)-Lipschitz} and highlight its significance.

\subsection{Equivalent Formulation of $\alpha$-Symmetric $(L_0, L_1)$-Lipschitz Assumption}
In this subsection, we consider the min-max optimization problem $\min_{w_1} \max_{w_2} \mathcal{L}(w_1, w_2)$. The corresponding operator $F$ and Jacobian $\mathbf{J}$ are defined as
\begin{equation}\label{eq:jacobian}
\textstyle
    F(x) = \begin{bmatrix}
        \nabla_{w_1} \mathcal{L}(w_1, w_2) \\
        - \nabla_{w_2} \mathcal{L}(w_1, w_2)
    \end{bmatrix} \text{ and  } 
    \mathbf{J}(x) = \begin{bmatrix}
    \nabla^2_{w_1 w_1} \mathcal{L}(w_1, w_2) & \nabla^2_{w_2 w_1} \mathcal{L}(w_1, w_2) \\
    -\nabla^2_{w_1 w_2} \mathcal{L}(w_1, w_2) & -\nabla^2_{w_2 w_2} \mathcal{L}(w_1, w_2)
    \end{bmatrix},
\end{equation}
where $x = (w_1^{\top}, w_2^{\top})^{\top}$. Assuming that $F$ is $\alpha$-symmetric $(L_0, L_1)$-Lipschitz, we obtain the following theorem.

\begin{theorem}\label{thm:equiv_formulation}
    Suppose $F$ is the differentiable operator associated with the problem $\min_{w_1} \max_{w_2} \mathcal{L}(w_1, w_2)$. Then $F$ satisfies the $\alpha$-symmetric $(L_0, L_1)$-Lipschitz condition~\eqref{eq:(L0,L1)-Lipschitz} if and only if
    \begin{equation}\label{eq:equiv_formulation}
    \textstyle
        \|\mathbf{J}(x)\| \leq L_0 + L_1 \|F(x)\|^{\alpha}.
    \end{equation}
    Here $\mathbf{J}(x)$ is the Jacobian defined in \eqref{eq:jacobian} and $\| \mathbf{J}(x)\| = \sigma_{\max}(\mathbf{J}(x))$ i.e. maximum singular value of $\mathbf{J}(x)$. In particular, we have $\| \mathbf{J}(x)\| \leq L$ when operator $F$ is $L$-Lipschitz.
\end{theorem}
This result provides an equivalent characterization of the $\alpha$-symmetric $(L_0, L_1)$-Lipschitz condition~\eqref{eq:(L0,L1)-Lipschitz} for min-max optimization problems. In practice, it is often easier to verify~\eqref{eq:equiv_formulation} than to directly check~\eqref{eq:(L0,L1)-Lipschitz}. In Appendix~\ref{appendix:equiv_formulation}, we provide an example where we use Theorem \ref{thm:equiv_formulation} to verify if an operator satisfies \eqref{eq:(L0,L1)-Lipschitz}. 

\subsection{Examples of $\alpha$-Symmetric $(L_0, L_1)$-Lipschitz Operators}\label{sec:examples}
To motivate the significance of this relaxed assumption~\eqref{eq:(L0,L1)-Lipschitz}, we present a few instances of $\alpha$-symmetric $(L_0, L_1)$-Lipschitz operators that highlight its advantages over the conventional $L$-Lipschitz assumption.\\
\newline
\textbf{Example 1}\cite{gorbunov2024methods}: We start with an example from the minimization setting. Consider the logistic regression loss function $f(x) = \log \left( 1 + \exp{\left( - a^{\top}x \right)}\right)$. Then the corresponding gradient operator $F = \nabla f$ satisfies the $L$-Lipschitz assumption with $L = \|a\|^2$ and $\alpha$-symmetric $(L_0, L_1)$-Lipschitz assumption with $L_0 = 0, L_1 = \|a\|, \alpha = 1$. Therefore, when $\|a\| \gg 1$, the bound provided by $1$-symmetric $(L_0, L_1)$-Lipschitz can be significantly tighter than the one imposed by the $L$-Lipschitz condition. This example emphasizes the benefit of the $\alpha$-symmetric $(L_0, L_1)$-Lipschitz framework in scenarios where standard Lipschitz constants are overly pessimistic.

\textbf{Example 2}: Consider the operator $F(x) = (u_1^2, u_2^2)$ for $x = (u_1, u_2) \in \R^2$ with $x_* = (0, 0)$. Then we can show that
\begin{eqnarray*}
    \|F(x) - F(y)\| & \leq & 2 \left\|F \left(\frac{x+y}{2} \right) \right\|^{\nicefrac{1}{2}} \|x - y\| \\
    & \leq & 2 \left\| \max_{\theta \in [0, 1]} F \left(\theta x + (1 - \theta) y \right) \right\|^{\nicefrac{1}{2}} \|x - y\|.
\end{eqnarray*}
This establishes that $F$ is $\nicefrac{1}{2}$-symmetric $(0, 2)$-Lipschitz operator. However, this operator $F$ does not satisfy the standard $L$-Lipschitz assumption for any finite choice of $L$. We add the related details to Appendix \ref{appendix:examples}. Therefore, this example highlights the need for relaxed assumptions on operators beyond standard $L$-Lipschitz~\eqref{eq:L-Lipschitz}.

We provide additional examples in Appendix~\ref{appendix:examples} to illustrate cases where the operator associated with a bilinearly coupled min-max optimization problem or an $N$-player game satisfies the $\alpha$-symmetric $(L_0, L_1)$-Lipschitz condition.

\section{Convergence Analysis}\label{sec:convergence_analysis}
In this section, we present the convergence guarantees of \algname{EG} for solving monotone, strongly monotone, and weakly Minty operators. For strongly monotone operators, we have linear convergence, while for monotone and weak Minty operators, we provide sublinear convergence guarantees. To prove these results, we rely on the similar expression presented in~\cite{chen2023generalized} for the $(L_0, L_1)$-smooth minimization problem. For completeness, we include the proof for $\alpha$-symmetric $(L_0, L_1)$-Lipschitz operators in the Appendix. 

\begin{proposition} \label{prop:equiv_formulation}
    Suppose $F$ is $\alpha$-symmetric $(L_0, L_1)$-Lipschitz operator. Then, for $\alpha = 1$
    \begin{equation}\label{eq:alpha=1}
        \| F(x) - F(y) \| \leq (L_0 + L_1 \| F(x)\|) \exp{(L_1 \|x - y\|)} \| x- y \|,
    \end{equation}
    and for $\alpha \in (0, 1)$ we have
    \begin{equation}\label{eq:alpha(0,1)}
        \| F(x) - F(y) \| \leq \left(K_0 + K_1 \|F(x)\|^{\alpha} + K_2 \|x - y\|^{\nicefrac{\alpha}{1 - \alpha}} \right) \|x - y\|
    \end{equation}
    where $K_0 = L_0 (2^{\nicefrac{\alpha^2}{1 - \alpha}} + 1)$, $K_1 = L_1 \cdot 2^{\nicefrac{\alpha^2}{1 - \alpha}}$ and $K_2 = L_1^{\nicefrac{1}{1 - \alpha}} \cdot 2^{\nicefrac{\alpha^2}{1 - \alpha}} \cdot 3^{\alpha} (1 - \alpha)^{\nicefrac{\alpha}{1 - \alpha}}$.
\end{proposition}
Proposition \ref{prop:equiv_formulation} eliminates the maximum over $\theta \in [0, 1]$ in \eqref{eq:(L0,L1)-Lipschitz} and provides a simpler upper bound on the $\|F(x) - F(y)\|$. We divide the rest of the section into three subsections based on the structure of operators. Moreover, each of these subsections is divided into two parts depending on the value of $\alpha$, i.e. $\alpha = 1$ and $\alpha \in (0, 1)$.

\subsection{Convergence Guarantees for Strongly Monotone Operators}
In case of strongly monotone operators~\eqref{eq:strong_monotone}, we achieve linear convergence rates, analogous to those obtained under standard $L$-Lipschitz assumptions~\cite{tseng1995linear, mokhtari2020unified}. For $\alpha = 1$, operator $F$ satisfies the condition \eqref{eq:alpha=1}. To guarantee convergence for this class of operators, we use the \algname{EG} with step size $\gamma_k = \omega_k = \nicefrac{\nu}{\left( L_0 + L_1 \|F(x_k)\| \right)}$ and $\nu > 0$. \cite{gorbunov2024methods} used similar step sizes for the Gradient Descent algorithm to solve convex minimization problems. 
\begin{theorem}\label{theorem:1symm_strongmonotone}
    Suppose $F$ is $\mu$-strongly monotone and $1$-symmetric $(L_0, L_1)$-Lipschitz operator. Then \algname{EG} with step size $\gamma_k = \omega_k = \frac{\nu}{L_0 + L_1 \|F(x_k)\|}$ satisfy
    \begin{eqnarray*}
        \|x_{k+1} - x_*\|^2 & \leq & \left( 1 - \frac{\nu \mu}{L_0 \left( 1 + L_1 \exp{\left( L_1 \|x_0 - x_*\|\right)} \|x_0 - x_* \|\right)} \right)^{k+1} \|x_0 - x_*\|^2
    \end{eqnarray*}
    where $\nu > 0$ satisfy $ 1 - 2 \nu - \nu^2 \exp{2\nu} = 0$.
\end{theorem}
The equation $1 - 2 \nu - \nu^2 \exp{(2\nu)} = 0$ admits a positive solution, approximately $\nu \approx 0.363$. 

Specifically, to ensure $\|x_K - x_*\|^2 \leq \varepsilon$, we require 
$$K = \mathcal{O} \left( \left(\frac{L_0}{\mu} + \frac{L_0 L_1 \|x_0 - x_*\| \exp{ \left(L_1\|x_0 - x_*\| \right)}}{\mu} \right) \log \frac{1}{\varepsilon} \right)$$
iterations. When $L_1 = 0$, we recover the best-known results for the strongly monotone $L$-Lipschitz setting~\cite{tseng1995linear, mokhtari2020unified}. \cite{pmlr-v235-vankov24a} also studied constrained strongly monotone problems and obtained similar guarantees with an alternative step size scheme.

However, using a refined proof technique, we can eliminate the $\exp(L_1\|x_0 - x_*\|)$ term from the convergence rate and establish a tighter bound.

One of the intermediate steps of Theorem \ref{theorem:1symm_strongmonotone} is proving a lower bound on the step size $\gamma_k$, which can be very loose for large $k$.

We now show that after a certain number of iterations $K'$~\eqref{eq:1symm_strongmonoton_noexp_eq2}, the operator norm satisfies $\|F(x_k)\| \leq \nicefrac{L_0}{L_1}$ for all $k \geq K'$, which implies $\gamma_k = \omega_k \geq \nicefrac{\nu}{2L_0}$ for all $k \geq K'$.

\begin{corollary}\label{corollary:1symm_strongmonotone}
    Suppose $F$ is a $\mu$-strongly monotone and $1$-symmetric $(L_0, L_1)$-Lipschitz operator. Then, \algname{EG} with step sizes $\gamma_k = \omega_k = \frac{\nu}{L_0 + L_1 \|F(x_k)\|}$ guarantees $\|x_{K+1} - x_*\|^2 \leq \varepsilon$ after at most
    \begin{equation}\label{corollary:1symm_strongmonotone_eq1}
        K = \underbrace{\frac{2 L_0}{\nu \mu} \log \left( \frac{\|x_0 - x_*\|^2}{\varepsilon}\right)}_{\text{Term I}} + \underbrace{\frac{1}{\zeta \mu} \log \left( \frac{2L_1 \|x_0 - x_*\|^2}{\zeta^2 L_0} \right)}_{\text{Term II}} 
    \end{equation}
    iterations, where we have $$\zeta \coloneqq \frac{\nu}{L_0 \left( 1 + L_1 \exp\left( L_1 \|x_0 - x_*\|\right) \|x_0 - x_* \| \right)},$$
    and $\nu > 0$ satisfies $1 - 4 \nu - 2\nu^2 \exp(2\nu) = 0$. 
\end{corollary}
This result shows that to reach an accuracy of $\varepsilon > 0$, we need \eqref{corollary:1symm_strongmonotone_eq1} iterations. Importantly, Term II in \eqref{corollary:1symm_strongmonotone_eq1} is independent of $\varepsilon$, and the Term I of \eqref{corollary:1symm_strongmonotone_eq1} no longer depends on $\exp(L_1\|x_0 - x_*\|)$. Technically, Term II corresponds to the number of iterations required for the step sizes $\gamma_k$ and $\omega_k$ to exceed $\nicefrac{\nu}{2L_0}$, while Term I captures the iteration complexity of \algname{EG} with a fixed step size $\nicefrac{\nu}{2L_0}$.

Now, we investigate the behavior of $\alpha$-symmetric $(L_0, L_1)$-Lipschitz operators for $0 < \alpha < 1$. In this regime, we adopt a step size of the order $\mathcal{O}\left( \|F(x_k)\|^{-\alpha} \right)$ and prove the following result.

\begin{theorem}\label{theorem:1symm_strongmonotone_alhpha01}
    Suppose $F$ is $\mu$-strongly monotone and $\alpha$-symmetric $(L_0, L_1)$-Lipschitz operator with $\alpha \in (0, 1)$. Then \algname{EG} with $\gamma_k = \omega_k = \frac{\nu}{2K_0 + \left( 2 K_1 + 2^{1 - \alpha} K_2^{1 - \alpha} \right) \| F(x_k)\|^{\alpha}}$ satisfy
    \begin{eqnarray*}
        \|x_{k+1} - x_*\|^2 \leq \left( 1 - \varrho \right)^{k+1} \|x_0 - x_*\|^2
    \end{eqnarray*}
    where $$\varrho = \frac{\nu \mu}{2 K_0 + (2 K_1 + 2^{1 - \alpha} K_2^{1 - \alpha})(K_0 + K_2 \|x_0 - x_*\|^{\nicefrac{\alpha}{1 - \alpha}})^{\alpha} \|x_0 - x_*\|^{\alpha}}$$ and $\nu \in (0, 1)$ is a constant such that $1 - \nu - \nu^2 = 0$.
\end{theorem}
This result establishes linear convergence. In particular, to ensure $\|x_{K} - x_*\|^2 \leq \varepsilon$, it suffices to run $$K = \mathcal{O} \left(\left( \frac{K_0}{\mu} + \frac{(K_1 K_2^{\alpha} + K_2)\|x_0 - x_*\|^{\nicefrac{\alpha}{1- \alpha}}}{\mu}\right) \log \frac{1}{\varepsilon} \right)$$ iterations. Compared to the $L$-Lipschitz setting, the bound here includes an additional dependence on $\|x_0 - x_*\|^{\nicefrac{\alpha}{1 - \alpha}}$, which grows larger as $\alpha \to 1$.

\subsection{Convergence Guarantees for Monotone Operators} 
In this subsection, we focus on the monotone operators~\eqref{eq:monotone}. Here we provide the first analysis for the monotone $1$-symmetric $(L_0, L_1)$-Lipschitz operators.

\begin{theorem}\label{theorem:1symm_monotone}
    Suppose $F$ is monotone and $1$-symmetric $(L_0, L_1)$-Lipschitz operator. Then \algname{EG} with step size $\gamma_k = \omega_k = \frac{\nu}{L_0 + L_1 \|F(x_k)\|}$ satisfy 
    \begin{equation}\label{eq:1symm_monotone}
        \min_{0 \leq k \leq K} \|F(x_k)\|^2 \leq \frac{2L_0^2 \left( 1 + L_1 \exp{\left( L_1 \|x_0 - x_*\|\right) \|x_0 - x_*\|}\right)^2 \|x_0 - x_*\|^2}{\nu^2(K+1)}.
    \end{equation}
    where $\nu \exp{\nu} = \nicefrac{1}{\sqrt{2}}$.
\end{theorem}
Note that the solution of $\nu \exp{\nu} = \nicefrac{1}{\sqrt{2}}$ is approximately $0.45$. Hence, this result proves sublinear convergence of \algname{EG} when $F$ is monotone. Moreover, \eqref{eq:1symm_monotone} implies, \algname{EG} will  need $$K = \mathcal{O} \left( \frac{L_0^2 \|x_0 - x_*\|^2}{\varepsilon} + \frac{L_0^2 L_1^2 \exp{\left(2 L_1 \|x_0 - x_*\| \right) \|x_0 - x_*\|^4}}{\varepsilon} \right)$$ iterations to get $\|F(x_k)\|^2 \leq \varepsilon$ for some $k \leq K$. Therefore the convergence rate exponentially depends on $\|x_0 - x_*\|$ when $L_1 > 0$. This shows that $1$-symmetric $(L_0, L_1)$-Lipschitz operators potentially require more iterations of \algname{EG} compared to $L$-Lipschitz operators when initialization $x_0$ is far from $x_*$. However, \eqref{eq:1symm_monotone} recovers the best known dependence on $\|x_0 - x_*\|$ as a special case when $L_1 = 0$, i.e. $F$ is a standard Lipschitz operator \cite{gorbunov2022extragradient}.

Theorem \ref{theorem:1symm_monotone} shows that the \algname{EG}'s convergence rate has an extra term $\exp{(L_1\|x_0 - x_*\|)}$ compared to the results of the Lipschitz setting. One of the intermediate steps in this proof involves an upper bound on $\sum_{k = 0}^K \gamma_k^2 \|F(x_k)\|^2$ (see \eqref{eq:1symm_monotone_eq3} in Appendix \ref{appendix:convergence_analysis}). Then the simple approach is to get a lower bound on $\gamma_k^2$ for all $k$ and derive \eqref{eq:1symm_monotone}. This lower bound on $\gamma_k^2$ involves the $\exp{(L_1\|x_0 - x_*\|)}$ term (see \eqref{eq:1symm_monotone_eq2} in Appendix \ref{appendix:convergence_analysis}) and can be potentially very small. However, it is possible to eliminate this exponential dependence using a refined proof technique.
\begin{theorem}\label{theorem:1symm_monotone_noexp}
    Suppose $F$ is monotone and $1$-symmetric $(L_0, L_1)$-Lipschitz operator. Then \algname{EG} with step size $\gamma_k = \omega_k = \frac{\nu}{L_0 + L_1 \|F(x_k)\|}$ satisfy 
    \begin{equation*}
        \min_{0 \leq k \leq K} \| F(x_k)\| \leq \frac{\sqrt{2} L_0 \|x_0 - x_*\|}{\nu \sqrt{K+1} - \sqrt{2} L_1 \|x_0 - x_*\|}
    \end{equation*}
    where $\nu \exp{\nu} = \nicefrac{1}{\sqrt{2}}$ and $K+1 \geq \frac{2L_1^2 \|x_0 - x_*\|^2}{\nu^2}$.
\end{theorem}
Note that to obtain this convergence guarantee, a sufficiently large number of iterations is required, specifically $K+1 \geq \frac{2L_1^2 \|x_0 - x_*\|^2 }{\nu^2}$. \cite{gorbunov2024methods} employed a similar proof technique to eliminate the exponential dependence on the initial distance $\exp(L_1\|x_0 - x_*\|)$ in the context of the Adaptive Gradient method.

Next we state our result for $\alpha$-symmetric $(L_0, L_1)$-Lipschitz monotone operator with $\alpha \in (0, 1).$
\begin{theorem}\label{theorem:alpha01}
    Suppose $F$ is monotone and $\alpha$-symmetric $(L_0, L_1)$-Lipschitz operator with $\alpha \in (0, 1)$. Then \algname{EG} with $\gamma_k = \omega_k = \frac{1}{2 \sqrt{2} K_0  + \left( 2\sqrt{2} K_1 + 2^{\nicefrac{3 (1 - \alpha)}{2}} K_2^{1 - \alpha} \right) \|F(x_k)\|^{\alpha}}$ satisfy 
    \begin{eqnarray*}
        \min_{0 \leq k \leq K} \|F(x_k)\|^2 \leq \frac{16 \left( K_0 + (K_1 + 2^{\nicefrac{-3}{2}} K_2^{1 - \alpha}) (K_0 + K_2 \|x_0 - x_*\|^{\nicefrac{\alpha}{1 - \alpha}})^{\alpha}\|x_0 - x_*\|^{\alpha}\right)^2 \|x_0 - x_*\|^2}{K+1}.
    \end{eqnarray*}
\end{theorem}
This theorem establishes a sublinear convergence rate for $\alpha \in (0, 1)$. In the special case where $L_1 = 0$ (i.e., the standard $L$-Lipschitz setting), we have $K_1 = K_2 = 0$ by Proposition~\ref{prop:equiv_formulation}. Thus, our result recovers the best-known rate $\mathcal{O} \left( \frac{L_0^2 \|x_0 - x_*\|^2}{K+1} \right)$ from~\cite{gorbunov2022extragradient}. On the other hand, when $L_1 > 0$, we obtain a convergence rate of $\mathcal{O} \left( \frac{\|x_0 - x_*\|^{\frac{2 + 4 \alpha - 2\alpha^2}{1 - \alpha}}}{K+1}\right)$. Furthermore, as $\alpha \to 0$—which corresponds again to the $L$-Lipschitz setting—our step sizes $\gamma_k$ and $\omega_k$ become constant, and we recover the standard convergence rate $\mathcal{O} \left( \frac{\|x_0 - x_*\|^2}{K+1}\right)$. This matches the classical result for monotone $L$-Lipschitz operators up to constants, emphasizing the tightness of our analysis.
\vspace{-1mm}
\subsection{Local Convergence Guarantees for Weak Minty Operators}
Beyond the monotone operators, it is also possible to provide convergence for weak Minty operators~\eqref{eq: weak MVI} under some restrictions on $\rho > 0$. In contrast to the monotone problems where we used the same extrapolation and update step $\gamma_k, \omega_k$, here we use smaller update step size $\omega_k$. Specifically, we employ $\omega_k = \nicefrac{\gamma_k}{2}$, similar to \cite{diakonikolas2021efficient} for handling weak Minty $L$-Lipschitz operators. 
\begin{theorem}\label{theorem:weak_minty_alpha1}
     Suppose $F$ is weak Minty and $1$-symmetric $(L_0, L_1)$-Lipschitz assumption. Moreover we assume 
    \begin{equation}\label{eq:restriction_rho}
        \Delta_1 \eqdef \frac{\nu}{L_0 \left( 1 + L_1 \|x_0 - x_*\| e^{L_1 \|x_0 - x_*\|}\right)} - 4 \rho > 0.
    \end{equation} Then \algname{EG} with step size $\gamma_k = \frac{\nu}{L_0 + L_1 \|F(x_k)\|}$ and $\omega_k = \nicefrac{\gamma_k}{2}$ satisfies
    \begin{eqnarray}\label{eq:monotone_asymmetric}
        && \min_{0 \leq k \leq K} \|F(\hx_k)\|^2 \leq \frac{4 L_0 \left( 1 + L_1 \exp{\left( L_1 \|x_0 - x_*\|\right) \|x_0 - x_*\|}\right) \|x_0 - x_*\|^2}{\nu \Delta_1 (K+1)} 
    \end{eqnarray}
    where $\nu \exp{\nu} = 1$.
\end{theorem}
To the best of our knowledge, this is the first result establishing convergence guarantees for weak Minty, $\alpha$-symmetric $(L_0, L_1)$-Lipschitz operators. Similar to the monotone case, we obtain a sublinear convergence rate for weak Minty operators. However, the condition in~\eqref{eq:restriction_rho} indicates that the initialization point $x_0$ must be sufficiently close to the solution $x_*$ in order to ensure convergence. Consequently, Theorem~\ref{theorem:weak_minty_alpha1} only provides a local convergence guarantee.

In the special case where $L_1 = 0$ i.e., the standard $L$-Lipschitz setting, condition~\eqref{eq:restriction_rho} reduces to the simpler requirement $\rho < \frac{\nu}{4L_0}$. Similar assumptions on $\rho$ have been made in prior works such as~\cite{diakonikolas2021efficient} and~\cite{pethick2023escaping} for the $L$-Lipschitz weak Minty setting. Finally, we extend our analysis to the case $\alpha \in (0, 1)$, and present a corresponding theorem establishing sublinear convergence under analogous restrictions on $\rho$.

\begin{theorem}\label{theorem:weak_minty_alpha01}
    Suppose $F$ is weak Minty and $\alpha$-symmetric $(L_0, L_1)$-Lipschitz operator with $\alpha \in (0, 1)$. Moreover we assume 
    \begin{equation}\label{eq:restriction_rho_alpha}
        \Delta_{\alpha} \eqdef \frac{1}{2 \sqrt{2} K_0 + 2 \sqrt{2} (K_1 + 2^{\nicefrac{-3}{2}} K_2^{1 - \alpha}) (K_0 + K_2 \|x_0 - x_*\|^{\nicefrac{\alpha}{1 - \alpha}})^{\alpha} \|x_0 - x_*\|^{\alpha}} - 4 \rho > 0.
    \end{equation} Then \algname{EG} with step size $\gamma_k = \frac{1}{2 \sqrt{2} K_0 + \left(2 \sqrt{2}K_1 + 2^{\nicefrac{3 (1 - \alpha)}{2}} K_2^{1 - \alpha} \right) \|F(x_k)\|^{\alpha}}$ and $\omega_k = \nicefrac{\gamma_k}{2}$ satisfy 
    \begin{eqnarray*}
        \min_{0 \leq k \leq K} \|F(\hx_k)\|^2 \leq \frac{4 \left( K_0 + \left(K_1 + 2^{\nicefrac{-3}{2}} K_2^{1 - \alpha} \right) (K_0 + K_2 \|x_0 - x_*\|^{\nicefrac{\alpha}{1 - \alpha}})^{\alpha} \|x_0 - x_*\|^{\alpha}\right) \|x_0 - x_*\|^2}{\Delta_{\alpha}(K+1)}.
    \end{eqnarray*}
\end{theorem}

\section{Numerical Experiments}\label{sec:experiments}

\vspace{-2mm}
In this section, we conduct experiments to validate the efficiency of our proposed step size strategy $\gamma_k = \frac{1}{c_0 + c_1 \|F(x_k)\|^{\alpha}}$ with $\alpha = 1$. In the first experiment, we compare our step size choice with that of \cite{pmlr-v235-vankov24a} on a strongly monotone problem, and in the second experiment, we make a comparison with the constant step size strategy for solving a monotone problem. Finally, we evaluate our scheme for solving the GlobalForsaken problem from \cite{pethick2023escaping}. All experiments in this chapter were conducted using a personal MacBook with an Apple M3 chip and 16GB of RAM. 
\begin{figure}[h]
\centering
\begin{subfigure}[b]{0.4\textwidth}
    \centering
    \includegraphics[width=\textwidth]{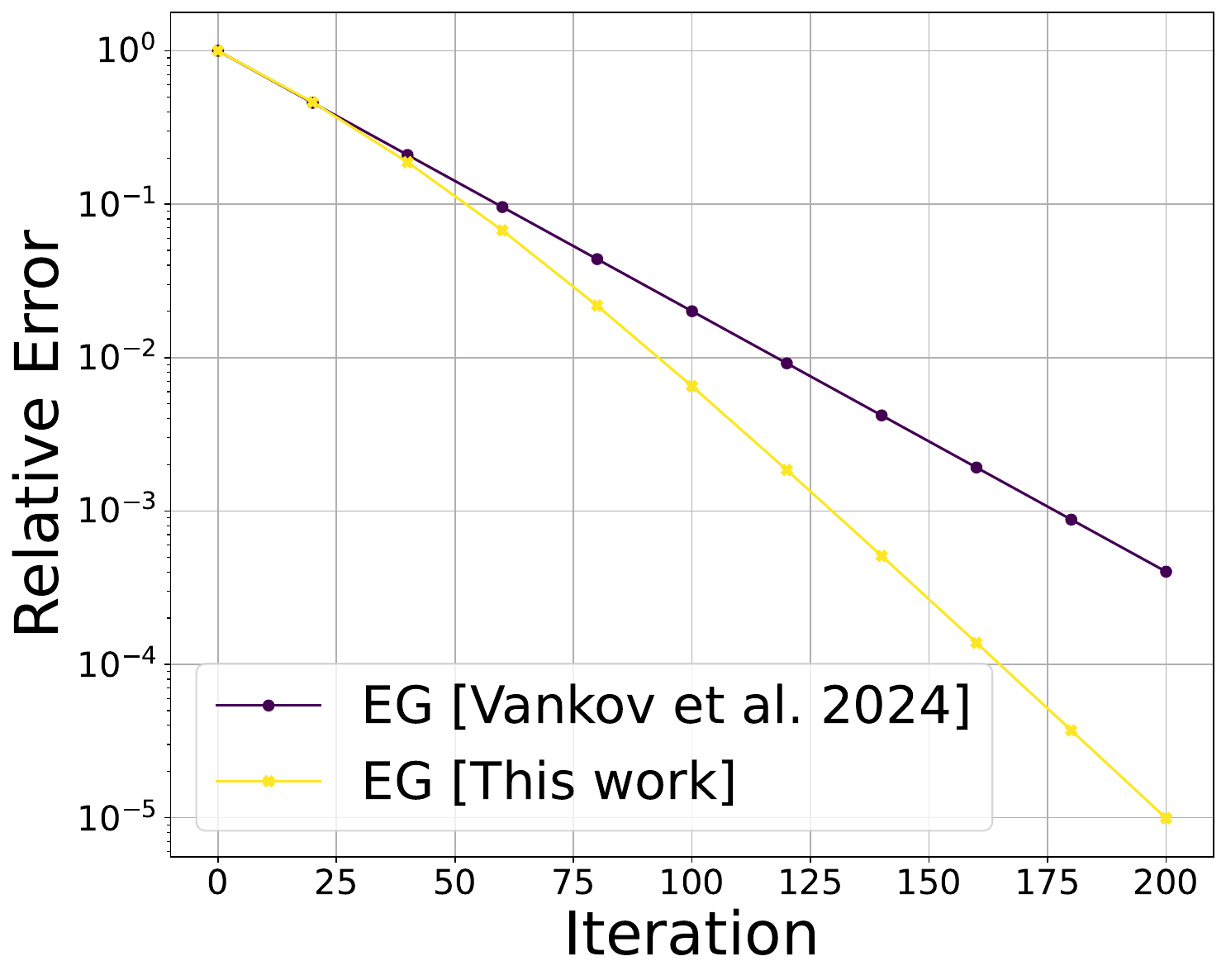}
    \caption{}\label{fig:L0L1comparison_opt_dist}
\end{subfigure}
\hfill
\begin{subfigure}[b]{0.4\textwidth}
    \centering
    \includegraphics[width=\textwidth]{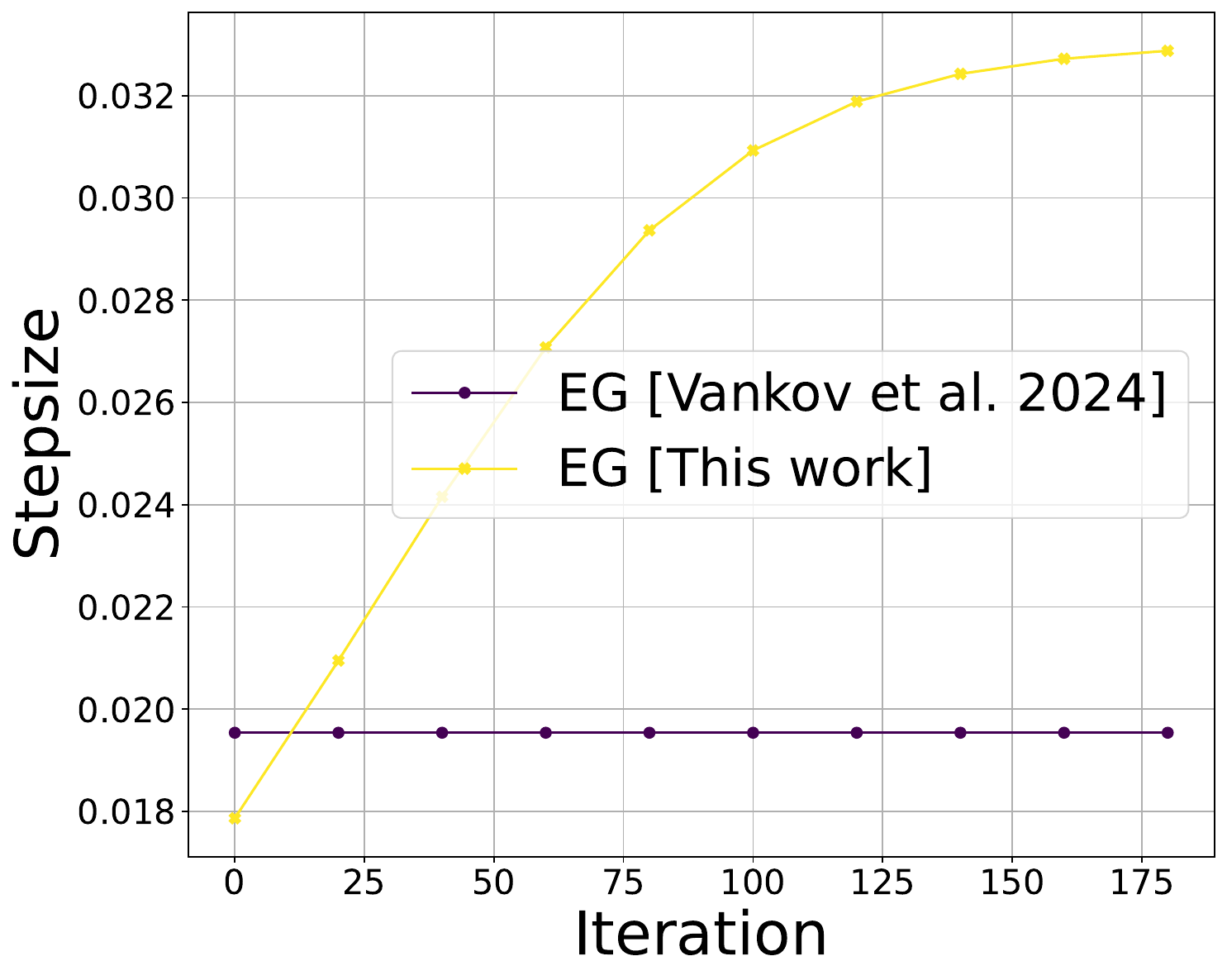}
    \caption{}\label{fig:L0L1comparison_stepsize}
\end{subfigure}
\begin{subfigure}[b]{0.4\textwidth}
    \centering
    \includegraphics[width=\textwidth]{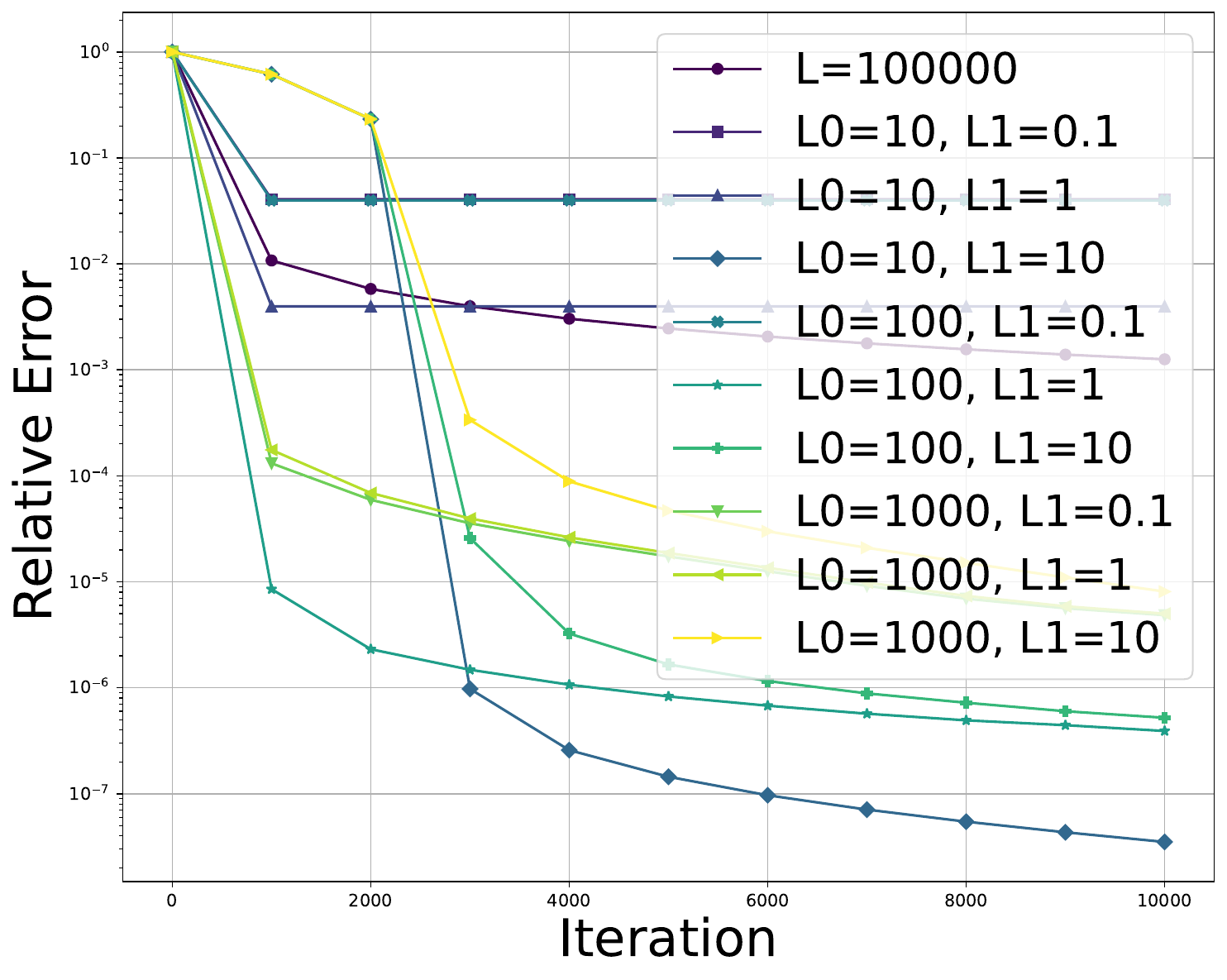}
    \caption{}\label{fig:monotone_cubic_relative_error}
\end{subfigure}
\hfill
\begin{subfigure}[b]{0.4\textwidth}
    \centering
    \includegraphics[width=\textwidth]{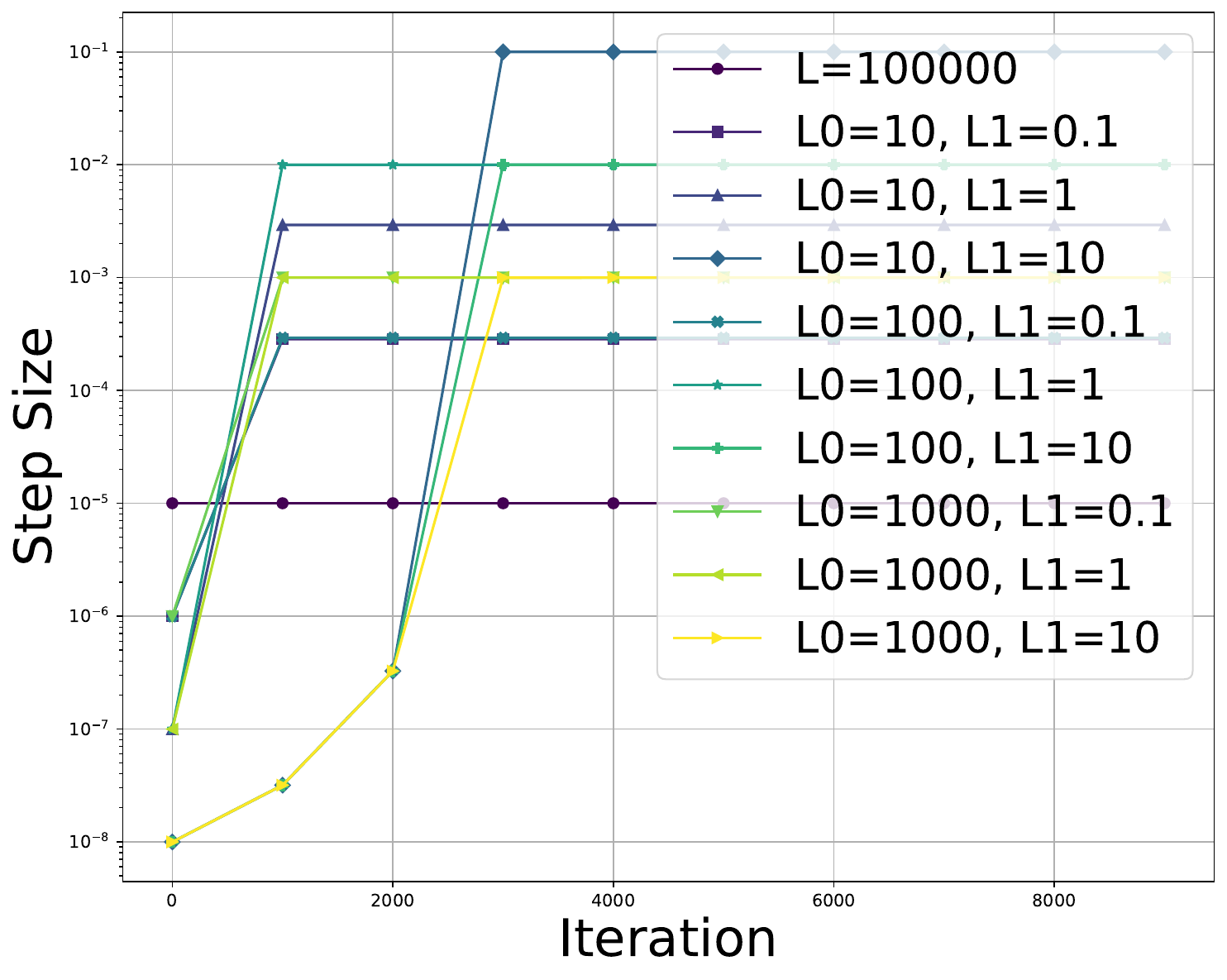}
    \caption{}\label{fig:monotone_cubic_step_size}
\end{subfigure}
    \caption{In Figures~\ref{fig:L0L1comparison_opt_dist} and~\ref{fig:L0L1comparison_stepsize}, we compare our proposed adaptive step size strategy with that of \cite{pmlr-v235-vankov24a}. In Figures~\ref{fig:monotone_cubic_relative_error} and~\ref{fig:monotone_cubic_step_size}, we evaluate the performance of the \algname{EG} method on the problem in~\eqref{eq:min_max_cubic_Rd}, using both a constant step size and the $(L_0, L_1)$-adaptive step size. For both sets of experiments, we report the relative error and the magnitude of the step size over iterations.
    }\label{fig:monotone_and_vankov}
\end{figure}

\textbf{Performance on a Strongly Monotone Problem.}
In this experiment, we compare our theoretical step sizes with those from \cite{pmlr-v235-vankov24a}. Here, we implement \algname{EG} for solving the operator $F(x) = (\text{sign}\left(u_1 \right) \left|u_1 \right| + u_2, \text{sign}\left(u_2 \right) \left|u_2 \right| - u_1)$.

This problem has constants $L_0 = 1 + 2 \sqrt{2}$ and $L_1 = 2\sqrt{2}$. For our method, we use $\gamma_k = \omega_k =  \frac{\nu}{L_0 + L_1 \|F(x_k)\|}$ while for \algname{EG}~\cite{pmlr-v235-vankov24a} we use stepsize $\gamma_k = \omega_k = \min \left\{ \frac{1}{4 \mu}, \frac{1}{2 \sqrt{2}e L_0}, \frac{1}{2 \sqrt{2}e L_1 \|F(x_k)\|} \right\}$. In Figure \ref{fig:L0L1comparison_opt_dist}, we plot the relative error $\frac{\|x_k - x_*\|^2}{\|x_0 - x_*\|^2}$ on the $y$-axis while number of iterations on the $x$-axis. We find that our proposed step size outperforms that of \cite{pmlr-v235-vankov24a}. Moreover, in Figure \ref{fig:L0L1comparison_stepsize}, we compare the magnitude of the step size and how it evolves over the iterations. We find that the step size of \cite{pmlr-v235-vankov24a} remains constant at approximately $0.02$, whereas our proposed step size increases to a value larger than $0.032$. These experiments highlight the efficiency of our proposed step size. 

\textbf{Performance on a Monotone Problem.} Here we consider the following min-max optimization problem
\begin{eqnarray}\label{eq:min_max_cubic_Rd}
    \min_{w_1 \in \R^d} \max_{w_2 \in \R^d} \mathcal{L}(w_1, w_2) = \frac{1}{3} \left( w_1^{\top} \A w_1 \right)^{\nicefrac{3}{2}} + w_1^{\top} \B w_2 - \frac{1}{3} \left(w_2^{\top} \C w_2 \right)^{\nicefrac{3}{2}}.
\end{eqnarray}
where $\A, \B, \C \in \R^{d \times d}$ are positive definite matrices. Note that, when $d = 1$, and $\A, \B, \C$ are just scalars equal to $1$, this problem reduces to \eqref{eq:min_max_cubic}.

The corresponding operator of this problem is given by 
\begin{equation*}
    F(x)  
    = \begin{bmatrix}
        \nabla_{w_1} \mathcal{L}(w_1, w_2) \\
        - \nabla_{w_2} \mathcal{L}(w_1, w_2)
    \end{bmatrix}
    = \begin{bmatrix}
        \left( w_1^{\top} \A w_1\right)^{\nicefrac{1}{2}} \A w_1 + \B w_2 \\
        \left( w_2^{\top} \C w_2\right)^{\nicefrac{1}{2}} \C w_2 - \B^{\top} w_1
    \end{bmatrix}.
\end{equation*} 
Furthermore, we show that $\mathcal{L}$ is convex-concave and has an equilibrium only at $w_1, w_2 = 0 \in \R^d$ (check Appendix \ref{appendix:num_exp}). To solve~\eqref{eq:min_max_cubic_Rd}, we implement the \algname{EG} method using two types of step size strategies: (1) a constant step size $\gamma_k = \omega_k = \nicefrac{1}{c}$, and (2) an adaptive step size $\gamma_k = \omega_k = \nicefrac{1}{\left(c_0 + c_1 \|F(x_k)\| \right)}$. For the constant step size \algname{EG}, we perform a grid search over $c \in \{10^2, 10^3, 10^4, 10^5, 10^6, 10^7\}$. We find that $c = 10^5$ yields the best performance: larger values lead to slower convergence, while smaller values cause divergence. Figures~\ref{fig:monotone_cubic_relative_error} and~\ref{fig:monotone_cubic_step_size} present the relative error and step size for the case $c = 10^5$. For our adaptive \algname{EG} method, we perform a grid search over $c_0 \in \{10, 100, 1000\}$ and $c_1 \in \{0.1, 1, 10\}$, evaluating all $9$ possible combinations. The performance of all adaptive variants is plotted in Figures~\ref{fig:monotone_cubic_relative_error} and~\ref{fig:monotone_cubic_step_size}. We observe that most combinations outperform the constant step size \algname{EG}, with $(c_0, c_1) = (10, 10)$ achieving the best results (see Figure \ref{fig:monotone_cubic_relative_error}).

Interestingly, while the adaptive \algname{EG} starts with smaller step sizes compared to the constant step size \algname{EG}, its step sizes increase over time and eventually surpass those of the constant step size approach (see Figure \ref{fig:monotone_cubic_step_size}). This highlights the practical effectiveness of our proposed method in handling non-Lipschitz operators such as the one in~\eqref{eq:min_max_cubic_Rd}.

\begin{figure}
	\centering
	\includegraphics[width=0.6\linewidth]{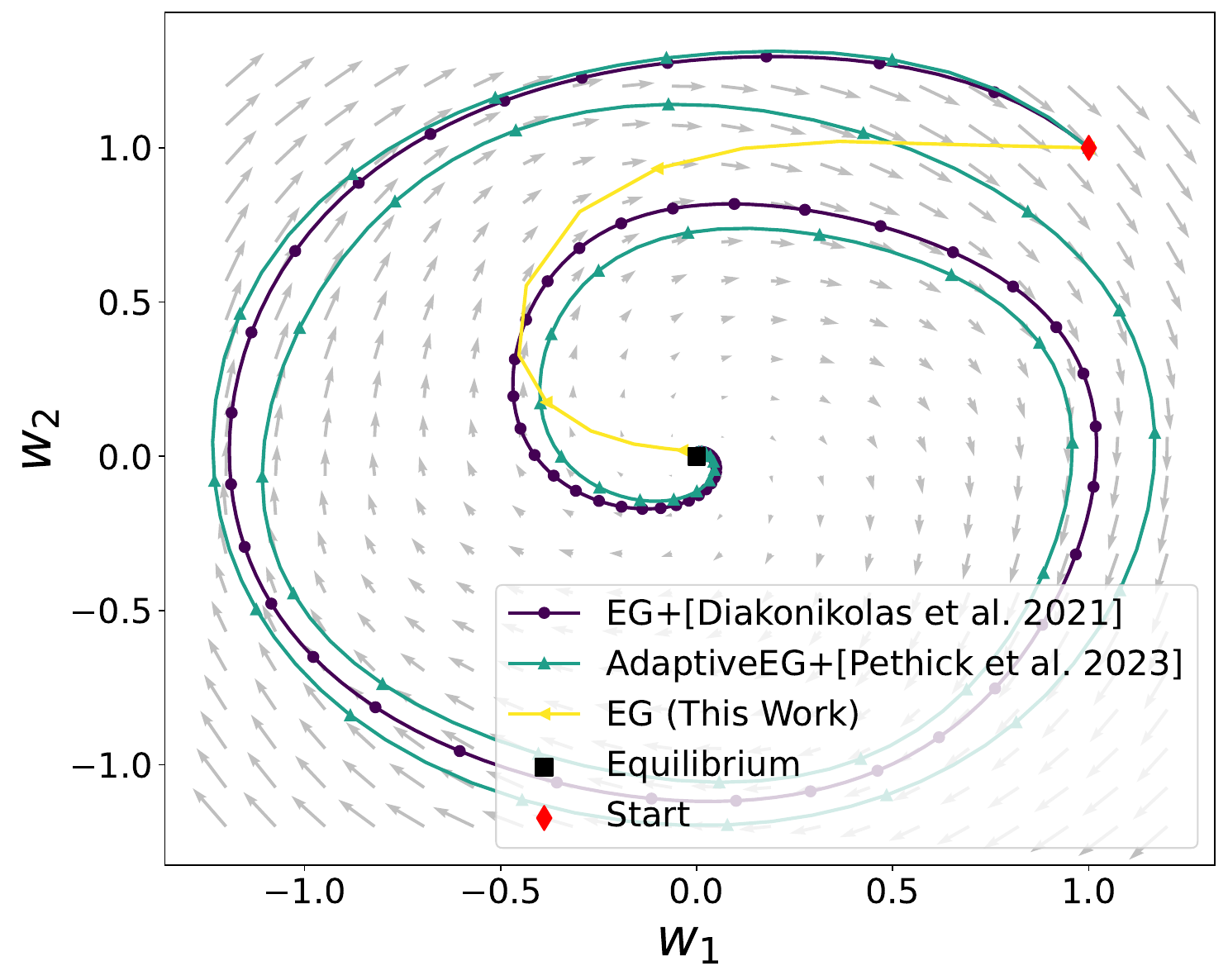}
	\caption{\small Trajectories of algorithms for solving problem~\eqref{eq:globalforsaken}.}\label{fig:globalforsaken}
\end{figure}

\textbf{Performance on a Weak Minty Problem.} Here we consider the unconstrained GlobalForsaken problem from \cite{pethick2023escaping} given by
\begin{equation}\label{eq:globalforsaken}
    \min_{w_1 \in \R} \max_{w_2 \in \R} \mathcal{L}(w_1, w_2) \eqdef w_1 w_2 + \psi(w_1) - \psi(w_2), 
\end{equation}
where $\psi(w) = \frac{2w^6}{21} - \frac{w^4}{3} + \frac{w^2}{3}$. As shown in~\cite{pethick2023escaping}, the saddle-point problem in~\eqref{eq:globalforsaken} admits a global Nash equilibrium at $(w_1, w_2) = (0, 0)$ and satisfies the weak Minty condition~\eqref{eq: weak MVI} with parameter $\rho \approx 0.119732$. We implement \algname{AdaptiveEG+}~\cite{pethick2022escaping}, \algname{EG+}~\cite{diakonikolas2021efficient}, and \algname{EG} with our step size strategy to solve this problem. For each algorithm, we perform step size tuning on a grid of $\gamma_k \in \{10^{-5}, 10^{-4}, \cdots, 10^2 \}$. We observe that both \algname{AdaptiveEG+} $\left( \text{extrapolation step size }\omega_k = \frac{\left\langle F(\hx_k), x_k - \hx_k \right\rangle}{ \left\| F(\hx_k) \right\|^2} \right)$ and \algname{EG+} $\left( \text{extrapolation step size }\omega_k = \frac{\gamma_k}{2} \right)$ perform best with a fixed step size of $\gamma_k = 0.1$. For our method, we set the step size parameters as $(c_0, c_1) = (1, 1)$. In Figure \ref{fig:globalforsaken}, we present the trajectory plots of these algorithms, all initialized at $(w_1, w_2) = (1, 1)$. The dots on these trajectories mark every $10$ iterations of the algorithm. Our findings indicate that all algorithms eventually converge to the equilibrium $(0, 0)$, but the convergence of our method is significantly faster. This demonstrates the advantage of our step-size strategy in solving challenging problems that satisfy only weak Minty conditions.

\chapter{Communication-Efficient Gradient Descent Ascent Methods for Distributed Min-Max Optimization} \label{chap:chap-5}

\section{Introduction}

Federated learning (FL) \cite{konevcny2016federated,mcmahan2017communication,kairouz2021advances} has become a fundamental distributed machine learning framework in which multiple clients collaborate to train a model while keeping their data decentralized. 
Communication overhead is one of the main bottlenecks in FL \cite{karimireddy2020scaffold}, which motivates the use by the practitioners of advanced algorithmic strategies to alleviate the communication burden. One of the most popular and well-studied strategies to reduce the communication cost is increasing the number of local steps between the communication rounds~\cite{mcmahan2017communication,stich2018local,assran2019stochastic,khaled2020tighter,koloskova2020unified}.

Currently, federated learning algorithms and techniques are associated primarily with minimization problems. However, with the increase of game-theoretical formulations in machine learning, the necessity of designing efficient federated learning approaches for these problems is apparent. In this chapter, we are interested in the design of communication-efficient federated learning algorithms suitable for multi-player game formulations. In particular, we consider a more
abstract formulation and focus on solving the following distributed/federated variational inequality problem (VIP):
\begin{equation}
  \label{eq:FedVIP}
  \text{Find } z_*\in\mathbb{R}^{d'}, \text{ such that } \autoprod{F(z_*),z - z_*}\geq 0,\quad
   \forall z\in\mathbb{R}^{d'},
\end{equation}
 where $F:\mathbb{R}^{d'}\rightarrow\mathbb{R}^{d'}$ is a structured non-monotone operator. We assume the operator $F$ is distributed across $n$ different nodes/clients and written as $F(z)=\frac{1}{n}\sum_{i=1}^n H_i(z)$. In this setting, the operator $H_i:\mathbb{R}^{d'}\rightarrow\mathbb{R}^{d'}$  is
owned by and stored on client $i$. Since we are in the unconstrained scenario, problem~\eqref{eq:FedVIP} can be equivalently written as finding $z_*\in\mathbb{R}^{d'}$, such that $F(z_*)= 0$~\cite{loizou2021stochastic,gorbunov2022stochastic}.

In the distributed/federated learning setting, formulation~\eqref{eq:FedVIP} captures both classical FL minimization~\cite{mcmahan2017communication} and FL minimax optimization problems~\cite{sharma2022federated, deng2021local} as special cases.

\subsection{Main Contributions}
\label{MainCont}

Our key contributions are summarized as follows:
\begin{itemize}[leftmargin=1em]
\item \textbf{VIPs and Federated Learning:} We present the first connection between regularized VIPs and federated learning setting by explaining how one can construct a consensus reformulation of problem\eqref{eq:FedVIP} by appropriately selecting the regularizer term $R(x)$ (see discussion in Section~\ref{sec:connection_VIP_FL}). 

    \item \textbf{Unified Framework:} 
    We provide a unified theoretical framework for the design of efficient local training methods for solving the distributed VIP~\eqref{eq:FedVIP}. For our analysis, we use a general key assumption on the stochastic estimates that allows us to study, under a single framework, several stochastic local variants of the proximal gradient descent-ascent (\algname{GDA}) method \footnote{Throughout this chapter we use this suggestive name (\algname{GDA} and Stochastic \algname{GDA}/ \algname{SDA}) motivated by the minimax formulation, but we highlight that our results hold for the more general VIP~\eqref{eq:FedVIP}}. 

    \item    
    \textbf{Communication Acceleration:} In practice, local algorithms are superior in communication complexity compared to their non-local counterparts. However, in the literature, the theoretical communication complexity of \algname{Local GDA} does not improve upon the vanilla distributed \algname{GDA} (i.e., communication in every iteration). Here we close the gap and provide the first communication-accelerated local GDA methods. For the deterministic strongly monotone and smooth setting, our method requires $\mathcal{O}\autopar{\kappa\ln\frac{1}{\epsilon}}$ communication rounds over the $\mathcal{O}\autopar{\kappa^2\ln\frac{1}{\epsilon}}$ of vanilla \algname{GDA} (and previous analysis of \algname{Local GDA}), where $\kappa$ is the condition number. See Table~\ref{TableFirst} for further details .

    \item \textbf{Heterogeneous Data:} Designing algorithms for federated minimax optimization and distributed VIPs,
    is a relatively recent research topic, and existing works heavily rely on the bounded heterogeneity assumption, which may be unrealistic in FL setting~\cite{hou2021efficient, beznosikov2020distributed,sharma2022federated}. Thus, they can only solve problems with similar data between clients/nodes. In practical scenarios, the private data stored by a user on a mobile device and the data of different users can be arbitrarily heterogeneous. Our analysis does not assume bounded heterogeneity, and the proposed algorithms (\algname{ProxSkip-VIP-FL} and \algname{ProxSkip-L-SVRGDA-FL}) guarantee convergence with an improved communication complexity.

    \item \textbf{Sharp rates for known special cases:} For the known methods/settings fitting our framework, our general theorems recover the best rates known for these methods. For example, the convergence results of the Proximal (non-local) \algname{SGDA} algorithm for regularized VIPs \cite{beznosikov2022stochastic} and the ProxSkip algorithm for composite minimization problems \cite{mishchenko2022proxskip} can be obtained as special cases of our analysis, showing the tightness of our convergence guarantees.

    \item \textbf{Numerical Evaluation:} 
    In numerical experiments, we illustrate the most important properties of the proposed methods by comparing them with existing algorithms in federated minimax learning tasks. The numerical results corroborate our theoretical findings.
\end{itemize}

\begin{table}[tbp]
    \centering
    \footnotesize
    \renewcommand{\arraystretch}{1.25}
    \caption{
    Summary and comparison of algorithms for solving strongly-convex-strongly-concave federated minimax optimization problems (a special case of the distributed VIPs \eqref{eq:FedVIP}).
    }
    \begin{threeparttable}[htb]
        \begin{tabular}{c | c | c | c }
            \hline \hline
            \textbf{Algorithm}\tnote{1}
            & \textbf{Acceleration?}
            & \begin{tabular}{c}
            \textbf{Variance}\\
            \textbf{Reduction?}
            \end{tabular}  
            & \textbf{Communication Complexity}
            \\
            \hline \hline
            \makecell[c]{
                \algname{GDA/SGDA}
                \\
                \cite{fallah2020optimal}
            }            
            & ---
            & \xmark
            & $\mathcal{O}\autopar{\max\autobigpar{\kappa^2, \frac{\sigma_*^2}{\mu^2\eps}}\ln\frac{1}{\eps}}$
            \\
            \hline
            \makecell[c]{
                \algname{Local GDA/SGDA}
                \\
                \cite{deng2021local}
            }            
            & \xmark
            & \xmark
            & $\mathcal{O}\autopar{\sqrt{\frac{\kappa^2(\sigma_*^2+\Delta^2)}{\mu\epsilon}}}$
            \\
            \hline
            \cellcolor{bgcolor2}\begin{tabular}{c}
            \algname{ProxSkip-GDA/SGDA-FL}\\
            (This chapter)
            \end{tabular} 
            & \cellcolor{bgcolor2}\cmark
            & \cellcolor{bgcolor2}\xmark
            & \cellcolor{bgcolor2} {\small $\mathcal{O}\autopar{\sqrt{\max\autobigpar{\kappa^2, \frac{\sigma_*^2}{\mu^2\epsilon}}}\ln\frac{1}{\eps}}$}
            \\
            \hline
            \cellcolor{bgcolor2}\begin{tabular}{c}
            \algname{ProxSkip-L-SVRGDA-FL}\\
            (This chapter)
            \end{tabular} 
            & \cellcolor{bgcolor2}\cmark
            & \cellcolor{bgcolor2}\cmark
            & \cellcolor{bgcolor2} $\mathcal{O}\autopar{\kappa\ln\frac{1}{\eps}}$
            \\
            \hline \hline
       \end{tabular}
       \begin{tablenotes}
            \footnotesize
            \item[1] ``Acceleration"~=~whether the algorithm enjoys acceleration in communication compared to its non-local counterpart, ``Variance Red.?"~=~whether the algorithm (in stochastic setting) applies variance reduction, ``Communication Complexity"~=~the communication complexity of the algorithm.  Here $\kappa=L/\mu$ denotes the condition number where $L$ is the Lipschitz parameter and $\mu$ is the modulus of strong convexity, $\sigma_*^2$ corresponds to the variance level at the optimum ($\sigma_*^2=0$ in deterministic setting), and $\Delta$ captures the bounded variance.  For a more detailed comparison of complexities, please also refer to Table~\ref{table:comparison_v2}.
        \end{tablenotes}
    \end{threeparttable}
\label{TableFirst}
\end{table}

\section{Technical Preliminaries}
\subsection{Regularized VIP and Consensus Reformulation}
\label{sec:connection_VIP_FL}

Following classical techniques from \cite{parikh2014proximal}, the distributed VIP~\eqref{eq:FedVIP} can be recast into a consensus form:

\newpage

\begin{equation}
\label{eq:objective_FL_reformulation}
   \text{Find } x_*\in\mathbb{R}^d, \text{such that }  \autoprod{F(x_*),x-x_*}+R(x)-R(x_*)\geq 0,\quad
    \forall x\in\mathbb{R}^d,
\end{equation}
where $d=nd'$ and
\begin{equation}
    \label{eq:mpFL_VIP_form}
    \small
    F(x)\triangleq\sum_{i=1}^n F_i(x_i),
    \quad
    R(x)
    \triangleq
    \begin{cases}
        0 & \text{if } x_1=x_2=\cdots=x_n\\
        +\infty & \text{otherwise}.
    \end{cases}
\end{equation}
Here $x=\autopar{x_1, x_2, \cdots, x_n}\in\mathbb{R}^d$, $x_i\in\mathbb{R}^{d'}$, $F:\mathbb{R}^d\rightarrow\mathbb{R}^d$, $F_i:\mathbb{R}^{d'}\rightarrow\mathbb{R}^d$ and $F_i(x_i)=(0, \cdots, H_i(x_i), \cdots, 0)$ where $[F(x)]_{(id'+1):(id'+d')}=H_i(x_i)$.
Note that the reformulation requires a dimension expansion on the variable enlarged from $d'$ to $d$. It is well-known that the two problems \eqref{eq:FedVIP} and \eqref{eq:objective_FL_reformulation} are equivalent in terms of the solution, which we detail in Appendix \ref{apdx:thm_FL_Operator_Check}.

Having explained how the problem~\eqref{eq:FedVIP} can be converted into a regularized VIP~\eqref{eq:objective_FL_reformulation}, let us now present the Stochastic Proximal Method \cite{parikh2014proximal, beznosikov2022stochastic}, one of the most popular algorithms for solving  general regularized VIPs\footnote{We call general regularized VIPs, the problem \eqref{eq:objective_FL_reformulation} where $F:\mathbb{R}^d\rightarrow\mathbb{R}^d$ is an operator and $R:\mathbb{R}^d\rightarrow\mathbb{R}$ is a regularization term (a proper lower semicontinuous convex function). 
}. The update rule of the Stochastic Proximal Method defines as follows:
\begin{equation}
    \label{proxOpe}
    x_{k+1}=\prox_{\gamma R}\autopar{x_k-\gamma g_k}
    \quad\text{where}\ \ 
    \prox_{\gamma R}\autopar{x}
    \triangleq
    \argmin_{v\in\mathbb{R}^d}\autobigpar{R(v)+\frac{1}{2\gamma}\autonorm{v-x}^2}.
\end{equation}

Here $g_k$ is an unbiased estimator of $F(x_k)$ and $\gamma>0$ is the step-size of the method.
As explained in \cite{mishchenko2022proxskip}, it is typically assumed that the proximal computation~\eqref{proxOpe} can be evaluated in closed form (has exact value), and its computation it is relatively cheap. That is, the bottleneck in the update rule of the Stochastic Proximal Method is the computation of $g_k$. This is normally the case when the regularizer term $R(x)$ in general regularized VIP has a simple expression. For $R(v) = \frac{1}{2}\|v\|_2^2$, the proximal operator admits a closed-form expression given by $
\mathbf{prox}_{\gamma R}(x) = \frac{x}{1 + \gamma}$.
For $R(v) = \|v\|_1$, the proximal operator is the well-known soft-thresholding operator, expressed as  $\mathbf{prox}_{\gamma R}(x) = \operatorname{sign}(x) \odot \max\{ |x| - \gamma,\, 0 \},$ where \(\odot\) denotes elementwise multiplication. For more details on closed-form expression~\eqref{proxOpe} under simple choices of regularizers $R(x)$, we refer readers to check \cite{parikh2014proximal}.

However, in the consensus reformulation of the distributed VIP, the regularizer $R(x)$ has a specific expression~\eqref{eq:mpFL_VIP_form} that makes the proximal computation~\eqref{proxOpe} expensive compared to the evaluation of $g_k$.

In particular, note that by following the definition of $R(x)$ in \eqref{eq:mpFL_VIP_form}, we get that $\Bar{x}=\frac{1}{n}\sum_{i=1}^n x_i$ and $\prox_{\gamma R}\autopar{x}=\autopar{\Bar{x}, \Bar{x}, \cdots, \Bar{x}}.$

So evaluating $\prox_{\gamma R}\autopar{x}$ is equivalent to taking the average of the variables $x_i$~\cite{parikh2014proximal},

meaning that it involves a high communication cost in distributed/federated learning settings. This was exactly the motivation behind the proposal of the \algname{ProxSkip} algorithm in \cite{mishchenko2022proxskip}, which reduced the communication cost by allowing the expensive proximal
operator to be skipped in most iterations.

In this chapter, inspired by the \algname{ProxSkip} approach of \cite{mishchenko2022proxskip}, we provide a unified framework for analyzing efficient algorithms for solving the distributed VIP~\eqref{eq:FedVIP} and its consensus reformulation problem~\eqref{eq:objective_FL_reformulation}. In particular in Section \ref{sec:centralized}, we provide algorithms for solving general regularized VIPs (not necessarily a distributed setting). Later in Section~\ref{asdas} we explain how the proposed algorithms can be interpreted as distributed/federated learning methods.

\subsection{Main Assumptions}

Having presented the consensus reformulation of distributed VIPs, let us now provide the main conditions of problem \eqref{eq:objective_FL_reformulation} assumed throughout the paper.
\begin{assumption}
    \label{assume:main}
    We assume that Problem \eqref{eq:objective_FL_reformulation} has a unique solution $x_*$ and
    \begin{itemize}
        \item The operator $F$ is $\mu$-quasi-strongly monotone and $\ell$-star-cocoercive with $\mu, \ell>0$, i.e., $\forall x\in\mathbb{R}^d$, 
        \begin{eqnarray*}
            \autoprod{F(x)-F(x_*), x-x_*}& \geq & \mu\autonorm{x-x_*}^2, \\
            \autoprod{F(x)-F(x_*), x-x_*} & \geq & \frac{1}{\ell}\autonorm{F(x)-F(x_*)}^2.
        \end{eqnarray*}
        \item The function $R(\cdot)$ is a proper lower semicontinuous convex function.
    \end{itemize}
\end{assumption}

Assumption~\ref{assume:main} 
is weaker than the classical strong-monotonicity and Lipschitz continuity assumptions commonly used in analyzing methods for solving problem \eqref{eq:objective_FL_reformulation} and captures non-monotone and non-Lipschitz problems as special cases~\cite{loizou2021stochastic}. In addition, given that the operator $F$ is $L$-Lipschitz continuous and $\mu$-strongly monotone, it can be shown that the operator $F$ is ($\kappa L$)-star-cocoercive where $\kappa\triangleq L/\mu$ is the condition number of the operator~\cite{loizou2021stochastic}. 

Motivated by recent applications in machine learning, in this chapter, we mainly focus on the case where we have access to unbiased estimators of the operator $F$. Regarding the inherent stochasticity, we further use the following key assumption, previously used in \cite{beznosikov2022stochastic} for the analysis of \algname{Proximal SGDA}, to characterize the behavior of the operator estimation.
\begin{assumption}[Estimator]
    \label{assume:stochastic}

    For all $k\geq 0$, we assume that the estimator $g_k$ is unbiased ($\EE[g_k]=F(x_k)$). Next, we assume that there exist non-negative constants $A, B, C, D_1, D_2\geq 0$, $\rho\in(0,1]$ and a sequence of (possibly random) non-negative variables $\{\sigma_{k}\}_{k\geq0}$ such that for all $k \geq 0$:
        \begin{equation*} 
            \begin{split}
                \mathbb{E}\autonorm{g_k-F(x_*)}^2
                \leq &\ 
                2A\autoprod{F(x)-F(x_*), x-x_*}+B\sigma_k^2+D_1,\\
                \mathbb{E}[\sigma_{k+1}^2]
                \leq &\ 
                2C\autoprod{F(x)-F(x_*), x-x_*}+(1-\rho)\sigma_k^2+D_2.
            \end{split}
        \end{equation*}   
\end{assumption}

Variants of Assumption~\ref{assume:stochastic} have first proposed in classical minimization setting for providing unified analysis for several stochastic optimization methods and, at the same time, avoid the more restrictive bounded gradients and bounded variance assumptions~\cite{gorbunov2020unified,khaled2020unified}. Recent versions of Assumption~\ref{assume:stochastic} have been used in the analysis of Stochastic Extragradient~\cite{gorbunov2022stochastic} and stochastic gradient-descent assent~\cite{beznosikov2022stochastic} for solving minimax optimization and VIP problems. To our knowledge, analysis of existing local training methods for distributed VIPs depends on bounded variance conditions. Thus, via Assumption~\ref{assume:stochastic}, our convergence guarantees hold under more relaxed assumptions as well. Note that Assumption~\ref{assume:stochastic} covers many well-known conditions: for example, when $\sigma_k \equiv 0$ and $F(x_*)=0$, it will recover the recently introduced expected co-coercivity condition \cite{loizou2021stochastic} where $A$ corresponds to the modulus of expected co-coercivity, and $D_1$ corresponds to the scale of estimator variance at $x_*$. In Section~\ref{sec:centralized_special_case}, we explain how Assumption~\ref{assume:stochastic} is satisfied for many well-known estimators $g_k$, including the vanilla mini-batch estimator and variance reduced-type estimator. That is, for the different estimators (and, as a result, different algorithms), we prove the closed-form expressions of the parameters $A, B, C, D_1, D_2\geq 0$, $\rho\in(0,1]$ for which Assumption~\ref{assume:stochastic} is satisfied.

\section{General Framework: \algname{ProxSkip-VIP}}
\label{sec:centralized}
\vspace{-3mm}

Here we provide a unified algorithm framework (Algorithm~\ref{alg:Stoc-ProxSkip-VIP}) for solving regularized VIPs~\eqref{eq:objective_FL_reformulation}. Note that in this section, the guarantees hold under the general assumptions~\ref{assume:main} and~\ref{assume:stochastic}. In Section~\ref{asdas} we will further specify our results to the consensus reformulation setting where \eqref{eq:mpFL_VIP_form} also holds.

Algorithm~\ref{alg:Stoc-ProxSkip-VIP} is inspired by the \algname{ProxSkip} proposed in \cite{mishchenko2022proxskip} for solving composite minimization problems. The two key elements of the algorithm are the randomized prox-skipping, and the control variate $h_k$. Via prox-skipping, the proximal oracle is rarely called if $p$ is small, which helps to reduce the computational cost when the proximal oracle is expensive. Also, for general regularized VIPs we have $F(x_*)\neq 0$ at the optimal point $x_*$ due to the composite structure. Thus skipping the proximal operator will not allow the method to converge. On this end, the introduction of $h_k$ alleviates such drift and stabilizes the iterations toward the optimal point $(h_k \rightarrow F(x_*))$.

We also highlight that Algorithm~\ref{alg:Stoc-ProxSkip-VIP} is a general update rule and can vary based on the selection of the unbiased estimator $g_k$ and the probability $p$. For example, if $p=1$, and $h_k\equiv 0$, the algorithm is reduced to the \algname{Proximal SGDA} algorithm proposed in \cite{beznosikov2022stochastic} (with the associated theory). We will focus on further important special cases of Algorithm \ref{alg:Stoc-ProxSkip-VIP} below.
\vspace{-3mm}

\begin{algorithm}[h]
    \caption{\algname{ProxSkip-VIP}}
    \label{alg:Stoc-ProxSkip-VIP}
    \begin{algorithmic}[1]
        \REQUIRE Initial point $x_0$, parameters $\gamma, p$, initial control variate $h_0$, number of iterations $T$
        \FORALL{$k = 0,1,..., K$}
            \STATE $\widehat{x}_{k+1}=x_k-\gamma (g_k-h_k)$
            \STATE Flip a coin $\theta_k$, and $\theta_k=1$ w.p. $p$, otherwise $0$
            \IF{$\theta_k=1$}
            \STATE $x_{k+1}=\prox_{\frac{\gamma}{p} R}\autopar{\widehat{x}_{k+1}-\frac{\gamma}{p} h_k}$
            \ELSE
            \STATE $x_{k+1}=\widehat{x}_{k+1}$
            \ENDIF
            \STATE $h_{k+1}=h_k+\frac{p}{\gamma}(x_{k+1}-\widehat{x}_{k+1})$ 
        \ENDFOR
    \end{algorithmic}
\end{algorithm}

\subsection{Convergence of \algname{ProxSkip-VIP}} 
\label{sec:general_result_proxskip-vip}
\vspace{-2mm}
Our main convergence quarantees are presented in the following Theorem and Corollary.
\begin{theorem}[Convergence of \algname{ProxSkip-VIP}]
    \label{thm:convergence_Stoc-ProxSkip-VIP}
    With Assumptions \ref{assume:main} and \ref{assume:stochastic}, let 
    $
    \gamma\leq \min\autobigpar{\frac{1}{\mu}, \frac{1}{2(A+MC)}}$, $\tau\triangleq \min\autobigpar{\gamma\mu, p^2, \rho-\frac{B}{M} }$, for some $M>\frac{B}{\rho}$. Denote $$V_k\triangleq\autonorm{x_k-x_*}^2+ \frac{\gamma^2}{p^2} \autonorm{h_k- F(x_*)}^2+M\gamma^2\sigma_k^2,$$ Then the iterates of \algname{ProxSkip-VIP} (Algorithm~\ref{alg:Stoc-ProxSkip-VIP}), satisfy:
    \begin{equation*}
        \mathbb{E}\automedpar{V_K}\leq
        \autopar{1-\tau}^TV_0+\frac{\gamma^2\autopar{D_1+MD_2}}{\tau}.
    \end{equation*}

\end{theorem}

Theorem~\ref{thm:convergence_Stoc-ProxSkip-VIP} show that ProxSkip-VIP converges linearly to the neighborhood of the solution. The neighborhood is proportional to the step-size $\gamma$ and the parameters $D_1$ and $D_2$ of Assumption~\ref{assume:stochastic}. 
In addition, we highlight that if we set  $p=1$, and $h_k\equiv 0$, then Theorem~\ref{thm:convergence_Stoc-ProxSkip-VIP} recovers the convergence guarantees of Proximal-SGDA of \cite[Theorem 2.2]{beznosikov2022stochastic}.
As a corollary of Theorem~\ref{thm:convergence_Stoc-ProxSkip-VIP}, we can also obtain the following corresponding complexity results.

\begin{corollary}
    \label{cor:Stoc-ProxSkip-VIP-Complexity}
    With the setting in Theorem \ref{thm:convergence_Stoc-ProxSkip-VIP}, if we set $M=\frac{2B}{\rho}$, $p=\sqrt{\gamma\mu}$ and $$\gamma\leq \min\autobigpar{\frac{1}{\mu}, \frac{1}{2(A+MC)}, \frac{\rho}{2\mu}, \frac{\mu\epsilon}{2\autopar{D_1+\frac{2B}{\rho} D_2}}},$$ we have $\mathbb{E}\automedpar{V_K}\leq\epsilon$ with iteration complexity and the number of calls to the proximal oracle $\prox(\cdot)$ as 
    \begin{eqnarray*}
        \mathcal{O}\autopar{\max\autobigpar{\frac{A+\frac{BC}{\rho}}{\mu}, \frac{1}{\rho}, \frac{D_1+\frac{B}{\rho} D_2}{\mu^2\epsilon}}\ln\frac{V_0}{\epsilon}}
    \end{eqnarray*}
    and
    \begin{eqnarray*}
        \mathcal{O}\autopar{\sqrt{\max\autobigpar{\frac{A+\frac{BC}{\rho}}{\mu}, \frac{1}{\rho}, \frac{D_1+\frac{B}{\rho} D_2}{\mu^2\epsilon}}}\ln\frac{V_0}{\epsilon}}.
    \end{eqnarray*}
\end{corollary}

\subsection{Special Cases of General Analysis}
\label{sec:centralized_special_case}
\vspace{-2mm}
Theorem~\ref{thm:convergence_Stoc-ProxSkip-VIP} holds under the general key Assumption~\ref{assume:stochastic} on the stochastic estimates. In this subsection, via Theorem~\ref{thm:convergence_Stoc-ProxSkip-VIP}, we explain how different selections of the unbiased estimator $g_k$ in Algorithm~\ref{alg:Stoc-ProxSkip-VIP} lead to various convergence guarantees.
In particular, here we cover (i) \algname{ProxSkip-SGDA}, (ii) \algname{ProxSkip-GDA}, and (iii)variance-reduced method \algname{ProxSkip-L-SVRGDA}. To the best of our knowledge, none of these algorithms have been proposed and analyzed before for solving VIPs.

\vspace{-2mm}
\paragraph{(i) Algorithm: \algname{ProxSkip-SGDA}.} 
Let us have the following assumption:
\begin{assumption}[Expected Cocoercivity]
    \label{assume:ECC}
    We assume that for all $k \geq 0$, the stochastic operator $g_k\triangleq g(x_k)$, which is an unbiased estimator of $F(x_k)$, satisfies expected cocoercivity, i.e., for all $x \in \R^d$ there is $L_g>0$ such that
    \begin{equation*}
        \mathbb{E}\autonorm{g(x)-g(x_*)}^2\leq L_g \autoprod{F(x)-F(x_*), x-x_*}.
    \end{equation*}
\end{assumption}

The expected cocoercivity condition was first proposed in \cite{loizou2021stochastic} to analyze SGDA and the stochastic consensus optimization algorithms efficiently. It is strictly weaker compared to the bounded variance assumption and “growth
conditions,” and it implies the star-cocoercivity of the operator $F$. Assuming expected co-coercivity allows us to characterize the estimator $g_k$. 
\begin{lemma}[\cite{beznosikov2022stochastic}]
    \label{thm:property_local_estimator}
    Let Assumptions \ref{assume:main} and \ref{assume:ECC} hold and let $\sigma_*^2 \triangleq \E \left[\|g(x_*) - F(x_*)\|^2\right] <+\infty$.
    Then $g_k$ satisfies Assumption \ref{assume:stochastic} with 
    \begin{equation*}
        A=L_g,\ D_1=2\sigma_*^2, \rho=1, \text{ and }  B=C=D_2=\sigma_k^2\equiv 0.
    \end{equation*}
\end{lemma}
By combining Lemma~\ref{thm:property_local_estimator} and Corollary~\ref{cor:Stoc-ProxSkip-VIP-Complexity}, we obtain the following result.

\begin{corollary}[Convergence of \algname{ProxSkip-SGDA}]
    \label{thm:convergence_ProxSkip-SGDA}
    With Assumption \ref{assume:main} and \ref{assume:ECC}, if we further set 
    $$\gamma\leq \min\autobigpar{\frac{1}{2L_g}, \frac{1}{2\mu}, \frac{\mu\epsilon}{8\sigma_*^2}}$$
    and $p=\sqrt{\gamma\mu}$, then for the iterates of \algname{ProxSkip-SGDA}, we have $\mathbb{E}\automedpar{V_K}\leq\epsilon$ with iteration complexity and the number of calls of the proximal oracle $\prox(\cdot)$ as
    \begin{eqnarray*}
        \mathcal{O}\autopar{\max\autobigpar{\frac{L_g}{\mu}, \frac{\sigma_*^2}{\mu^2\epsilon}}\ln\frac{1}{\epsilon}}
    \end{eqnarray*}
    and
    \begin{eqnarray*}
        \mathcal{O}\autopar{\sqrt{\max\autobigpar{\frac{L_g}{\mu}, \frac{\sigma_*^2}{\mu^2\epsilon}}}\ln\frac{1}{\epsilon}},
    \end{eqnarray*}
    respectively.
\end{corollary}

\paragraph{ (ii) Deterministic Case: \algname{ProxSkip-GDA}.} In the deterministic case where $g(x_k)= F(x_k)$, the algorithm \algname{ProxSkip-SGDA} is reduced to \algname{ProxSkip-GDA}. Thus, Lemma~\ref{thm:property_local_estimator} and Corollary~\ref{cor:Stoc-ProxSkip-VIP-Complexity} but with $\sigma_*^2=0$. In addition, the expected co-coercivity parameter becomes $L_g=\ell$ by Assumption ~\ref{assume:main}.

\begin{corollary}[Convergence of \algname{ProxSkip-GDA}]
\label{cor:ProxSkip_GDA_complexity}
    With the same setting in Corollary \ref{thm:convergence_ProxSkip-SGDA}, if we set the estimator $g_k= F(x_k)$, and 
 
    $\gamma= \frac{1}{2\ell}$
    the iterates of \algname{ProxSkip-GDA} satisfy:
    $$\mathbb{E} \automedpar{V_K}\leq
        \autopar{1-\min\autobigpar{\gamma\mu, p^2 }}^TV_0,$$

    and we get $\mathbb{E}\automedpar{V_K}\leq\epsilon$ with iteration complexity and number of calls of the proximal oracle $\prox(\cdot)$ as 

    \begin{equation*}
        \mathcal{O}\autopar{\frac{L_g}{\mu}\ln\frac{1}{\epsilon}}
        \ \ \text{and}\ \ 
        \mathcal{O}\autopar{\sqrt{\frac{L_g}{\mu}}\ln\frac{1}{\epsilon}},
    \end{equation*}
    respectively.

\end{corollary}
Note that if we further have $F$ to be $L$-Lipschitz continuous and $\mu$-strongly monotone, then $\ell=\kappa L$ where $\kappa= \nicefrac{L}{\mu}$ is the condition number~\cite{loizou2021stochastic}. Thus the number of iteration and calls of the proximal oracles are $\mathcal{O}\autopar{\kappa^2\ln\frac{1}{\epsilon}}$ and $\mathcal{O}\autopar{\kappa\ln\frac{1}{\epsilon}}$ respectively. In the minimization setting, these two complexities are equal to $\mathcal{O}\autopar{\kappa\ln\frac{1}{\epsilon}}$ and $\mathcal{O}\autopar{\sqrt{\kappa}\ln\frac{1}{\epsilon}}$ since $L_g=\ell=L$. In this case, our result recovers the result of the original \algname{ProxSkip} method in \cite{mishchenko2022proxskip}.

\paragraph{(iii) Algorithm: \algname{ProxSkip-L-SVRGDA}.}
Here, we focus on a variance-reduced variant of the proposed \algname{ProxSkip} framework (Algorithm~\ref{alg:Stoc-ProxSkip-VIP}). We further specify the operator $F$ in \eqref{eq:objective_FL_reformulation} to be in a finite-sum formulation: $F(x)=\frac{1}{n}\sum_{i=1}^nF_i(x).$

We propose the ProxSkip-Loopless-SVRG (\algname{ProxSkip-L-SVRGDA}) algorithm (Algorithm~\ref{alg:Stoc-ProxSkip-L-SVRGDA} in Appendix) which generalizes the L\nobreakdash-SVRGDA proposed in \cite{beznosikov2022stochastic}. In this setting, we need to introduce the following assumption:

\begin{assumption}
    \label{assume:average_star_coco}
    We assume that there exist a constant $\widehat{\ell}$ such that for all $x \in \R^d$: 
    \begin{equation*}
        \frac{1}{n}\sum_{i=1}^n\autonorm{F_i(x)-F_i(x_*)}^2\leq\widehat{\ell}\autoprod{F(x)-F(x_*), x-x_*}.
    \end{equation*}
\end{assumption}

If each $F_i$ is $\ell_i$-cocoercive, then Assumption \ref{assume:average_star_coco} holds with $\widehat{\ell}\leq\max_{i\in[n]}L_i$. Using Assumption~\ref{assume:average_star_coco} we obtain the following result~\cite{beznosikov2022stochastic}.
\begin{lemma}[\cite{beznosikov2022stochastic}]
    \label{lm:ABC_SVRG}
    With Assumption~\ref{assume:main} and \ref{assume:average_star_coco}, we have the estimator $g_k$ in Algorithm~\ref{alg:Stoc-ProxSkip-L-SVRGDA} satisfies Assumption \ref{assume:main} with 

    \begin{equation*}
        A=\widehat{\ell},\ B=2,\ D_1=D_2=0,\ C=\frac{q\widehat{\ell}}{2},\ \rho=q,\ \text{ and } \sigma_k^2=\frac{1}{n}\sum_{i=1}^n\autonorm{F_i(x_k)-F_i(x_*)}^2.
    \end{equation*}
\end{lemma}

By combining Lemma~\ref{lm:ABC_SVRG} and Corollary~\ref{cor:Stoc-ProxSkip-VIP-Complexity}, we obtain the following result for \algname{ProxSkip-L-SVRGDA}:

\begin{corollary}[Complexities of \algname{ProxSkip-L-SVRGDA}]
    \label{cor:ProxSkip-L-SVRGDA-FL_complexity}
    Let Assumption \ref{assume:main} and \ref{assume:average_star_coco} hold. If we further set 
    
    \begin{equation*}
        q=2\gamma\mu, M=\frac{4}{q}, p=\sqrt{\gamma\mu} \text{ and }  \gamma=\min\autobigpar{\frac{1}{\mu}, \frac{1}{6\widehat{\ell}}},
    \end{equation*}
    then we obtain $\mathbb{E}\automedpar{V_K}\leq\epsilon$ with iteration complexity and the number of calls of the proximal oracle $\prox(\cdot)$ 
    as $\mathcal{O}(\widehat{\ell}/\mu\ln\frac{1}{\epsilon})$ and the number of calls of the proximal oracle $\prox(\cdot)$ as $\mathcal{O}(\sqrt{\widehat{\ell}/\mu}\ln\frac{1}{\epsilon})$.
 
\end{corollary}
If we further consider Corollary~\ref{cor:ProxSkip-L-SVRGDA-FL_complexity} in the minimization setting and assume each $F_i$ is $L$-Lipschitz and $\mu$-strongly convex, the number of iteration and calls of the proximal oracles will be $\mathcal{O}\autopar{\kappa\ln\frac{1}{\epsilon}}$ and $\mathcal{O}\autopar{\sqrt{\kappa}\ln\frac{1}{\epsilon}}$, which recover the results of variance-reduced \algname{ProxSkip} in \cite{malinovsky2022variance}.

\section{Application of \algname{ProxSkip} to Federated Learning}
\label{asdas}

In this section, by specifying the general problem to the expression of~\eqref{eq:mpFL_VIP_form} as a special case of \eqref{eq:objective_FL_reformulation},  we explain how the proposed algorithmic framework can be interpreted as federated learning algorithms. As discussed in Section~\ref{sec:connection_VIP_FL}, in the distributed/federated setting, evaluating the $\prox_{\gamma R}\autopar{x}$ is equivalent to a communication between the $n$ workers. In the FL setting, Algorithm~\ref{alg:Stoc-ProxSkip-VIP} can be expressed as Algorithm~\ref{alg:ProxSkip-VIP-FL}. Note that skipping the proximal operator in Algorithm~\ref{alg:Stoc-ProxSkip-VIP} corresponds to local updates in Algorithm~\ref{alg:ProxSkip-VIP-FL}. In Algorithm~\ref{alg:ProxSkip-VIP-FL}, $g_{i,k}=g_i(x_{i,k})$ is the unbiased estimator of the $H_i(x_{i,k})$ of the original problem \eqref{eq:FedVIP}, while the client control vectors $h_{i,k}$ satisfy $h_{i,k} \rightarrow H_i(z_*)$. The probability $p$ in this setting shows how often a communication takes place (averaging of the workers' models).
\begin{algorithm}[htbp]
    \caption{\algname{ProxSkip-VIP-FL}}
    \label{alg:ProxSkip-VIP-FL}
    \begin{algorithmic}[1]
        \REQUIRE Initial points $\{x_{i,0}\}_{i=1}^n$ and $\{h_{i,0}\}_{i=1}^n$, parameters $\gamma, p, K$
        \FORALL{$k = 0,1,..., K$}
            \STATE \textbf{Server:} 
            Flip a coin $\theta_k$, $\theta_k=1$ w.p. $p$, otherwise $0$.
            Send $\theta_k$ to all workers
            \FOR{{\bf each workers $i\in\automedpar{n}$ in parallel}} 
                \STATE $\widehat{x}_{i, k+1}=x_{i, t}-\gamma (g_{i,k}-h_{i,k})$             
                \COMMENT{Local update with control variate}
                \IF{$\theta_k=1$}
                \STATE Worker: $x_{i, k+1}'=\widehat{x}_{i, k+1}-\frac{\gamma}{p} h_{i,k}$, sends $x_{i, k+1}'$ to the server
                \STATE Server: computes $x_{i, k+1}=\frac{1}{n}\sum_{i=1}^n x_{i, k+1}'$ and send to workers
                \COMMENT{Communication}
                \ELSE
                \STATE $x_{i, k+1}=\widehat{x}_{i, k+1}$
                \COMMENT{Otherwise skip the communication step}
                \ENDIF
                \STATE $h_{i, k+1}=h_{i, k}+\frac{p}{\gamma}(x_{i, k+1}-\widehat{x}_{i, k+1})$ 
            \ENDFOR
        \ENDFOR
        \vspace{-1mm}
    \end{algorithmic}
\end{algorithm}

\textbf{Algorithm: \algname{ProxSkip-SGDA-FL}.}
The first implementation of the framework we consider is \algname{ProxSkip-SGDA-FL}. Similar to \algname{ProxSkip-SGDA} in the centralized setting, here we set the estimator to be the vanilla estimator of $F$, i.e., 
$g_{i,k}=g_i(x_{i,k})$, where $g_i$ is an unbiased estimator of $H_i$. We note that if we set $h_{i,k}\equiv 0$, then Algorithm~\ref{alg:ProxSkip-VIP-FL} is reduced to the typical \algname{Local SGDA}~\cite{deng2021local}.

To proceed with the analysis in FL, we require the following assumption:

\begin{assumption}
    \label{assume:additional_FL}
    The problem \eqref{eq:FedVIP} attains a unique solution $z_*\in\mathbb{R}^{d'}$. Each $H_i$ in \eqref{eq:FedVIP}, it is $\mu$-quasi-strongly monotone around $z_*$, i.e., for any $x_i \in\mathbb{R}^{d'}$,  $$\autoprod{H_i(x_i)-H_i(z_*), x_i-z_*} \geq\mu\autonorm{x_i-z_*}^2. $$ Operator $g_i(x_i)$, is an unbiased estimator of $H_i(x_i)$, and for all $x_i\in\mathbb{R}^{d'}$ we have 
    $\mathbb{E}\autonorm{g_i(x_i)-g_i(z_*)}^2\leq L_g \autoprod{H_i(x_i)-H_i(z_*), x_i-z_*}.$
\end{assumption}

\vspace{-2mm}
The assumption on the uniqueness of the solution is pretty common in the literature. For example, in the (unconstrained) minimization case, quasi-strong monotonicity implies uniqueness $z_*$~\cite{hinder2020near}. Assumption~\ref{assume:additional_FL} is required as through it, we can prove that the operator $F$ in \eqref{eq:objective_FL_reformulation} satisfies Assumption~\ref{assume:main}, and its corresponding estimator satisfies Assumption~\ref{assume:stochastic}, which we detail in Appendix 
\ref{apdx:thm_FL_Operator_Check}. 

Let us now present the convergence guarantees.
\begin{theorem}[Convergence of \algname{ProxSkip-SGDA-FL}]
    \label{thm:complexity_FL_ProxSkip}
    With Assumption~\ref{assume:additional_FL}, then \algname{ProxSkip-VIP-FL}~(Algorithm~\ref{alg:ProxSkip-VIP-FL}) achieves $\mathbb{E}\automedpar{V_K}\leq\epsilon$ (where $V_K$ is defined in Theorem \ref{thm:convergence_Stoc-ProxSkip-VIP}), with iteration complexity 

     $$\mathcal{O} \left(\max\autobigpar{\frac{L_g}{\mu}, \frac{\sigma_*^2}{\mu^2\epsilon}}\ln\frac{1}{\epsilon} \right)$$ 
    and communication complexity  $$\mathcal{O} \left(\sqrt{\max\autobigpar{\frac{L_g}{\mu}, \frac{\sigma_*^2}{\mu^2\epsilon}}}\ln\frac{1}{\epsilon} \right).$$

\end{theorem}

\textbf{Comparison with Literature.}

Note that Theorem~\ref{thm:complexity_FL_ProxSkip} is quite general, which holds under any reasonable, unbiased estimator. In the special case of federated minimax problems, 
one can use the same (mini-batch) gradient estimator from \algname{Local SGDA}~\cite{deng2021local} in Algorithm~\ref{alg:ProxSkip-VIP-FL} and our results still hold. The benefit of our approach compared to \algname{Local SGDA} is the communication acceleration, as pointed out in Table~\ref{TableFirst}. In addition, in the deterministic setting ($g_i(x_{i,k})=H_i(x_{i,k})$) we have $\sigma_*^2=0$ and Theorem~\ref{thm:complexity_FL_ProxSkip} reveals  $\mathcal{O} \left( \nicefrac{\ell}{\mu} \ln\frac{1}{\epsilon} \right)$ iteration complexity and $\mathcal{O} \left(\sqrt{\nicefrac{\ell}{\mu}}\ln\frac{1}{\epsilon} \right)$ communication complexity for \algname{ProxSkip-GDA-FL}. In Table~\ref{table:comparison_v2} of the appendix, we provide a more detailed comparison of our Algorithm~\ref{alg:ProxSkip-VIP-FL} (Theorem~\ref{thm:complexity_FL_ProxSkip}) with existing literature in FL. The proposed approach outperforms other algorithms (\algname{Local SGDA}, \algname{Local SEG}, \algname{FedAvg-S}) in terms of iteration and communication complexities.

As the baseline, the distributed (centralized) gradient descent-ascent (\algname{GDA}) and extragradient (\algname{EG}) algorithms achieve $\mathcal{O} \left(\kappa^2\ln\frac{1}{\eps} \right)$ and $\mathcal{O} \left(\kappa\ln\frac{1}{\eps} \right)$ communication complexities, respectively~\cite{fallah2020optimal,mokhtari2019convergence}.
We highlight that our analysis does not require an assumption on bounded heterogeneity~/~dissimilarity, and as a result, we can solve problems with heterogeneous data.

Finally, as reported by \cite{beznosikov2020distributed}, the lower communication complexity bound for problem \eqref{eq:FedVIP} is given by $\Omega \left(\kappa\ln\frac{1}{\eps} \right)$, which further highlights the optimality of our proposed \algname{ProxSkip-SGDA-FL} algorithm.

\paragraph{Algorithm: \algname{ProxSkip-L-SVRGDA-FL}.}
Next, we focus on variance-reduced variants of \algname{ProxSkip-SGDA-FL}
and we further specify the operator $F$ in \eqref{eq:objective_FL_reformulation} as $F(x)\triangleq\sum_{i=1}^n F_i(x_i)$ and $F_i(x)\triangleq \frac{1}{m_i}\sum_{j=1}^{m_i} F_{i,j}(x_i)$.

The proposed algorithm, \algname{ProxSkip-L-SVRGDA-FL}, is presented in the Appendix as Algorithm \ref{alg:ProxSkip-L-SVRGDA-FL}. 
In this setting, we need the following assumption on $F_i$ to proceed with the analysis.

\begin{assumption}
    \label{assume:additional_SVRG_FL}
    The Problem \eqref{eq:FedVIP} attains a unique solution $z_*$. Also for each $H_i$ in \eqref{eq:FedVIP}, it is $\mu$-quasi-strongly monotone around $z_*$. Moreover we assume for all $x_i\in\mathbb{R}^{d'}$ we have 
    $$\frac{1}{m_i}\sum_{j=1}^{m_i}\autonorm{F_{i,j}(x_i)-F_{i,j}(z_*)}^2\leq \widehat{\ell} \autoprod{H_i(x_i)-H_i(z_*), x_i-z_*}.$$
\end{assumption}

Similar to the \algname{ProxSkip-SGDA-FL} case, we can show that under Assumption \eqref{assume:additional_SVRG_FL}, the operator and unbiased estimator fit into the setting of Assumptions~\ref{assume:main} and \ref{assume:stochastic} (see derivation in Appendix \ref{apdx:FL_operator_check}). As a result, we can obtain the following complexity result.

\begin{theorem}[Convergence of \algname{ProxSkip-L-SVRGDA-FL}]
\label{dnaoao}
    Let Assumption \ref{assume:additional_SVRG_FL} hold. Then the iterates of \algname{ProxSkip-L-SVRGDA-FL} achieve $\mathbb{E}[V_K]\leq\epsilon$ (where $V_K$ is defined in Theorem \ref{thm:convergence_Stoc-ProxSkip-VIP})with iteration complexity 

    $\mathcal{O} \left(\widehat{\ell}/\mu\ln\frac{1}{\epsilon} \right)$ and communication complexity $\mathcal{O} \left(\sqrt{\widehat{\ell}/\mu}\ln\frac{1}{\epsilon} \right)$.
\end{theorem}
\vspace{-2mm}
In the special case of minimization problems $\min_{z\in\mathbb{R}^d}f(z)$, i.e., $F(z)=\nabla f(z)$, we have Theorem~\ref{dnaoao} recovers the theoretical result of \cite{malinovsky2022variance} which focuses on ProxSkip methods for minimization problems, showing the tightness of our approach.

Similar to the \algname{ProxSkip-VIP-FL}, the \algname{ProxSkip-L-SVRGDA-FL} enjoys a communication complexity improvement in terms of the condition number $\kappa$. For more details, please refer to Table~\ref{TableFirst} (and Table~\ref{table:comparison_v2} in the appendix).

\section{Numerical Experiments}
\vspace{-2mm}
We corroborate our theory with the experiments, and test the performance of the proposed algorithms \algname{ProxSkip-(S)GDA-FL} (Algorithm~\ref{alg:ProxSkip-VIP-FL} with $g_i(x_{i,k})=H_i(x_{i,k})$ or its unbiased estimator)  and \algname{ProxSkip-L-SVRGDA-FL} (Algorithm~\ref{alg:ProxSkip-L-SVRGDA-FL}). We focus on two classes of problems: (i) strongly monotone quadratic games and (ii) robust least squares. See Appendix~\ref{apdx:numerical_experiments} and \ref{apdx:more_experiment} for more details and extra experiments.

To evaluate the performance, we use the relative error measure $\frac{\|x_k - x_*\|^2}{\|x_0 - x_*\|^2}$. The 
horizontal axis corresponds to the number of communication rounds.
For all experiments, we pick the step-size $\gamma$ and probability $p$ for different algorithms according to our theory. That is, ProxSkip-(S)GDA-FL based on Corollary~\ref{thm:convergence_ProxSkip-SGDA} and \ref{cor:ProxSkip_GDA_complexity}, \algname{ProxSkip-L-SVRGDA-FL} by Cororllary~\ref{cor:ProxSkip-L-SVRGDA-FL_complexity}. See also Table~\ref{table:stepsize} in the appendix for the settings of parameters. We compare our methods to \algname{Local SGDA}~\cite{deng2021local}, \algname{Local SEG}~\cite{beznosikov2020distributed} (and their deterministic variants), and \algname{FedGDA-GT}, a deterministic algorithm proposed in ~\cite{sun2022communication}. For all methods, we use parameters based on their theoretical convergence guarantees. For more details on our implementations and additional experiments, see Appendix~\ref{apdx:numerical_experiments} and \ref{apdx:more_experiment}.

\begin{figure}
 \begin{subfigure}[b]{0.45\textwidth}
        \centering
        \includegraphics[width=\textwidth]{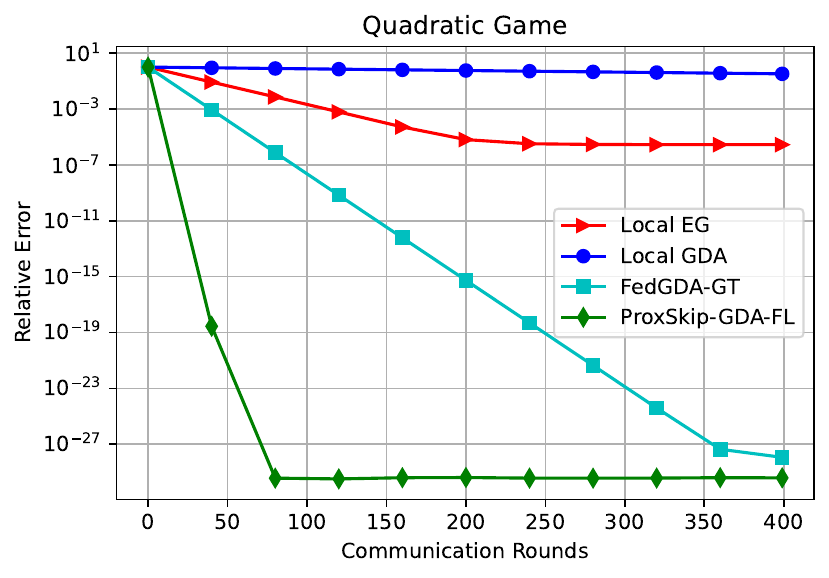}
        \caption{Deterministic Case}
        \label{fig:deterministicXtheoreticalXDatasetI}
    \end{subfigure}
    \hspace{1em}
    \begin{subfigure}[b]{0.45\textwidth}
        \centering
        \includegraphics[width=\textwidth]{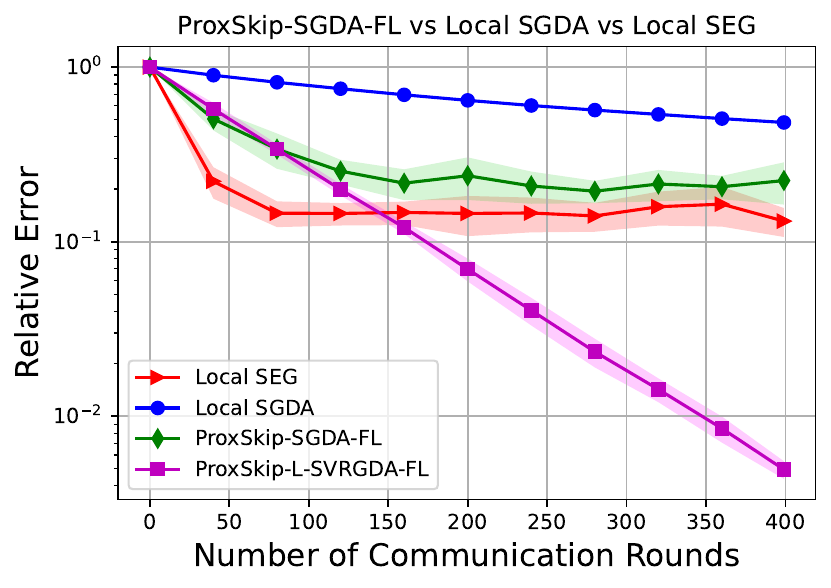}
        \caption{Stochastic Case}
        \label{fig:stochasticXtheoreticalXDatasetI}
    \end{subfigure}
    \caption{Comparison of algorithms on the strongly-monotone quadratic game \eqref{quadraticgame}.}
    \label{fig: Comparison of ProxSkip-VIP-FL vs Local SGDA vs Local SEG on Heterogeneous Data in the deterministic setting.}
\end{figure}

\paragraph{Strongly-monotone Quadratic Games.}
In the first experiment, we consider a problem of the form 
\begin{equation}
\label{quadraticgame}
    \min_{w_1 \in \mathbb{R}^d} \max_{w_2 \in \mathbb{R}^d} \frac{1}{n} \sum_{i = 1}^n \frac{1}{m_i} \sum_{j = 1}^{m_i} \mathcal{L}_{ij}(w_1, w_2) ,
\end{equation}
where $\mathcal{L}_{ij}(w_1, w_2) \triangleq \frac{1}{2} w_1^{\intercal}\A_{ij} w_1 + w_1^{\intercal}\B_{ij} w_2 - \frac{1}{2}w_1^{\intercal}\C_{ij} w_1
+ a_{ij}^{\intercal}w_1 - c_{ij}^{\intercal}w_2.$ Here we set the number of clients $n = 20$, $m_i = 100$ for all $i \in [n]$ and $d = 20$. We generate positive semidefinite matrices $\A_{ij}, \B_{ij}, \C_{ij} \in \mathbb{R}^{d \times d}$ such that eigenvalues of $\A_{ij}, \C_{ij}$ lie in the interval $[0.01, 1]$ while those of $\B_{ij}$ lie in $[0,1]$. The vectors $a_{ij}, c_{ij} \in \mathbb{R}^{d}$ are generated from $\mathcal{N}_d(0, I_d)$ distribution. This data generation process ensures that the quadratic game satisfies the assumptions of our theory. To make the data heterogeneous, we produce different $\A_{ij}, \B_{ij}, \C_{ij}, a_{ij}, c_{ij}$ across the clients indexed by $i \in [n]$.

We present the results in Figure~\ref{fig: Comparison of ProxSkip-VIP-FL vs Local SGDA vs Local SEG on Heterogeneous Data in the deterministic setting.} for both deterministic and stochastic settings. As our theory predicted, our proposed methods are always faster in terms of communication rounds than Local (S)GDA. The current analysis of \algname{Local EG} requires the bounded heterogeneity assumption which leads to convergence to a neighborhood even in a deterministic setting. As also expected by the theory, our proposed variance-reduced algorithm converges linearly to the exact solution. 

\paragraph{Robust Least Square.}\label{sec:RobustLearning}

\begin{figure}
     \centering
    \begin{subfigure}[b]{0.45\textwidth}
        \centering
        \includegraphics[width=\textwidth]{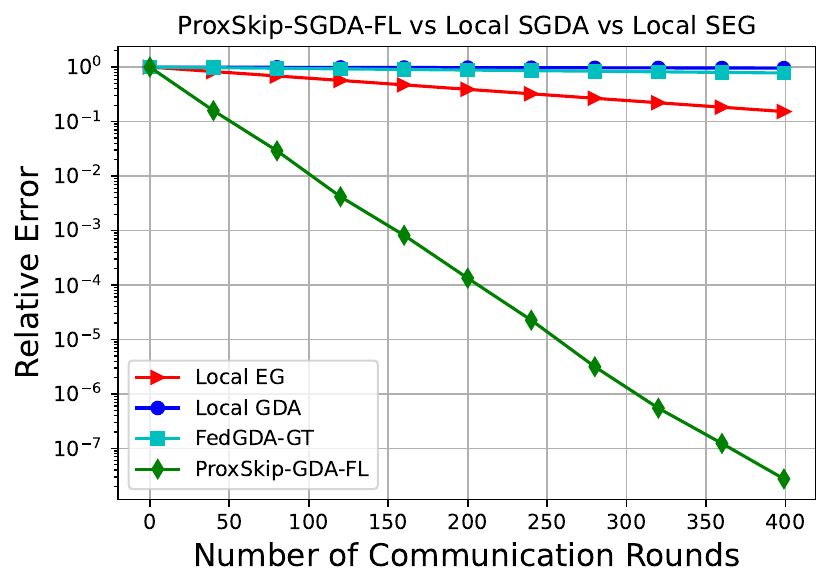}
        \caption{Deterministic Case}
        \label{fig:deterministicXtheoreticalXDatasetII}
    \end{subfigure}
    \hspace{1em}
    \begin{subfigure}[b]{0.45\textwidth}
        \centering
        \includegraphics[width=\textwidth]{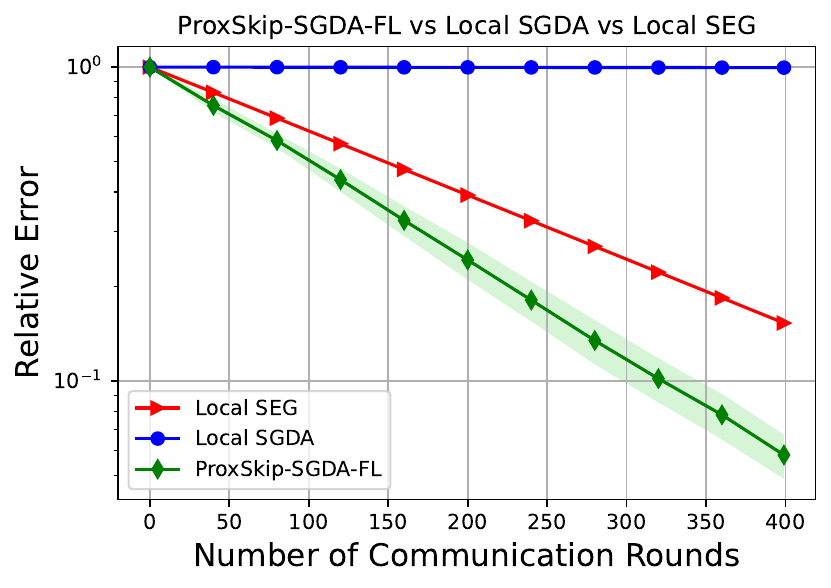}
        \caption{Stochastic Case}
        \label{fig:stochasticXtheoreticalXDatasetII}
    \end{subfigure}
    \caption{Comparison of algorithms on the Robust Least Square~\eqref{robustleastsquare}}
    \label{fig: Comparison of ProxSkip-VIP-FL vs Local SGDA vs Local SEG on Heterogeneous Data in the stochastic setting.}
\end{figure}

In the second experiment, we consider the robust least square (RLS) problem~\cite{el1997robust, yang2020global} with the coefficient matrix ${\bf A} \in \R^{r \times s}$ and noisy vector $y_0 \in \R^r$. We assume $y_0$ is corrupted by a bounded perturbation $\delta$ (i.e.,$\|\delta\| \leq \delta_0$). RLS minimizes the
worst case residual and can be formulated as follows: 
$$
\min_{\beta} \max_{\delta: \|\delta\| \leq \delta_0} \|{\bf A}\beta - y\|^2\quad\text{where}\quad\delta = y_0 - y.
$$
In this chapter, we consider the following penalized version of the RLS problem: 
\begin{eqnarray}
\label{robustleastsquare}
    \min_{\beta \in \R^s} \max_{y \in \R^r} \|{\bf A}\beta - y\|^2 - \lambda \|y - y_0\|^2
\end{eqnarray}
This objective function is strongly-convex-strongly-concave when $\lambda > 1$~\cite{thekumparampil2022lifted}. We run our experiment on the ``California Housing" dataset from scikit-learn package~\cite{pedregosa2011scikit}, using $\lambda = 50$ in \eqref{robustleastsquare}. Please see Appendix~\ref{apdx:numerical_experiments} for further detail on the setting.

In Figure~\ref{fig: Comparison of ProxSkip-VIP-FL vs Local SGDA vs Local SEG on Heterogeneous Data in the stochastic setting.}, we show the trajectories of our proposed algorithms compared to Local EG, Local GDA and their stochastic variants. In both scenarios (deterministic and stochastic), our \algname{ProxSkip-GDA/SGDA-FL} algorithms outperform the other methods in terms of communication rounds.
\chap{Conclusion} 

In this thesis, we focused on the development of optimization algorithms to address min-max optimization problems, a fundamental challenge in modern machine learning and game theory.

In Chapter~\ref{chap:chap-2}, we analyzed single-call \algname{SEG} methods under weaker assumptions on stochastic estimates, thereby ensuring applications in a broader class of problems. 

In Chapter~\ref{chap:chap-3}, we proposed adaptive \algname{SEG} algorithms that incorporate a Polyak-type update step size and a line-search based extrapolation strategy. These design choices eliminate the need for manual step size tuning. We further demonstrated that our adaptive algorithm outperforms existing methods across various experiments.

In Chapter~\ref{chap:chap-4}, we extended the analysis of \algname{EG} algorithms to settings where the operator satisfies the $\alpha$-symmetric $(L_0, L_1)$-Lipschitz condition. This is a significant step forward, as recent studies suggest that such relaxed conditions more accurately model the behavior of modern neural networks, in contrast to the classical $L$-Lipschitz assumption. Our work lays a theoretical foundation for developing extragradient-based algorithms tailored to modern applications.

Finally, in Chapter~\ref{chap:chap-5}, we proposed a modification of the distributed \algname{SGDA} algorithm by incorporating local updates, achieving optimal communication complexity in the distributed setting.

\noindent\textbf{Future Extensions.} Our work opens up several promising directions for future research. Below, we outline a few potential extensions:
\begin{enumerate}
    \item \textbf{Stochastic Algorithms for $\alpha$-Symmetric $(L_0, L_1)$-Lipschitz Operators:} 
    While recent advances have been made for deterministic algorithms under $\alpha$-symmetric $(L_0, L_1)$-Lipschitz conditions, the stochastic counterpart remains largely unexplored. In future work, we aim to develop stochastic algorithms with provable convergence guarantees for this broader class of operators.

    \item \textbf{Relaxed Conditions for Stochastic Estimates:}
    In Chapter~\ref{chap:chap-2}, we introduced the Expected Residual condition (Assumption~\ref{as:expected_residual}), which holds under Lipschitz continuity of operators $F_i$ in Equation~\eqref{eq:chap2_finitesum}. However, this condition may not be satisfied under $\alpha$-symmetric $(L_0, L_1)$-Lipschitz operators. A key future direction is to identify relaxed conditions that naturally hold for such operators and to develop a rigorous theoretical framework under these assumptions.

    \item \textbf{Solving Structured Non-Monotone Variational Inequality Problems (VIPs):}
    Many min-max problems arising in practice are non-monotone and fall outside the scope of existing theory based on the weak Minty assumption (which requires $\rho < \nicefrac{1}{L}$). We plan to explore more general frameworks or alternative regularity conditions that can capture a wider class of non-monotone min-max problems encountered in real-world applications.

    \item \textbf{\algname{EG} for Distributed Setting:}
    In Chapter~\ref{chap:chap-5}, we introduced the \algname{ProxSkip-VIP-FL} algorithm, a distributed method based on \algname{SGDA} with local steps that achieves optimal communication efficiency. While extragradient methods (EG) are known to offer faster convergence than gradient descent methods (\algname{GDA}) in centralized settings, it remains unclear whether distributed extragradient variants with local updates can offer similar advantages. We intend to investigate this direction, aiming to design a distributed \algname{EG} algorithm that outperforms existing methods such as \algname{ProxSkip-VIP-FL}.
\end{enumerate}

This concludes our study on designing algorithms for min-max optimization problems. The algorithms and theoretical tools developed in this thesis offer a solid foundation for future research in optimization and its applications in machine learning and beyond.




\BibTextSpacing                         
\fancyhead[L]{\nouppercase \leftmark}
\printbibliography[title={Bibliographic references},
    heading=bibintoc, notcategory=mypapers]
\clearpage                              

\MainTextSpacing                    
\fancyhead[L]{\appendixname\ \thechapter. \nouppercase \leftmark}

\appendix 
\makeatletter
\addtocontents{toc}{\protect\renewcommand\protect\cftchappresnum{\@chapapp\ }}
\makeatother
\renewcommand{\thechapter}{\Alph{chapter}}

\chapter{Appendices for Chapter \ref{chap:chap-2}} \label{chap:appendix-a}

\section{Further Related Work for Chapter \ref{chap:chap-2}}\label{section: Related Work}

The references necessary to motivate chapter \ref{chap:chap-2} and connect it to the most relevant literature are included in the appropriate sections of the main body of the paper. In this section, we present a broader view of the
literature, including more details on closely related work and more references to papers that are not directly related to our main results. 

\begin{itemize}[leftmargin=*]
\setlength{\itemsep}{0pt}
\item \textbf{Classes of Structured Non-monotone Operators.} 
With an increasing interest in improved computational speed, first-order methods are the primary choice for solving VIPs. However, computation of an approximate first-order locally optimal solution of a general non-monotone VIP is intractable \cite{daskalakis2021complexity,lee2021fast}. It motivates us to exploit the additional structures prevalent in large classes of non-monotone VIPs. Recently \cite{gorbunov2022stochastic, hsieh2019convergence} provide convergence guarantees of stochastic methods for solving quasi-strongly monotone VIPs, while \cite{hsieh2020explore} for problems satisfying error-bound conditions. \cite{diakonikolas2021efficient} defined the notion of a weak MVI~\eqref{eq: weak MVI} covering classes of non-monotone VIPs. 

\vspace{2mm}

\item \textbf{Assumptions on Operator Noise.} The standard analysis of stochastic methods for solving VIPs relies on bounded variance assumption. \cite{bohm2022solving, diakonikolas2021efficient, hsieh2019convergence, gidel2018variational} use bounded variance assumption (i.e. $\E \|F_i(x) - F(x)\|^2 \leq \sigma^2$ for all $x$) while \cite{nemirovski2009robust, abernethy2021last} assume bounded operators for their analysis. However, there are examples of simple quadratic games that do not satisfy these conditions. It has motivated researchers to look for alternative/relaxed assumptions on distributions. \cite{loizou2021stochastic} provides convergence of Stochastic Gradient Descent Ascent Method under Expected Cocoercivity. \cite{hsieh2020explore, mishchenko2020revisiting} considered alternative assumptions for analyzing Stochastic Extragradient Methods that do not imply boundedness of the variance. However, there is no analysis of single-call extragradient methods without bounded variance assumption.

\vspace{2mm}

\item \textbf{Weak Minty Variational Inequalities.}  Numerous contemporary studies look to identify first-order methods for efficiently solving min-max optimization problems. It varies from simple convex-concave to nontrivial nonconvex nonconcave objectives. Though there has been a significant development in the convex-concave setting, \cite{daskalakis2021complexity} demonstrates that even finding local solutions are intractable for general nonconvex nonconcave objectives. Therefore, researchers seek to identify the structure of objective functions for which it is possible to resolve the intractability issues. \cite{diakonikolas2021efficient} proposes the notion of non-monotonicity, which generalizes the existence of a Minty solution (i.e., $\rho = 0 $ in \eqref{eq: weak MVI}). This problem is known as weak Minty variational inequality in the literature. \cite{diakonikolas2021efficient, pethick2022escaping} provides convergence guarantees of the Extragradient Method for weak Minty variational inequality. They establish a convergence rate of $\cO(1/k)$ for the squared operator norm. \cite{lee2021fast} shows that it is possible to have an accelerated extragradient method even for non-monotone problems. Furthermore, \cite{bohm2022solving} provides a convergence guarantee for the \algname{SOG} with a complexity bound of $\cO(\varepsilon^{-2})$. However, all papers exploring stochastic extragradient methods for solving weak Minty variational inequality consider bounded variance assumption~\cite{bohm2022solving, diakonikolas2021efficient}. Moreover, all algorithms solving Weak Minty variational inequality require increasing batchsize. Recently, \cite{pethick2023solving} introduced \algname{BCSEG+} which can solve weak minty variational inequality without increasing batchsize. \algname{BCSEG+} involves three oracle calls per iteration and addition of a bias-corrected term in the extrapolation step.

\vspace{2mm}

\item \textbf{Arbitrary Sampling Paradigm.} As we mentioned in the main paper, the stochastic reformulation \eqref{Reformulation} of the original problem \eqref{eq: Variational Inequality Definition} allows us to analyze single-call extragradient methods under the arbitrary sampling paradigm. That is, provide  a unified analysis for \algname{SPEG} that captures multiple sampling strategies, including $\tau$-minibatch and importance samplings. An arbitrary sampling analysis of a stochastic optinmization method was first proposed in the context of the randomized coordinate descent method for solving strongly convex functions in \cite{richtarik2013optimal}. Since then, several other stochastic methods were studied in this regime, including accelerated coordinate descent algorithms \cite{qu2016coordinate, hanzely2019accelerated}, randomized iterative methods for solving consistent linear systems~\cite{richtarik2020stochastic, loizou2020momentum,loizou2020convergence}, randomized gossip algorithms \cite{loizou2016new, loizou2021revisiting}, stochastic gradient descent (\algname{SGD})~\cite{gower2019sgd,gower2021sgd}, and variance reduced methods~\cite{qian2019saga, horvath2019nonconvex, khaled2020unified}. The first analysis of stochastic algorithms under the arbitrary sampling paradigm for solving variational inequality problems was proposed in ~\cite{loizou2020stochastic,loizou2021stochastic}.  In \cite{loizou2020stochastic,loizou2021stochastic}, the
authors focus on algorithms like the stochastic Hamiltonian method, the stochastic gradient descent ascent,
and the stochastic consensus optimization. These ideas were later extended to the case of Stochastic Extragradient by \cite{gorbunov2022stochastic}. To the best of our knowledge, our work is the first that provides an analysis of single-call extragradient methods under the arbitrary sampling paradigm.

\vspace{2mm}

\item \textbf{Overparameterized Models and Interpolation.} For a function $f(x) \coloneqq \frac{1}{n} \sum_{i = 1}^n f_i(x)$ we say that interpolation condition holds if there exists $x^*$ such that $\min_{x} f_i(x) = f_i(x^*)$ for all $i \in [n]$ (or equivalently $\nabla f_i(x^*) = 0$ for smooth convex functions)~\cite{gower2021sgd}. The interpolation condition is satisfied when the underlying models are sufficiently overparameterized ~\cite{vaswani2019fast}. Some known examples include deep matrix factorization and classification using neural networks \cite{arora2019implicit,rolinek2018l4, vaswani2019fast}. The interpolated model structure enables \algname{SGD} and other optimization algorithms to have faster convergence \cite{gower2021sgd,pmlr-v130-loizou21a, d2021stochastic}. 
Inspired by this, one can extend the notion of the interpolation condition to operators. In this scenario, we say that the VIP \eqref{eq: Variational Inequality Definition} is interpolated if there exists solution $x^*$ of \eqref{eq: Variational Inequality Definition} such that $F_i(x^*) = 0$ for all $i \in [n]$. This concept has been explored for analyzing the stochastic extragradient method in \cite{vaswani2019painless, li2022convergence}. We highlight that our proposed theorems show fast convergence of \algname{SPEG} in this interpolated regime (when $\sigma_*^2=0$). To the best of our knowledge, our work is the first that proves such convergence for \algname{SPEG}. In Fig.~\ref{fig: Interpolated Model}, we experimentally verify the fast convergence for solving a strongly monotone interpolated problem.

\vspace{2mm}

\item \textbf{Deterministic Extragradient Methods.} The Extragradient method (\algname{EG})~\cite{korpelevich1976extragradient} and its single-call variant, Optimistic Gradient (\algname{OG}) \cite{popov1980modification}, were proposed to overcome the convergence issues of gradient descent-ascent method for solving monotone problems. Since their introduction, these methods have been revisited and explored in various ways. \cite{mokhtari2020unified} analyzed \algname{EG} and \algname{OG} as an approximation of the Proximal Point method to solve bilinear and strongly convex-strongly concave min-max problems. \cite{solodov1999hybrid} and \cite{ryu2019ode} provide the best-iterate convergence guarantees of \algname{EG} and \algname{OG} with a rate of $\cO(\nicefrac{1}{K})$ for solving monotone problems. However, providing a last-iterate convergence rate of \algname{EG} and \algname{OG} for monotone VIPs has been a long-lasting open problem that was only recently resolved. The works of \cite{golowich2020tight, gorbunov2022extragradient, cai2022tight} prove a last-iterate $\cO(\nicefrac{1}{K})$ convergence rate for these methods. Finally, in the deterministic setting, some recent works provide convergence analysis of \algname{EG} and \algname{OG} for solving weak MVI~\eqref{eq: weak MVI} \cite{diakonikolas2021efficient, pethick2022escaping, bohm2022solving, gorbunov2022convergence}.

\end{itemize}

\section{Technical Preliminaries}\label{section: tachnical preliminaries}
Throughout the chapter, we assume 
\begin{assumption}
Operator $F$ in \eqref{eq: Variational Inequality Definition} is L Lipschitz, i.e.,  $\forall x,y \in \mathbb{R}^d$ operator $F$ satisfies
\begin{equation}\label{eq: F lipschitz}
    \|F(x) - F(y)\| \leq L \|x - y\|.
\end{equation}
Operators $F_i: \mathbb{R}^d \to \mathbb{R}^d $  of problem \eqref{eq: Variational Inequality Definition} are $L_i$- Lipschitz, i.e., $\forall x,y \in \mathbb{R}^d$ operator $F_i$ satisfies
\begin{equation}\label{eq: F_i lipschitz}
    \|F_i(x) - F_i(y)\| \leq L_i \|x - y\|.
\end{equation}
\end{assumption}
In our proofs, we often use the following simple inequalities.

\begin{lemma}\label{Lemma: Young's Inequality} For all $a, b, a_1, a_2, \cdots a_n \in \mathbb{R}^d, n \geq 1, \alpha > 0$, we have the following inequalities:

\begin{eqnarray}
     \la a, b \ra &\leq& \|a\| \|b\|, \label{eq: Cauchy Schwarz Inequality} \\
     \la a,b \ra &\leq& \frac{1}{2\alpha}\|a\|^2 + \frac{\alpha}{2}\|b\|^2,\\
    \|a+b\|^2 &\leq& 2\|a\|^2 + 2\|b\|^2,  \label{eq: Young's Inequality}\\
    \|a\|^2 &\geq& \frac{1}{2}\|a+b\|^2 - \|b\|^2, \\
    \bigg\|\sum_{i = 1}^n a_i \bigg\|^2 &\leq& n \sum_{i = 1}^n \|a_i\|^2. \label{eq: n dimensional young's inequality}    
\end{eqnarray}

\end{lemma}
Inequality~\eqref{eq: Young's Inequality} is well known as Young's Inequality. Now, we present a simple property of unbiased estimators.

\vspace{2mm}

\begin{lemma}\label{Lemma: variance of an unbiased estimator}
For an unbiased estimator $g$ of operator $F$ i.e. $\E[g(x)] = F(x)$ we have 
\begin{equation}\label{eq: variance of an unbiased estimator}
    \E \|g(x) - F(x)\|^2 = \E \|g(x)\|^2 - \|F(x)\|^2.
\end{equation}
\end{lemma}
Next, we present the following lemma from \cite{stich2019unified}, which plays a vital role in proving the convergence guarantee of Theorem \ref{Theorem: Total number of iteration knowledge}.
\begin{lemma}(Simplified Verison of Lemma 3 from \cite{stich2019unified})\label{Lemma: Stich lemma}
Let the non-negative sequence $\{r_k\}_{k \geq 0}$ satisfy the relation $r_{k+1} \leq (1 - a \gamma_k)r_k + c \gamma_k^2$ for all $k \geq 0$, parameters $a,c \geq 0$ and any non-negative sequence $\{\g_k\}_{k \geq 0}$ such that $\gamma_k \leq \frac{1}{h}$ for some $h \geq a, h >0$. Then for any $K \geq 0$ one can choose $\{\gamma_k\}_{k\geq 0}$ as follows:
\begin{equation*}
    \begin{split}
        \text{if $K \leq \frac{h}{a}$}, & \qquad \gamma_k = \frac{1}{h},\\
        \text{if $K > \frac{h}{a}$ and $k < k_0$}, & \qquad \gamma_k = \frac{1}{h},\\
        \text{if $K > \frac{h}{a}$ and $k \geq k_0$}, & \qquad \gamma_k = \frac{2}{a(\kappa + k - k_0)},
    \end{split}
\end{equation*}
where $\kappa = \frac{2h}{a}$ and $k_0 = \left[ \frac{K}{2} \right]$. For this choice of $\g_k$ the following inequality holds:
$$
r_K \leq \frac{32hr_0}{a} \exp{\Bigg(-\frac{aK}{2h}\Bigg)} + \frac{36c}{a^2K}.
$$
\end{lemma}
We use the next lemma to bound the trace of matrix products. 
\vspace{2mm}

\begin{lemma}
For positive semidefinite matrices $A,B \in \mathbb{R}^{d \times d}$ we have 
\begin{equation}\label{eq: trace inequality}
    \text{tr}(AB) \leq \lambda_{\max}(B) \text{tr}(A),
\end{equation}
where $\lambda_{\max}(B)$ denotes the maximum eigenvalue of $B$.
\end{lemma}
Next lemma proves equivalence of \algname{SPEG} and \algname{SOG}:
\begin{proposition}[\textbf{Equivalence of \algname{SPEG} and \algname{SOG}}] \label{proposition: equivalence of SPEG and SOG}
Consider the iterates of \algname{SPEG} $\{x_k, \hat{x}_k\}_{k = 1}^{\infty}$ with constant step-sizes $\om_k = \om, \gamma_k = \gamma$ in \eqref{SPEG_UpdateRule}. Then $\hat{x}_k$ follows the iteration rule of \algname{SOG} i.e.
\begin{eqnarray}
     \hat{x}_{k+1} = \hat{x}_k - \om_k F_{v_k}(\hat{x}_k) - \gamma_k \large[F_{v_k}(\hat{x}_k) - F_{v_{k-1}}(x_{k-1})\large]
\end{eqnarray}
\end{proposition}
\begin{proof}
From the update rule of \algname{SPEG}~\eqref{SPEG_UpdateRule} we get
\begin{equation*}
    \begin{split}
        \hat{x}_{k+1} = \quad & x_{k + 1} - \gamma F_{v_k}(\hat{x}_k) \\
        = \quad & x_k - \om F_{v_k}(\hat{x}_k) - \gamma F_{v_k}(\hat{x}_k) \\
        = \quad & x_k - (\om + \gamma) F_{v_k} (\hat{x}_k) \\
        = \quad & \hat{x}_k + \gamma F_{v_{k - 1}}(\hat{x}_{k - 1}) - (\om + \gamma) F_{v_k} (\hat{x}_k) \\
        = \quad & \hat{x}_k  - \om  F_{v_k} (\hat{x}_k) - \gamma \Big(F_{v_k} (\hat{x}_k) - F_{v_{k - 1}}(\hat{x}_{k - 1}) \Big).
    \end{split}
\end{equation*}
This shows that \algname{SPEG} iterations are equivalent to \algname{SOG}, with $\hat{x}_k$ being the $k$-th iterate of \algname{SOG}.
\end{proof}

\section{Example:  A Problem where the Bounded Variance Condition not Hold}\label{sec:BoundedVarianceCounterExample}
Here, we provide a simple problem that does not satisfy the bounded variance assumption. Consider the linear regression problem
\begin{eqnarray*}
        \min_{x \in \mathbb{R}}f(x) := \frac{1}{2} (a_1x - b_1)^2 + \frac{1}{2} (a_2x - b_2)^2
\end{eqnarray*}
where $x \in \mathbb{R}$. Here $f_1(x) = (a_1x - b_1)^2$ and $f_2(x) = (a_2x - b_2)^2$. Now consider the estimator $g(x)$ of $\nabla f(x)$ under uniform sampling i.e. $g(x)$ takes the value $\nabla f_1(x)$ with probability $\frac{1}{2}$ and $\nabla f_2(x)$ with probability $\frac{1}{2}$. Then we have
\begin{eqnarray*}
        \Exp \|g(x) -  \nabla f(x)\|^2 & = & \frac{1}{2} \|\nabla f_1 (x) - \nabla f(x)\|^2 + \frac{1}{2} \|\nabla f_2(x) - \nabla f(x)\|^2 \\
        & = & \frac{1}{2} \cdot \frac{1}{4}\|\nabla f_1 (x) - \nabla f_2(x)\|^2 + \frac{1}{2} \cdot \frac{1}{4}\|\nabla f_2 (x) - \nabla f_1(x)\|^2 \\
        & = & \frac{1}{4} \|\nabla f_1(x) - \nabla f_2(x)\|^2 \\
        & = & \frac{1}{4} \left( 2(a_1x - b_1)a_1 - 2(a_2x - b_2)a_2 \right)^2 \\
        & = & \left( (a_1^2 - a_2^2) x - (a_1b_1 - a_2b_2) \right)^2
\end{eqnarray*}
Therefore, $\Exp \|g(x) -  \nabla f(x)\|^2$ is a quadratic function of $x$ with the coefficient of $x$ being positive. Hence, as $x \to \infty$, we have $\Exp \|g(x) -  \nabla f(x)\|^2 \to \infty$, which means that a constant can not bound the variance.

\newpage
\section{Proofs of Results on Expected Residual}\label{section: proofs of results on expected residual}
\subsection{Proof of Lemma \ref{Lemma: variance bound}}
\begin{proof}
Using Young's Inequality~\eqref{eq: Young's Inequality}, we get 
\begin{equation*}
    \begin{split}
        \E \|g(x) - F(x)\|^2 \overset{\eqref{eq: Young's Inequality}}{\leq} \quad & 2 \E \|g(x) - F(x) - g(x^*)\|^2 + 2 \E \|g(x^*)\|^2 \\
        \overset{\eqref{eq: ER Condition}}{\leq} \quad & \delta \|x -x^*\|^2 + 2 \E \|g(x^*)\|^2. \\
    \end{split}
\end{equation*}
Then breaking down the RHS, we obtain 
\begin{equation*}
    \begin{split}
         \E \|g(x)\|^2 - \|F(x)\|^2 \overset{\eqref{eq: variance of an unbiased estimator}}{\leq} \delta \|x -x^*\|^2 + 2 \E \|g(x^*)\|^2.
    \end{split}
\end{equation*}
Now we rearrange the terms and set $\sigma_*^2 = \E \|g(x^*)\|^2$ to complete the proof of this Lemma.
\end{proof}

\begin{proposition}\label{Proposition lipschitz implies ER}
 If $F_i$ are $L_i$-lipschitz then Expected Residual condition~\eqref{eq: ER Condition} holds. In that case
\begin{equation*}
    \begin{split}
        \delta =  \frac{2}{n} \sum_{i = 1}^n L_i^2 \E (v_i^2). 
    \end{split}
\end{equation*}
In addition, if $F$ is $\mu$-quasi strongly monotone \eqref{eq: Strong Monotonicity} then we have 
\begin{equation*}
    \begin{split}
       \delta = \frac{2}{n} \sum_{i = 1}^n L_i^2 \E (v_i^2) - 2\mu^2. 
    \end{split}
\end{equation*}
\end{proposition}

\begin{proof}
Note that
\begin{eqnarray}
        \E\|(F_v(x) - F_v(x^*)) - (F(x) - F(x^*))\|^2 &=&  \E\|F_v(x) - F_v(x^*)\|^2 + \|F(x) - F(x^*)\|^2 \notag\\
        && \quad - 2 \E \la F_v(x) - F_v(x^*), F(x) - F(x^*)\ra \notag\\
        &=&  \E \|F_v(x) - F_v(x^*)\|^2 - \|F(x) - F(x^*)\|^2 \notag\\
        &=&  \E \|F_v(x) - F_v(x^*)\|^2  - \|F(x)\|^2  \notag\\
        &=& \E \bigg\| \frac{1}{n} \sum_{i = 1}^n v_i (F_i(x) - F_i(x^*)) \bigg\|^2  - \|F(x)\|^2 \notag\\
        &=& \frac{1}{n^2} \E \bigg \| \sum_{i = 1}^n v_i (F_i(x) - F_i(x^*)) \bigg\|^2  - \|F(x)\|^2 \notag\\
        &\overset{\eqref{eq: n dimensional young's inequality}}{\leq}& \frac{1}{n} \sum_{i = 1}^n \E(v_i^2) \|F_i(x) - F_i(x^*)\|^2  - \|F(x)\|^2 \notag\\
        &\overset{\eqref{eq: F_i lipschitz}}{\leq}& \frac{\|x - x^*\|^2}{n} \sum_{i = 1}^n \E(v_i^2) L_i^2  - \|F(x)\|^2 \label{eq: Expected Residual bound for lipschitz quasi strongly monotone F}.
\end{eqnarray}
The first part of the lemma follows by ignoring the positive term $\|F(x)\|^2$. For the second part we assume $F$ is $\mu$-quasi strongly monotone. Then we have
$$\mu \|x - x^*\|^2 \overset{\eqref{eq: Strong Monotonicity}}{\leq} \la F(x), x - x^*\ra \overset{\eqref{eq: Cauchy Schwarz Inequality}}{\leq} \|F(x)\| \|x - x^*\|.$$ 
Cancelling $\|x - x^*\|$ from both sides we get 
\begin{equation}\label{eq: lower bound on norm of quasi strongly monotone F}
\mu \|x - x^*\| \leq \|F(x)\|.    
\end{equation}
Therefore we have the following bound for $\mu$-quasi strongly monotone operator $F$:
\begin{equation*}
        \E \|(F_v(x) - F_v(x^*)) - (F(x) - F(x^*))\|^2
        \overset{\eqref{eq: Expected Residual bound for lipschitz quasi strongly monotone F}, \eqref{eq: lower bound on norm of quasi strongly monotone F}}{\leq} \Bigg( \frac{1}{n} \sum_{i = 1}^n \E(v_i^2) L_i^2 - \mu^2 \Bigg) \|x - x^*\|^2.
\end{equation*}
This proves the second part of the lemma. This lemma ensures that the Lipschitz property is sufficient to guarantee Expected Residual~\eqref{eq: ER Condition} condition. 
\end{proof}

\subsection{Proof of Proposition \ref{Prop_SufficientCondition}}
\begin{proof}
    Proposition \ref{Proposition lipschitz implies ER} implies that Lipschitzness of all operators $F_i$ is enough to ensure that \ref{eq: ER Condition} holds. For $\tau$- minibatch sampling, denote the matrix $\textbf{R} = \Big(F_1(x) - F_1(x^*), \cdots, F_n(x) - F_n(x^*)\Big) \in \mathbb{R}^{d \times n}$. Then we obtain the following bound:
\begin{eqnarray*}
\E \|F_v(x) - F_v(x^*) - (F(x) - F(x^*))\|^2 \hspace{-3mm} &=& \hspace{-3mm} \E \left\| \frac{1}{n} \sum_{i = 1}^n v_i (F_i(x) - F_i(x^*)) - (F_i(x) - F_i(x^*)) \right \|^2 \\
&=& \hspace{-3mm}\frac{1}{n^2} \E \bigg \| \sum_{i = 1}^n (v_i -1) (F_i(x) - F_i(x^*)) \bigg\|^2 \\
&=& \hspace{-3mm}\frac{1}{n^2} \E \big \| \textbf{R}(v - \mathbf{1}) \big\|^2 \\
&=& \hspace{-3mm} \frac{1}{n^2} \E (v - \mathbf{1})^{\intercal} \textbf{R}^{\intercal}\textbf{R} (v - \mathbf{1}) \\
&=& \hspace{-3mm} \frac{1}{n^2} \E \bigg( \text{tr} \bigg( \textbf{R}^{\intercal}\textbf{R} (v - \mathbf{1}) (v - \mathbf{1})^{\intercal} \bigg) \bigg) \\
&=& \hspace{-3mm}\frac{1}{n^2} \text{tr} \bigg( \textbf{R}^{\intercal}\textbf{R} \E \bigg((v - \mathbf{1}) (v - \mathbf{1})^{\intercal} \bigg) \bigg) \\
&=& \hspace{-3mm}\frac{1}{n^2} \text{tr} \bigg( \textbf{R}^{\intercal}\textbf{R} \textbf{Var}[v] \bigg) \bigg) \\
&\overset{\eqref{eq: trace inequality}}{\leq}& \hspace{-3mm} \frac{\lambda_{\max}\big(\textbf{Var}[v]\big)}{n^2} \text{tr}(\textbf{R}^{\intercal}\textbf{R}) \\
&=& \hspace{-3mm} \frac{\lambda_{\max}\big(\textbf{Var}[v]\big)}{n^2} \sum_{i = 1}^n \|F_i(x) - F_i(x^*)\|^2 \\
&\overset{\eqref{eq: F_i lipschitz}}{\leq}& \frac{\lambda_{\max}(\textbf{Var}[v]) \|x - x^*\|^2}{n^2}\sum_{i = 1}^n L_i^2.   
\end{eqnarray*}
From the proof details of Lemma F.3 in \cite{sebbouh2019towards} we have $\lambda_{\max}(\textbf{Var}[v]) = \frac{n(n - \tau)}{\tau (n - 1)}$ for $\tau$-minibatch sampling. Thus we obtain
\begin{equation*}
    \E \big\|F_v(x) - F_v(x^*) - (F(x) - F(x^*)) \big  
    \|^2 \leq \frac{2(n - \tau)}{n \tau (n - 1)} \sum_{i = 1}^n L_i^2 \|x - x^*\|^2.
\end{equation*}
Now we focus on the derivation of $\sigma_*^2 = \E\|F_v(x^*)\|^2$ for $\tau$-minibatch sampling. We expand $\E \|F_v(x^*)\|^2$ as follows:
\begin{eqnarray}
        \E \|F_v(x^*)\|^2 &=& \frac{1}{n^2} \E \bigg\|\sum_{i = 1}^n v_iF_i(x^*) \bigg\|^2 \notag \\
        &=& \frac{1}{n^2} \E \bigg\|\sum_{i \in S} \frac{1}{p_i}F_i(x^*) \bigg\|^2 \notag\\
        &=& \frac{1}{n^2} \E \bigg\|\sum_{i = 1}^ n \textbf{1}_{i \in S}\frac{1}{p_i}F_i(x^*) \bigg\|^2 \notag\\
        &=& \frac{1}{n^2}\E \bigg \langle \sum_{i = 1}^n \textbf{1}_{i \in S}\frac{1}{p_i} F_i(x^*), \sum_{j = 1}^n \textbf{1}_{j \in S}\frac{1}{p_j} F_j(x^*) \bigg \rangle \notag\\
        &=& \frac{1}{n^2} \sum_{i,j = 1}^n \frac{P_{ij}}{p_i p_j} \langle F_i(x^*), F_j(x^*) \rangle, \label{eq: expansion for sigma}
\end{eqnarray}
where $P_{ij} = P(i, j \in S)$ and $p_i = P(i \in S)$. For $\tau$-minibatch sampling, we obtain $P_{ij} = \frac{\tau(\tau - 1)}{n(n-1)}$ and $p_i = \frac{\tau}{n}$. Plugging in these values of $P_{ij}$ and $p_i$ in \eqref{eq: expansion for sigma} we get the closed-form expression of $\sigma_*^2$. This completes the proof of Proposition \ref{Prop_SufficientCondition}.
\end{proof}

\subsection{Proof of Proposition \ref{Proposition connecting assumptions}}
Here we enlist the assumptions made on operators. Suppose $g$ is an estimator of operator $F$.
\begin{equation*}
    \begin{split}
        & \textbf{1. Bounded Operator:} \quad  \E\|g(x)\|^2 \leq \sigma^2 \\
        & \textbf{2. Bounded Variance:} \quad  \E\|g(x) - F(x)\|^2 \leq \sigma^2 \\
        & \textbf{3. Growth Condition:} \quad  \E\|g(x)\|^2 \leq \alpha \|F(x)\|^2 + \beta \\
        & \textbf{4. Expected Co-coercivity:} \quad \E \|g(x) - g(x^*)\|^2 \leq l_F \la F(x), x-x^* \ra \\
        & \textbf{5. Expected Residual:} \quad \E \|(g(x) - g(x^*)) - (F(x) - F(x^*))\|^2 \leq \frac{\delta}{2} \|x-x^*\|^2 \\
        & \textbf{6. Bound from Lemma \ref{Lemma: variance bound}:} \quad\E\|g(x)\|^2 \leq \delta \|x - x^*\|^2 + \|F(x)\|^2 + 2\sigma_*^2 \\
        & \textbf{7. $F_i$ are Lipschitz:} \quad \|F_i(x) - F_i(y)\| \leq L_i\|x - y\|\quad \forall \; i=1,\ldots,n
    \end{split}
\end{equation*}

\begin{proof} Here we will prove Proposition \ref{Proposition connecting assumptions}
\begin{itemize}
    \item $1 \implies 2$. Note that $\E \|g(x)\|^2 \leq \sigma^2 \leq \|F(x)\|^2 + \sigma^2 \implies \E \|g(x) - F(x)\| \leq \sigma^2$.
    \item $2 \implies 3$. Here $\E\|g(x) - F(x)\|^2 \leq \sigma^2 \implies \E \|g(x)\|^2 \leq \|F(x)\|^2 + \sigma^2$ as $g$ is an unbiased for estimator of $F$. Then take $\alpha = 1$ and $\beta = \sigma^2$.
    \item $3 \implies 6$. Note that $\E \|g(x)\|^2 \leq \alpha \|F(x)\|^2 + \beta \leq \alpha L^2 \|x - x^*\|^2 + \beta$. The last inequality follows from lipschitz property of $F$ and $F(x^*) = 0$. Then choose $\delta = \alpha L^2$ and $\sigma_*^2 = \nicefrac{\beta}{2}$ to get the result.
    \item $4\implies 5$. Note that expected cocoercivity and $L$-Lipschitzness of $F$ imply $\E \|(g(x) - g(x^*)) - (F(x) - F(x^*))\|^2 = \E \|g(x) - g(x^*)\|^2 - \|F(x) - F(x^*)\|^2 \leq \E \|g(x) - g(x^*)\|^2 \leq l_F \la F(x), x-x^* \ra \overset{\eqref{Lemma: Young's Inequality}}{\leq} \frac{l_F}{2L} \|F(x)\|^2 + \frac{l_F L}{2} \|x -x^*\|^2 \leq l_F L \|x -x^*\|^2$.
    \item $7 \implies 5$. This follows from Proposition \ref{Proposition lipschitz implies ER}.
    \item $5 \implies 6$. This follows from Lemma \ref{Lemma: variance bound}
\end{itemize}
\end{proof}

\newpage
\section{Main Convergence Analysis Results}\label{section: main convergence analysis results}
First, we present some results followed by iterates of \algname{SPEG}. These will play a key role in proving the Theorems later in this section. Recall that iterates of \algname{SPEG} are
\begin{equation*}
    \begin{split}
        \hat{x}_k & = x_k - \gamma_k F_{v_{k - 1}}(\hat{x}_{k - 1}), \\
    x_{k + 1} & = x_k - \om_k F_{v_k}(\hat{x}_k).
    \end{split}
\end{equation*}

\begin{lemma}\label{Lemma: Breakdown} For \algname{SPEG} iterates with step-size $\omega_k = \gamma_k = \om$, we have
\begin{equation}\label{eq: Breakdown}
    \begin{split}
        \|x_{k+1} - x^*\|^2 = \|x_{k+1} - \hat{x}_k\|^2 + \|x_k - x^*\|^2 - \|\hat{x}_k - x_k\|^2 - 2 \om \la F_{v_k}(\hat{x}_k), \hat{x}_k - x^*\ra.
    \end{split}
\end{equation}
\end{lemma}

\begin{proof}
We have
\begin{eqnarray*}
    \|x_{k+1} - x^*\|^2 &=& \|x_{k+1} - \hat{x}_k + \hat{x}_k - x_k + x_k - x^*\|^2 \\
    &=& \|x_{k+1} - \hat{x}_k\|^2 + \|\hat{x}_k - x_k\|^2 + \|x_k - x^*\|^2 +  2 \la \hat{x}_k - x_k, x_k - x^* \ra \\
    && \quad + 2 \la x_{k+1} - \hat{x}_k, \hat{x}_k - x_k \ra + 2 \la x_{k+1} - \hat{x}_k, x_k - x^* \ra \\
    &=& \|x_{k+1} - \hat{x}_k\|^2 + \|\hat{x}_k - x_k\|^2 + \|x_k - x^*\|^2 + 2 \la x_{k+1} - \hat{x}_k, \hat{x}_k - x^* \ra \\
    && \quad + 2 \la \hat{x}_k - x_k, x_k - x^* \ra\\
    &=& \|x_{k+1} - \hat{x}_k\|^2 + \|\hat{x}_k - x_k\|^2 + \|x_k - x^*\|^2 + 2 \la x_{k+1} - \hat{x}_k, \hat{x}_k - x^* \ra \\
    && \quad + 2 \la \hat{x}_k - x_k, x_k - \hat{x}_k + \hat{x}_k - x^* \ra \\
    &=&\|x_{k+1} - \hat{x}_k\|^2 + \|\hat{x}_k - x_k\|^2 + \|x_k - x^*\|^2 + 2 \la x_{k+1} - \hat{x}_k, \hat{x}_k - x^* \ra \\
    && \quad + 2 \la \hat{x}_k - x_k,  \hat{x}_k - x^* \ra - 2\|\hat{x}_k - x_k\|^2 \\
    &=&\|x_{k+1} - \hat{x}_k\|^2 - \|\hat{x}_k - x_k\|^2 + \|x_k - x^*\|^2 + 2 \la x_{k+1} - \hat{x}_k, \hat{x}_k - x^* \ra \\
    &&\quad + 2 \la \hat{x}_k - x_k,  \hat{x}_k - x^* \ra  \\
    &=& \|x_{k+1} - \hat{x}_k\|^2 - \|\hat{x}_k - x_k\|^2 + \|x_k - x^*\|^2 + 2 \la x_{k+1} - x_k, \hat{x}_k - x^* \ra \\
    &=& \|x_{k+1} - \hat{x}_k\|^2 - \|\hat{x}_k - x_k\|^2 + \|x_k - x^*\|^2 - 2 \om \la F_{v_k}(\hat{x}_k), \hat{x}_k - x^* \ra.
\end{eqnarray*}
\end{proof}

\begin{lemma}\label{Lemma: bound on difference of gradients}
Let $F$ be $L$-Lipschitz, and let \ref{eq: ER Condition} hold. Then \algname{SPEG} iterates satisfy
\begin{equation}\label{eq: bound on difference of gradients}
    \begin{split}
        \E_{\mathcal{D}} \|F_{v_k}(\hat{x}_k) - F_{v_{k-1}}(\hat{x}_{k-1})\|^2 \leq \quad & \delta \|\hat{x}_k - x^*\|^2 + 2\delta \|\hat{x}_{k-1} - x^*\|^2 + 2 L^2 \|\hat{x}_k - \hat{x}_{k-1}\|^2  + 6 \sigma_*^2.
    \end{split}
\end{equation}
\end{lemma}
\begin{proof}
\begin{eqnarray*}
    \E_{\mathcal{D}} \|F_{v_k}(\hat{x}_k) - F_{v_{k-1}}(\hat{x}_{k-1})\|^2 &=&\E_{\mathcal{D}} \|F_{v_k}(\hat{x}_k) - F(\hat{x}_{k})\|^2 + \E_{\mathcal{D}} \|F(\hat{x}_k) - F_{v_{k-1}}(\hat{x}_{k-1})\|^2  \\
    && \quad + 2 \E_{\mathcal{D}} \la F_{v_k}(\hat{x}_k) - F(\hat{x}_{k}), F(\hat{x}_k) - F_{v_{k-1}}(\hat{x}_{k-1}) \ra \\
    &=& \E_{v_{k}} \|F_{v_k}(\hat{x}_k) - F(\hat{x}_{k})\|^2 + \E_{\mathcal{D}} \|F(\hat{x}_k) - F_{v_{k-1}}(\hat{x}_{k-1})\|^2 \\
    &\overset{\eqref{eq: Young's Inequality}}{\leq}& \E_{\mathcal{D}} \|F_{v_k}(\hat{x}_k) - F(\hat{x}_{k})\|^2 + 2\E_{\mathcal{D}} \|F(\hat{x}_k) - F(\hat{x}_{k-1})\|^2 \\
    && \quad + 2\E_{\mathcal{D}} \|F(\hat{x}_{k-1}) - F_{v_{k-1}}(\hat{x}_{k-1})\|^2 \\
    &=& \E_{\mathcal{D}} \|F_{v_k}(\hat{x}_k)\|^2 - \|F(\hat{x}_{k})\|^2 + 2 \|F(\hat{x}_k) - F(\hat{x}_{k-1})\|^2\\
    && \quad + 2\E_{\mathcal{D}} \|F_{v_{k-1}}(\hat{x}_{k-1}) \|^2 - 2 \|F(\hat{x}_{k-1})\|^2 \\
    &\overset{\eqref{eq: variance bound}}{\leq}& \delta \|\hat{x}_k - x^*\|^2 + 2\delta \|\hat{x}_{k-1} - x^*\|^2 + 6\sigma_*^2\\
    && \quad + 2 \|F(\hat{x}_k) - F(\hat{x}_{k-1})\|^2 \\
    &\overset{\eqref{eq: F lipschitz}}{\leq}& \delta \|\hat{x}_k - x^*\|^2 + 2\delta \|\hat{x}_{k-1} - x^*\|^2 + 6\sigma_*^2\\
    && \quad + 2 L^2 \|\hat{x}_k - \hat{x}_{k-1}\|^2.
\end{eqnarray*}
\end{proof}

\begin{lemma}\label{Lemma: conditions}
For $\om \in \bigg[0,  \frac{1}{4L} \bigg]$ the following two conditions hold:
\begin{align}
        & 2 \om(\mu - \om \delta) + 8 \om^2 L^2 -1 \leq 0,  \label{eq: conditions 1}\\
        \text{and}\quad & 8 \om^2 (\delta + L^2) \leq 1 - \om \mu + 9 \om^2\delta.  \label{eq: conditions 2}
\end{align}
\end{lemma}

\begin{proof}
Note that for $\om \in \bigg[0,  \frac{1}{4L} \bigg]$, we have $$2 \om(\mu - \om \delta) + 8 \om^2 L^2 -1 \overset{\om^2\delta \geq 0}{\leq} 2 \om\mu  + 8 \om^2 L^2 -1 \overset{\om \leq \frac{1}{4L}}{\leq} \frac{\mu}{2L} + \frac{1}{2} - 1 \overset{\mu \leq L}{\leq} 0.$$
This proves the first condition. The second condition is equivalent to $ \om(\mu - \om \delta) + 8 \om^2 L^2 -1 \leq 0 $, which is again true using similar arguments.
\end{proof}

\subsection{Proof of Theorem \ref{Theorem: constant stepsize theorem}}
\begin{proof}
For $\om \in \bigg[0, \frac{\mu}{18 \delta}\bigg]$ we have $\om(\mu - 9\om \delta) \geq 0$ and $1 - \om(\mu - 9 \om \delta) \leq 1 - \frac{\om \mu}{2}$. Then we derive
\begin{eqnarray*}
    \E_{\mathcal{D}}[\|x_{k+1} - x^*\|^2 + \|x_{k+1} - \hat{x}_k\|^2] &\overset{\eqref{eq: Breakdown}}{=}&\|x_k - x^*\|^2 + 2\E_{\mathcal{D}} \|x_{k+1} - \hat{x}_k\|^2 - \|\hat{x}_k - x_k\|^2 \\
    && \quad - 2 \om \E_{\mathcal{D}} \la F_{v_{k}}(\hat{x}_k), \hat{x}_k - x^*\ra \\
    &=&\|x_k - x^*\|^2 + 2\E_{\mathcal{D}} \|x_{k+1} - \hat{x}_k\|^2 - \|\hat{x}_k - x_k\|^2\\
    && \quad - 2 \om \la F(\hat{x}_k), \hat{x}_k - x^*\ra \\
    &\overset{\eqref{eq: Strong Monotonicity}}{\leq}&\|x_k - x^*\|^2 + 2\E_{\mathcal{D}} \|x_{k+1} - \hat{x}_k\|^2 - \|\hat{x}_k - x_k\|^2\\
    && \quad - 2 \om \mu \| \hat{x}_k - x^* \|^2 \\
    &=&\|x_k - x^*\|^2 + 2\om^2 \E_{\mathcal{D}} \|F_{v_k}(\hat{x}_k) - F_{v_{k-1}}(\hat{x}_{k-1})\|^2 \\
    && \quad - \|\hat{x}_k - x_k\|^2 - 2 \om \mu \| \hat{x}_k - x^* \|^2 \\
    &\overset{\eqref{eq: bound on difference of gradients}}{\leq}&\|x_k - x^*\|^2 + 2\om^2 \bigg(\delta \|\hat{x}_k - x^*\|^2 \\
    &&\quad  + 2\delta \|\hat{x}_{k-1} - x^*\|^2 + 2 L^2 \|\hat{x}_k - \hat{x}_{k-1}\|^2  + 6\sigma_*^2 \bigg)  \\
    && \quad - \|\hat{x}_k - x_k\|^2 - 2 \om \mu \| \hat{x}_k - x^* \|^2 \\
    &=&\|x_k - x^*\|^2 - 2 \om (\mu - \om \delta)  \| \hat{x}_k - x^*\|^2 \\
    && \quad + 4 \om^2 \delta  \|\hat{x}_{k-1} - x^*\|^2 + 4 \om^2 L^2 \|\hat{x}_k - \hat{x}_{k-1}\|^2 \\
    && \quad - \|\hat{x}_k - x_k\|^2 + 12 \om^2 \sigma_*^2 \\
    &\overset{\eqref{eq: Young's Inequality}}{\leq}& \|x_k - x^*\|^2 - \om (\mu - \om \delta) \| x_k - x^*\|^2 \\
    && \quad + 2 \om (\mu - \om \delta) \| x_k - \hat{x}_k\|^2 + 4 \om^2 \delta \|\hat{x}_{k-1} - x^*\|^2 \\
    &&\quad + 4 \om^2 L^2 \|\hat{x}_k - \hat{x}_{k-1}\|^2 - \|\hat{x}_k - x_k\|^2\\
    &&\quad + 12 \om^2 \sigma_*^2 \\
    &\overset{\eqref{eq: Young's Inequality}}{\leq}& \|x_k - x^*\|^2 - \om (\mu - \om \delta) \| x_k - x^*\|^2 \\
    &&\quad + 2 \om (\mu - \om \delta) \| x_k - \hat{x}_k\|^2 + 8 \om^2 \delta \|\hat{x}_{k-1} - x_k\|^2 \\
    &&\quad + 8 \om^2 \delta \|x_k - x^*\|^2+ 8 \om^2 L^2 \|\hat{x}_k - x_k\|^2 \\
    &&\quad  + 8 \om^2 L^2 \|x_k - \hat{x}_{k-1}\|^2 - \|\hat{x}_k - x_k\|^2 + 12 \om^2 \sigma_*^2 \\
    &=& (1 - \om \mu + 9 \om^2 \delta ) \|x_k - x^*\|^2 \\
    &&\quad + (8\om^2 \delta + 8\om^2 L^2 ) \|x_k - \hat{x}_{k-1} \|^2 \\
    && \quad  + (2\om(\mu - \om\delta) +  8 \om^2L^2 -1) \|x_k - \hat{x}_k \|^2 \\
    && + 12\om^2 \sigma_*^2.    
\end{eqnarray*}
Then using \eqref{eq: conditions 1}, \eqref{eq: conditions 2} we have
\begin{eqnarray*}
    \E_{\mathcal{D}}[\|x_{k+1} - x^*\|^2 + \|x_{k+1} - \hat{x}_k\|^2] &\leq& (1 - \om \mu + 9 \om^2\delta) \bigg(\|x_k - x^*\|^2 + \|x_k - \hat{x}_{k-1}\|^2 \bigg) \\
    && \quad + 12 \om^2 \sigma_*^2.
\end{eqnarray*}
Then we take total expectation with respect to the algorithm to obtain the following recurrence:
\begin{equation}\label{eq: recurrece after total expectation}
    R_{k+1}^2 \leq (1 - \omega \mu + 9 \omega^2 \delta) R_k^2 + 12 \omega^2 \sigma_*^2.
\end{equation}
Using the inequality $ 1- \omega(\mu - 9 \omega \delta) \leq 1 - \frac{\omega \mu}{2}$, we have
\begin{equation}\label{eq: before recurrence}
    \begin{split}
    \E\bigg[\|x_{k+1} - x^*\|^2 + \|x_{k+1} - \hat{x}_k\|^2 \bigg] & \leq \bigg(1 - \frac{\om \mu}{2}\bigg) \E \bigg[\|x_k - x^*\|^2 + \|x_k - \hat{x}_{k-1}\|^2 \bigg] + 12 \om^2 \sigma_*^2.
    \end{split}
\end{equation}
The theorem follows by unrolling the above recurrence. In order to compute the iteration complexity of \algname{SPEG}, we consider any arbitrary $\varepsilon > 0$. Then we choose the step-size $\om$ such that $\frac{24 \om \sigma_*^2}{\mu} \leq \frac{\varepsilon}{2}$ i.e. $\om \leq \frac{\varepsilon \mu}{48 \sigma_*^2}$. Next we will choose the number of iterations $k$ such that $(1 - \frac{\om \mu}{2})^k R_0^2 \leq \frac{\varepsilon}{2}$. It is equivalent to choosing $k$ such that 
\begin{equation*}
    \log \bigg( \frac{2 R_0^2}{\varepsilon} \bigg) \leq k \log \bigg(  \frac{1}{1 - \frac{\om \mu}{2}}\bigg).
\end{equation*}
Now using the fact $\log \big( \frac{1}{\rho}\big) \geq 1 - \rho$ for $0 < \rho \leq 1$, we get $\log \Big( \frac{2 R_0^2}{\varepsilon} \Big) \leq  \frac{k\om \mu}{2}$, or equivalently $k \geq \frac{2}{\om \mu} \log \Big( \frac{2 R_0^2}{\varepsilon} \Big)$. Therefore, with step-size $\om = \min \left\{\frac{\mu}{18 \delta}, \frac{1}{4L}, \frac{\varepsilon \mu}{48 \sigma_*^2} \right\}$ we get the following lower bound on the number of iterations
\begin{equation*}
    k \geq \max \bigg\{\frac{8L}{\mu}, \frac{36 \delta}{\mu^2}, \frac{96\sigma_*^2}{\varepsilon \mu^2} \bigg\} \log \bigg( \frac{2 R_0^2}{\varepsilon} \bigg).
\end{equation*}
\end{proof}

\subsection{Proof of Theorem \ref{SPEG switching rule}}
\begin{proof}
For $\om \leq \min \big\{\frac{1}{4L}, \frac{\mu}{18 \delta} \big\}$, from Theorem \ref{Theorem: constant stepsize theorem} we obtain 
$$
R_{k+1}^2 \leq \bigg(1 - \frac{\om \mu}{2} \bigg)^{k+1} R_0^2 + \frac{24 \om \sigma_*^2}{\mu}. 
$$
Let the step-size $\om_k = \frac{2k+1}{(k+1)^2} \frac{2}{\mu}$ and $k^*$ be an integer that satisfies $\om_{k^*} \leq \Bar{\om}$. In particular this holds when $k^* \geq \left[ \frac{4}{\mu \Bar{\om}} - 1 \right]$. Note that $\om_k$ is decreasing in $k$ and consequently $\om_k \leq \Bar{\om}$ for all $k \geq k^*$. Therefore, from \eqref{eq: before recurrence} we derive 
$$
R_{k+1}^2 \leq \bigg(1 - \om_k\frac{\mu}{2} \bigg) R_{k}^2 + 12\om_k^2 \sigma_*^2
$$
for all $k \geq k^*$. Then we replace $\om_k$ with $\frac{2k+1}{(k+1)^2} \frac{2}{\mu}$ to obtain
\begin{eqnarray*}
    R_{k+1}^2 &\leq& \bigg(1 - \frac{2k +1}{(k+1)^2} \bigg) R_{k}^2 + 48 \sigma_*^2 \frac{(2k + 1)^2}{\mu^2(k+1)^4} \\
     &=&\frac{k^2}{(k+1)^2} R_{k}^2 + 48 \sigma_*^2 \frac{(2k + 1)^2}{\mu^2(k+1)^4}.
\end{eqnarray*}
Multiplying both sides by $(k+1)^2$ we get
\begin{eqnarray*}
    (k+1)^2 R_{k+1}^2 &\leq& k^2 R_k^2 + \frac{48 \sigma_*^2}{\mu^2} \bigg(\frac{2k+1}{k+1} \bigg)^2 \\
    &\leq& k^2 R_k^2 + \frac{192\sigma_*^2}{\mu^2},
\end{eqnarray*}
where in the last line follows from $\frac{2k + 1}{k + 1} < 2$. Rearranging and summing the last expression for $t = k^*,\cdots, k$ we obtain 
\begin{equation*}
    \begin{split}
        & \sum_{t = k^*}^k (t+1)^2 R_{t+1}^2 -  t^2 R_t^2 \leq \frac{192 \sigma_*^2}{\mu^2}(k - k^*).
    \end{split}
\end{equation*}
Using telescopic sum and dividing both sides by $(k+1)^2$ we obtain
\begin{equation}\label{eq: decreasing stepsize bound}
    R_{k+1}^2 \leq \bigg(\frac{k^*}{k+1}\bigg)^2 R_{k^*}^2+ \frac{192 \sigma_*^2 (k - k^*)}{\mu^2 (k+1)^2}.
\end{equation}
Suppose for $k \leq k^*$, we have $\om_k = \Bar{\om} = \min \Big\{\frac{1}{4L}, \frac{\mu}{18 \delta} \Big\}$ i.e. constant step-size. Then from \eqref{eq:SPEG_const_steps_neighborhood}, we obtain $R_{k^*}^2 \leq \Big(1 - \frac{\mu \Bar{\om}}{2}\Big)^{k^*} R_0^2 + \frac{24 \Bar{\om} \sigma_*^2}{\mu}$. This bound on $R_{k^*}^2$, combined with \eqref{eq: decreasing stepsize bound} yields
\begin{equation*}
    \begin{split}
        R_{k+1}^2 \leq \bigg(\frac{k^*}{k+1}\bigg)^2 \bigg(1 - \frac{\mu \Bar{\om}}{2}\bigg)^{k^*} R_0^2 + \bigg(\frac{k^*}{k+1}\bigg)^2 \frac{24 \Bar{\om} \sigma_*^2}{\mu} + \frac{192 \sigma_*^2 (k - k^*)}{\mu^2 (k+1)^2}.
    \end{split}
\end{equation*}
Now we want to choose $k^*$ which minimizes the expression $\big(\frac{k^*}{k+1}\big)^2 \frac{24\Bar{\om} \sigma_*^2}{\mu} + \frac{192 \sigma_*^2 (k - k^*)}{\mu^2 (k+1)^2}$. Note that, it is minimized at $\frac{4}{\mu \Bar{\om}}$, hence we choose $k^* = \left[ \frac{4}{\mu \Bar{\om}} \right]$. Therefore, using this value of $k^*$, we obtain
\begin{eqnarray*}
    R_{k+1}^2 &\leq&\bigg(\frac{k^*}{k+1}\bigg)^2 \bigg(1 - \frac{2}{k^*} \bigg)^{k^*} R_0^2 + \frac{24 \sigma_*^2}{\mu^2 (k+1)^2} (8k - 4k^*) \\
    &\leq&\bigg(\frac{k^*}{k+1}\bigg)^2 \bigg(1 - \frac{2}{k^*} \bigg)^{k^*} R_0^2 + \frac{192 k \sigma_*^2}{\mu^2 (k+1)^2}\\
    &\leq&\bigg(\frac{k^*}{k+1}\bigg)^2 \frac{R_0^2}{e^2} + \frac{192 \sigma_*^2}{\mu^2 (k+1)}.
\end{eqnarray*}
The last line follows from $\Big(1 - \frac{1}{x} \Big)^x \leq e^{-1}$ for all $x \geq 1$. This completes the proof.
\end{proof}

\subsection{Proof of Theorem \ref{Theorem: Total number of iteration knowledge}}
\begin{proof}
For $ 0 < \om_k \leq \big\{ \frac{1}{4L},\frac{\mu}{18 \delta} \big\}$ we obtain the following bound from Theorem \ref{Theorem: constant stepsize theorem}:
$$
R_k^2 \leq \bigg(1 - \frac{\mu \om_k}{2} \bigg) R_{k-1}^2 + 12 \om_k^2 \sigma_*^2.
$$
Then using Lemma \ref{Lemma: Stich lemma} with $a = \frac{\mu}{2}, h = \frac{1}{\Bar{\om}} $ and $c = 12 \sigma_*^2$ we complete the proof of this Theorem.
\end{proof}

\subsection{Proof of Theorem \ref{thm:weak_MVI}}
\vspace{5mm}
\begin{theorem}\label{thm:weak_MVI}
    Let $F$ be $L$-Lipschitz and satisfy Weak Minty condition with parameter $\rho < \nicefrac{1}{(2L)}$. Assume that inequality \eqref{eq: variance bound} holds (e.g., it holds whenever Assumption~\ref{as:expected_residual} holds, see Lemma~\ref{Lemma: variance bound}). Assume that $\gamma_k = \gamma$, $\omega_k = \omega$ and
    \begin{equation*}
        \max\left\{2\rho, \frac{1}{2L}\right\} < \gamma < \frac{1}{L},\quad 0 < \omega < \min\left\{\gamma - 2\rho, \frac{1}{4L} - \frac{\gamma}{4}\right\}, \quad \delta \leq \frac{(1-L\gamma)L^3\omega}{32}.
    \end{equation*}
    Then, for all $K \geq 2$ the iterates produced by \algname{SPEG} satisfy
    \begin{eqnarray}
        \min\limits_{0\leq k \leq K-1}\E\left[\|F(\hat x_k)\|^2\right] &\leq& \frac{(1 + 8\omega\gamma (\delta + L^2) - L\gamma)\left(1+\frac{48\omega\gamma \delta}{(1-L\gamma)^2}\right)^{K-1}\|x_0 - x^*\|^2}{\omega\gamma (1 - L(\gamma + 4\omega)) (K-1)}\notag \\
        &&\quad + \frac{8 \left(8 + \frac{(1-L\gamma)^2}{K-1}\left(1+\frac{48\omega\gamma \delta}{(1-L\gamma)^2}\right)^{K-1}\right) \sigma_*^2}{(1-L\gamma)^2(1-L(\gamma + 4\omega))}. \label{eq:SPEG_weak_MVI_result_appendix}
    \end{eqnarray}
\end{theorem}

\begin{proof}
    The proof closely follows the proof of Lemma C.3 and Theorem C.4 from \cite{gorbunov2022convergence}. The update rule of \algname{SPEG} implies for $k \geq 1$
    \begin{eqnarray*}
        \|x_{k+1} - x^*\|^2 &=& \|x_k - x^*\|^2 - 2\omega \langle x_k - x^*, F_{v_k}(\hx_k) \rangle + \omega^2 \|F_{v_k}(\hx_k)\|^2\\
        &=& \|x_k - x^*\|^2 - 2\omega \langle \hx_k - x^*, F_{v_k}(\hx_k) \rangle - 2\omega\gamma \langle F_{v_{k-1}}(\hx_{k-1}), F_{v_k}(\hx_k) \rangle \\
        && \quad + \omega^2 \|F_{v_k}(\hx_k)\|^2\\
        &=& \|x_k - x^*\|^2 - 2\omega \langle \hx_k - x^*, F_{v_k}(\hx_k) \rangle - \omega\gamma\|F_{v_{k-1}}(\hx_{k-1})\|^2\\
        &&\quad - \omega(\gamma - \omega)\|F_{v_k}(\hx_k)\|^2 + \omega \gamma \|F_{v_k}(\hx_k) - F_{v_{k-1}}(\hx_{k-1})\|^2,
    \end{eqnarray*}
    where in the last step we apply $2\langle a, b \rangle = \|a\|^2 + \|b\|^2 - \|a - b\|^2$, which holds for all $a,b \in \R^d$. Taking the full expectation and using $\E[\E_{v_k}[\cdot]] = \E[\cdot]$ and Weak Minty condition, we derive
    \begin{eqnarray}
        \E\left[\|x_{k+1} - x^*\|^2\right] &\leq& \E\left[\|x_k - x^*\|^2\right] - 2\omega \E\left[\langle \hx_k - x^*, F(\hx_k) \rangle\right] - \omega\gamma\E\left[\|F_{v_{k-1}}(\hx_{k-1})\|^2\right] \notag\\
        &&\quad - \omega(\gamma - \omega)\E\left[\|F_{v_k}(\hx_k)\|^2\right] + \omega \gamma \E\left[\|F_{v_k}(\hx_k) - F_{v_{k-1}}(\hx_{k-1})\|^2\right] \notag\\
        &\overset{\eqref{eq: weak MVI}}{\leq}& \E\left[\|x_k - x^*\|^2\right] + 2\omega\rho \E\left[\|F(\hx_k)\|^2\right] - \omega\gamma\E\left[\|F_{v_{k-1}}(\hx_{k-1})\|^2\right] \notag\\
        &&\quad - \omega(\gamma - \omega)\E\left[\|F_{v_k}(\hx_k)\|^2\right] + \omega \gamma \E\left[\|F_{v_k}(\hx_k) - F_{v_{k-1}}(\hx_{k-1})\|^2\right] \notag\\
        &\leq& \E\left[\|x_k - x^*\|^2\right] - \omega\gamma\E\left[\|F_{v_{k-1}}(\hx_{k-1})\|^2\right] \notag \\
        &&\quad - \omega(\gamma - 2\rho - \omega)\E\left[\|F_{v_k}(\hx_k)\|^2\right] \notag \\
        &&\quad + \omega \gamma \E\left[\|F_{v_k}(\hx_k) - F_{v_{k-1}}(\hx_{k-1})\|^2\right]\notag\\
        &\leq& \E\left[\|x_k - x^*\|^2\right] - \omega\gamma\E\left[\|F_{v_{k-1}}(\hx_{k-1})\|^2\right] \notag\\
        &&\quad  + \omega \gamma \E\left[\|F_{v_k}(\hx_k) - F_{v_{k-1}}(\hx_{k-1})\|^2\right], \label{eq:MVI_technical_ineq_1}
    \end{eqnarray}
    where we apply Jensen's inequality $\|F(\hx_k)\|^2 = \|\E_{v_k}[F_{v_k}(\hx_k)]\|^2 \leq \E_{v_k}[\|F_{v_k}(\hx_k)\|^2]$ and $\gamma > 2\rho + \omega$. For $k = 0$ we have $x_1 = x_0 - \omega F_{v_0}(\hx_0) = x_0 - \omega F_{v_0}(x_0)$ and
    \begin{eqnarray}
        \E\left[\|x_1 - x^*\|^2\right] &=& \|x_0 - x^*\|^2  - 2\omega \E\left[ \langle x_0 - x^*, F_{v_0}(x_0)  \rangle\right] + \omega^2 \E\left[\|F_{v_0}(x_0)\|^2\right] \notag\\
        &=& \|x_0 - x^*\|^2  - 2\omega  \langle x_0 - x^*, F(x_0)  \rangle+ \omega^2 \E\left[\|F_{v_0}(x_0)\|^2\right]. \notag
    \end{eqnarray}
    Applying Weak Minty condition, we get
    \begin{eqnarray}
        \E\left[\|x_1 - x^*\|^2\right] &=& \|x_0 - x^*\|^2  + 2\omega\rho  \|F(x_0)\|^2 + \omega^2 \E\left[\|F_{v_0}(x_0)\|^2\right] \notag\\
        &\leq& \|x_0 - x^*\|^2 + \omega(\omega + 2\rho)\E\left[\|F_{v_0}(x_0)\|^2\right]. \label{eq:MVI_technical_ineq_1_1}
    \end{eqnarray}
    
    \newpage 
    The next step of our proof is in estimating the last term from \eqref{eq:MVI_technical_ineq_1}. Using Young's inequality $\|a+b\|^2  \leq (1 + \alpha)\|a\|^2 + (1 + \alpha^{-1})\|b\|^2$, which holds for any $a,b \in \R^d$, $\alpha > 0$, we get for all $k \geq 2$
    \begin{eqnarray*}
        \E\left[\|F_{v_k}(\hx_k) - F_{v_{k-1}}(\hx_{k-1})\|^2\right] \hspace{-3mm} &\leq& \hspace{-3mm}(1+\alpha)\E\left[\|F(\hx_k) - F(\hx_{k-1})\|^2\right]\\
        && \hspace{-3mm}+ (1+\alpha^{-1})\E\left[\|F_{v_k}(\hx_k) - F(\hx_k) \right. \\
        && \hspace{-3mm}\quad \left. - (F_{v_{k-1}}(\hx_{k-1}) - F(\hx_{k-1}))\|^2\right]\\
        &\leq& \hspace{-3mm}(1+\alpha)L^2\E\left[\|\hx_k - \hx_{k-1}\|^2\right]\\
        &&\hspace{-3mm} \quad + 2(1+\alpha^{-1})\E\left[\|F_{v_k}(\hx_k) - F(\hx_k)\|^2 \right. \\
        && \hspace{-3mm} \quad \left. + \|F_{v_{k-1}}(\hx_{k-1}) - F(\hx_{k-1})\|^2\right]\\
        &\overset{\eqref{eq: variance bound}}{\leq}& \hspace{-3mm} (1+\alpha)L^2\E\left[\|\hx_k - x_k + x_k - x_{k-1} + x_{k-1} - \hx_{k-1}\|^2\right]\\
        &&\hspace{-3mm} \quad + 2(1+\alpha^{-1})\delta\E\left[\|\hx_k - x^*\|^2 + \|\hx_{k-1} - x^*\|^2\right] \\
        && \hspace{-3mm} \quad + 8(1+\alpha^{-1})\sigma_*^2\\
        &\leq& \hspace{-3mm}  (1+\alpha)L^2\E\left[\|(\gamma + \omega)F_{v_{k-1}}(\hx_{k-1}) - \gamma F_{v_{k-2}}(\hx_{k-2})\|^2\right]\\
        && \hspace{-3mm}\quad + 4(1+\alpha^{-1})\delta\E\left[\|x_k - x^*\|^2 + \|x_{k-1} - x^*\|^2\right]\\
        &&\hspace{-3mm} \quad + 4(1+\alpha^{-1})\delta\gamma^2\E\left[\|F_{v_{k-1}}(\hx_{k-1})\|^2 + \|F_{v_{k-2}}(\hx_{k-2})\|^2\right]\\
        &&\hspace{-3mm} \quad + 8(1+\alpha^{-1})\sigma_*^2\\
        &=&\hspace{-3mm} (1+\alpha)L^2(\gamma + \omega)^2\E\left[\|F_{v_{k-1}}(\hx_{k-1})\|^2\right] \\
        && \hspace{-3mm}\quad + (1+\alpha)L^2\gamma^2\E\left[\|F_{v_{k-2}}(\hx_{k-2})\|^2\right]\\
        &&\hspace{-3mm}\quad - 2(1+\alpha)L^2\gamma(\gamma+\omega)\E\left[\langle F_{v_{k-1}}(\hx_{k-1}), F_{v_{k-2}}(\hx_{k-2}) \rangle\right]\\
        &&\hspace{-3mm}\quad + 4(1+\alpha^{-1})\delta\E\left[\|x_k - x^*\|^2 + \|x_{k-1} - x^*\|^2\right]\\
        &&\hspace{-3mm}\quad + 4(1+\alpha^{-1})\delta\gamma^2\E\left[\|F_{v_{k-1}}(\hx_{k-1})\|^2 + \|F_{v_{k-2}}(\hx_{k-2})\|^2\right] \\
        &&\hspace{-3mm} \quad + 8(1+\alpha^{-1})\sigma_*^2\\
        &=&\hspace{-3mm} (1+\alpha)L^2(\gamma + \omega)^2\E\left[\|F_{v_{k-1}}(\hx_{k-1})\|^2\right] \\
        &&\hspace{-3mm} \quad + (1+\alpha)L^2\gamma^2\E\left[\|F_{v_{k-2}}(\hx_{k-2})\|^2\right]\\
        &&\hspace{-3mm}\quad - (1+\alpha)L^2\gamma(\gamma+\omega)\E\left[\| F_{v_{k-1}}(\hx_{k-1})\|^2 \right]\\
        &&\hspace{-3mm}\quad - (1+\alpha)L^2\gamma(\gamma+\omega)\E\left[\| \|F_{v_{k-2}}(\hx_{k-2})\|^2\right]\\
        &&\hspace{-4mm}\quad + (1+\alpha)L^2\gamma(\gamma+\omega)\E\left[\|F_{v_{k-1}}(\hx_{k-1}) - F_{v_{k-2}}(\hx_{k-2})\|^2\right]\\
        &&\hspace{-4mm}\quad + 4(1+\alpha^{-1})\delta\E\left[\|x_k - x^*\|^2 + \|x_{k-1} - x^*\|^2\right]\\
        &&\hspace{-4mm}\quad + 4(1+\alpha^{-1})\delta\gamma^2\E\left[\|F_{v_{k-1}}(\hx_{k-1})\|^2 + \|F_{v_{k-2}}(\hx_{k-2})\|^2\right] \\
        && \hspace{-4mm} \quad + 8(1+\alpha^{-1})\sigma_*^2\\
        &=& \hspace{-3mm} (1+\alpha)L^2\omega(\gamma + \omega)\E\left[\|F_{v_{k-1}}(\hx_{k-1})\|^2\right]\\
        && \hspace{-4mm}\quad - (1+\alpha)L^2\gamma\omega\E\left[\|F_{v_{k-2}}(\hx_{k-2})\|^2\right]\\
        && \hspace{-4mm}\quad + (1+\alpha)L^2\gamma(\gamma+\omega)\E\left[\|F_{v_{k-1}}(\hx_{k-1}) - F_{v_{k-2}}(\hx_{k-2})\|^2\right]\\
        && \hspace{-4mm}\quad + 4(1+\alpha^{-1})\delta\E\left[\|x_k - x^*\|^2 + \|x_{k-1} - x^*\|^2\right]\\
        && \hspace{-4mm}\quad + 4(1+\alpha^{-1})\delta\gamma^2\E\left[\|F_{v_{k-1}}(\hx_{k-1})\|^2 + \|F_{v_{k-2}}(\hx_{k-2})\|^2\right]\\
        && \hspace{-4mm} \quad + 8(1+\alpha^{-1})\sigma_*^2.
    \end{eqnarray*}
    Since $\hx_0 = x_0$ and $\hx_1 = x_1 - \gamma F_{v_0}(x_0) = x_0 - (\gamma + \omega)F_{v_0}(x_0)$, for $k = 1$ we have
    \begin{eqnarray*}
        \E\left[\|F_{v_1}(\hx_1) - F_{v_{0}}(\hx_{0})\|^2\right] &=& \E\left[\|F_{v_1}(\hx_1) - F_{v_{0}}(x_{0})\|^2\right]\\
        &\leq& (1+\alpha)\E\left[\|F(\hx_1) - F(x_{0})\|^2\right]\\
        &&\quad + (1+\alpha^{-1})\E\left[\|F_{v_1}(\hx_1) - F(\hx_1) - (F_{v_{0}}(x_{0}) - F(x_{0}))\|^2\right]\\
        &\leq& (1+\alpha)L^2\E\left[\|\hx_1 - x_0\|^2\right]\\
        && \hspace{-4mm} + 2(1+\alpha^{-1})\E\left[\|F_{v_1}(\hx_1) - F(\hx_1)\|^2 + \|F_{v_{0}}(x_{0}) - F(x_{0})\|^2\right]
    \end{eqnarray*}
    Then using \eqref{eq: variance bound} we get,
    \begin{eqnarray*}
        \E\left[\|F_{v_1}(\hx_1) - F_{v_{0}}(\hx_{0})\|^2\right]
        &\overset{\eqref{eq: variance bound}}{\leq}& (1+\alpha)L^2(\gamma+\omega)^2\E\left[\|F_{v_0}(x_0)\|^2\right]\\
        &&\quad + 2(1+\alpha^{-1})\delta \E\left[\|\hx_1 - x^*\|^2 + \|x_0 - x^*\|^2\right] \\
        && + 8(1+\alpha)\sigma_*^2\\
        &\leq& \left((1+\alpha)L^2 + 4(1+\alpha^{-1})\delta\right)(\gamma+\omega)^2\E\left[\|F_{v_0}(x_0)\|^2\right]\\
        &&\quad + 6(1+\alpha^{-1})\delta \|x_0 - x^*\|^2 + 8(1+\alpha)\sigma_*^2.
    \end{eqnarray*}
    Let $\{w_k\}_{k=0}^{K-1}$ be a non-increasing sequence of positive numbers that will be specified later and $W_K = \sum_{k=0}^{K-1} w_k$. Summing up the above two inequalities with weights $\{w_k\}_{k=1}^{K-1}$, we derive
    \begin{eqnarray*}
        && \sum\limits_{k=1}^{K-1}w_k \E\left[\|F_{v_k}(\hx_k) - F_{v_{k-1}}(\hx_{k-1})\|^2\right] \\
        \leq && \hspace{-6.5mm} (1+\alpha)L^2\sum\limits_{k=1}^{K-3}\left(\omega(\gamma + \omega)w_{k+1} \E\left[\|F_{v_{k}}(\hx_{k})\|^2\right] \right. \\
        && \hspace{-6.5mm} \left. -\gamma\omega w_{k+2}\E\left[\|F_{v_{k}}(\hx_{k})\|^2\right]\right)\\
        && \hspace{-6.5mm} + (1+\alpha)L^2\omega(\gamma + \omega)w_{K-1}\E\left[\|F_{v_{K-2}}(\hx_{K-2})\|^2\right]\\
        && \hspace{-6.5mm} - (1+\alpha)L^2 \gamma\omega w_2 \E\left[\|F_{v_{0}}(x_{0})\|^2\right]\\
        && \hspace{-6.5mm} + (1+\alpha)L^2\gamma(\gamma+\omega)\sum\limits_{k=1}^{K-2}w_{k+1}\E\left[\|F_{v_k}(\hx_k) \right. \\
        && \hspace{-6.5mm} \left. - F_{v_{k-1}}(\hx_{k-1})\|^2\right] + 4(1+\alpha^{-1})\delta\sum\limits_{k=2}^{K-1}w_k\E\left[\|x_k - x^*\|^2 \right] \\
        && \hspace{-6.5mm} + w_k\E \left[ \|x_{k-1} - x^*\|^2\right] \\
        && \hspace{-6.5mm} + 4(1+\alpha^{-1})\delta\gamma^2\sum\limits_{k=1}^{K-2}w_{k+1}\E\left[\|F_{v_{k}}(\hx_{k})\|^2 \right. \\
        && \hspace{-6.5mm} \left. + \|F_{v_{k-1}}(\hx_{k-1})\|^2\right] + 8(1+\alpha^{-1})(W_K - w_0 - w_1)\sigma_*^2\\
        && \hspace{-6.5mm} + \left((1+\alpha)L^2 + 4(1+\alpha^{-1})\delta\right)(\gamma+\omega)^2w_1\E\left[\|F_{v_0}(x_0)\|^2\right]\\
        && \hspace{-6.5mm} + 6(1+\alpha^{-1})\delta w_1 \|x_0 - x^*\|^2 + 8(1+\alpha)w_1\sigma_*^2.
    \end{eqnarray*}
    Next, we rearrange the terms using $w_k \geq w_{k+1}$ and new notation $$\Delta_k = \E\left[\|F_{v_k}(\hx_k) - F_{v_{k-1}}(\hx_{k-1})\|^2\right]$$ to get
    \begin{eqnarray*}
        && \left(1 - (1+\alpha)L^2\gamma(\gamma+\omega)\right)\sum\limits_{k=1}^{K-1}w_k\Delta_k \\
        &\leq& \hspace{-3.5mm} \sum\limits_{k=1}^{K-2}(1+\alpha)L^2\omega(\gamma+\omega) w_k\E\left[\|F_{v_{k}}(\hx_{k})\|^2\right]\\
        && \hspace{-3mm} + 8(1+\alpha^{-1})\delta\gamma^2 w_k\E\left[\|F_{v_{k}}(\hx_{k})\|^2\right]\\
        && \hspace{-3mm}+ \left((1+\alpha)L^2 + 8(1+\alpha^{-1})\delta\right)(\gamma+\omega)^2w_0\E\left[\|F_{v_0}(x_0)\|^2\right]\\
        && \hspace{-3mm}+ 12(1+\alpha^{-1})\delta\sum\limits_{k=1}^{K-1}w_k\E\left[\|x_k - x^*\|^2\right] \\
        && \hspace{-3mm} + 8(1+\alpha^{-1})(W_K - w_0)\sigma_*^2.
    \end{eqnarray*}
    To simplify the above inequality we choose $\alpha = \frac{1}{2L^2\gamma(\gamma+\omega)} - \frac{1}{2}$, which is positive due to $\gamma < \nicefrac{1}{L}$ and $\gamma+\omega < \nicefrac{1}{L}$. In this case, we have
    \begin{eqnarray*}
        (1+\alpha)L^2\gamma(\gamma + \omega) &=& \frac{1}{2}L^2\gamma(\gamma + \omega) + \frac{1}{2},\\
        (1+\alpha)L^2(\gamma + \omega)^2 &=& \frac{1}{2}L^2(\gamma + \omega)^2 + \frac{\gamma+\omega}{2\gamma} \leq \frac{3}{2},\\
        (1+\alpha)L^2\omega (\gamma + \omega) &=& \frac{1}{2}L^2\omega(\gamma + \omega) + \frac{\omega}{2\gamma} = \frac{L\omega}{2}\left(L(\gamma + \omega) + \frac{1}{\gamma L}\right) \leq \frac{3 L\omega}{2},\\
        1 + \alpha^{-1} &=& 1 + \frac{2L^2\gamma (\gamma + \omega)}{1 - L^2\gamma (\gamma + \omega)} = \frac{1 + L^2\gamma (\gamma + \omega)}{1 - L^2\gamma (\gamma + \omega)} \leq \frac{2}{1 - L^2\gamma (\gamma + \omega)},
    \end{eqnarray*}
    where we also use $\nicefrac{1}{2L} < \gamma < \nicefrac{1}{L}$ and $\gamma + \omega < \nicefrac{1}{L}$. Using these relations, we can continue our derivation as follows:
    \begin{eqnarray*}
        && \frac{1}{2}\left(1 - L^2\gamma(\gamma+\omega)\right)\sum\limits_{k=1}^{K-1}w_k\Delta_k \\
        &\leq& \sum\limits_{k=1}^{K-2}\left(\frac{3L\omega}{2} + \frac{16}{1 - L^2\gamma (\gamma + \omega)}\delta\gamma^2\right) w_k\E\left[\|F_{v_{k}}(\hx_{k})\|^2\right]\\
        &&\quad + \left(\frac{3}{2} + \frac{16}{1 - L^2\gamma (\gamma + \omega)}\delta(\gamma+\omega)^2\right)w_0\E\left[\|F_{v_0}(x_0)\|^2\right]\\
        &&\quad + \frac{24}{1 - L^2\gamma (\gamma + \omega)}\delta\sum\limits_{k=1}^{K-1}w_k\E\left[\|x_k - x^*\|^2\right]\\
        && \quad + \frac{16}{1 - L^2\gamma (\gamma + \omega)}(W_K - w_0)\sigma_*^2.
    \end{eqnarray*}
    Dividing both sides by $\frac{1}{2}\left(1 - L^2\gamma(\gamma+\omega)\right)$, we derive
    \begin{eqnarray}
        \sum\limits_{k=1}^{K-1}w_k\Delta_k &\leq& \sum\limits_{k=1}^{K-2}\left(\frac{3L\omega}{1 - L^2\gamma (\gamma + \omega)} + \frac{32}{(1 - L^2\gamma (\gamma + \omega))^2}\delta\gamma^2\right) w_k\E\left[\|F_{v_{k}}(\hx_{k})\|^2\right]\notag\\
        &&\quad + \left(\frac{3}{1 - L^2\gamma (\gamma + \omega)} + \frac{32}{(1 - L^2\gamma (\gamma + \omega))^2}\delta(\gamma+\omega)^2\right)w_0\E\left[\|F_{v_0}(x_0)\|^2\right]\notag\\
        &&\quad + \frac{48}{(1 - L^2\gamma (\gamma + \omega))^2}\delta\sum\limits_{k=1}^{K-1}w_k\E\left[\|x_k - x^*\|^2\right] \notag\\
        && \quad + \frac{32}{(1 - L^2\gamma (\gamma + \omega))^2}(W_K - w_0)\sigma_*^2 \notag\\
        &=& \sum\limits_{k=1}^{K-2}C_1 w_k\E\left[\|F_{v_{k}}(\hx_{k})\|^2\right] + C_2 w_0\E\left[\|F_{v_0}(x_0)\|^2\right]\notag\\
        &&\quad + 3C_3\delta\sum\limits_{k=1}^{K-1}w_k\E\left[\|x_k - x^*\|^2\right] + 2C_3W_K\sigma_*^2, \label{eq:MVI_technical_ineq_2}
    \end{eqnarray}
    where $C_1 = \frac{3L\omega}{1 - L^2\gamma (\gamma + \omega)} + \frac{32}{(1 - L^2\gamma (\gamma + \omega))^2}\delta\gamma^2$, $C_2 = \frac{3}{1 - L^2\gamma (\gamma + \omega)} + \frac{32}{(1 - L^2\gamma (\gamma + \omega))^2}\delta(\gamma+\omega)^2$, and $C_3 = \frac{16}{(1 - L^2\gamma (\gamma + \omega))^2}$. Summing up inequalities \eqref{eq:MVI_technical_ineq_1} for $k = 1,\ldots, K-1$ with weights $w_1,\ldots, w_{K-1}$ and \eqref{eq:MVI_technical_ineq_1_1} with weight $w_0$, we get
    \begin{eqnarray*}
        \sum\limits_{k=0}^{K-1}w_k\E\left[\|x_{k+1} - x^*\|^2\right] &\leq& \sum\limits_{k=0}^{K-1}w_k\E\left[\|x_{k} - x^*\|^2\right] - \omega\gamma \sum\limits_{k=1}^{K-1}w_k \E\left[\|F_{v_{k-1}}(\hx_{k-1})\|^2\right]\\
        &&\quad  + \omega\gamma \sum\limits_{k=1}^{K-1}w_k\Delta_k  + \omega (\omega + 2\rho) w_0\E\left[\|F_{v_0}(x_0)\|^2\right].
    \end{eqnarray*}
    Since $w_k \geq w_{k+1}$, we can continue the derivation as follows:
    \begin{eqnarray*}
        \sum\limits_{k=0}^{K-1}w_k\E\left[\|x_{k+1} - x^*\|^2\right] &\leq& \sum\limits_{k=0}^{K-1}w_k\E\left[\|x_{k} - x^*\|^2\right] - \omega\gamma \sum\limits_{k=0}^{K-2}w_{k} \E\left[\|F_{v_{k}}(\hx_{k})\|^2\right]\\
        &&\quad  + \omega\gamma \sum\limits_{k=1}^{K-1}w_k\Delta_k  + \omega (\omega + 2\rho) w_0\E\left[\|F_{v_0}(x_0)\|^2\right]\\
        &\overset{\eqref{eq:MVI_technical_ineq_2}}{\leq}& \sum\limits_{k=0}^{K-1}(1+ 3C_3\omega\gamma \delta)w_k\E\left[\|x_{k} - x^*\|^2\right]\\
        && \quad - \omega\gamma (1 - C_1)\sum\limits_{k=0}^{K-2}w_{k} \E\left[\|F_{v_{k}}(\hx_{k})\|^2\right]\\
        &&\quad + 2\omega\gamma C_2 w_0 \E\left[\|F_{v_{0}}(\hx_{0})\|^2\right] + 2\omega\gamma C_3 W_K \sigma_*^2.
    \end{eqnarray*}
    Now we need to specify the weights $w_{-1}, w_0, w_1, \ldots, w_{K-1}$. Let $w_{K-2} = 1$ and $w_{k-1} = (1+ 3C_3\omega\gamma \delta)w_k$. Then, rearranging the terms, dividing both sides by $\omega\gamma (1 - C_1)W_{K-1}$, we get
    \begin{eqnarray*}
        \min\limits_{0\leq k \leq K-1}\E\left[\|F(\hat x_k)\|^2\right] &\leq& \min\limits_{0\leq k \leq K-1}\E\left[\|F_{v_k}(\hat x_k)\|^2\right]\\
        &\leq& \sum\limits_{k=0}^{K-2}\frac{w_k}{W_{K-1}}\E\left[\|F_{v_{k}}(\hx_{k})\|^2\right]\\
        &\leq& \frac{1}{\omega\gamma (1 - C_1) W_{K-1}}\sum\limits_{k=0}^{K-1}\left(w_{k-1}\E\left[\|x_{k} - x^*\|^2\right] \right.\\
        && \quad \left. - w_k\E\left[\|x_{k+1} - x^*\|^2\right]\right) + \frac{2C_2 w_0 \E\left[\|F_{v_{0}}(\hx_{0})\|^2\right]}{(1-C_1)W_{K-1}} \\
        && \quad + \frac{2C_3 W_K \sigma_*^2}{(1-C_1)W_{K-1}}\\
        &\leq&  \frac{w_{-1}\|x_0 - x^*\|^2}{\omega\gamma (1 - C_1) W_{K-1}} + \frac{2C_2 w_0 \E\left[\|F_{v_{0}}(\hx_{0})\|^2\right]}{(1-C_1)W_{K-1}} + \frac{2C_3 W_K \sigma_*^2}{(1-C_1)W_{K-1}}.
    \end{eqnarray*}
    It remains to simplify the right-hand side of the above inequality. First, we notice that $W_{K-1} = \sum_{k=0}^{K-2} w_k \geq (K-1)w_{K-2} = K-1$ since $w_{k} \geq w_{k+1}$. Moreover, $w_{-1} = (1 + 3C_3 \omega\gamma\delta)^{K-1}$. Next,
    \begin{eqnarray*}
        C_1 &=& \frac{3L\omega}{1 - L^2\gamma (\gamma + \omega)} + \frac{32}{(1 - L^2\gamma (\gamma + \omega))^2}\delta\gamma^2\\
        &\leq& \frac{3L\omega}{1 - L\gamma} + \frac{32}{(1 - L\gamma)^2} \cdot \frac{(1-L\gamma)L^3\omega}{32} \cdot \gamma^2 \leq \frac{4L\omega}{1 - L\gamma},\\
        C_2 &=& \frac{3}{1 - L^2\gamma (\gamma + \omega)} + \frac{32}{(1 - L^2\gamma (\gamma + \omega))^2}\delta(\gamma+\omega)^2\\
        &\leq& \frac{3}{1 - L\gamma} + \frac{32}{(1 - L\gamma)^2}\cdot \frac{(1-L\gamma)L^3\omega}{32} \cdot (\gamma+\omega)^2 \leq \frac{4}{1 - L\gamma},\\
        C_3 &=& \frac{16}{(1 - L^2\gamma (\gamma + \omega))^2} \leq \frac{16}{(1 - L\gamma)^2},
    \end{eqnarray*}
    where we use $\delta \leq \nicefrac{(1-L\gamma)L^3\omega}{16}$ and $\gamma + \omega < \nicefrac{1}{L}$. Using these inequalities, we simplify the bound as follows:
    \begin{eqnarray}
        \min\limits_{0\leq k \leq K-1}\E\left[\|F(\hat x_k)\|^2\right] &\leq& \frac{(1 - L\gamma)(1+3C_3\omega\gamma \delta)^{K-1}\|x_0 - x^*\|^2}{\omega\gamma (1 - L(\gamma + 4\omega)) (K-1)} \notag \\
        && \quad + \frac{8 (1+3C_3\omega\gamma \delta)^{K-2} \E\left[\|F_{v_{0}}(\hx_{0})\|^2\right]}{(1 - L(\gamma + 4\omega))(K-1)} \notag \\
        &&\quad + \frac{32 \sigma_*^2}{(1-L\gamma)(1-L(\gamma + 4\omega))} \notag\\
        &\leq& \frac{(1 - L\gamma)\left(1+\frac{48\omega\gamma \delta}{(1-L\gamma)^2}\right)^{K-1}\|x_0 - x^*\|^2}{\omega\gamma (1 - L(\gamma + 4\omega)) (K-1)} \notag \\
        && \quad + \frac{8 \left(1+\frac{48\omega\gamma \delta}{(1-L\gamma)^2}\right)^{K-2} \E\left[\|F_{v_{0}}(\hx_{0})\|^2\right]}{(1 - L(\gamma + 4\omega))(K-1)} \notag\\
        &&\quad + \frac{32 \sigma_*^2}{(1-L\gamma)(1-L(\gamma + 4\omega))} \label{eq:MVI_technical_ineq_3}
    \end{eqnarray}
    where we use $W_{K} = W_{K-1} + w_{K-1} \leq W_{K-1} + w_{K-2} \leq 2W_{K-1}$. Finally, we use \eqref{eq: variance bound} to upper-bound $\E\left[\|F_{v_{0}}(\hx_{0})\|^2\right]$:
    \begin{eqnarray*}
        \E\left[\|F_{v_{0}}(\hx_{0})\|^2\right] &=& \E\left[\|F_{v_{0}}(x_{0})\|^2\right] \overset{\eqref{eq: variance bound}}{\leq} \delta \|x_0 - x^*\|^2 + \|F(x_0)\|^2 + 2\sigma_*^2\\
        &\leq& (\delta + L^2)\|x_0 - x^*\|^2 + 2\sigma_*^2.
    \end{eqnarray*}
    Plugging this inequality in \eqref{eq:MVI_technical_ineq_3}, we derive
    \begin{eqnarray*}
        \min\limits_{0\leq k \leq K-1}\E\left[\|F(\hat x_k)\|^2\right] &\leq& \frac{(1 + 8\omega\gamma (\delta + L^2) - L\gamma)\left(1+\frac{48\omega\gamma \delta}{(1-L\gamma)^2}\right)^{K-1}\|x_0 - x^*\|^2}{\omega\gamma (1 - L(\gamma + 4\omega)) (K-1)}\\
        &&\quad + \frac{4 \left(8 + \frac{1-L\gamma}{K-1}\left(1+\frac{48\omega\gamma \delta}{(1-L\gamma)^2}\right)^{K-1}\right)\sigma_*^2}{(1-L\gamma)(1-L(\gamma + 4\omega))},
    \end{eqnarray*}
    which concludes the proof.
\end{proof}

\subsection{Proof of Theorem \ref{cor:weak_MVI_convergence}}
\label{AppendixE5}

\begin{theorem}\label{cor:weak_MVI_convergence_appendix}
    Let $F$ be $L$-Lipschitz and satisfy Weak Minty condition with parameter $\rho < \nicefrac{1}{(2L)}$. Assume that inequality \eqref{eq: variance bound} holds (e.g., it holds whenever Assumption~\ref{as:expected_residual} holds, see Lemma~\ref{Lemma: variance bound}). Assume that $\gamma_k = \gamma$, $\omega_k = \omega$ and
    \begin{equation*}
        \max\left\{2\rho, \frac{1}{2L}\right\} < \gamma < \frac{1}{L},\quad 0 < \omega < \min\left\{\gamma - 2\rho, \frac{1}{4L} - \frac{\gamma}{4}\right\}.
    \end{equation*}
    Then, for all $K \geq 2$ the iterates produced by mini-batched \algname{SPEG} with batch-size 
    \begin{equation}
        \tau \geq \max\left\{1, \frac{32\delta}{(1-L\gamma)L^3\omega}, \frac{48\omega\gamma \delta(K-1)}{(1 - L\gamma)^2}, \frac{2\omega\gamma\sigma_*^2(K-1)}{(1-L\gamma)\|x_0 - x^*\|^2}\right\} \label{eq:SPEG_weak_MVI_batchsize_appendix}
    \end{equation}
    satisfy
    \begin{eqnarray}
        \min\limits_{0\leq k \leq K-1}\E\left[\|F(\hat x_k)\|^2\right] \leq \frac{48\|x_0 - x^*\|^2}{\omega\gamma (1 - L(\gamma + 4\omega)) (K-1)}. \label{eq:SPEG_weak_MVI_result_corollary_appendix}
    \end{eqnarray}
\end{theorem}
\begin{proof}
    Mini-batched \algname{SPEG} uses estimator
    \begin{eqnarray*}
        F_{v_k}(\hx_k) = \frac{1}{\tau}\sum\limits_{i=1}^{\tau} F_{v_{k,i}}(\hx_k),
    \end{eqnarray*}
    where $F_{v_{k,1}}(\hx_k),\ldots, F_{v_{k,\tau}}(\hx_k)$ are independent samples satisfying \eqref{eq: variance bound} with parameters $\delta$ and $\sigma_*^2$. Using variance decomposition and independence of $F_{v_{k,1}}(\hx_k),\ldots, F_{v_{k,\tau}}(\hx_k)$, we get
    \begin{eqnarray*}
        \E_{v_k}\left[\|F_{v_k}(\hx_k)\|^2\right] &=& \E_{v_k}\left[\|F_{v_k}(\hx_k) - F(\hx_k)\|^2\right] + \|F(\hx_k)\|^2\\
        &=& \E_{v_k}\left[\left\|\frac{1}{\tau}\sum\limits_{i=1}^b (F_{v_{k,i}}(\hx_k) - F(\hx_k))\right\|^2\right] + \|F(\hx_k)\|^2\\
        &=& \frac{1}{\tau^2}\sum\limits_{i=1}^{\tau}\E_{v_{k}}\left[\|F_{v_{k,i}}(\hx_k) - F(\hx_k)\|^2\right] + \|F(\hx_k)\|^2\\
        &\overset{\eqref{eq: variance bound}}{\leq}& \frac{\delta}{\tau}\|\hx_k - x^*\|^2 + \|F(\hx_k)\|^2 + \frac{2\sigma_*^2}{\tau}. 
    \end{eqnarray*}
    That is, mini-batched estimator $F_{v_k}(\hx_k)$ satisfies \eqref{eq: variance bound} with parameters $\nicefrac{\delta}{\tau}$ and $\nicefrac{\sigma_*^2}{\tau}$. Therefore, Theorem~\ref{thm:weak_MVI} implies
    \begin{eqnarray}
        \min\limits_{0\leq k \leq K-1}\E\left[\|F(\hat x_k)\|^2\right] &\leq& \frac{(1 + 4\omega\gamma \left(\frac{\delta}{\tau} + L^2\right) - L\gamma)\left(1+\frac{48\omega\gamma \delta}{(1-L\gamma)^2\tau}\right)^{K-1}\|x_0 - x^*\|^2}{\omega\gamma (1 - L(\gamma + 4\omega)) (K-1)} \notag\\
        &&\quad + \frac{8 \left(8 + \frac{1-L\gamma}{K-1}\left(1+\frac{48\omega\gamma \delta}{(1-L\gamma)^2\tau}\right)^{K-1}\right) \sigma_*^2}{(1-L\gamma)(1-L(\gamma + 4\omega))\tau}. \label{eq:MVI_technical_ineq_4}
    \end{eqnarray}
    Since $\tau$ satisfies \eqref{eq:SPEG_weak_MVI_batchsize_appendix} and $\gamma \leq \nicefrac{1}{L}$, $\omega \leq \nicefrac{1}{4L}$, we have
    \begin{eqnarray*}
        4\omega\gamma\left(\frac{\delta}{\tau} + L^2\right) &\leq& \frac{1}{4L^2}\left(\delta \cdot \frac{(1-L\gamma)L^3\omega}{16\delta} + L^2\right) \leq 1,\\
        \left(1+\frac{48\omega\gamma \delta}{(1-L\gamma)^2\tau}\right)^{K-1} &\leq& \left(1+\frac{48\omega\gamma \delta}{(1-L\gamma)^2} \cdot \frac{(1 - L\gamma)^2}{48\omega\gamma \delta(K-1)}\right)^{K-1} \\
        && = \left(1 + \frac{1}{K-1}\right)^{K-1} \leq \exp(1) < 3.
    \end{eqnarray*}
    Using this, we can simplify \eqref{eq:MVI_technical_ineq_4} as follows:
    \begin{eqnarray*}
        \min\limits_{0\leq k \leq K-1}\E\left[\|F(\hat x_k)\|^2\right] &\leq& \frac{6\|x_0 - x^*\|^2}{\omega\gamma (1 - L(\gamma + 4\omega)) (K-1)} + \frac{88 \sigma_*^2}{(1-L\gamma)(1-L(\gamma + 4\omega))\tau}\\
        &\overset{\eqref{eq:SPEG_weak_MVI_batchsize_appendix}}{\leq}& \frac{6\|x_0 - x^*\|^2}{\omega\gamma (1 - L(\gamma + 4\omega)) (K-1)} \\
        && \quad + \frac{88 \sigma_*^2}{(1-L\gamma)(1-L(\gamma + 4\omega))} \cdot \frac{(1-L\gamma)\|x_0 - x^*\|^2}{2\omega\gamma\sigma_*^2}\\
        &=& \frac{48\|x_0 - x^*\|^2}{\omega\gamma (1 - L(\gamma + 4\omega)) (K-1)}.
    \end{eqnarray*}
    This concludes the proof.
\end{proof}

\paragraph{On Oracle Complexity of Theorem \ref{cor:weak_MVI_convergence}.} Let us now express the result of Theorem~\ref{cor:weak_MVI_convergence} via oracle complexity.

Oracle complexity captures the computational requirements required to solve a specific optimization problem. That is, given a prespecified accuracy $\varepsilon > 0$, it measures the number of oracle calls needed to solve the problem to this $\varepsilon$ accuracy. In our setting,  an oracle call indicates the computation of one operator, $F_i$ (for some $i \in [n]$). Therefore, in Theorem \ref{cor:weak_MVI_convergence}, where a mini-batch of size $\tau$ is required in each iteration of the update rule, we have $\tau$ many oracle calls per iteration. In that scenario, the total number of oracle calls required to obtain specific accuracy $\varepsilon > 0$ is given by $K \tau$ (multiplication of $K$ iterations with $\tau$ oracle calls).

Note that according to Theorem \ref{cor:weak_MVI_convergence} to achieve an $\varepsilon$ accuracy, we need $K \geq \frac{C \left\| x_0 - x^* \right\|^2}{\epsilon}$ iterations. This follows trivially by
\begin{equation}
\label{naosdnao}
    \min\limits_{0\leq k \leq K-1}\E\left[\|F(\hat x_k)\|^2\right] \overset{\text{Theorem}~\ref{cor:weak_MVI_convergence}}{\leq} \frac{C\|x_0 - x^*\|^2}{K-1} \leq \varepsilon.
\end{equation}
 Therefore, using $K \geq \frac{C \left\| x_0 - x^* \right\|^2}{\epsilon}$ in combination with the lower bound on $\tau$ from \eqref{eq:SPEG_weak_MVI_batchsize}, the total number of oracle calls to satisfy \eqref{naosdnao} is given by:  
\begin{eqnarray*}
   K\tau \geq \max \left\lbrace \frac{C|| x_0 - x^{\ast} ||^2}{\epsilon}, \frac{32 C \delta || x_0 - x^{\ast} ||^2}{(1 - L \gamma) L^3 \omega \epsilon}, \frac{48 C^2 \omega \gamma \delta || x_0 - x^{\ast}||^4}{(1 - \gamma L)^2 \epsilon^2}, \frac{2 C^2 \omega \gamma \sigma_\ast^2 || x_0 - x^{\ast}||^2}{(1 - L \gamma) \epsilon^2}\right\rbrace. 
\end{eqnarray*}

\newpage
\section{Further Results on Arbitrary Sampling}\label{section: further results on arbitrary sampling}
\subsection{Proof of Proposition \ref{Prop_SingleElement}}
Expanding the left hand side of Expected Residual~\eqref{eq: ER Condition} condition we have
\begin{eqnarray}
    \E \|(F_v(x) - F_v(x^*)) - (F(x) - F(x^*))\|^2 
   &\overset{\eqref{eq: variance of an unbiased estimator}}{=} & \E\|(F_v(x) - F_v(x^*)) \|^2 - \|F(x) - F(x^*)\|^2 \notag \\
   &\leq& \E \|F_v(x) - F_v(x^*)\|^2 . \label{eq: bound on expected residual}
\end{eqnarray}
For any $x$ and $y$ with $v_i = \frac{1}{p_i}$ we obtain  
\begin{eqnarray*}
\|F_v(x) - F_v(y)\|^2 &=& \frac{1}{n^2} \bigg\| \sum_{i \in S} \frac{1}{p_i} (F_i(x) - F_i(y))\bigg\|^2 \\
&=& \sum_{i,j \in S} \bigg \langle \frac{1}{n p_i}(F_i(x) - F_i(y)), \frac{1}{n p_j}(F_j(x) - F_j(y))\bigg \rangle.
\end{eqnarray*}
Then taking expectation on both sides we get
\begin{eqnarray*}
    \E \|F_v(x) - F_v(y)\|^2 &=& \sum_{C} p_C \sum_{i,j \in C} \bigg\langle \frac{1}{n p_i}(F_i(x) - F_i(y)), \frac{1}{n p_j}(F_j(x) - F_j(y)) \bigg\rangle \\
    &=& \sum_{i,j = 1}^n \sum_{C: i,j \in C} p_C \bigg \langle \frac{1}{n p_i}(F_i(x) - F_i(y)), \frac{1}{n p_j}(F_j(x) - F_j(y)) \bigg \rangle \\
    &=& \sum_{i,j = 1}^n \frac{P_{ij}}{p_i p_j} \bigg \langle \frac{1}{n}(F_i(x) - F_i(y)), \frac{1}{n}(F_j(x) - F_j(y)) \bigg \rangle.
\end{eqnarray*}
Now we consider the case, where the ratio $\frac{P_{ij}}{p_i p_j} = c_2$ i.e. constant for $i \neq j$ and $P_{ii} = p_i$. Then from the above computations we derive
\begin{eqnarray*}
    \E \|F_v(x) - F_v(y)\|^2 &=& \sum_{i \neq j}^n c_2 \bigg \langle \frac{1}{n} (F_i(x) - F_i(y)), \frac{1}{n} (F_i(x) - F_i(y)) \bigg \rangle \\
    && \quad + \sum_{i = 1}^n \frac{1}{n^2 p_i} \|F_i(x) - F_i(y)\|^2 \\
    &=& \sum_{i,j = 1}^n c_2 \bigg \langle \frac{1}{n} (F_i(x) - F_i(y)), \frac{1}{n} (F_i(x) - F_i(y)) \bigg \rangle\\
    && \quad + \sum_{i = 1}^n \frac{1 - p_i c_2}{n^2 p_i} \|F_i(x) - F_i(y)\|^2\\
    &\overset{\eqref{eq: F_i lipschitz}}{\leq}&  c_2 \|F(x) - F(y)\|^2 + \sum_{i = 1}^n \frac{1 - p_i c_2}{n^2 p_i} L_i^2\|x - y\|^2 \\
    &\overset{\eqref{eq: F lipschitz}}{\leq}& \bigg(c_2 L^2 + \frac{1}{n^2} \sum_{i = 1}^n \frac{1 - p_i c_2}{p_i} L_i^2 \bigg) \|x - y\|^2.
\end{eqnarray*}
Thus replacing $y = x^*$ and combining with \eqref{eq: bound on expected residual} we get the following bound on the Expected Residual:
\begin{equation}\label{eq: single element sampling bound on ER 2}
    \E \|(F_v(x) - F_v(x^*)) - (F(x) - F(x^*))\|^2 \leq \bigg(c_2 L^2 + \frac{1}{n^2} \sum_{i = 1}^n  \frac{1 - p_i c_2}{p_i} L_i^2 \bigg) \|x - x^*\|^2.
\end{equation}
For single-element sampling $c_2 = 0$ (as probability of two points appearing in same sample is zero for single element sampling i.e. $P_{ij} = 0$). Then we obtain
$$
\delta \leq  \frac{2}{n^2} \sum_{i = 1}^n \frac{L_i^2}{p_i}
$$
from \eqref{eq: single element sampling bound on ER 2}. This completes the derivation of $\delta$ for single element sampling. To compute $\sigma_*^2$ for single element sampling, we replace
\begin{equation*}
    P_{ij} = \begin{cases}
    p_i & \text{if } i = j \\
    0 & \text{otherwise}
    \end{cases}
\end{equation*}
in \eqref{eq: expansion for sigma} to get
\begin{equation*}
    \sigma_*^2 = \frac{1}{n^2} \sum_{i=1}^n \frac{1}{p_i} \|F_i(x^*)\|^2.
\end{equation*}

\newpage
\section{Numerical Experiments}
\label{Appendix_AddExperiments}
In Appendix \ref{subsec:FurtherDetails}, we add more details on the experiments discussed in the main paper. Furthermore, in Appendix \ref{subsec:ComparisonConstantvsSwitching}, we run more experiments to evaluate the performance of \algname{SPEG} on quasi-strongly monotone and weak MVI problems.

\subsection{More Details on the Numerical Experiments of Section~\ref{Numerical Experiments}}\label{subsec:FurtherDetails}

\paragraph{On Constant vs Switching Stepsize Rule.} We run the experiments on two synthetic datasets. In Fig.~\ref{fig: Synthetic Dataset 1} of the main paper, we take $\mu_A = \mu_C = 0.6$. Here we include one more plot with a similar flavor but in a different setting. For Fig.~\ref{fig: Synthetic Dataset 2}, we generate the data such that eigenvalues of $A_1, B_1, C_1$ are generated uniformly from the interval $[0.1, 10]$. In the new plot, similar to the main paper, we can see the benefit of switching the step-size rule of Theorem \ref{SPEG switching rule}.

\begin{figure}[h]
\centering
    \includegraphics[width=.48\textwidth]{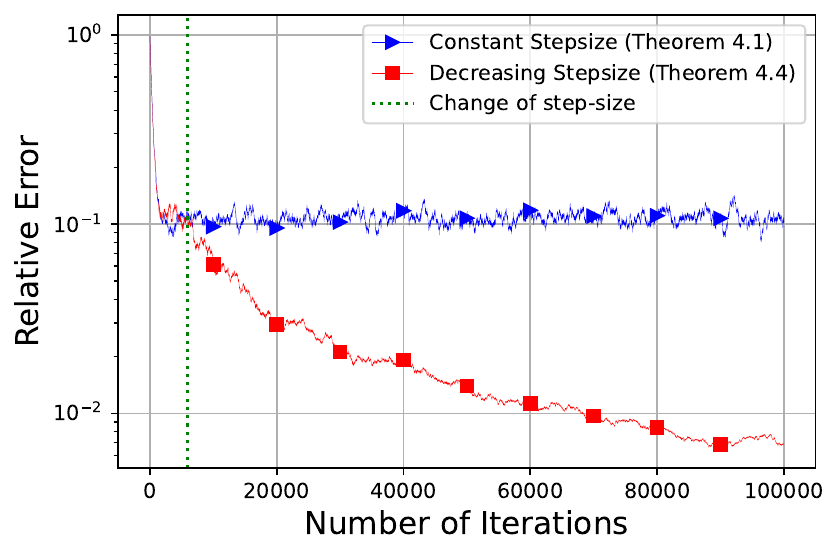}
\caption{Comparison of the constant step-size rule~\eqref{eq:constant_stepsize} with the switching step-sizes~\eqref{eq:stepsize_switching_1} on the strongly monotone quadratic game.}\label{fig: Synthetic Dataset 2}
\end{figure}

\paragraph{On Weak Minty VIPs.}
In this experiment, we generate $\xi_i, \zeta_i$ such that $\frac{1}{n}\sum_{i = 1}^n \xi_i = \sqrt{63}$ and $\frac{1}{n} \sum_{i = 1}^n \zeta_i = - 1$. This choice of $\xi_i, \zeta_i$ ensures that $L = 8$ and $\rho = \nicefrac{1}{32}$ for the min-max problem we considered in Section \ref{sec: Experiment on WMVI}. In Fig. \ref{fig:performance of SPEG on WMVI_b}, we again implement the \algname{SPEG} on \eqref{asoxasl} with batchsize = $0.15 \times n$ (different batchsize compare to the plot of the main paper). 
 \begin{figure}[h]
     \centering
    \includegraphics[width=.48\textwidth]{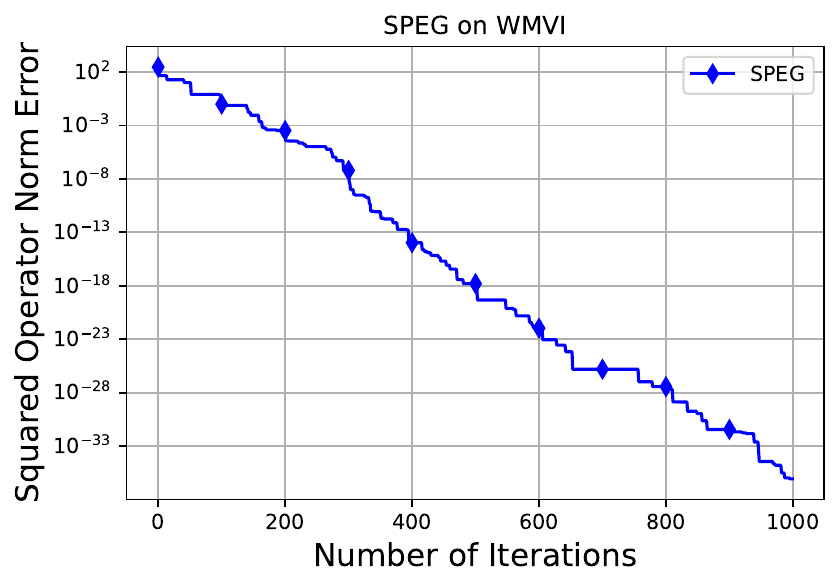}
    \caption{Trajectory of \algname{SPEG} for solving weak MVI using a batchsize = $0.15 \times n$.}
    \label{fig:performance of SPEG on WMVI_b}
\end{figure}

\subsection{Additional Experiments}\label{subsec:ComparisonConstantvsSwitching}
In this subsection, we include more experiments to evaluate the performance of \algname{SPEG} on quasi-strongly monotone and weak MVI problems. 
First, we run the experiment comparing constant and switching step-size rules on a different setup than the one we included in the main paper to analyze the performance of \algname{SPEG} under different condition numbers. Then, we implement \algname{SPEG} on the weak MVI of \eqref{asoxasl}. To evaluate the performance in this experiment, we plot $\nicefrac{\|F(\hat{x}_k)\|^2}{\|F(x_0)\|^2}$ on the $y$-axis.
\subsubsection{Strongly Monotone Quadratic Game:} In this experiment, we compare the proposed constant step-size~\eqref{eq:constant_stepsize} and the switching step-size rule~\eqref{eq:stepsize_switching_1}. We implement our algorithm on operator $F: \R^{4} \to \R^{4}$ given by 
\begin{equation*}
   F(x) \eqdef \frac{1}{3} \left(M_1(x - x_1^*) + M_2 (x - x_2^*) + M_3 (x - x_3^*)\right), 
\end{equation*}
where $M_1$, $M_2$ and $M_3$ are the diagonal matrices, 
\begin{eqnarray*}
    M_1 = \begin{pmatrix}
        \Delta & & &\\
        & 1 & & \\
        & & 1 & \\
        & & & 1
    \end{pmatrix}, \quad M_2 = \begin{pmatrix}
        1 & & &\\
        & \Delta & & \\
        & & 1 & \\
        & & & 1
    \end{pmatrix}, \quad M_3 = \begin{pmatrix}
        1 & & &\\
        & 1 & & \\
        & & \Delta & \\
        & & & 1
    \end{pmatrix}
\end{eqnarray*}
and
\begin{eqnarray*}
    x_1^* = \begin{pmatrix}
    \Delta \\
    0 \\
    0 \\
    \Delta
    \end{pmatrix}, \quad x_2^* = \begin{pmatrix}
    0 \\
    \Delta \\
    0 \\
    0
    \end{pmatrix}, \quad x_3^* = \begin{pmatrix}
    0 \\
    0 \\
    \Delta \\
    0
    \end{pmatrix}. 
\end{eqnarray*}
This choice of $M_i$ and $x_i^*$ ensures that the Lipschitz constant of operator $F$ is $\frac{\Delta + 2}{3}$ while quasi-strong monotonicity parameter~\eqref{eq: Strong Monotonicity} is $\mu = 1$. Hence the condition number of $F$ is given by $\frac{\Delta + 2}{3}$. This allows us to vary the condition number of operator $F$ by changing the value of $\Delta$. For Fig. \ref{fig:condition_no_1} we take $\Delta = 3$ (condition number = $1.67$) while for Fig. \ref{fig:condition_no_10} we choose $\Delta = 10$ (condition number = $10.67$). The vertical dotted line in plots of Fig.~\ref{fig:constant_vs_switch_exp_2} marks the transition point from constant to switching step-size rule as predicted by our theoretical result in Theorem~\ref{SPEG switching rule}.

\begin{figure}[h]
\centering
\begin{subfigure}[b]{.48\textwidth}
    \centering
    \includegraphics[width=\textwidth]{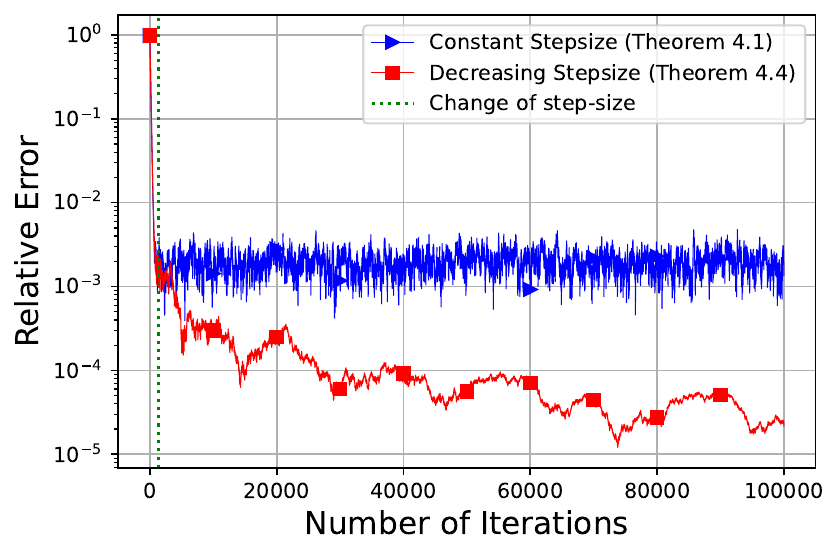}
    \caption{Condition Number $\frac{L}{\mu} = 1.67$.}\label{fig:condition_no_1}
\end{subfigure}
\begin{subfigure}[b]{0.48\textwidth}
    \centering
    \includegraphics[width=\textwidth]{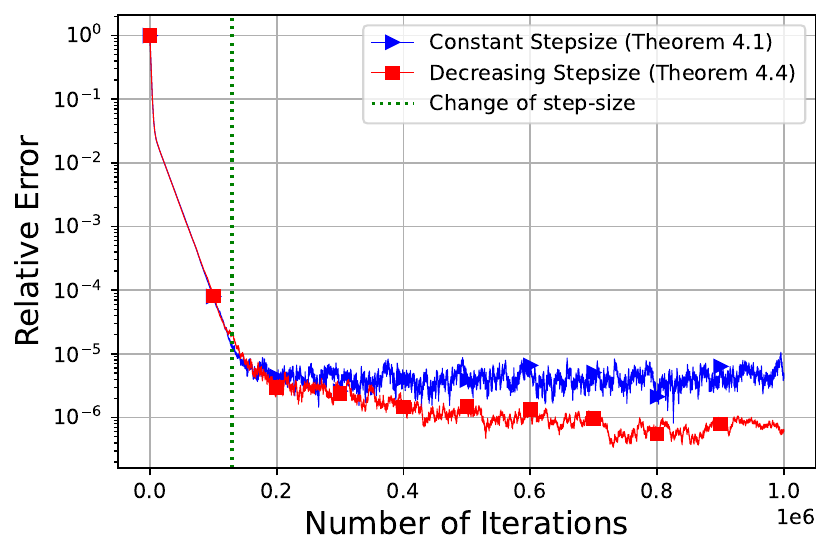}
    \caption{Condition Number $\frac{L}{\mu} = 10.67$.}\label{fig:condition_no_10}
\end{subfigure}
    \caption{\emph{Illustration of switching rule~\eqref{eq:stepsize_switching_1} in Theorem \ref{SPEG switching rule}. The dotted line marks the transition from phase 1 (where we use constant step-size) to phase 2 (where we use decreasing step-size).}}
\label{fig:constant_vs_switch_exp_2}
\end{figure}

\subsubsection{Weak Minty VIPs Continued}
\label{subsec:WeakMVI}
In this experiment, we reevaluate the performance of \algname{SPEG} on weak MVI example of \eqref{asoxasl}. That is, we generate the data in exactly the same way as the ones in section \ref{sec: Experiment on WMVI} with $n = 100$. In Fig. \ref{fig:batchsize=0.1} and \ref{fig:batchsize=0.15}, we implement \algname{SPEG} with batchsize $10$ and $15$, respectively (we note that in this setting the full-gradient evaluation requires a batchsize of $100$). For these plots, we use the relative operator norm on the $y$-axis, i.e. $\nicefrac{\|F(\hat{x}_k)\|^2}{\|F(x_0)\|^2}$, where $x_0$ denotes the starting point of \algname{SPEG}. As expected, the plots illustrate that \algname{SPEG} performs better as we increase the batchsize. From Fig.~\ref{fig:SPEG_weak_MVI_operator_norm} it is clear that with batchsize 15 \algname{SPEG} reaches an accuracy close to $10^{-10}$ while when we use a batchsize of $10$ for the same number of iterations we are only able to converge to an accuracy of $10^{-4}$. 
\begin{figure}[H]
\centering
\begin{subfigure}[b]{.48\textwidth}
    \centering
    \includegraphics[width=\textwidth]{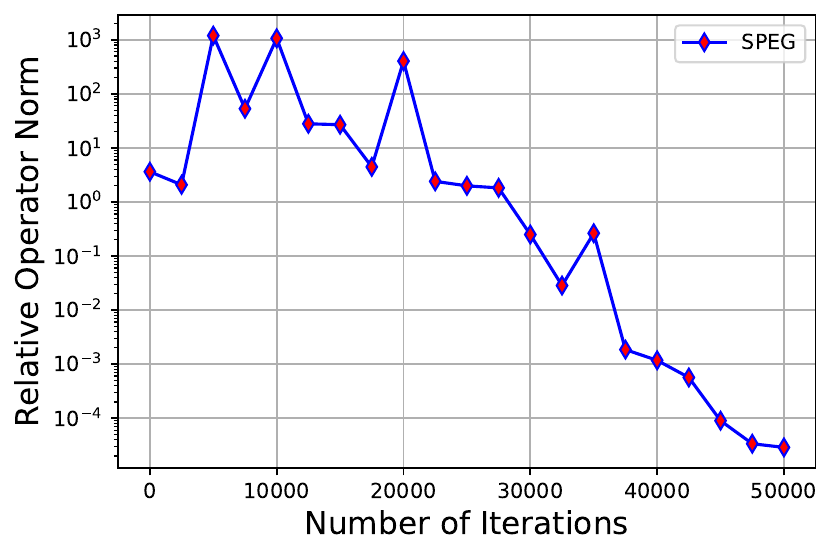}
    \caption{Batchsize = $0.1 \times n.$}\label{fig:batchsize=0.1}
\end{subfigure}
\begin{subfigure}[b]{0.48\textwidth}
    \centering
    \includegraphics[width=\textwidth]{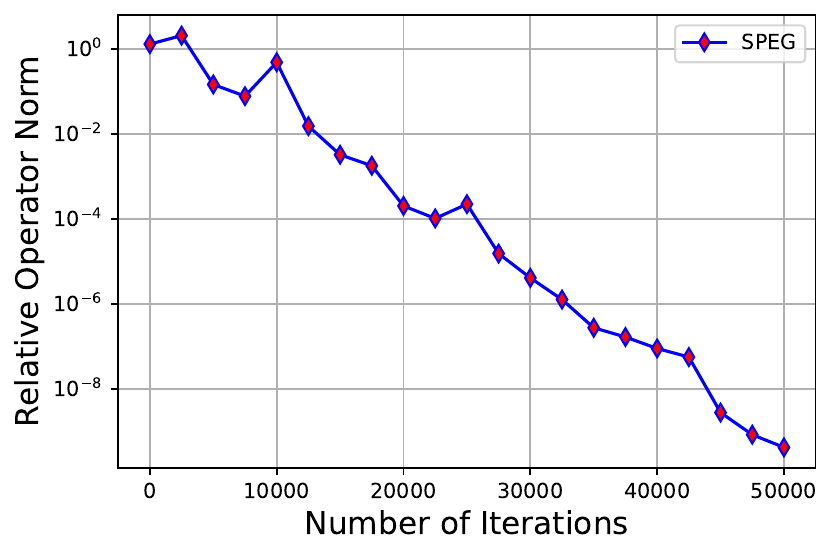}
    \caption{Batchsize = $0.15 \times n.$}\label{fig:batchsize=0.15}
\end{subfigure}
    \caption{\emph{Performance of \algname{SPEG} for solving weak MVI with different batchsizes. In plot (a) we use a batchsize of $10$ while in plot (b) we use $15$.}}
\label{fig:SPEG_weak_MVI_operator_norm}
\end{figure}
\chapter{Appendices for Chapter \ref{chap:chap-3}} \label{chap:appendix-b}

\section{Further Related Work for Chapter \ref{chap:chap-3}}\label{appendix:related_work}

In this section, we present a broader view of the literature, including more details on closely related work and more references to papers that are not directly related to our main results.

\paragraph{Extragradient Methods} The \algname{EG} method is a widely utilized algorithm for solving variational inequality problems (VIPs). It can be regarded as an approximation of the Proximal Point method~\cite{mokhtari2020unified} and is effective in solving specific classes of structured non-monotone VIPs, such as \textit{weak Minty} VIPs~\cite{diakonikolas2021efficient} and \textit{negatively comonotone} VIPs~\cite{gorbunov2023convergence}. For deterministic VIPs, \algname{EG} achieves a convergence rate of $\mathcal{O}\left( \nicefrac{1}{K}\right)$~\cite{gorbunov2022last}, where $K$ is the total number of iterations, for monotone problems, and exhibits a linear rate for strongly monotone problems~\cite{mokhtari2020unified}. In contrast, the stochastic extragradient method converges to a neighbourhood of the solution at a rate of $\mathcal{O}\left( \nicefrac{1}{K}\right)$~\cite{gorbunov2022stochastic}.

Moreover, recent advancements have introduced single-call \algname{EG} methods~\cite{popov1980modification}, which match the convergence guarantees of the traditional \algname{EG} method in both deterministic and stochastic settings while utilizing only a single oracle call per iteration~\cite{gorbunov2023convergence, hsieh2019convergence, choudhury2024single}. Researchers have also developed an enhanced version of \hyperref[eq:EG]{\algname{EG}} with anchoring, which achieves the optimal convergence rate of $\mathcal{O}(\nicefrac{1}{K^2})$ for monotone problems \cite{yoon2021accelerated}. 

\paragraph{Polyak Step Size for Minimization}\cite{polyak1987introduction} initially introduced the Polyak step size to address non-convex minimization problems. This choice of stepsize enables the algorithm to adapt to the curvature and local smoothness of the problem. Recent advancements have extended the application of the Polyak step size to stochastic settings, particularly for solving finite-sum minimization problems of the form:
\begin{eqnarray}
\min_x f(x) \eqdef \frac{1}{n} \sum_{i = 1}^n f_i(x).
\end{eqnarray}
\cite{loizou2021stochastic} proposed the stochastic Polyak step size (\algname{SPS}) to solve the problem in \eqref{eq:min}. This step size selection can significantly accelerate the training process of overparameterized models~\cite{loizou2021stochastic}. However, \algname{SPS} requires knowledge of the optimal values $f_i^*$, typically known only in specific scenarios. Additionally, \algname{SPS} can ensure convergence only to a neighbourhood of the solution for smooth convex problems. To address these limitations, \cite{orvieto2022dynamics} introduced \algname{DecSPS}, which employs a decreasing step size with a lower bound on $f_i^*$ to ensure exact convergence. In a parallel line of research, \cite{gower2021stochastic} reinterpreted \algname{SPS} from a subsampled Newton-Raphson perspective and designed new algorithms that leverage this insight.

\paragraph{Projection Interpretation of Extragradient Method} \cite{solodov1996modified} introduced the concept of interpreting the update step of the \hyperref[eq:EG]{\algname{EG}} method as a projection onto a hyperplane. This perspective has inspired recent advancements in the field. For instance, \cite{pethick2023escaping} utilized this approach in the algorithm known as \algname{AdaptiveEG+} to address \textit{weak Minty} VIPs. Additionally, they proposed \algname{CurvatureEG+}, which employs a backtracking line-search to determine the extrapolation stepsize $\gamma$. Building on this idea, \cite{fan2023weaker} developed a method featuring multiple extrapolations, enabling the solution of a broader class of \textit{weak minty} VIPs. Chapter \ref{chap:chap-3} focuses on \algname{AdaptiveEG+} and explores its connection to the Polyak step size literature.

\paragraph{Adaptive Gradient Descent Ascent Methods} Beyond \algname{EG}, stochastic adaptive step sizes have also been investigated for \algname{GDA} algorithms. For instance, \cite{yang2022nest} introduced \algname{NeAda}, a two-loop algorithm where the variable $w_2$ undergoes multiple updates until a stopping criterion is met before updating $w_1$ (here, $w_1, w_2$ are variables from \eqref{eq: MinMax}). However, this nested approach has several drawbacks: in the stochastic setting, the number of inner loop iterations increases with the number of outer loop iterations. Additionally, achieving a near-optimal convergence rate requires a large batch size, which is often impractical in real-world applications. To address these limitations, \cite{li2022tiada} proposed \algname{TiAda}, a single-loop algorithm employing AdaGrad-type step sizes. However, their convergence guarantees rely on second-order Lipschitz continuity assumptions, which our proof technique does not require. Instead of a nested loop, our algorithm utilizes a simple line-search strategy that terminates after a fixed number of iterations (Lemma \ref{lemma:LS-step-size-lower-bound}).

\section{Projection Based Interpretation of Polyak Step Size}\label{appendix:motivation}
In this section, we explore the projection-based interpretation of the Polyak step size for minimization and demonstrate how it can be used to design Polyak steps for the \algname{EG} method.

\subsection{Polyak Step Size as Projection onto Hyperplane} Interestingly, the Polyak step size can also be derived from a projection perspective~\cite{gower2021stochastic}. To see this, consider the goal of solving the nonlinear equation $f(x) = f(x_*)$ (note that it is equivalent to solving $\min_{x \in \R^d} f(x)$). Using a first-order Taylor expansion of $f$ around $x_k$, we approximate this equation by the linear equation:
\begin{eqnarray}\label{eq:hyperplane1}
    f(x_*) \approx f(x_k) + \la \nabla f(x_k), x - x_k \ra.
\end{eqnarray}
This linear equation represents a hyperplane in the variable $x$. Therefore projecting a point $x'$ onto this hyperplane~\eqref{eq:hyperplane1} yields the following expression:
\begin{eqnarray*}
    x' - \frac{\la \nabla f(x_k), x' - x_k \ra + (f(x_k) - f(x_*))}{\| \nabla f(x_k)\|^2} \nabla f(x_k).
\end{eqnarray*}
In particular, projecting $x_k$ onto this hyperplane gives us
\begin{eqnarray*}
    x_k - \frac{f(x_k) - f(x_*)}{\| \nabla f(x_k)\|^2} \nabla f(x_k).
\end{eqnarray*}
Therefore, the Polyak step size for \algname{GD} can also be interpreted as a projection algorithm. This projection interpretation provides additional insight into the geometric motivation behind the Polyak step size.

\begin{figure}[H]
    \centering
    \includegraphics[width=0.6\linewidth]{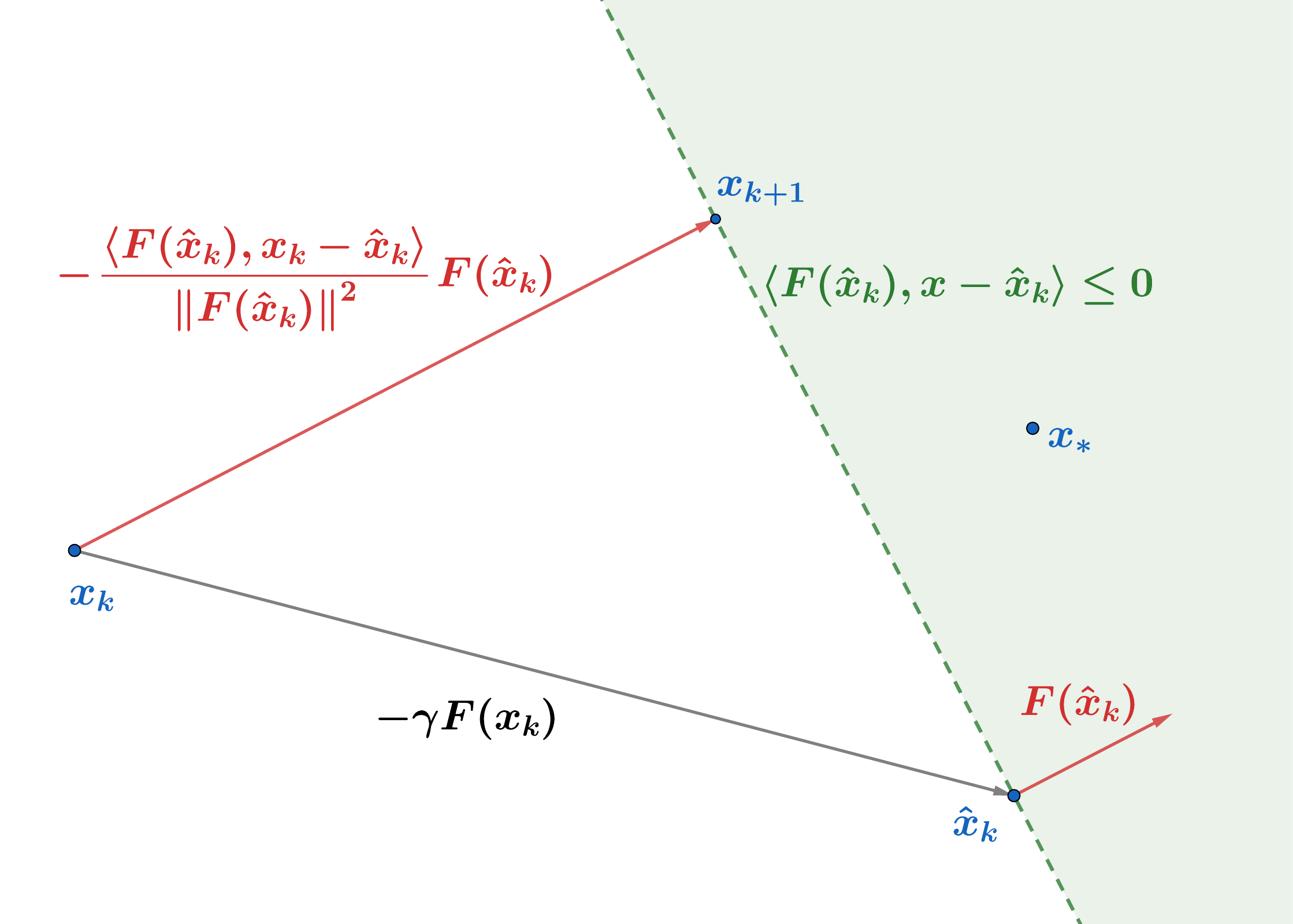}
    \caption{\small Geometric interpretation of using \hyperref[eq:EG]{\algname{EG}} with $\alpha_k = \frac{\la F(\hx_k), x_k - \hx_k \ra}{\| F(\hx_k)\|^2}$ and $\gamma_k = \gamma > 0$. The algorithm iteratively moves from the current point $x_k$ to a point $\hx_k$ via a gradient step and then corrects this move by projecting onto the hyperplane $\la F(\hx_k), x - \hx_k\ra = 0$, leading to the next point $x_{k+1}$.}
    \label{fig:polyakEG_geometric}
\end{figure}

\subsection{Geometric Interpretation of Polyak Step Size for Extragradient}
\cite{solodov1996modified} also developed similar step sizes, where they interpreted the update step as a projection onto a hyperplane. Given the extrapolation iterate $\hx_k$, a halfspace can be constructed as follows
\begin{eqnarray*}
    \mathcal{D} \left(\hx_k \right) \eqdef \left\{ x \in \R^d \mid \la F(\hx_k), \hx_k - x \ra \geq 0 \right\}.
\end{eqnarray*}
It is important to note that the following holds from monotonicity of $F$:
\begin{equation}\label{eq:x_star_in_halfspace}
    \la F(\hx_k), \hx_k - x_* \ra  \geq \la F(x_*), \hx_k - x_* \ra  = 0.
\end{equation} 
Here, the second equality follows the definition of $x_*$. Consequently, we have $x_* \in \mathcal{D} \left( \hx_k \right)$ for all $k$. Since $\mathcal{D} \left( \hx_k \right)$ contains the solution $x_*$, it is possible to move closer to $x_*$ by projecting $x_k$ onto $\mathcal{D} \left( \hx_k \right)$ i.e.
\begin{eqnarray*}
    x_{k+1} & = & \Pi_{\mathcal{D} \left( \hx_k \right)} \left[ x_k \right].
\end{eqnarray*}
As $\mathcal{D} \left(\hx_k\right)$ is a halfspace, the projection of $x_k$ onto $\mathcal{D} \left( \hx_k\right)$ is given by $\Pi_{\mathcal{D} \left( \hx_k \right)} \left[ x_k \right] = x_k - \frac{\la F(\hx_k), x_k - \hx_k \ra}{\|F(\hx_k)\|^2} F(\hx_k)$. Therefore, the above update iteration simplifies to
\begin{eqnarray*}
    x_{k+1} & = & x_k - \frac{\la F(\hx_k), x_k - \hx_k \ra}{\|F(\hx_k)\|^2} F(\hx_k),
\end{eqnarray*}
which is the same as what we propose in \hyperref[alg:PolyakEG]{\algname{PolyakEG}}.

\paragraph{Geometric Interpretation of \hyperref[alg:PolyakEG]{\algname{PolyakEG}}.} In Figure \ref{fig:polyakEG_geometric}, we provide a detailed geometric interpretation of the \hyperref[alg:PolyakEG]{\algname{PolyakEG}} method. Starting from the current iterate $x_k$, we compute the vector $-\gamma F(x_k)$ and use it to obtain the intermediate point $\hx_k = x_k - \gamma F(x_k)$. At this point, we construct a hyperplane defined by the equation $\la F(\hx_k), x - \hx_k \ra = 0$. This hyperplane is perpendicular to the vector $F(\hx_k)$ and passes through the point $\hx_k$.\\
\newline
It is important to note that the optimal point $x_*$ lies within the half-space $\la F(\hx_k), x - \hx_k \ra \leq 0$ (follows from \eqref{eq:x_star_in_halfspace}), which is depicted as the green-shaded region in Figure \ref{fig:polyakEG_geometric}. The next iterate, $x_{k+1}$, is then selected as the point on the hyperplane $\la F(\hx_k), x - \hx_k \ra = 0$ such that the vector $x_{k+1} - x_k$ is perpendicular to this hyperplane. Note that
\begin{eqnarray*}
    \la F(\hx_k), x_{k+1} - \hx_k \ra & = & \quad \la F(\hx_k), x_k - \hx_k \ra - \alpha_k \la F(\hx_k), F(\hx_k) \ra \\
    & = & \quad \la F(\hx_k), x_k - \hx_k \ra - \alpha_k \| F(\hx_k) \|^2 \\
    & = & \quad 0.
\end{eqnarray*}
The last line follows from the expression of $\alpha_k = \frac{\la F(\hx_k), x_k - \hx_k\ra}{\|F(\hx_k)\|^2}$.

\newpage

\section{Technical Lemmas}\label{sec:technical_lemma}
In this section, we present some technical lemmas, which will be used repeatedly to prove the main results of this chapter in subsequent sections.

\begin{lemma}[Young's Inequality]
\label{lemma:young-ineq}
    For any $a, b \in \R^d$ and $\omega > 0$ we have
    \begin{eqnarray}\label{eq:young1}
        \| a + b \|^2 & \leq & (1 + \omega) \|a\|^2 + (1 + \omega^{-1}) \|b\|^2.
    \end{eqnarray}
    In particular, for $\omega = 1$ we have
    \begin{eqnarray}\label{eq:young2}
        \| a + b \|^2 & \leq & 2 \|a\|^2 + 2 \|b\|^2.
    \end{eqnarray}
\end{lemma}

\begin{lemma}
\label{lemma:variant-young-ineq}
    For any $a, b \in \R^d$ and $\omega > 0$ we have
    \begin{eqnarray}\label{eq:young4}
        -\| a\|^2 & \leq &  -\frac{1}{1 +  \omega}\|a + b\|^2 + \frac{1}{\omega} \|b\|^2.
    \end{eqnarray}
    In particular for $\omega = 1$ we get
    \begin{eqnarray}\label{eq:young3}
        -\| a\|^2 & \leq &  -\frac{1}{2}\|a + b\|^2 + \|b\|^2.
    \end{eqnarray}
\end{lemma}

\begin{lemma}[Cauchy-Schwarz Inequality]
    For any $a, b \in \R^d$ we have
    \begin{equation}\label{eq:CS1}
        \la a, b\ra \leq \|a\| \|b\|.
    \end{equation}
\end{lemma}

\begin{lemma}
\label{lemma:CS2}
    For any $a, b \in \R^d$ and $\lambda > 0$ we have
    \begin{eqnarray}\label{eq:CS2}
        \la a, b\ra & \leq & \frac{\lambda}{2} \|a\|^2 + \frac{1}{2 \lambda} \|b\|^2.
    \end{eqnarray}
\end{lemma}

\begin{lemma}
    For any $a, b \in \R^n$ we have
    \begin{eqnarray}\label{eq:expand_innerprod}
        \la a, b\ra & = & \frac{1}{2}\|a\|^2 + \frac{1}{2}\|b\|^2 - \frac{1}{2}\|a - b\|^2.
    \end{eqnarray}
\end{lemma}

\begin{lemma}\label{lemma:sum_1_by_sqrt_k}
    For $K \in \mathbb{N}$ we have
    \begin{equation}\label{eq:sum_1_by_sqrt_k}
        \sum_{k = 0}^K \frac{1}{\sqrt{k+1}} \leq 2 \sqrt{K+1}.
    \end{equation}
\end{lemma}

\begin{proof}
    For any $k \leq z \leq k+1$ we have $\frac{1}{\sqrt{k+1}} \leq \frac{1}{\sqrt{z}}$. Then 
    \begin{eqnarray*}
        \sum_{k = 0}^K \frac{1}{\sqrt{k+1}} & = & \sum_{k = 0}^K \frac{1}{\sqrt{k+1}} \int_{k}^{k+1} dz \notag\\
        & \leq & \sum_{k = 0}^K \int_{k}^{k+1} \frac{1}{\sqrt{x}} dz \notag\\
        & = & 2 \sum_{k = 0}^K  \left(\sqrt{k+1} - \sqrt{k} \right) \notag\\
        & = & 2 \sqrt{K+1}. 
\end{eqnarray*}
\end{proof}

\begin{lemma}[\cite{orvieto2022dynamics}]\label{lemma:orvieto}
    Let $z_{k+1} = A_k z_k + \varepsilon_k$ with $A_k = \left( 1- \frac{a}{\sqrt{k+1}}\right)$ and $\varepsilon_k = \frac{b}{\sqrt{k+1}}$. If $z_0 > 0$, $0 < a \leq 1$, $b > 0$ then 
    \begin{eqnarray*}
        z_{k+1} \leq \max \left\{ z_0, \frac{b}{a} \right\} \qquad \forall \quad k \geq 0.
    \end{eqnarray*}
\end{lemma}

\newpage

\section{Missing Proofs}\label{sec:missing_proof}

\subsection{Deterministic Setting}
In this section, we present the missing proofs from Section \ref{sec:conv_ana1}.

\subsubsection{Proof of Lemma \ref{lemma:alpha-lower-bound}}
\begin{lemma}
\vspace{1mm}
If \eqref{eqn:gamma-critical-condition} is satisfied with $A \in (0,1]$, then we have 
\begin{align*}
    \textstyle
    \inprod{F(\hat{x}_k)}{F(x_k)} \ge \frac{1}{2} \left( \sqnorm{F(\hat{x}_k)} + (1 - A^2) \sqnorm{F(x_k)} \right)
\end{align*}
which implies $\alpha_k \ge \frac{\gamma_k}{1+A}$.
\end{lemma}

\begin{proof}
From the critical condition \eqref{eqn:gamma-critical-condition}, we have
\begin{align*}
    \inprod{F(\hat{x}_k)}{F(x_k)} & \overset{\eqref{eq:expand_innerprod}}{=} \frac{1}{2} \left( \sqnorm{F(\hat{x}_k)} - \sqnorm{F(\hat{x}_k) - F(x_k)} + \sqnorm{F(x_k)} \right) \\
    & \overset{\eqref{eqn:gamma-critical-condition}}{\ge} \frac{1}{2} \left( \sqnorm{F(\hat{x}_k)} - A^2 \sqnorm{F(x_k)} +\sqnorm{F(x_k)} \right) \\
    & = \frac{1}{2} \left( \sqnorm{F(\hat{x}_k)} + (1 - A^2) \sqnorm{F(x_k)} \right) .
\end{align*}
This completes the proof of the first part. Then, using
\begin{eqnarray*}
    \norm{F(\hx_k)} & \le & \norm{F(x_k)} + \norm{F(\hx_k) - F(x_k)} \\
    & \overset{\eqref{eqn:gamma-critical-condition}}{\leq} & \norm{F(x_k)} + A\norm{F(x_k)} \\
    & = & (1+A) \norm{F(x_k)}
\end{eqnarray*}
we can further lower bound the inner product $\inprod{F(\hat{x}_k)}{F(x_k)}$ as
\begin{eqnarray*}
    \inprod{F(\hat{x}_k)}{F(x_k)} & \ge & \frac{1}{2} \left( \sqnorm{F(\hx_k)} + \frac{1-A^2}{(1+A)^2} \sqnorm{F(\hx_k)} \right) \\
    & = & \frac{1}{1+A} \sqnorm{F(\hx_k)} .
\end{eqnarray*}
This shows that
\begin{eqnarray*}
    \alpha_k & = & \frac{\inprod{F(\hx_k)}{x_k - \hx_k}}{\sqnorm{F(\hx_k)}} \\
    & = & \frac{\gamma_k \inprod{F(\hat{x}_k)}{F(x_k)}}{\sqnorm{F(\hx_k)}} \\
    & \ge & \frac{\gamma_k}{1+A} .
\end{eqnarray*}
\end{proof}

\subsubsection{Proof of Theorem \ref{theorem:deterministic-master-theorem}}
\begin{theorem}
\vspace{1mm}
Let $F$ is $L$-Lipschitz and (strongly) monotone. Suppose that we choose $\gamma_k$ so that \eqref{eqn:gamma-critical-condition} holds for all $k\ge 0$. Then, \hyperref[alg:PolyakEG]{\algname{PolyakEG}} satisfies, when $F$ is
\begin{align*}
    \hspace{-3mm} &\text{$\bullet$ monotone and $ A \in (0, 1]$,} \textstyle \min_{k=0,\dots,K} \gamma_k^2 \sqnorm{F(\hat{x}_k)} \le \frac{(1+A)^2 \sqnorm{x_0 - x_*}}{K+1}. \\
    \hspace{-1mm} &\text{If $A \in (0, 1)$, we additionally have,} \textstyle \min_{k=0,\dots,K} \gamma_k^2 \sqnorm{F(x_k)} \le \frac{(1+A) \sqnorm{x_0 - x_*}}{(1-A) (K+1)}.  \\
    \hspace{-3mm} &\text{$\bullet$ strongly monotone and $A \in (0, 1)$,} \textstyle \sqnorm{x_{k+1} - x_*} \le \prod_{j=0}^{k} \left( 1 - \frac{2(1-A) \gamma_j \mu}{(1+A)^2} \right) \sqnorm{x_0 - x_*}
\end{align*}
\end{theorem}

\begin{proof}
\begin{align*}
    \sqnorm{x_{k+1} - x_*} & = \sqnorm{x_k - \alpha_k F(\hat{x}_k) - x_*} \nonumber \\
    & = \sqnorm{x_k - x_*} - 2\alpha_k \inprod{F(\hat{x}_k)}{x_k - x_*} + \alpha_k^2 \sqnorm{F(\hat{x}_k)} \nonumber \\
    & = \sqnorm{x_k - x_*} - 2\alpha_k \inprod{F(\hat{x}_k)}{x_k - \hat{x}_k} - 2\alpha_k \inprod{F(\hat{x}_k)}{\hat{x}_k - x_*} \nonumber \\
    & + \alpha_k^2 \sqnorm{F(\hat{x}_k)} \nonumber \\
    & \le \sqnorm{x_k - x_*} - 2\alpha_k \inprod{F(\hat{x}_k)}{x_k - \hat{x}_k} - 2\alpha_k \mu \sqnorm{\hat{x}_k - x_*} + \alpha_k^2 \sqnorm{F(\hat{x}_k)} .
\end{align*}
We apply Lemma \ref{lemma:variant-young-ineq} with $a=\hx_k - x_*$ and $b=x_k - \hx_k$ to upper bound the last expression (we determine $\omega$ later), which gives
\begin{align}
    \sqnorm{x_{k+1} - x_*} & \le \sqnorm{x_k - x_*} - 2\alpha_k \inprod{F(\hat{x}_k)}{x_k - \hat{x}_k} - \frac{2\alpha_k \mu}{1+\omega} \sqnorm{x_k - x_*} \nonumber \\
    & + \frac{2\alpha_k \mu}{\omega} \sqnorm{x_k - \hat{x}_k} + \alpha_k^2 \sqnorm{F(\hat{x}_k)} \nonumber \\
    & = \left( 1 - \frac{2\alpha_k \mu}{1+\omega} \right) \sqnorm{x_k - x_*} + \frac{2\alpha_k \mu}{\omega} \sqnorm{x_k - \hat{x}_k} - 2\alpha_k \inprod{F(\hat{x}_k)}{x_k - \hat{x}_k}  \nonumber \\
    & + + \alpha_k^2 \sqnorm{F(\hat{x}_k)} \nonumber \\
    & = \left( 1 - \frac{2\alpha_k \mu}{1+\omega} \right) \sqnorm{x_k - x_*} + \frac{2\alpha_k \gamma_k^2 \mu}{\omega} \sqnorm{F(x_k)} - \alpha_k \inprod{F(\hat{x}_k)}{x_k - \hat{x}_k} \nonumber \\
    & = \left( 1 - \frac{2\alpha_k \mu}{1+\omega} \right) \sqnorm{x_k - x_*} + \frac{2\alpha_k \gamma_k^2 \mu}{\omega} \sqnorm{F(x_k)} - \alpha_k \gamma_k \inprod{F(\hat{x}_k)}{F(x_k)} \nonumber \\
    & = \left( 1 - \frac{2\alpha_k \mu}{1+\omega} \right) \sqnorm{x_k - x_*} - \alpha_k \gamma_k \left( \inprod{F(\hat{x}_k)}{F(x_k)} - \frac{2\gamma_k \mu}{\omega} \sqnorm{F(x_k)} \right) .
    \label{eqn:LS-proof-bound}
\end{align}
Here, for the third line, we use the definition of $\alpha_k$ to replace $\alpha_k^2 \sqnorm{F(\hx_k)} = \alpha_k \inprod{F(\hx_k)}{x_k - \hx_k}$.

\vspace{.2cm}
\noindent
\textbf{$\bullet$ $F$ is monotone.} 
From \eqref{eqn:LS-proof-bound}, using Lemma \ref{lemma:alpha-lower-bound} we have
\begin{align*}
    \sqnorm{x_{k+1} - x_*} & \le \sqnorm{x_k - x_*} - \alpha_k \gamma_k \inprod{F(\hat{x}_k)}{F(x_k)} \\
    & \le \sqnorm{x_k - x_*} - \frac{\alpha_k \gamma_k}{2} \left( \sqnorm{F(\hat{x}_k)} + (1 - A^2) \sqnorm{F(x_k)} \right) \\
    & \le \sqnorm{x_k - x_*} - \frac{\gamma_k^2}{2(1+A)} \left( \sqnorm{F(\hat{x}_k)} + (1 - A^2) \sqnorm{F(x_k)} \right) .
\end{align*}
Summing the above inequality for $k=0,1,\dots,K$ we obtain
\begin{align}
\label{eqn:monotone-bound-Fxk-Fhxk-combined}
    \sum_{k=0}^{K} \frac{\gamma_k^2}{2(1+A)}  \left( \sqnorm{F(\hat{x}_k)} + (1 - A^2) \sqnorm{F(x_k)} \right) & \le \sqnorm{x_0 - x_*} - \sqnorm{x_{K+1} - x_*} \nonumber \\
    & \le \sqnorm{x_0 - x_*} .
\end{align}
Now note that we have
\begin{align}
\label{eqn:Fhxk-upper-bound-with-Fxk}
    \norm{F(\hx_k)} & \le \norm{F(x_k)} + \norm{F(\hx_k) - F(x_k)} \nonumber \\
    & \le \norm{F(x_k)} + A\norm{F(x_k)} = (1+A) \norm{F(x_k)}
\end{align}
and
\begin{align}
\label{eqn:Fhxk-lower-bound-with-Fxk}
    \norm{F(\hx_k)} & \ge \norm{F(x_k)} - \norm{F(\hx_k) - F(x_k)} \nonumber \\
    & \ge \norm{F(x_k)} - A\norm{F(x_k)} = (1-A) \norm{F(x_k)} .
\end{align}
Therefore, further lower bounding the left hand side of \eqref{eqn:monotone-bound-Fxk-Fhxk-combined} in terms of $\norm{F(\hx_k)}$ using \eqref{eqn:Fhxk-upper-bound-with-Fxk} we have
\begin{align*}
    \sqnorm{x_0 - x_*} & \ge \sum_{k=0}^K \frac{\gamma_k^2}{2(1+A)} \left( \sqnorm{F(\hx_k)} + \frac{1-A^2}{(1+A)^2} \sqnorm{F(\hx_k)} \right) \\
    & = \sum_{k=0}^K \frac{\gamma_k^2}{(1+A)^2} \sqnorm{F(\hx_k)} \\
    & = \frac{(K+1)}{(1+A)^2} \min_{k=0,\dots,K} \gamma_k^2 \sqnorm{F(\hx_k)} .
\end{align*}
This proves the first part. On the other hand, lower bounding \eqref{eqn:monotone-bound-Fxk-Fhxk-combined} in terms of $\norm{F(x_k)}$ using \eqref{eqn:Fhxk-lower-bound-with-Fxk} we have
\begin{align*}
    \sqnorm{x_0 - x_*} & \ge \sum_{k=0}^K \frac{\gamma_k^2}{2(1+A)} \left( (1-A)^2 \sqnorm{F(x_k)} + (1-A^2) \sqnorm{F(x_k)} \right) \\
    & = \sum_{k=0}^K \frac{1-A}{1+A} \gamma_k^2 \sqnorm{F(x_k)} \\
    & = \frac{(K+1) (1-A)}{1+A} \min_{k=0,\dots,K} \gamma_k^2 \sqnorm{F(x_k)} .
\end{align*}
This proves the second part of the monotone setting. \\
\newline

\textbf{$\bullet$ $F$ is strongly monotone.}
Strong monotonicity of $F$ implies
\begin{align*}
    \inprod{x_k - \hat{x}_k}{F(x_k) - F(\hat{x}_k)} \ge \mu \sqnorm{x_k - \hat{x}_k} .
\end{align*}
Substituting $x_k - \hat{x}_k = \gamma_k F(x_k)$ in the above we get
\begin{align*}
    \inprod{F(x_k)}{F(x_k) - F(\hat{x}_k)} \ge \gamma_k \mu \sqnorm{F(x_k)} .
\end{align*}
We apply the above inequality to \eqref{eqn:LS-proof-bound}:
\begin{align*}
    \sqnorm{x_{k+1} - x_*} & \le \left( 1 - \frac{2\alpha_k \mu}{1+\omega} \right) \sqnorm{x_k - x_*} - \alpha_k \gamma_k \left( \inprod{F(\hat{x}_k)}{F(x_k)} - \frac{2}{\omega} \gamma_k \mu \sqnorm{F(x_k)} \right) \\
    & \le \left( 1 - \frac{2\alpha_k \mu}{1+\omega} \right) \sqnorm{x_k - x_*} \\
    & - \alpha_k \gamma_k \left( \inprod{F(\hat{x}_k)}{F(x_k)} - \frac{2}{\omega} \inprod{F(x_k)}{F(x_k) - F(\hat{x}_k)} \right) \\
    & = \left( 1 - \frac{2\alpha_k \mu}{1+\omega} \right) \sqnorm{x_k - x_*} \\
    & - \alpha_k \gamma_k \left( \left(1 + \frac{2}{\omega}\right) \inprod{F(\hat{x}_k)}{F(x_k)} - \frac{2}{\omega} \sqnorm{F(x_k)} \right) .
\end{align*}
Now we use Lemma \ref{lemma:alpha-lower-bound} to obtain
\begin{align*}
    \sqnorm{x_{k+1} - x_*} & \le \left( 1 - \frac{2\alpha_k \mu}{1+\omega} \right) \sqnorm{x_k - x_*} \\
    & - \alpha_k \gamma_k \left( \frac{1}{2} \left(1 + \frac{2}{\omega}\right) \left( \sqnorm{F(\hat{x}_k)} + (1 - A^2) \sqnorm{F(x_k)} \right) - \frac{2}{\omega} \sqnorm{F(x_k)} \right) \\
    & \le \left( 1 - \frac{2\alpha_k \mu}{1+\omega} \right) \sqnorm{x_k - x_*} \\
    & - \frac{\alpha_k \gamma_k}{2} \left(1 + \frac{2}{\omega}\right) \left( (1-A)^2 \sqnorm{F(x_k)} + (1 - A^2) \sqnorm{F(x_k)} \right) \\
    & + \frac{2 \alpha_k \gamma_k}{\omega} \sqnorm{F(x_k)} \\
    & \le \left( 1 - \frac{2\alpha_k \mu}{1+\omega} \right) \sqnorm{x_k - x_*} - \alpha_k \gamma_k \left( 1 - A - \frac{2A}{\omega} \right) \sqnorm{F(x_k)} .
\end{align*}
where the second inequality uses \eqref{eqn:Fhxk-lower-bound-with-Fxk}.
Now taking $\omega = \frac{2A}{1-A}$ in the last inequality gives
\begin{align*}
    \sqnorm{x_{k+1} - x_*} & \le \left( 1 - \frac{2(1-A)}{1+A} \alpha_k \mu \right) \sqnorm{x_k - x_*} \\
    & \le \left( 1 - \frac{2(1-A)}{(1+A)^2} \gamma_k \mu \right) \sqnorm{x_k - x_*} .
\end{align*}
where for the second line we use \eqref{lemma:alpha-lower-bound}.
Applying this recursively, we obtain the desired result.
\end{proof}

\subsubsection{Proof of Theorem \ref{lemma:LS-step-size-lower-bound}}

\begin{theorem}
\vspace{1mm}
    Suppose $F$ is $L$-Lipschitz. Then the total number of while loop calls in \hyperref[alg:PolyakEG_linesearch]{\hyperref[alg:PolyakEG_linesearch]{\algname{PolyakEG-LS}}} is at most $\left\lfloor \frac{\log (L\gamma_{-1} / A)}{\log 1/\beta} \right\rfloor + 1$, and $$\gamma_k \ge \min \left\{ \frac{\beta A}{L} , \gamma_{-1} \right\} \text{ for all } k\ge 0.$$
\end{theorem}

\begin{proof}
First, observe that the sequence $\{\gamma_k\}$ is monotonically decreasing for $k=-1,0,1,\dots$ (as it is modified only by multiplying the factor $\beta < 1$ to the previous values).
Assuming that $\gamma_{k_0} \le \frac{A}{L}$ is reached at some $k_0$, the while loop call does not occur for $k\ge k_0$ because with $\hat{x}_k = x_k - \gamma_{k_0} F(x_k)$ we have
\begin{align}
\label{eqn:while-loop-escape}
    \norm{F(x_k) - F(\hat{x}_k)} \le L\norm{x_k - \hat{x}_k} \le \gamma_{k_0} L \norm{F(x_k)} \le A \norm{F(x_k)} .
\end{align}
Therefore, if $\gamma_{-1} \le \frac{A}{L}$ then line search iterations do not occur at all, and there is nothing to prove.

Now assume without loss of generality that $\gamma_{-1} > \frac{A}{L}$.
The maximum number of while loop calls for which $\gamma_k > \frac{A}{L}$ is retained is the largest natural number $N$ with
\[
    \beta^N \gamma_{-1} > \frac{A}{L} \iff N \log \beta + \log \frac{L \gamma_{-1}}{A} > 0 \iff N < \frac{\log(L \gamma_{-1} / A)}{\log 1/\beta} .
\]
Therefore, if the while loop is called $\left\lfloor \frac{\log (L\gamma_{-1} / A)}{\log 1/\beta} \right\rfloor + 1$ times then we have $\gamma_k \le \frac{1}{L}$ and it does not change afterwards.
Otherwise, the while loop is called less than or equal to $\left\lfloor \frac{\log (L\gamma_{-1} / A)}{\log 1/\beta} \right\rfloor$ times in total.
In any case the statement regarding the total number of while loops is proved.

Now we prove that if $\gamma_{-1} \ge \frac{\beta A}{L}$, then $\gamma_k \ge \frac{\beta A}{L}$ for all $k\ge 0$ (which is equivalent to showing that $\gamma_k \ge \min\left\{ \frac{\beta A}{L}, \gamma_{-1} \right\}$).
Suppose that the set $I = \left\{ k = 0,1,\dots \,\middle|\, \gamma_k < \frac{A}{L} \right\}$ is not empty. (If this is not the case, we already have $\gamma_k \ge \frac{A}{L} > \frac{\beta A}{L}$ for all $k$.)
Let $k_0$ be the smallest element of $I$.
Then by definition, $\gamma_{k_0 - 1} \ge \frac{A}{L} > \gamma_{k_0}$ and therefore, $N \ge 1$ while loops must have run in the iteration $k_0$ (of the for loop).
Assume to the contrary that $\gamma_{k_0} < \frac{\beta A}{L}$.
This implies that the most recent while loop was called even if $\frac{\gamma_{k_0}}{\beta} < \frac{A}{L}$ was used as the extrapolation step size, which contradicts \eqref{eqn:while-loop-escape}.
Therefore, we must have $\gamma_{k_0} \ge \frac{\beta A}{L}$, and for all $k\ge k_0$ we have $\gamma_k = \gamma_{k_0}$, because $\gamma_{k_0} < \frac{A}{L}$, and therefore $\gamma_k$ stays constant afterwards.
This shows that $\gamma_k \ge \frac{\beta A}{L}$ for all $k \ge 0$.
\end{proof}

\subsection{Stochastic Setting}
In this section, we present the missing proofs from Section \ref{sec:stochastic}. For this, we first state two lemmas, which will be used to analyze both \algname{PolyakSEG} and \algname{DecPolyakSEG} later in this section.


\begin{lemma}
\label{lemma:DecPolyakSEGLS-gamma-lower-bound}
    Suppose each $F_i$ is $L$-Lipschitz and $\gamma_{-1} \ge \frac{A}{L}$.
    During the runtime of \hyperref[alg:decPolyakSEG_linesearch]{\algname{DecPolyakSEG-LS}} with $c_{-1} = 1$ and $c_k = \sqrt{k+1}$ for $k\ge 0$, the total number of while loop calls is at most $\left\lfloor \frac{\log (L\gamma_{-1} / A)}{\log 1/\beta} \right\rfloor + 1$, and 
    \begin{eqnarray}\label{eq:dec_LS-step-size-lower-bound}
        \gamma_k \ge \frac{\beta A}{L \sqrt{k+1}}
    \end{eqnarray} 
    for all $k\ge 0$.
\end{lemma}

\begin{proof}
The sequence $\{\gamma_k\}$ is monotonically decreasing (as it is modified only by multiplying the factor $\beta < 1$ or $\frac{c_{k-1}}{c_{k}} < 1$ to the previous values).
As in the proof of Lemma \ref{lemma:LS-step-size-lower-bound}, once $\gamma_{k_0} \le \frac{A}{L}$ is reached at some $k_0$, the while loop call does not occur thereafter.
If $N = \left\lfloor \frac{\log (L\gamma_{-1} / A)}{\log 1/\beta} \right\rfloor + 1$ while loop calls occur, then we necessarily have $\gamma_k \le \beta^N \gamma_{-1} \le \frac{A}{L}$ and no more while loops are called.
Thus $\left\lfloor \frac{\log (L\gamma_{-1} / A)}{\log 1/\beta} \right\rfloor + 1$ is the upper bound on the total number of while loop calls.

Now, to prove that $\gamma_k \ge \frac{\beta A}{L\sqrt{k+1}}$, let $k_0$ be the smallest element of $$I = \left\{k=0,1,\dots,\middle| \gamma_k < \frac{A}{L} \right\}$$ (if $I$ is empty, there is nothing to prove).
Then $\gamma_{k_0 - 1} \ge \frac{A}{L} > \gamma_{k_0}$ by minimality of $k_0$ (even if $k_0 = 0$, we still have $\gamma_{-1} \ge \frac{A}{L}$ by assumption).
We have two possibilities:
\begin{enumerate}
    \item $\frac{c_{k_0-1}}{c_{k_0}} \gamma_{k_0 - 1} < \frac{A}{L}$: In this case, $\gamma_{k_0}$ has already crossed the $\frac{A}{L}$ border due to the multiplication of the factor $\frac{c_{k_0-1}}{c_{k_0}}$. 
    In this case the while loop call has not occurred at iteration $k_0$, and we have $\gamma_{k_0} = \frac{c_{k_0 - 1}}{c_{k_0}} \gamma_{k_0 - 1}$.
    If $k_0 = 0$ we have $\frac{c_{k_0 - 1}}{c_{k_0}} = 1 = \frac{1}{\sqrt{k_0 + 1}}$, and if $k_0 \ge 1$ we have $\frac{c_{k_0 - 1}}{c_{k_0}} = \sqrt{\frac{k_0}{k_0 + 1}} \ge \frac{1}{\sqrt{k_0 + 1}}$.
    In any case, $\gamma_{k_0} \ge \frac{1}{\sqrt{k_0 + 1}} \frac{A}{L} > \frac{1}{\sqrt{k_0 + 1}} \frac{\beta A}{L}$.
    
    \item $\frac{c_k}{c_{k+1}} \gamma_{k_0 - 1} \ge \frac{A}{L}$: In this case, $\gamma_{k_0}$ has crossed the $\frac{A}{L}$ border due to the while loop calls, and by the same argument as in Lemma \ref{lemma:LS-step-size-lower-bound}, we have $\gamma_{k_0} \ge \frac{\beta A}{L} > \frac{1}{\sqrt{k_0 + 1}} \frac{\beta A}{L}$.
\end{enumerate}

In any case we have $\gamma_{k_0} \ge \frac{\beta A}{L\sqrt{k_0 + 1}}$, and therefore, for every $k \ge k_0$ we have
\begin{align*}
    \gamma_k = \frac{c_{k-1}}{c_k} \dots \frac{c_{k_0}}{c_{k_0 + 1}} \gamma_{k_0} = \sqrt{\frac{k_0 + 1}{k + 1}} \gamma_{k_0} \ge \frac{1}{\sqrt{k+1}} \frac{\beta A}{L} .
\end{align*}

\end{proof}


\begin{lemma}
\label{lemma:PolyakSEG-LS-key-properties}
\vspace{1mm}
Suppose $F_i$ is $L$-Lipschitz.
Let $x_k$ be an iterate from either \hyperref[alg:PolyakSEG]{\algname{PolyakSEG}} or \hyperref[alg:decPolyakSEG]{\algname{DecPolyakSEG}}, and suppose that $\gamma_k$ satisfies the critical condition (Definition \ref{def:critical-condition-stochastic}). 
Then, for $k=0,1,\dots$,
\begin{enumerate}[label=(\alph*)]
    \item $(1-A) \norm{F_{\mathcal{S}_k}(x_k)} \le \norm{F_{\mathcal{S}_k}(\hx_k)} \le (1+A)\norm{F_{\mathcal{S}_k}(x_k)}$
    \item $\inprod{F_{\mathcal{S}_k}(\hx_k)}{F_{\mathcal{S}_k}(x_k)} 
    \ge \max \left\{ (1-A) \sqnorm{F_{\mathcal{S}_k}(x_k)} , \frac{\sqnorm{F_{\mathcal{S}_k}(\hx_k)}}{1+A} \right\}$
    \item $\frac{\gamma_k}{1+A} \le \alpha_k \le \frac{\gamma_k}{1-A}$
\end{enumerate}
holds almost surely.
\end{lemma}

\begin{proof}

By the critical condition \eqref{eqn:gamma-critical-condition-stochastic}, we obtain
\begin{align*}
    \norm{F_{\mathcal{S}_k}(\hx_k)} & \le \norm{F_{\mathcal{S}_k}(x_k)} + \norm{F_{\mathcal{S}_k}(\hx_k) - F_{\mathcal{S}_k}(x_k)} \\
    & \le \norm{F_{\mathcal{S}_k}(x_k)} + A\norm{F_{\mathcal{S}_k}(x_k)} \\
    & = (1+A) \norm{F_{\mathcal{S}_k}(x_k)}
\end{align*}
and
\begin{align*}
    \norm{F_{\mathcal{S}_k}(\hx_k)} & \ge \norm{F_{\mathcal{S}_k}(x_k)} - \norm{F_{\mathcal{S}_k}(\hx_k) - F_{\mathcal{S}_k}(x_k)} \\
    & \ge \norm{F_{\mathcal{S}_k}(x_k)} - A\norm{F_{\mathcal{S}_k}(x_k)} \\
    & \ge (1-A) \norm{F_{\mathcal{S}_k}(x_k)}
\end{align*}
which proves (a).
Next, observe that
\begin{align*}
    \inprod{F_{\mathcal{S}_k}(\hx_k)}{F_{\mathcal{S}_k}(x_k)} & = \frac{1}{2} \left( \sqnorm{F_{\mathcal{S}_k}(\hx_k)} - \sqnorm{F_{\mathcal{S}_k}(\hx_k) - F_{\mathcal{S}_k}(x_k)} + \sqnorm{F_{\mathcal{S}_k}(x_k)} \right) \\
    & \ge \frac{1}{2} \left( \sqnorm{F_{\mathcal{S}_k}(\hx_k)} - A^2 \sqnorm{F_{\mathcal{S}_k}(x_k)} +\sqnorm{F_{\mathcal{S}_k}(x_k)} \right) \\
    & = \frac{1}{2} \left( \sqnorm{F_{\mathcal{S}_k}(\hx_k)} + (1 - A^2) \sqnorm{F_{\mathcal{S}_k}(x_k)} \right) 
\end{align*}
where the inequality uses \eqref{eqn:gamma-critical-condition-stochastic}.
Applying the first inequality of (a) to the last expression, we obtain
\begin{align*}
    \inprod{F_{\mathcal{S}_k}(\hx_k)}{F_{\mathcal{S}_k}(x_k)} & \ge \frac{1}{2} \left( (1-A)^2 \sqnorm{F_{\mathcal{S}_k}(x_k)} + (1 - A^2) \sqnorm{F_{\mathcal{S}_k}(x_k)} \right) \\
    & = (1-A) \sqnorm{F_{\mathcal{S}_k}(x_k)} ,
\end{align*}
while applying the second inequality of (a) gives
\begin{align*}
    \inprod{F_{\mathcal{S}_k}(\hx_k)}{F_{\mathcal{S}_k}(x_k)} & \ge \frac{1}{2} \left( \sqnorm{F_{\mathcal{S}_k}(\hx_k)} + (1 - A^2) \frac{1}{(1+A)^2} \sqnorm{F_{\mathcal{S}_k}(\hx_k)} \right) \\
    & = \frac{\sqnorm{F_{\mathcal{S}_k}(\hx_k)}}{1+A} .
\end{align*}
This proves (b).
Finally, for \algname{PolyakSEG} we immediately obtain
\begin{align*}
    \alpha_k = \frac{\inprod{F_{\mathcal{S}_k}(\hx_k)}{x_k - \hx_k}}{\sqnorm{F_{\mathcal{S}_k}(\hx_k)}} = \frac{\gamma_k \inprod{F_{\mathcal{S}_k}(\hx_k)}{F_{\mathcal{S}_k}(x_k)}}{\sqnorm{F_{\mathcal{S}_k}(\hx_k)}} \ge \frac{\gamma_k}{1+A} .
\end{align*}
For the case of \algname{DecPolyakSEG}, we have $\gamma_k \le \gamma_{k-1}$ and $\alpha_k \le \alpha_{k-1}$ for all $k=0,1,\dots$ by construction.
As $\alpha_{-1} = \infty$, 
\begin{align*}
    \alpha_0 = \min \left\{ \frac{\inprod{F(\hx_0, \xi_0)}{x_0 - \hx_0}}{\sqnorm{F(\hx_0, \xi_0)}} , \alpha_{-1} \right\} = \frac{\inprod{F(\hx_0, \xi_0)}{x_0 - \hx_0}}{\sqnorm{F(\hx_0, \xi_0)}} \ge \frac{\gamma_0}{1+A} .
\end{align*}
Now we use induction on $k=1,2,\dots$: assuming that $\alpha_{k-1} \ge \frac{\gamma_{k-1}}{1+A}$, we have $\alpha_{k-1} \ge \frac{\gamma_{k}}{1+A}$ (since $\gamma_{k} \le \gamma_{k-1}$). 
We further have $\frac{\inprod{F_{\mathcal{S}_k}(\hx_k)}{x_k - \hx_k}}{\sqnorm{F_{\mathcal{S}_k}(\hx_k)}} \ge \frac{\gamma_k}{1+A}$ by (b), which implies
\begin{align*}
    \alpha_k = \min \left\{ \alpha_{k-1}, \frac{\inprod{F_{\mathcal{S}_k}(\hx_k)}{x_k - \hx_k}}{\sqnorm{F_{\mathcal{S}_k}(\hx_k)}} \right\} \ge \frac{\gamma_k}{1+A} ,
\end{align*}
completing the induction. This proves the first inequality in (c).
Finally, for the second inequality in (c), observe that for both \algname{PolyakSEG} and \algname{DecPolyakSEG},
\begin{align*}
    \alpha_k & \le \frac{\inprod{F_{\mathcal{S}_k}(\hx_k)}{x_k - \hx_k}}{\sqnorm{F_{\mathcal{S}_k}(\hx_k)}} \\
    & = \frac{\gamma_k \inprod{F_{\mathcal{S}_k}(\hx_k)}{F(x_k,\xi_k)}}{\sqnorm{F_{\mathcal{S}_k}(\hx_k)}} \\
    & \le \frac{\gamma_k}{2} \frac{\frac{1}{1-A} \sqnorm{F_{\mathcal{S}_k}(\hx_k)} + (1-A) \sqnorm{F(x_k,\xi_k)}}{\sqnorm{F_{\mathcal{S}_k}(\hx_k)}} \\
    & \le \frac{\gamma_k}{2} \frac{\frac{1}{1-A} \sqnorm{F_{\mathcal{S}_k}(\hx_k)} + \frac{1}{1-A} \sqnorm{F_{\mathcal{S}_k}(\hx_k)}}{\sqnorm{F_{\mathcal{S}_k}(\hx_k)}} \\
    & = \frac{\gamma_k}{1-A}
\end{align*}
where the third line uses Lemma \ref{lemma:CS2} with $\lambda = \frac{1}{1-A}$, and the fourth line uses (a).
\end{proof}

\subsection{Analysis of \hyperref[alg:PolyakSEG]{\algname{PolyakSEG}}}

Given the interpolation condition, the stochastic analysis is essentially the repetition of the proof of Theorem \ref{theorem:deterministic-master-theorem} except for minimal additional arguments regarding the expectations.
For the sake of completeness, however, we detail the crucial steps as below.

\subsubsection{Proof of Theorem \ref{theorem:PolyakSEG-convergence}}
\begin{theorem}
\vspace{1mm}
    Let $F_i$ are $L$-Lipschitz, (strongly) monotone
    and there exists an interpolating solution $x_*$ for which $F_i(x_*) = 0$ for all $i \in [n]$ almost surely.
    Then, provided that $\gamma_k$ satisfy the critical condition ~\eqref{eqn:gamma-critical-condition-stochastic} with $0 < A < 1$, \algname{PolyakSEG} satisfies 
    \newline
    $\bullet$ When $F_i$ are monotone~\eqref{eq:monotone}, $\min_{0 \leq k \leq K} \Exp{\gamma_k^2 \|F(x_k)\|^2} \leq \frac{(1+A) \|x_0 - x_*\|^2}{(1-A)(K+1)}.$ \\
    $\bullet$ When  $F_i$ are strongly monotone~\eqref{eq:strong_monotone}, $$\Exp{\|x_{k+1} - x_*\|^2} \le \Exp{ \left(1 - \frac{2 (1-A) \gamma_k \mu}{(1+A)^2} \right)\sqnorm{x_k - x_*}}.$$
\end{theorem}

\begin{proof}
\textbf{$\bullet$ $F_i$ are monotone.}
\begin{align}
    \sqnorm{x_{k+1} - x_*} & = \sqnorm{x_k - \alpha_k F_{\mathcal{S}_k}(\hx_k) - x_*} \nonumber \\
    & = \sqnorm{x_k - x_*} - 2\alpha_k \inprod{F_{\mathcal{S}_k}(\hx_k)}{x_k - x_*} + \alpha_k^2 \sqnorm{F_{\mathcal{S}_k}(\hx_k)} \nonumber \\
    & = \sqnorm{x_k - x_*} - 2\alpha_k \inprod{F_{\mathcal{S}_k}(\hx_k)}{x_k - \hat{x}_k} - 2\alpha_k \inprod{F_{\mathcal{S}_k}(\hx_k)}{\hat{x}_k - x_*} \nonumber \\
    & + \alpha_k^2 \sqnorm{F_{\mathcal{S}_k}(\hx_k)} \nonumber \\
    & = \sqnorm{x_k - x_*} - 2\alpha_k \inprod{F_{\mathcal{S}_k}(\hx_k)}{x_k - \hat{x}_k} \nonumber \\
    & - 2\alpha_k \inprod{F_{\mathcal{S}_k}(\hx_k) - F_{\mathcal{S}_k}(x_*)}{\hx_k - x_*}  + \alpha_k^2 \sqnorm{F_{\mathcal{S}_k}(\hx_k)} \nonumber \\
    & \le \sqnorm{x_k - x_*} - 2\alpha_k \inprod{F_{\mathcal{S}_k}(\hx_k)}{x_k - \hat{x}_k} + \alpha_k^2 \sqnorm{F_{\mathcal{S}_k}(\hx_k)} \nonumber \\
    & = \sqnorm{x_k - x_*} - \alpha_k \inprod{F_{\mathcal{S}_k}(\hx_k)}{x_k - \hat{x}_k} \label{eqn:PolyakSEG-LS-monotone-first-bound}
\end{align}
where the second last inequality follows from monotonicity of $F_{\mathcal{S}_k}$ and the last equality from the definition of $\alpha_k$.
Now by Lemma \ref{lemma:PolyakSEG-LS-key-properties} we have
\begin{align*}
    \alpha_k \inprod{F_{\mathcal{S}_k}(\hx_k)}{x_k - \hat{x}_k} & = \alpha_k \gamma_k \inprod{F_{\mathcal{S}_k}(\hx_k)}{F_{\mathcal{S}_k}(x_k)} \\
    & \ge \alpha_k \gamma_k (1-A) \sqnorm{F_{\mathcal{S}_k}(x_k)} \\
    & \ge \frac{(1-A) \gamma_k^2}{1+A} \sqnorm{F_{\mathcal{S}_k}(x_k)} .
\end{align*}
Plugging this back into \eqref{eqn:PolyakSEG-LS-monotone-first-bound} and taking expectation (conditioned on $x_k$), we obtain
\begin{align*}
    \mathbb{E}_k \left[ \sqnorm{x_{k+1} - x_*} \right] & \le \sqnorm{x_k - x_*} - \frac{1-A}{1+A} \gamma_k^2 \Expk{\sqnorm{F_{\mathcal{S}_k}(x_k)}} \\
    & \le \sqnorm{x_k - x_*} - \frac{1-A}{1+A} \gamma_k^2 \sqnorm{\Expk{F_{\mathcal{S}_k}(x_k)}} \\
    & \le \sqnorm{x_k - x_*} - \frac{1-A}{1+A} \gamma_k^2 \sqnorm{F(x_k)}
\end{align*}
Here, the third line follows from Jensen's Inequality. Now we rearrange the terms and take total expectation to get
\begin{eqnarray*}
    \frac{1-A}{1+A} \, \Exp{\gamma_k^2 \|F(x_k)\|^2} & \leq & \Exp{\| x_{k} - x_* \|^2} - \Exp{\| x_{k+1}- x_*\|^2}
\end{eqnarray*}
Summing this up for $k=0,\dots,K-1$ and lower bounding the left hand side with the minimum over $K+1$ iterates, we conclude
\begin{eqnarray*}
    \min_{k = 0,\dots,K} \Exp{\gamma_k^2 \|F(x_k)\|^2} & \leq & \frac{(1+A) \|x_0 - x_*\|^2}{(1 - A) (K+1)}.
\end{eqnarray*}

\textbf{$\bullet$ $F_i$ are strongly monotone}
\begin{align}
    \sqnorm{x_{k+1} - x_*} & = \sqnorm{x_k - \alpha_k F_{\mathcal{S}_k}(\hx_k) - x_*} \nonumber \\
    & = \sqnorm{x_k - x_*} - 2\alpha_k \inprod{F_{\mathcal{S}_k}(\hx_k)}{x_k - x_*} + \alpha_k^2 \sqnorm{F_{\mathcal{S}_k}(\hx_k)} \nonumber \\
    & = \sqnorm{x_k - x_*} - 2\alpha_k \inprod{F_{\mathcal{S}_k}(\hx_k)}{x_k - \hat{x}_k} - 2\alpha_k \inprod{F_{\mathcal{S}_k}(\hx_k)}{\hat{x}_k - x_*} \nonumber \\
    & + \alpha_k^2 \sqnorm{F_{\mathcal{S}_k}(\hx_k)} \nonumber \\
    & = \sqnorm{x_k - x_*} - 2\alpha_k \inprod{F_{\mathcal{S}_k}(\hx_k)}{x_k - \hat{x}_k}  \nonumber \\
    & - 2\alpha_k \inprod{F_{\mathcal{S}_k}(\hx_k) - F_{\mathcal{S}_k}(x_*)}{\hx_k - x_*} + \alpha_k^2 \sqnorm{F_{\mathcal{S}_k}(\hx_k)} \nonumber \\
    & \le \sqnorm{x_k - x_*} - 2\alpha_k \inprod{F_{\mathcal{S}_k}(\hx_k)}{x_k - \hat{x}_k} \nonumber \\ 
    & - 2\alpha_k \mu \sqnorm{\hx_k - x_*} + \alpha_k^2 \sqnorm{F_{\mathcal{S}_k}(\hx_k)} \nonumber \\
    & = \sqnorm{x_k - x_*} - 2\alpha_k \mu \sqnorm{\hx_k - x_*} - \alpha_k \inprod{F_{\mathcal{S}_k}(\hx_k)}{x_k - \hat{x}_k} . \label{eqn:PolyakSEG-LS-strongly-monotone-first-bound}
\end{align}
As in the monotone case we can bound (using Lemma \ref{lemma:PolyakSEG-LS-key-properties}):
\begin{align}
\label{eqn:PolyakSEG-LS-strongly-monotone-second-bound}
    \alpha_k \inprod{F_{\mathcal{S}_k}(\hx_k)}{x_k - \hat{x}_k} \ge \alpha_k \gamma_k (1-A) \sqnorm{F_{\mathcal{S}_k}(x_k)}
\end{align}
and by Lemma \ref{lemma:variant-young-ineq} with $a=\hx_k - x_*$, $b=x_k - \hx_k$ and $\omega = \frac{2A}{1-A}$,
\begin{align}
    2\alpha_k \mu \sqnorm{\hx_k - x_*} & \ge 2\alpha_k \mu \left(\frac{1-A}{1+A}\right) \sqnorm{x_k - x_*} - \alpha_k \mu \left(\frac{1-A}{A}\right) \sqnorm{x_k - \hx_k} \nonumber \\
    & \ge \frac{2(1-A) \gamma_k \mu}{(1+A)^2} \sqnorm{x_k - x_*} - \left(\frac{1-A}{A}\right) \alpha_k \gamma_k^2 \mu \sqnorm{F_{\mathcal{S}_k}(x_k)} \nonumber \\
    & \ge \frac{2(1-A) \gamma_k \mu}{(1+A)^2} \sqnorm{x_k - x_*} - (1-A) \alpha_k \gamma_k \sqnorm{F_{\mathcal{S}_k}(x_k)} \label{eqn:PolyakSEG-LS-strongly-monotone-third-bound}
\end{align}
where the last inequality follows by combining strong monotonicity of $F_{\mathcal{S}_k}$ with Lemma \ref{lemma:PolyakSEG-LS-key-properties}:
\begin{align*}
    & \inprod{x_k - \hx_k}{F_{\mathcal{S}_k}(x_k) - F_{\mathcal{S}_k}(\hx_k)} \ge \mu \sqnorm{x_k - \hx_k} \\
    & \implies \gamma_k \mu \sqnorm{F_{\mathcal{S}_k}(x_k)} \le \inprod{F_{\mathcal{S}_k}(x_k)}{F_{\mathcal{S}_k}(x_k) - F_{\mathcal{S}_k}(\hx_k)} \le A \sqnorm{F_{\mathcal{S}_k}(x_k)} .
\end{align*}
Plugging \eqref{eqn:PolyakSEG-LS-strongly-monotone-second-bound} and \eqref{eqn:PolyakSEG-LS-strongly-monotone-third-bound} into \eqref{eqn:PolyakSEG-LS-strongly-monotone-first-bound} and taking (total) expectations, we obtain
\begin{align*}
    \Exp{\sqnorm{x_{k+1} - x_*}} \le \Exp{ \left(1 - \frac{2 (1-A) \gamma_k \mu}{(1+A)^2} \right)\sqnorm{x_k - x_*}} .
\end{align*}
\end{proof}

\paragraph{Discussion on Theorem \ref{theorem:PolyakSEG-convergence}}
If one can guarantee $\gamma_k \ge \underline{\gamma}$ for all $k=0,1,\dots$ for some fixed $\underline{\gamma} > 0$ (which is the case if $\gamma_k$ is either constant or found using the line search scheme), one can replace $\gamma_k$'s with $\underline{\gamma}$ and pull them out of the expectations.
This will imply
\begin{align*}
    \min_{k=0,\dots,K} \Exp{\|F(x_k)\|^2} \le \frac{(1+A) \sqnorm{x_0 - x_*}}{\underline{\gamma} (1-A) (K+1)}
\end{align*}
in the $\mu = 0$ case, and
\begin{align*}
    \Exp{\|x_{k+1} - x_*\|^2} \le \left( 1 - \frac{2(1-A) \underline{\gamma} \mu}{(1+A)^2} \right)^{k+1} \sqnorm{x_0 - x_*}
\end{align*}
in the $\mu > 0$ case (by unrolling the recursion).

\subsection{Analysis of \hyperref[alg:decPolyakSEG]{\algname{DecPolyakSEG}}}
In this subsection, we will present the results related to \algname{DecPolyakSEG}.
Before diving into the proof of Theorem \ref{theorem:DecPolyakSEGLS_monotone}, we will state following lemma which will be used for its analysis.

\begin{lemma}
\label{lemma:DecPolyakSEG-LS-gamma-upper-bound}
    For any positive sequence $\{c_k\}_{k=0}^{\infty}$, the step sizes of \hyperref[alg:decPolyakSEG]{\algname{DecPolyakSEG}} satisfy
    \begin{align}
    \label{eq:dec_gamma}
        \gamma_k \le \frac{c_0}{c_{k}} \gamma_{0}.
    \end{align}
\end{lemma}

\begin{proof}
Applying the inequality $c_k \gamma_k \le c_{k-1} \gamma_{k-1}$ recursively, we obtain $c_k \gamma_k \le c_0 \gamma_0$, or equivalently $\gamma_k \le \frac{c_0}{c_k} \gamma_k$.
\end{proof}

\subsubsection{Proof of Theorem \ref{theorem:DecPolyakSEGLS_monotone}}\label{sec:DecPolyakSEGLS_proof}

\begin{theorem}
    Let $F_i$ are $L$-Lipschitz, monotone, and Assumption~\ref{assumption:bounded-iterates} holds. Then $\overline{x}_K = \frac{1}{K+1} \sum_{k=0}^K x_k$ generated by \algname{DecPolyakEG} with $c_k = \sqrt{k+1}$ satisfies
    \begin{align*}
    \textstyle
        \Exp{\inprod{F(u)}{\overline{x}_K - u}}
        \le \frac{1}{K+1} \Exp{\frac{D^2}{\gamma_{K+1}}} + \frac{\sigma^2 \gamma_0}{(1-A)(K+1)} \sum_{k=0}^K \frac{1}{\sqrt{k+1}} 
    \end{align*}
    \vspace{-1.5mm}
    for any $u \in \mathcal{C}$. Here, $\sigma^2 \eqdef 2L^2 \max_{u_1, u_2 \in \mathcal{C}} \|u_1 - u_2\|^2 + 2 \Exp{\|F_{\mathcal{S}}(x_*)\|^2}$.
\end{theorem}

\begin{proof}
    For any $u \in \reals^d$, \hyperref[alg:decPolyakSEG]{\algname{DecPolyakSEG}} satisfies
\begin{eqnarray*}
    \|x_{k+1} - u\|^2 & = & \| x_k - \alpha_k F_{\mathcal{S}_k}(\hx_k) - u\|^2 \\
    & = & \| x_k - u\|^2 + \alpha_k^2 \| F_{\mathcal{S}_k}(\hx_k)\|^2 - 2 \alpha_k \la F_{\mathcal{S}_k}(\hx_k), x_k - u \ra \\
    & \le & \| x_k - u\|^2 + \alpha_k \la F_{\mathcal{S}_k}(\hx_k), x_k - \hx_k \ra - 2 \alpha_k \la F_{\mathcal{S}_k}(\hx_k), x_k - u \ra \\
    & = & \| x_k - u\|^2 + \alpha_k \la F_{\mathcal{S}_k}(\hx_k), x_k - \hx_k \ra - 2 \alpha_k \la F_{\mathcal{S}_k}(\hx_k), x_k - \hx_k + \hx_k - u \ra \\
    & = & \| x_k - u\|^2 + \alpha_k \la F_{\mathcal{S}_k}(\hx_k), x_k - \hx_k \ra - 2 \alpha_k \la F_{\mathcal{S}_k}(\hx_k), x_k - \hx_k \ra \\
    && - 2 \alpha_k \la F_{\mathcal{S}_k}(\hx_k), \hx_k - u \ra \\
    & = & \| x_k - u\|^2 - \alpha_k \la F_{\mathcal{S}_k}(\hx_k), x_k - \hx_k \ra - 2 \alpha_k \la F_{\mathcal{S}_k}(\hx_k), \hx_k - u \ra \\
    & = & \| x_k - u\|^2 - \alpha_k \la F_{\mathcal{S}_k}(\hx_k), x_k - \hx_k \ra - 2 \alpha_k \la F_{\mathcal{S}_k}(\hx_k) - F_{\mathcal{S}_k}(u), \hx_k - u \ra \\
    && - 2 \alpha_k \la F_{\mathcal{S}_k}(u), \hx_k - u \ra \\
    & \leq & \| x_k - u\|^2 - \alpha_k \la F_{\mathcal{S}_k}(\hx_k), x_k - \hx_k \ra - 2 \alpha_k \la F_{\mathcal{S}_k}(u), \hx_k - u \ra .
\end{eqnarray*}
The third line holds because $\alpha_k \le \frac{\inprod{F(\hx_k, \xi_k}{x_k - \hx_k}}{\sqnorm{F_{\mathcal{S}_k}(\hx_k)}}$, and the last line follows from monotonicity of $F(\cdot, \xi)$. 
Now we use Lemma \ref{lemma:PolyakSEG-LS-key-properties}(b) to bound the second term and plug $\hx_k = x_k - \gamma_k F_{\mathcal{S}_k}(x_k)$ into the third term to proceed as:
\begin{eqnarray*}
    \|x_{k+1} - u\|^2 & \leq & \| x_k - u\|^2 - \alpha_k \gamma_k (1-A) \sqnorm{F_{\mathcal{S}_k}(x_k)} - 2 \alpha_k \la F_{\mathcal{S}_k}(u), x_k - u \ra \\
    && + 2 \alpha_k \gamma_k \la F_{\mathcal{S}_k}(u), F_{\mathcal{S}_k}(x_k)\ra \\
    & \le & \| x_k - u \|^2 - 2 \alpha_k \la F_{\mathcal{S}_k}(u), x_k - u \ra - \alpha_k \gamma_k (1-A) \sqnorm{F_{\mathcal{S}_k}(x_k)} \\
    &&  + \alpha_k \gamma_k \left( \frac{1}{1-A} \sqnorm{F_{\mathcal{S}_k}(u)} + (1-A) \sqnorm{F_{\mathcal{S}_k}(x_k)} \right) \\
    & = & \| x_k - u \|^2 - 2 \alpha_k \la F_{\mathcal{S}_k}(u), x_k - u \ra + \frac{\alpha_k \gamma_k}{1-A} \sqnorm{F_{\mathcal{S}_k}(u)}
\end{eqnarray*}
where second inequality uses Lemma \ref{lemma:CS2} with $\lambda = \frac{1}{1-A}$.

Now we divide the both sides by $\alpha_k$ to obtain
\begin{eqnarray}\label{eq:divide_by_alpha_k}
   \frac{\|x_{k+1} - u\|^2}{\alpha_k} & \leq & \frac{\| x_k - u\|^2}{\alpha_k} - 2 \la F_{\mathcal{S}_k}(u), x_k - u \ra + \frac{\gamma_k}{1-A} \|F_{\mathcal{S}_k}(u)\|^2.
\end{eqnarray}
Now let $u \in \mathcal{C}$.
By Lemma \ref{lemma:DecPolyakSEG-LS-gamma-upper-bound} with $c_k = \sqrt{k+1}$, we have $\gamma_k \leq \frac{\gamma_0}{\sqrt{k+1}}$. 
Applying this bound to \eqref{eq:divide_by_alpha_k} and using the bounded iterates assumption $\sqnorm{x_{k+1} - u} \le D^2$ with the fact that $\alpha_{k+1} \le \alpha_k$, we obtain
\begin{eqnarray}\label{eq:dec_alpha_k}
      2 \la F_{\mathcal{S}_k}(u), x_k - u \ra & \leq & \frac{\| x_k - u\|^2}{\alpha_k} - \frac{\|x_{k+1} - u\|^2}{\alpha_k} + \frac{\gamma_0}{(1-A) \sqrt{k+1}} \sqnorm{F_{\mathcal{S}_k}(u)} \notag \\
      & \leq & \frac{\| x_k - u\|^2}{\alpha_k} - \frac{\|x_{k+1} - u\|^2}{\alpha_{k+1}} + \left( \frac{1}{\alpha_{k+1}} - \frac{1}{\alpha_k} \right) D^2  \notag \\
      && + \frac{\gamma_0}{(1-A) \sqrt{k+1}} \sqnorm{F_{\mathcal{S}_k}(u)} .
\end{eqnarray}
Note that $\|F_{\mathcal{S}_k}(u)\|^2$ can be bounded as following:
\begin{align*}
    \|F_{\mathcal{S}_k}(u)\|^2 & \leq 2\sqnorm{F_{\mathcal{S}_k}(x_*)} + 2\sqnorm{F_{\mathcal{S}_k}(x_*) - F_{\mathcal{S}_k}(u)} \\
    & \le 2 \|F_{\mathcal{S}_k}(x_*)\|^2 + 2L^2 \|x_* - u\|^2 \\
    & \le 2 \|F_{\mathcal{S}_k}(x_*)\|^2 + 2L^2 \max_{u_1, u_2 \in \mathcal{C}}\|u_1 - u_2\|^2 \\
    & = 2 \|F_{\mathcal{S}_k}(x_*)\|^2 + 2L^2 \Omega^2_{\mathcal{C}} 
\end{align*}
where $\Omega_{\mathcal{C}}^2 \eqdef \max_{u_1, u_2} \|u_1 - u_2\|^2$. Now, we sum up \eqref{eq:dec_alpha_k} for $k=0,\dots,K$ and apply the bound on $\|F_{\mathcal{S}_k}(u)\|^2$ to get
\begin{align}
    & 2\sum_{k=0}^K \inprod{F_{\mathcal{S}_k}(u)}{x_k - u} \\
    & \le \frac{\sqnorm{x_0 - u}}{\alpha_0} - \frac{\sqnorm{x_{K+1} - u}}{\alpha_{K+1}} + \sum_{k=0}^K  \left( \frac{1}{\alpha_{k+1}} - \frac{1}{\alpha_k} \right) D^2 + \sum_{k=0}^K \frac{\gamma_0}{(1-A) \sqrt{k+1}} \sqnorm{F_{\mathcal{S}_k}(u)} \nonumber \\
    & \le \frac{D^2}{\alpha_0} + \sum_{k=0}^K  \left( \frac{1}{\alpha_{k+1}} - \frac{1}{\alpha_k} \right) D^2 + \sum_{k=0}^K \frac{\gamma_0}{(1-A) \sqrt{k+1}} \left( 2 \|F_{\mathcal{S}_k}(x_*)\|^2 + 2L^2 \Omega^2_{\mathcal{C}} \right) \nonumber \\
    & = \frac{D^2}{\alpha_{K+1}} + \sum_{k=0}^K \frac{\gamma_0}{(1-A) \sqrt{k+1}} \left( 2 \|F_{\mathcal{S}_k}(x_*)\|^2 + 2L^2 \Omega^2_{\mathcal{C}} \right) \nonumber \\
    & \le \frac{2D^2}{\gamma_{K+1}} + \sum_{k=0}^K \frac{\gamma_0}{(1-A) \sqrt{k+1}} \left( 2 \|F_{\mathcal{S}_k}(x_*)\|^2 + 2L^2 \Omega^2_{\mathcal{C}} \right)
    \label{eqn:DecPolyakSEG-proof-key-eq-before-exp}
\end{align}
where 
the last inequality uses $\alpha_{K+1} \ge \frac{\gamma_{K+1}}{1+A} \ge \frac{\gamma_{K+1}}{2}$ (Lemma \ref{lemma:PolyakSEG-LS-key-properties}(c)). 
Finally we take the total expectation and divide both sides by $2(K+1)$ to obtain
\begin{align*}
    \Exp{\inprod{F(u)}{\overline{x}_K - u}}
    \le \frac{1}{K+1} \Exp{\frac{D^2}{\gamma_{K+1}}} + \frac{\sigma^2 \gamma_0}{(1-A)(K+1)} \sum_{k=0}^K \frac{1}{\sqrt{k+1}} ,
\end{align*}
where $\sigma^2 \eqdef 2L^2 \max_{u_1, u_2 \in \mathcal{C}} \|u_1 - u_2\|^2 + 2 \Exp{\|F_{\mathcal{S}}(x_*)\|^2}$. 
This completes the proof.

\end{proof}

\newpage
\section{More on Numerical Experiments}\label{sec:more_numerical_exp}
In this section, we provide more details on the experiments conducted in the main paper. In the first subsection, we provide the details related to Figure \ref{fig:2Dcontour}, while in the second subsection, we add more details on the setup of Section \ref{sec:numerical_experiment}.

\subsection{Details of Figure \ref{fig:2Dcontour}}\label{sec:compare_polyakeg_eg}
In Figure \ref{fig:2Dcontour}, we illustrate the advantages of using \hyperref[alg:PolyakEG]{\algname{PolyakEG}} over the standard \hyperref[eq:EG]{\algname{EG}} method on a two-dimensional min-max problem. Specifically, we consider the problem:
\begin{eqnarray*}
    \min_{y \in \R} \max_{z \in \R} f(y, z) \eqdef \frac{1}{2} y^2 + \frac{5}{2} yz - 25z^2.
\end{eqnarray*}
which has an equilibrium at the point $(y,z)=(0,0)$, as determined by the conditions $\nabla_y f(0, 0) = \nabla_z f(0, 0) = 0$. Figure \ref{fig:2Dcontour} provides a visual representation of the function $f(y,z)$ in a two-dimensional space, where the colour gradient indicates the value of $f(y,z)$ at each point $(y,z)$. The contour plot in Figure \ref{fig:2Dcontour} highlights the behaviour of the function and serves as the background for comparing the performance of \hyperref[alg:PolyakEG]{\algname{PolyakEG}} and \hyperref[eq:EG]{\algname{EG}}.
To solve the min-max problem, we apply both \hyperref[alg:PolyakEG]{\algname{PolyakEG}} and \hyperref[eq:EG]{\algname{EG}} methods, using a step size $\gamma = \nicefrac{1}{L}$, where $L$ is the \textit{Lipschitz} constant of the gradients $\nabla_y f$ and $\nabla_z f$. In Figure \ref{fig:2Dcontour}, we plot the trajectories of these algorithms over $10$ iterations. The results clearly show that \hyperref[alg:PolyakEG]{\algname{PolyakEG}} converges to the optimal point $(0,0)$ faster than \hyperref[eq:EG]{\algname{EG}}. This observation underscores the efficiency of \hyperref[alg:PolyakEG]{\algname{PolyakEG}} and motivates a deeper investigation into its theoretical properties.

\begin{figure*}[hbt!]
\centering
    \begin{subfigure}[b]{0.3\textwidth}
        \centering
        \includegraphics[width=\linewidth]{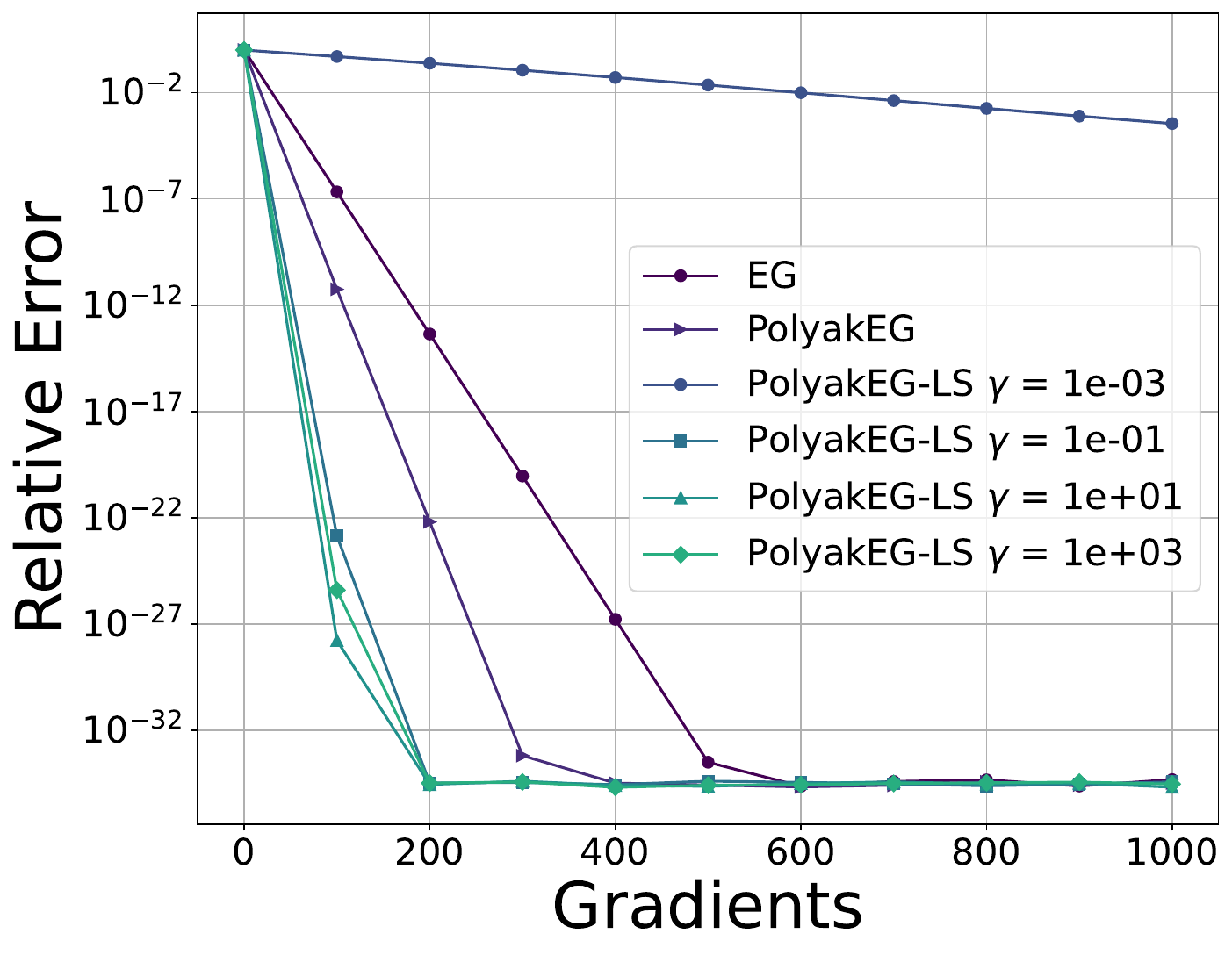}
        \caption{$\nicefrac{L}{\mu} \sim 10$}
        \label{fig:PolyakEG_theoretical_v2_gradients}
    \end{subfigure}
    
    \begin{subfigure}[b]{0.3\textwidth}
        \centering
        \includegraphics[width=\linewidth]{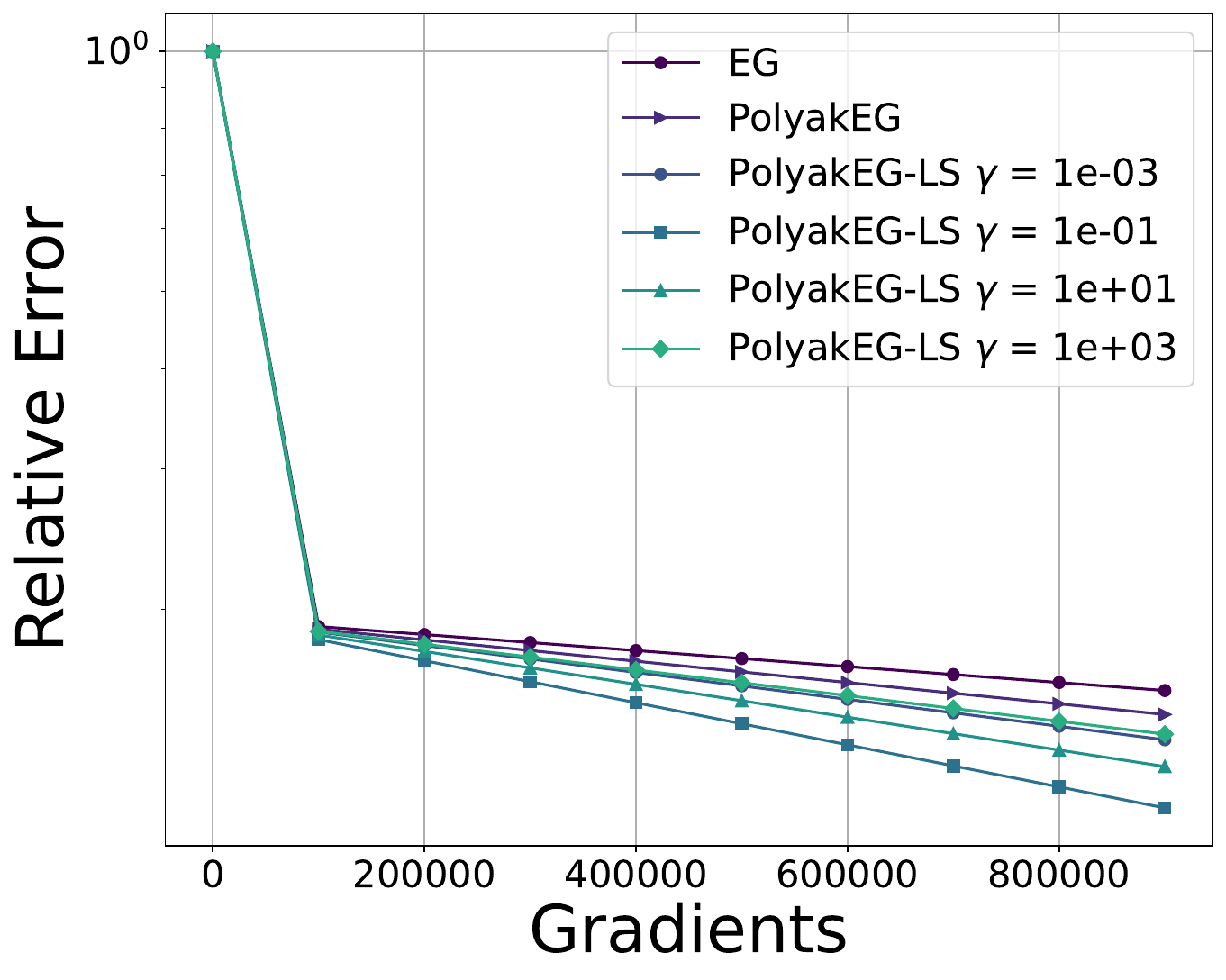}
        \caption{$\nicefrac{L}{\mu} \sim 10^8$}
        \label{fig:PolyakEG_theoretical_v3_gradients}
    \end{subfigure}

    \caption{\small In this experiment, we compare the performance of \hyperref[alg:PolyakEG]{\algname{PolyakEG}}, \hyperref[eq:EG]{\algname{EG}} (with their theoretical step sizes) and \hyperref[alg:PolyakEG_linesearch]{\algname{PolyakEG-LS}} for solving strongly monotone problems with different condition numbers $\nicefrac{L}{\mu}$.
    On the $x$-axis we have the number of gradient evaluations while on $y$-axis we plot the relative error $\nicefrac{\|x_k - x_*\|^2}{\|x_0 - x_*\|^2}$.}
\end{figure*}

\subsection{Details of Numerical Experiments}
In this section, we provide more details and explanation related to the numerical experiments in the main paper.

\paragraph{Comparison of \hyperref[eq:EG]{\algname{EG}}, \hyperref[alg:PolyakEG]{\algname{PolyakEG}} and \hyperref[alg:PolyakEG_linesearch]{\algname{PolyakEG-LS}}.}\label{subsec:compare_theoreticalstep_SM}
In Figure \ref{fig:PolyakEG_theoretical_gradients}, \ref{fig:PolyakEG_theoretical_v2_gradients} and \ref{fig:PolyakEG_theoretical_v3_gradients}, we compare the performance of the \hyperref[eq:EG]{\algname{EG}}, \hyperref[alg:PolyakEG]{\algname{PolyakEG}} and \hyperref[alg:PolyakEG_linesearch]{\algname{PolyakEG-LS}}. For this experiment, we focus on implementing our methods on the following strongly convex strongly concave min-max optimization problem $\min_{w_1 \in \mathbb{R}^{d_1}} \max_{w_2 \in \mathbb{R}^{d_2}} g(w_1, w_2) \eqdef \frac{1}{n} \sum_{i = 1}^n g_i(w_1, w_2),$~\cite{gorbunov2022stochastic} where 
\begin{equation}
    g_i(w_1, w_2) = \frac{1}{2}w_1^{\top}\mathbf{A}_i w_1 + w_1^{\top}\mathbf{B}_i w_2 - \frac{1}{2}w_2^{\top}\mathbf{C}_iw_2 + a_i^{\top}w_1 - c_i^{\top}w_2.
\end{equation}

For \hyperref[eq:EG]{\algname{EG}}, we use the theoretical step size $\gamma_k = \alpha_k = \nicefrac{1}{4L}$ from \cite{mokhtari2020unified} whereas for \hyperref[alg:PolyakEG]{\algname{PolyakEG}}, we use $\gamma_k = \nicefrac{1}{3L}$, as recommended by Theorem \ref{theorem:deterministic-master-theorem} for $\mu > 0$. For \hyperref[alg:PolyakEG_linesearch]{\algname{PolyakEG-LS}}, we start with different initialization of step size $\gamma_{-1} \in \left\{ 10^{-3}, 10^{-1}, 10, 10^{3}\right\}$. In all three figures, we have the number of gradient computations on $x$-axis while on $y$-axis we plot the relative error $\nicefrac{\|x_k - x_*\|^2}{\|x_0 - x_*\|^2}$. As evident from the figures, \hyperref[alg:PolyakEG]{\algname{PolyakEG}} performs better than \hyperref[eq:EG]{\algname{EG}} with theoretical step sizes. Moreover, we find for any initialization of $\gamma_{-1}$, \hyperref[alg:PolyakEG_linesearch]{\algname{PolyakEG-LS}} outperforms the other two algorithms. 
The superior performance of \hyperref[alg:PolyakEG_linesearch]{\algname{PolyakEG-LS}} is a consequence of the adaptive line-search procedure, which selects a step size $\gamma_k$ that satisfies the critical condition \eqref{eqn:gamma-critical-condition}. Therefore, the step size $\gamma_k$ can be much more aggressive than the constant step size $\nicefrac{1}{3L}$ of \algname{PolyakEG}, making \algname{PolyakEG-LS} faster. 
 
\paragraph{Comparison of \algname{PolyakSEG-LS} and \algname{SEG-LS}.} In this experiment, we compare the performance of \algname{PolyakSEG-LS} with \algname{SEG} from \cite{vaswani2019painless} for different line-search strategies on an interpolated model. For this, we consider the strongly convex- strongly concave quadratic min-max problem of \eqref{eq:quad_minmax} such that $\nabla_{w_1} g(w_{1*}, w_{2*}) = \nabla_{w_2} g(w_{1*}, w_{2*}) = 0$ (this ensures the model is interpolated). For \algname{SEG}, we use the iteration from \cite{vaswani2019painless}:
\begin{eqnarray*}
\textstyle
    \hat{x}_k & = & x_k - \gamma_k F_{\mathcal{S}_k}(x_k) \\
    x_{k+1} & = & x_k - \gamma_k F_{\mathcal{S}_k}(\hx_k)
\end{eqnarray*}
where $\gamma_k$ are computed using line-search similar to \algname{PolyakSEG-LS}. Note that \cite{vaswani2019painless} uses the same step size for both extrapolation and update step in contrast to our \algname{PolyakSEG-LS}, which uses different step sizes. In Figure \ref{fig:PolyakSEGLS_interpolation_v1} and \ref{fig:PolyakSEGLS}, we have \algname{SEG-LS0} which restarts line-search in every iteration from $\gamma_{-1}$ while \algname{SEG-LS1} uses the $\gamma_{k-1}$ from previous iteration similar to \algname{PolyakSEG-LS}. We find that \algname{PolyakSEG-LS} enjoys linear convergence and outperforms both of these methods from \cite{vaswani2019painless} for different condition numbers $\nicefrac{L}{\mu} \sim \left\{ 10, 10^4, 10^8\right\}$. Figure \ref{fig:PolyakSEGLS_interpolation_v1} corresponds to $L/\mu \sim 10^4$.

\begin{figure*}[hbt!]
    \centering   
\begin{subfigure}[b]{0.3\textwidth}
        \centering
        \includegraphics[width=\linewidth]{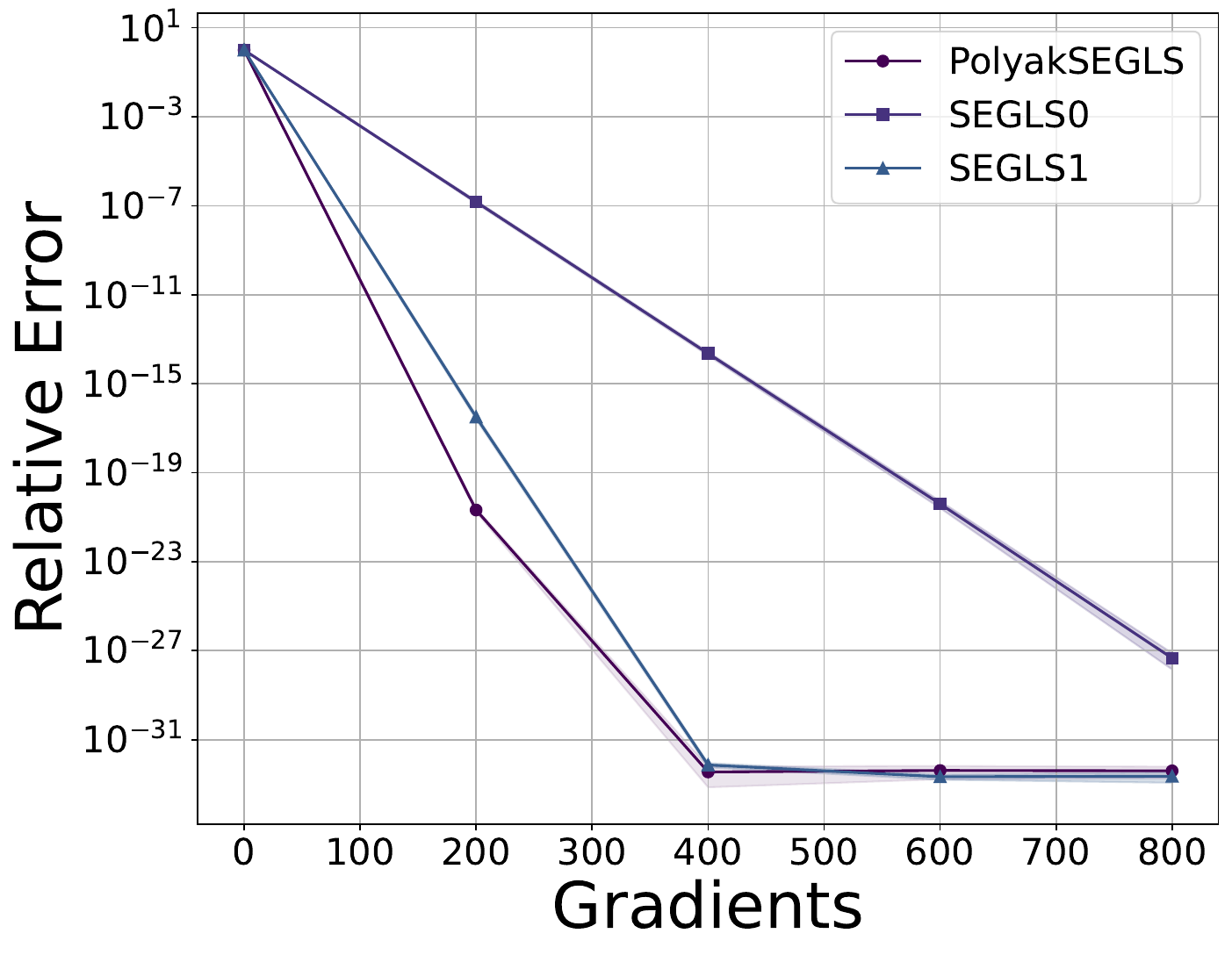}
        \caption{$\nicefrac{L}{\mu} \sim 10$}
        \label{fig:PolyakSEGLS_interpolation_v2}
    \end{subfigure}

    \begin{subfigure}[b]{0.3\textwidth}
        \centering
        \includegraphics[width=\linewidth]{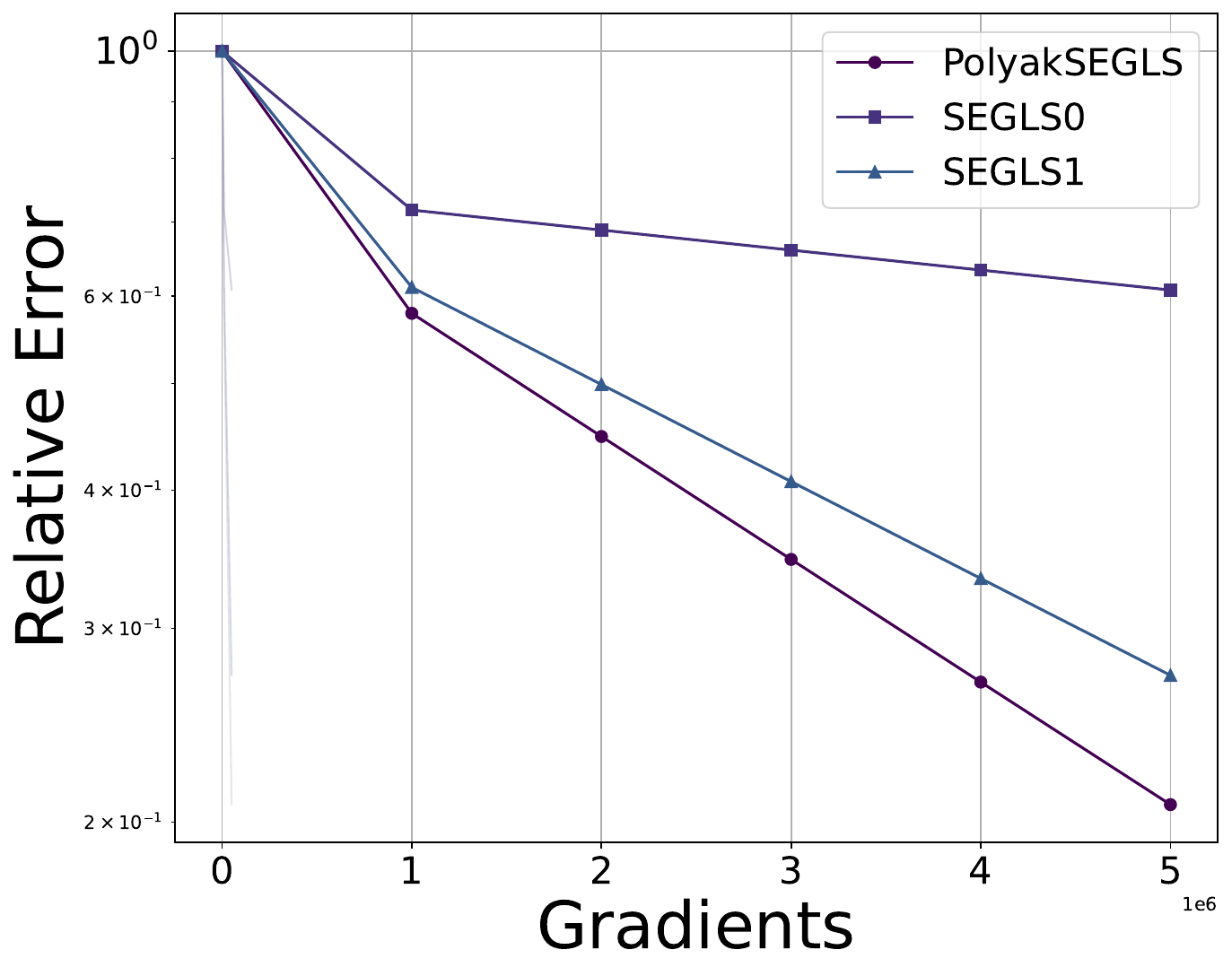}
        \caption{$\nicefrac{L}{\mu} \sim 10^8$}
        \label{fig:PolyakSEGLS_interpolation_v3}
    \end{subfigure}
    \caption{\small We compare the performance of \algname{PolyakSEG-LS} with \algname{SEG} for different line-search strategies. On the $y$-axis, we plot the relative error $\nicefrac{\|x_k - x_*\|^2}{\|x_0 - x_*\|^2}$, whereas on the $x$-axis we have the number of gradient computation. We find that \algname{PolyakSEGLS} outperforms \algname{SEG} for both of these line-search strategies.}\label{fig:PolyakSEGLS}
\end{figure*}

\paragraph{Robustness of \algname{DecPolyakSEG-LS}.} In this experiment, we compare the performance of different stochastic algorithms for varying step size initialization. Here, we focus on the Robust Least Square problem~\cite{el1997robust} with data matrix $\mathbf{A} \in \R^{r \times s}$ and a noisy vector $y_0 \in \R^r$ (with some bounded perturbation). Robust Least Square aims to minimize the worst-case residual and can be formulated as the following min-max optimization problem~\cite{yang2020global}:
\begin{eqnarray}\label{eq:RLS}
    \min_{v \in \R^s} \max_{y \in \R^r} \|\mathbf{A} v - y\|^2 - \lambda \|y - y_0\|^2.
\end{eqnarray}
For this experiment, we use the \textit{diabetes} dataset from \texttt{scikit-learn} for matrix $\mathbf{A}$, while for $y_0$, we generate $v_* \sim \mathcal{N}(0, I_s), \varepsilon \sim \mathcal{N}(0, I_r)$ and set $y_0 = \mathbf{A}v_* + \varepsilon$. Similar to \cite{ward2020adagrad}, we compare \hyperref[alg:decPolyakSEG_linesearch]{\algname{DecPolyakSEG-LS}} with \algname{SEG}, \algname{SDualExtra}, \algname{SOptDualAve} and \algname{AdaSEG} with different initialization of step size. For each of these algorithms, we start with a step size from the grid $\left\{10^{-3}, 10^{-2}, 10^{-1}, 1, 10, 10^{2} \right\}$ and plot $\nicefrac{\|F(\bar{x}_k)\|^2}{\|F(x_0)\|^2}$ after $10^4$ gradient computations in Figure \ref{fig:DecPolyakSEGLS_RL}. We find that \algname{AdaSEG} performs poorly for step sizes greater than $10^{-3}$. For step sizes smaller than $10^{-2}$, \algname{SEG}, \algname{SDualExtra} and \algname{SOptDualAve} perform very well. However, as we increase the step size, \algname{SEG} diverges while the performance of the other two degrades immensely. In contrast, we observe that the performance of \hyperref[alg:decPolyakSEG_linesearch]{\algname{DecPolyakSEG-LS}} remains consistent regardless of step size initialization. This experiment highlights that \textit{our proposed algorithm \hyperref[alg:decPolyakSEG_linesearch]{\algname{DecPolyakSEG-LS}} is robust to step size initialization.} 

\paragraph{Performance Comparison.} This experiment compares the performance of \\
\hyperref[alg:decPolyakSEG_linesearch]{\algname{DecPolyakSEG-LS}} with other stochastic algorithms to solve the Robust Least Square problem~\eqref{eq:RLS}. In Figure \ref{fig:DecPolyakSEGLS_RL_v3}, we plot the trajectories of \hyperref[alg:decPolyakSEG_linesearch]{\algname{DecPolyakSEG-LS}}, \algname{SEG} (Stochastic \algname{EG} with decreasing step size), Stochastic Dual Extrapolation (\algname{SDualExtra}), Stochastic Optimistic Dual Averaging (\algname{SOptDualAve})~\cite{antonakopoulos2021sifting} and Stochastic Adaptive Extragradient (\algname{AdaSEG})~\cite{antonakopoulos2020adaptive} with different initializations from $\left\{ 10^{-2}, 10^{-1}, 1\right\}$. On the $x$-axis, we have the number of gradient computations, while on the $y$-axis, we provide $\nicefrac{\|F(\bar{x}_k)\|^2}{\|F(x_0)\|^2}$. We observe that \hyperref[alg:decPolyakSEG_linesearch]{\algname{DecPolyakSEG-LS}} has outperformed others in solving this problem after $2,000$ iterations. 

\paragraph{Comparison on Matrix Game.} This experiment compares the performance of \hyperref[alg:decPolyakSEG_linesearch]{\algname{DecPolyakSEG-LS}} with other stochastic algorithms to solve the Policeman Burglar matrix game~\cite{nemirovski2013mini}:
\begin{eqnarray*}
    \min_{w_1 \in \Delta} \max_{w_2 \in \Delta} g(w_1, w_2) \eqdef \frac{1}{n} \sum_{i = 1}^n w_1^{\top} \mathbf{A}_i w_2
\end{eqnarray*}
where $\Delta = \left\{ w \in \R^d \mid \mathbf{1}^{\top}w = 1, w \geq 0 \right\}$. This matrix game is a constrained monotone problem, and for each of the algorithms, we project the iterates to $\Delta$ after every iteration. In Figure \ref{fig:DecPolyakSEGLS_matrix}, we plot the trajectories of \hyperref[alg:decPolyakSEG_linesearch]{\algname{DecPolyakSEG-LS}}, \algname{SEG} (Stochastic \algname{EG} with decreasing step size), Stochastic Dual Extrapolation (\algname{SDualExtra}), Stochastic Optimistic Dual Averaging (\algname{SOptDualAve})~\cite{antonakopoulos2021sifting} and Stochastic Adaptive Extragradient (\algname{AdaSEG})~\cite{antonakopoulos2020adaptive} with different initializations from $\left\{ 10^{-4}, 10^{-3}, 10^{-2}, 10^{-1}\right\}$. On the $x$-axis, we have the number of gradient computations, while on $y$-axis, we provide the duality gap. Note that, the duality gap at $(\hat{w}_1, \hat{w}_2)$ is given by $\max_{w_2 \in \Delta} g(\hat{w}_1, w_2) - \min_{w_1 \in \Delta} g(w_1, \hat{w}_2)$. We observe that \hyperref[alg:decPolyakSEG_linesearch]{\algname{DecPolyakSEG-LS}} outperforms all other algorithms for solving this problem.


\chapter{Appendices for Chapter \ref{chap:chap-4}} \label{chap:appendix-c}

\section{Technical Lemmas}\label{appendix:tech_lemma}
In this section, we present some technical lemmas, which will be used to prove the main results of the chapter \ref{chap:chap-4} in subsequent sections.

\begin{lemma}
    For $a, b \in \R^d$, we have
    \begin{equation}\label{eq:inner_prod}
        2 \la a, b\ra = \|a\|^2 + \|b\|^2 - \|a - b\|^2.
    \end{equation}
\end{lemma}

\begin{lemma}
    For $a, b \in \R^d$, we have
    \begin{equation}\label{eq:inequality1}
        -\|a\|^2 \leq - \frac{1}{2} \|a + b\|^2 + \|b\|^2.  
    \end{equation}
\end{lemma}

\begin{lemma}\label{lemma:reform_integration}\cite{chen2023generalized}
    Operator $F$ is $\alpha$-symmetric $(L_0, L_1)$-Lipschitz if and only if
    \begin{equation}\label{eq:reform_integration}
        \|F(x) - F(y)\| \leq \left( L_0 + L_1 \int_{0}^1 \left\| F(\theta x + (1 - \theta)y)\right\|^{\alpha} d\theta \right) \|x - y\| \qquad \forall x, y \in \R^d.
    \end{equation}
\end{lemma}

\begin{lemma}
    For a $2 \times 2$ symmetric matrix, the maximum eigenvalue is given by
    \begin{equation}\label{eq:eigen_2dmatrix}
        \lambda_{\max} \left( 
        \begin{bmatrix}
            a & b \\
            b & d
        \end{bmatrix}
        \right) = \frac{(a + d) + \sqrt{(a - d)^2 + 4b^2}}{2}
    \end{equation}
    where $a, b, d \in \R$.
\end{lemma}

\begin{proof}
        Let $A$ be a symmetric $2 \times 2$ matrix given by
        \[
        A = \begin{bmatrix}
        a & b \\
        b & d
        \end{bmatrix}
        \]
        where $a, b, d \in \R$. Since $A$ is symmetric, it has real eigenvalues. The eigenvalues of $A$ are the roots of its characteristic polynomial:
        \[
        \det(A - \lambda I) = 
        \det\left(
        \begin{bmatrix}
        a - \lambda & b \\
        b & d - \lambda
        \end{bmatrix}
        \right)
        = (a - \lambda)(d - \lambda) - b^2.
        \]
        
        Expanding the determinant, we obtain the characteristic equation:
        \[
        \lambda^2 - (a + d)\lambda + (ad - b^2) = 0.
        \]
        
        This is a quadratic equation in $\lambda$, and its solutions are:
        \[
        \lambda = \frac{(a + d) \pm \sqrt{(a - d)^2 + 4b^2}}{2}.
        \]
        
        Thus, the maximum eigenvalue is the larger of the two roots:
        \[
        \lambda_{\max}(A) = \frac{(a + d) + \sqrt{(a - d)^2 + 4b^2}}{2}.
        \]
        
        This completes the proof.
\end{proof}

\begin{lemma}\label{lemma:jacobian_norm_quad}
    For the quadratic problem $\min_{w_1} \max_{w_2} \mathcal{L}(w_1, w_2) = \frac{1}{2} w_1^2 + w_1 w_2 - \frac{1}{2}w_2^2$, the Jacobian $\mathbf{J}(x)$ is given by
    \begin{eqnarray*}
        \mathbf{J}(x) = \begin{bmatrix}
            1 & 1 \\
            -1 & 1
        \end{bmatrix}.
    \end{eqnarray*}
    In this case we get $\| \mathbf{J}(x)\| = \sigma_{\max} \left(\mathbf{J}(x) \right) = \sqrt{2}$.
\end{lemma}

\begin{proof}
    Note that 
    \begin{eqnarray*}
        \mathbf{J}(x)^{\top} \mathbf{J}(x) = \begin{bmatrix}
            2 & 0 \\
            0 & 2
        \end{bmatrix} = 2 \mathbf{I}.
    \end{eqnarray*}
    which has maximum eigenvalue $\sqrt{2}$. Hence, $\|\mathbf{J}(x)\| = \sigma_{\max}(\mathbf{J}(x)) = \sqrt{\lambda_{\max}(\mathbf{J}(x)^{\top} \mathbf{J}(x))}= \sqrt{2}.$
\end{proof}

\newpage
\section{Further Examples of $\alpha$-Symmetric $(L_0, L_1)$-Lipschitz Operators}\label{appendix:examples}
In this section, we provide further examples of operators that satisfy the $\alpha$-symmetric $(L_0, L_1)$-Lipschitz assumption. We first start with the details of Example 2 from Section \ref{sec:examples}, then we provide two more examples for the min-max optimization problem and $N$-player game that satisfy the assumption.\\
\newline
\textbf{Example 2:}\label{subsec:example2} Here we consider the operator $F(x) = (u_1^2, u_2^2)$ for $x = (u_1, u_2)$. Then for $y = (v_1, v_2)$, we have
\begin{eqnarray*}
    \|F(x) - F(y)\| & = & \left\|\left(u_1^2 - v_1^2, u_2^2 - v_2^2 \right) \right\| \\
    & = & \left( \left( u_1^2 - v_1^2\right)^2 + \left(u_2^2 - v_2^2 \right)^2 \right)^{\nicefrac{1}{2}} \\
    & = & \left( \left( u_1 - v_1\right)^2 \left( u_1 + v_1\right)^2 + \left(u_2 - v_2 \right)^2 \left(u_2 + v_2 \right)^2 \right)^{\nicefrac{1}{2}} \\
    & \leq & \left( (u_1 - v_1)^4 + (u_2 - v_2)^4 \right)^{\nicefrac{1}{4}} \left( (u_1 + v_1)^4 + (u_2 + v_2)^4 \right)^{\nicefrac{1}{4}} \\
    & \leq & \left( (u_1 - v_1)^2 + (u_2 - v_2)^2 \right)^{\nicefrac{1}{2}} \left( (u_1 + v_1)^4 + (u_2 + v_2)^4 \right)^{\nicefrac{1}{4}} \\
    & = &  \left( (u_1 + v_1)^4 + (u_2 + v_2)^4 \right)^{\nicefrac{1}{4}} \|x - y\| \\
    & = & 2 \left(\left(\frac{u_1 + v_1}{2} \right)^4 + \left(\frac{u_2 + v_2}{2} \right)^4 \right)^{\nicefrac{1}{4}} \|x - y\| \\
    & = & 2 \left\| \left( \frac{u_1 + v_1}{2}\right)^2, \left( \frac{u_2 + v_2}{2}\right)^2 \right\|^{\nicefrac{1}{2}} \|x - y\| \\
    & = & 2 \left\| F \left(\frac{x+y}{2} \right)\right\|^{\nicefrac{1}{2}} \|x - y\| \\
    & \leq & 2 \max_{\theta \in [0, 1]} \left\|  F \left( \theta x + (1 - \theta) y \right)\right\|^{\nicefrac{1}{2}} \|x - y\|.
\end{eqnarray*}
Here, the first inequality follows from the Cauchy-Schwarz inequality. This completes the proof of $\frac{1}{2}$-symmetric-$(0, 2)$ Lipschitz property of $F$. Now, we consider the vectors $x = \alpha \mathbf{1}_2$ and $y = \mathbf{1}_2$ where $\mathbf{1}_2 = (1, 1)$. Then we have 
\begin{eqnarray*}
    \frac{\| F(x) - F(y)\|}{\| x - y\|} & = & \frac{\sqrt{(\alpha^2 - 1)^2 + (\alpha^2 - 1)^2}}{\sqrt{(\alpha - 1)^2 + (\alpha - 1)^2}} \\
    & = & \frac{\sqrt{2(\alpha^2 - 1)^2}}{\sqrt{2(\alpha - 1)^2}} \\
    & = & \sqrt{\frac{(\alpha - 1)^2 (\alpha + 1)^2}{(\alpha - 1)^2}} \\ 
    & = & \lvert \alpha + 1 \rvert.
\end{eqnarray*}
Therefore, 
\begin{eqnarray*}
    \lim_{\alpha \to \infty} \frac{\| F(x) - F(y)\|}{\| x - y\|} & = & \lim_{\alpha \to \infty} \lvert \alpha + 1 \rvert \\
    & = & \infty
\end{eqnarray*}

\textbf{Example 3:}
Min-max optimization problems can be studied using variational inequality formulations. Here we consider one such example, the bilinearly coupled min-max optimization~\cite{chambolle2016ergodic} problem
\begin{equation*}
    \min_{\|w_1 \| \leq R} \max_{\|w_2 \| \leq R} \mathcal{L}(w_1, w_2) \eqdef f(w_1) + w_1^{\top} \B w_2 - g(w_2)
\end{equation*}
for matrix $\B \in \R^{d \times d}$ and functions $f, g: \R^d \to \R$. The associated operator for this problem is given by $F(x) = H(x) + \M x$ where
\begin{eqnarray*}
\textstyle
    H(x) = \begin{bmatrix}
        \nabla f(w_1) \\
        \nabla g(w_2)
    \end{bmatrix} \quad \text{ and } \quad \M = \begin{bmatrix}
        0 & \B \\
        - \B^{\top} & 0
    \end{bmatrix}.
\end{eqnarray*}
If $f, g$ are individually $(L_0, L_1)$-smooth~\cite{zhang2019gradient}, we show that
\begin{eqnarray*}
\textstyle
    \|F(x) - F(y)\| & \leq & \left( 2 L_0 + (1 + 2 L_1 R) \| \M \|+ \sqrt{2} L_1 \|F(x)\| \right) \|x - y\|.
\end{eqnarray*}
Thus, $F$ is $1$-symmetric $ \left(2 L_0 + (1 + 2 L_1 R) \| \M \|, \sqrt{2} L_1 \right)$-Lipschitz. Consider the min-max problem
\begin{equation*}
    \min_{\|w_1\| \leq R} \max_{\|w_2\| \leq R} \mathcal{L}(w_1, w_2) \eqdef f(w_1) + w_1^{\top} \B w_2 - g(w_2).
\end{equation*}
where $f, g$ are $(L_0, L_1)$-smooth. Then for $x = (w_1, w_2)$ we have
\begin{eqnarray*}
    F(x) = \begin{pmatrix}
        \nabla f(w_1) \\
        \nabla g(w_2)
    \end{pmatrix} + \begin{pmatrix}
        0 & \B \\
        - \B^{\top} & 0
    \end{pmatrix} \begin{pmatrix}
        w_1 \\
        w_2
    \end{pmatrix} = H(x) + \M x
\end{eqnarray*}
where $\M$ is a matrix and $H$ is an operator. Now note that for $x = (w_1, w_2)$ and $y = (v_1, v_2)$ we get\
\newpage
\begin{align*}
    & \|H(x) - H(y)\| \\
    = & \left\| \begin{bmatrix}
        \nabla f(w_1) - \nabla f(v_1) \\
        \nabla g(w_2) - \nabla g(v_2)
    \end{bmatrix}\right\| \\
    = & \left( \| \nabla f(w_1) - \nabla f(v_1)\|^2 + \| g(w_2) - g(v_2)\|^2 \right)^{\nicefrac{1}{2}} \\
    \leq & \left( \left(L_0 + L_1 \| \nabla f(w_1)\| \right)^2 \| w_1 - v_1\|^2 + \left(L_0 + L_1 \| \nabla g(w_2)\| \right)^2\| w_2 - v_2 \|^2 \right)^{\nicefrac{1}{2}} \\
    \leq & \left( \left(L_0 + L_1 \| \nabla f(w_1)\| \right)^4 + \left(L_0 + L_1 \| \nabla g(w_2)\| \right)^4 \right)^{\nicefrac{1}{4}} \left( \|w_1 - v_1\|^4 + \|w_2 - v_2\|^4 \right)^{\nicefrac{1}{4}} \\
    \leq & \left( \left(L_0 + L_1 \| \nabla f(w_1)\| \right)^2 + \left(L_0 + L_1 \| \nabla g(w_2)\| \right)^2 \right)^{\nicefrac{1}{2}} \left( \|w_1 - v_1\|^2 + \|w_2 - v_2\|^2 \right)^{\nicefrac{1}{2}} \\
    \leq & \left( 4 L_0^2 + 2L_1^2 \left(\| \nabla f(w_1) \|^2 + \| \nabla g(w_2)\|^2 \right) \right)^{\nicefrac{1}{2}} \left( \|w_1 - v_1\|^2 + \|w_2 - v_2\|^2 \right)^{\nicefrac{1}{2}} \\
    = & \left( 4 L_0^2 + 2L_1^2 \|H(x)\|^2 \right)^{\nicefrac{1}{2}} \|x - y\| \\
    \leq & \left( 2 L_0 + \sqrt{2} L_1 \|H(x)\| \right) \|x - y\|
\end{align*}

Therefore, using the above inequality, we derive
\begin{eqnarray*}
    \|F(x) - F(y)\| & \leq & \|H(x) - H(y)\| + \|\M x - \M y\| \\
    & \leq & \left( 2 L_0 + \sqrt{2} L_1 \|H(x)\| \right) \|x - y\|+ \|\M\| \|x - y\| \\
    & \leq & \left( 2 L_0 + \|\M\| + \sqrt{2} L_1 \|H(x)\| \right) \|x - y\| 
\end{eqnarray*}
Now using $\|H(x)\| = \|H(x) + \M x - \M x\| \leq \|F(x)\| + \| \M \| \|x\| \leq \|F(x)\| + \sqrt{2}R \|\M \|$ we get
\begin{eqnarray*}
    \|F(x) - F(y)\| & \leq & \left( 2 L_0 + (1 + 2L_1 R)\|\M\| + \sqrt{2} L_1 \|F(x)\| \right) \|x - y\|.
\end{eqnarray*}

\textbf{Example 4:}
In this example, we consider an $N$-player game~\cite{balduzzi2018mechanics, loizou2021stochastic, yoon2025multiplayer}, where each player $i \in [N]$ selects an action $w_i \in \R^{d_i}$, and the joint action vector of all players is denoted as $x = (w_1, w_2, \cdots, w_N) \in \R^{d_1 + \cdots d_N}$. Each player $i$ aims to minimize their loss function $f_i$ for their action $w_i$. The objective is to find an equilibrium $x_* = (w_{1*}, w_{2*}, \cdots, w_{N*})$ such that  
\begin{eqnarray*}
    w_{i*} = \argmin_{w_i \in \R^{d_i}} f_i (w_i, w_{-i*}).
\end{eqnarray*}
Here we abuse the notation to denote $w_{-i} = (w_1, \cdots, w_{i - 1}, w_{i+1}, \cdots, w_N)$ and $f_i(w_i, w_{-i}) = f_i (w_1, \cdots, w_N)$. When the functions $f_i$ are convex, this equilibrium corresponds to solving $F(x_*) = 0$, where the operator $F$ is defined as  
\begin{eqnarray*}
    F(x) = \left( \nabla_1 f_1 (x), \nabla_2 f_2 (x), \dots, \nabla_N f_N(x) \right).
\end{eqnarray*}

In case each of these partial gradients $\nabla_i f_i$ are $(L_0, L_1)$-Lipschitz i.e.
\begin{eqnarray*}
    \| \nabla_i f_i(x) - \nabla_i f_i (y)\| & \leq & (L_0 + L_1 \| \nabla_i f_i (x)\|) \| x - y\|
\end{eqnarray*}
then we obtain
\begin{eqnarray*}
    \| F(x) - F(y) \|^2 & = & \sum_{i = 1}^N \| \nabla_i f_i(x) - \nabla_i f_i(y) \|^2 \\
    & \leq & \sum_{i = 1}^N \left( L_0 + L_1 \left\| \nabla_i f_i(x) \right\| \right)^2 \|x - y\|^2 \\
    & \leq & \|x - y\|^2 \sum_{i = 1}^N (2L_0^2 + 2 L_1^2 \| \nabla_i f_i (x)\|^2) \\
    & = & \|x - y\|^2 (2N L_0^2 + 2 L_1^2 \| F(x)\|^2) \\
    & \leq & \|x - y\|^2 ( \sqrt{2N}L_0 + \sqrt{2} L_1 \| F(x)\|)^2.
\end{eqnarray*}
This completes the proof. 

\newpage
\section{Convergence Analysis}\label{appendix:convergence_analysis}
In this section, we present the missing proofs from Section \ref{sec:convergence_analysis}. We start with the proof of Proposition \ref{prop:equiv_formulation} and then provide the results related to strongly monotone, monotone and weak Minty problems. 

\subsection{Proof of Proposition \ref{prop:equiv_formulation}}

\begin{proposition}
    Suppose $F$ is $\alpha$-symmetric $(L_0, L_1)$-Lipschitz operator. Then, for $\alpha = 1$
    \begin{equation*}
    \textstyle
        \| F(x) - F(y) \| \leq (L_0 + L_1 \| F(x)\|) \exp{(L_1 \|x - y\|)} \| x- y \|,
    \end{equation*}
    and for $\alpha \in (0, 1)$ we have
    \begin{equation*}
    \textstyle
        \| F(x) - F(y) \| \leq \left(K_0 + K_1 \|F(x)\|^{\alpha} + K_2 \|x - y\|^{\nicefrac{\alpha}{1 - \alpha}} \right) \|x - y\|
    \end{equation*}
    where $K_0 = L_0 (2^{\nicefrac{\alpha^2}{1 - \alpha}} + 1)$, $K_1 = L_1 \cdot 2^{\nicefrac{\alpha^2}{1 - \alpha}}$ and $K_2 = L_1^{\nicefrac{1}{1 - \alpha}} \cdot 2^{\nicefrac{\alpha^2}{1 - \alpha}} \cdot 3^{\alpha} (1 - \alpha)^{\nicefrac{\alpha}{1 - \alpha}}$.
\end{proposition}

\begin{proof}
For proving this theorem, we follow the proof technique similar to \cite{chen2023generalized}. We start with $\alpha = 1$ case. Let $x, y \in \mathbb{R}^d$ and define $x_\theta := \theta x + (1 - \theta) y$. Since $F$ is $1$-symmetric $(L_0, L_1)$-Lipschitz , we have for all $\theta \in [0, 1]$,
$$
\| F(x_\theta) - F(y) \| \overset{\eqref{eq:reform_integration}}{\leq} \left(L_0 + L_1 \int_0^1 \| F(x_{\theta \tau}) \| d\tau \right) \| x_\theta - y \|.
$$
Note that
$$
x_{\theta \tau} = \tau x_\theta + (1 - \tau) y = \tau (\theta x + (1 - \theta) y) + (1 - \tau) y = \theta \tau x + (1 - \theta \tau) y.
$$
Let us define a function
$$
H(\theta) := L_0 \theta + L_1 \int_0^\theta \| F(x_u) \| du.
$$

Then, note that $H'(\theta) = L_0 + L_1 \| F(x_\theta) \|$. Moreover, we have
\begin{eqnarray*}
    \| F(x_\theta) - F(y) \| & \leq & \left(L_0 + L_1 \int_0^1 \| F(x_{\theta \tau}) \| d\tau \right) \| x_\theta - y \| \\
    & = & \left(L_0 + L_1 \int_0^1 \| F(x_{\theta \tau}) \| d\tau \right) \| \theta x + (1 - \theta) y - y \| \\
    & = & \left(L_0 + L_1 \int_0^1 \| F(x_{\theta \tau}) \| d\tau \right) \| \theta x - \theta y \| \\
    & = & \left(L_0 \theta + L_1 \int_0^1 \| F(x_{\theta \tau}) \| \theta d\tau \right) \| x -  y \| \\
    & = & \left(L_0 \theta + L_1 \int_0^1 \| F(\theta \tau x + (1 - \theta \tau) y) \| \theta d\tau \right) \| x -  y \| \\
    & = & \left(L_0 \theta + L_1 \int_0^{\theta} \| F(u x + (1 - u) y) \| du \right) \| x -  y \| \\
    & = & \left(L_0 \theta + L_1 \int_0^{\theta} \| F(x_u) \| du \right) \| x -  y \| \\
    & = & H(\theta) \| x -  y \|.
\end{eqnarray*}
Therefore we obtain
\begin{eqnarray*}
    H'(\theta) & = & L_0 + L_1 \| F(x_\theta) \|\\ 
    & \leq & L_0 + L_1 \left(\|F(x_\theta) - F(y)\| + \|F(y)\|\right)\\
    & \le & L_0 + L_1 \left(H(\theta) \|x - y\| + \|F(y)\|\right) \\
    & = & a H(\theta) + b,
\end{eqnarray*}
where  $a = L_1 \|x - y\|, b = L_0 + L_1 \|F(y)\|$.
Then we integrate both sides for $\theta \in [0, \theta']$ to get
$$
H(\theta') \le \frac{b}{a} (e^{a \theta'} - 1) .
$$
Here, we set $\theta' = 1$ to obtain
\begin{eqnarray*}
    H(1) &\le & \frac{b}{a} (e^{a} - 1) \\
    & = & \frac{L_0 + L_1 \|F(y)\|}{L_1 \|x - y\|} (e^{L_1 \|x - y\|} - 1).    
\end{eqnarray*}
Now, put this back into the original inequality
\begin{eqnarray*}
    \|F(x) - F(y)\| & \le & H(1) \|x - y\| \\
    & \le &  \left( L_0 + L_1 \|F(y)\| \right) \cdot \frac{e^{L_1 \|x - y\|} - 1}{L_1} \\
    & = & \left( \frac{L_0}{L_1} + \|F(y)\| \right) (e^{L_1 \|x - y\|} - 1).    
\end{eqnarray*}

Finally, using the inequality $e^z - 1 \le z e^z$ for $z \ge 0$, we get
\begin{eqnarray*}
    \|F(x) - F(y)\| & \le & \left( \frac{L_0}{L_1} + \|F(y)\| \right) L_1 \|x - y\| e^{L_1 \|x - y\|} \\
    & \leq & (L_0 + L_1 \|F(y)\|) e^{L_1 \|x - y\|} \|x - y\|
\end{eqnarray*}
This completes the proof for $\alpha = 1$. The proof for $\alpha \in (0, 1)$ follows similarly from \cite{chen2023generalized}. 
\end{proof}

\subsection{Convergence Guarantees for Strongly Monotone Operators}

\subsubsection{Proof of Theorem \ref{theorem:1symm_strongmonotone}}
\begin{theorem}
    Suppose $F$ is $\mu$-strongly monotone and $1$-symmetric $(L_0, L_1)$-Lipschitz operator. Then Extragradient method with step size $\gamma_k = \frac{\nu}{L_0 + L_1 \|F(x_k)\|}$ satisfy
    \begin{eqnarray*}
        \|x_{k+1} - x_*\|^2 & \leq & \left( 1 - \frac{\nu \mu}{L_0 \left( 1 + L_1 \exp{\left( L_1 \|x_0 - x_*\|\right)} \|x_0 - x_* \|\right)} \right)^{k+1} \|x_0 - x_*\|^2
    \end{eqnarray*}
    where $\nu > 0$ root of $ 1 - 2 \nu - \nu^2 \exp{2\nu} = 0$.
\end{theorem}

\begin{proof}
    From the update step of the Extragradient method, we obtain
    \begin{eqnarray*}
        \|x_{k+1} - x_*\|^2 & = & \| x_k - \gamma_k F(\hx_k) - x_* \|^2 \notag\\
        & = & \| x_k - x_* \|^2 - 2 \gamma_k \la F(\hx_k), x_k - x_* \ra + \gamma_k^2  \| F(\hx_k) \|^2 \notag\\
        & = & \| x_k - x_* \|^2 - 2 \gamma_k \la F(\hx_k), \hx_k - x_* \ra - 2 \gamma_k \la F(\hx_k), x_k - \hx_k \ra \notag \\
        && + \gamma_k^2  \| F(\hx_k) \|^2 \notag\\
        & \overset{\eqref{eq:strong_monotone}}{\leq} & \| x_k - x_* \|^2 - 2 \gamma_k \mu \| \hx_k - x_*\|^2  \notag \\
        && - 2 \gamma_k \la F(\hx_k), x_k - \hx_k \ra + \gamma_k^2  \| F(\hx_k) \|^2 \notag\\
        & \overset{\eqref{eq:inequality1}}{\leq} & \| x_k - x_* \|^2 - \gamma_k \mu \| x_k - x_*\|^2 + 2 \gamma_k \mu \|x_k - \hx_k\|^2  \notag \\ 
        && - 2 \gamma_k \la F(\hx_k), x_k - \hx_k \ra + \gamma_k^2  \| F(\hx_k) \|^2 \notag\\
        & = & (1 - \gamma_k \mu) \| x_k - x_* \|^2 + 2 \gamma_k \mu \|x_k - \hx_k\|^2 - 2 \gamma_k \la F(\hx_k), x_k - \hx_k \ra \notag \\
        && + \gamma_k^2  \| F(\hx_k) \|^2 \notag\\
        & = & (1 - \gamma_k \mu) \| x_k - x_* \|^2 + 2 \gamma_k^3 \mu \|F(x_k)\|^2 - 2 \gamma_k^2 \la F(\hx_k), F(x_k) \ra  \notag\\
        && + \gamma_k^2  \| F(\hx_k) \|^2 \notag \\
        & \overset{\eqref{eq:inner_prod}}{=} & (1 - \gamma_k \mu) \left\| x_k - x_* \right\|^2 + 2 \gamma_k^3 \mu \|F(x_k)\|^2 \notag \\
        && - \gamma_k^2 \left(\|F(\hx_k)\|^2 + \|F(x_k)\|^2 - \|F(x_k) - F(\hx_k)\|^2 \right) + \gamma_k^2 \|F(\hx_k)\|^2 \notag\\
        & = & (1 - \gamma_k \mu) \left\| x_k - x_* \right\|^2 -\gamma_k^2(1 - 2 \gamma_k \mu) \|F(x_k)\|^2 \notag \\
        && + \gamma_k^2 \|F(x_k) - F(\hx_k)\|^2 \notag\\
        & \overset{\eqref{eq:alpha=1}}{\leq} & (1 - \gamma_k \mu) \left\| x_k - x_* \right\|^2 -\gamma_k^2(1 - 2 \gamma_k \mu) \|F(x_k)\|^2 \notag \\
        && + \gamma_k^2 (L_0 + L_1 \|F(x_k)\|)^2 \exp{(2L_1 \|x_k - \hx_k\|)} \|x_k -\hx_k\|^2. \notag\\
    \end{eqnarray*}
    Thus we have 
    \begin{align}\label{eq:1symm_strongmonotone_eq}
        \|x_{k+1} - x_*\|^2 & \leq (1 - \gamma_k \mu) \|x_k - x_*\|^2 \notag \\
        & - \gamma_k^2 \left(1 - 2 \gamma_k \mu - \gamma_k^2 (L_0 + L_1 \|F(x_k)\|)^2 \exp{(2 \gamma_k L_1 \|F(x_k)\|)} \right) \|F(x_k)\|^2.
    \end{align}

    Now we set $\gamma_k = \frac{\nu}{L_0 + L_1 \|F(x_k)\|}$. Then we want to choose $\nu$ such that 
    \begin{eqnarray*}
        1 - \frac{2 \nu \mu}{L_0 + L_1 \|F(x_k)\|} - \nu^2 \exp{2 \nu} \geq 0.
    \end{eqnarray*}
    Note that $\mu \leq L_0 + L_1 \| F(x_k) \|$ for any $x_k$. Therefore, we get $1 - \frac{2 \nu \mu}{L_0 + L_1 \|F(x_k)\|} - \nu^2 \exp{2 \nu} \geq 1 - 2 \nu - \nu^2 \exp{2 \nu}$ and it is enough to choose $\nu$ such that 
    \begin{eqnarray*}
        1 - 2 \nu - \nu^2 \exp{2\nu} \geq 0.
    \end{eqnarray*}
    This inequality holds for any $\nu \leq 0.22$. Hence, for this choice of $\nu$ we get the following inequality from \eqref{eq:1symm_strongmonotone_eq}.
    \begin{eqnarray}\label{eq:1symm_strongmonotone_eq1}
        \|x_{k+1} - x_*\|^2 & \leq & (1 - \gamma_k \mu) \|x_k - x_*\|^2.
    \end{eqnarray}
    This proves that the distance of the iterates $x_k$ from $x_*$ are bounded by $\|x_0 - x_*\|$.
    Now note that using \eqref{eq:alpha=1} with $x = x_k, y = x_*$ with $\|x_k - x_*\| \leq \|x_0 - x_*\|$ we get
    \begin{eqnarray}\label{eq:1symm_strongmonotone_eq2}
        \|F(x_k)\| & \overset{\eqref{eq:alpha=1}}{\leq} & L_0 \exp{\left(L_1 \|x_k - x_*\| \right)} \|x_k - x_*\| \notag \\
        & \overset{\eqref{eq:1symm_strongmonotone_eq1}}{\leq} & L_0 \exp{\left(L_1 \|x_0 - x_*\| \right)} \|x_0 - x_*\|.
    \end{eqnarray}
    Therefore, we have the following lower bound on the step-size 
    \begin{eqnarray}\label{eq:1symm_strongmonotone_eq3}
        \gamma_k & = & \frac{\nu}{L_0 + L_1 \|F(x_k)\|} \notag \\
        & \overset{\eqref{eq:1symm_strongmonotone_eq2}}{\geq} & \frac{\nu}{L_0 \left( 1 + L_1 \exp{\left( L_1 \|x_0 - x_*\|\right)} \|x_0 - x_* \|\right)}.
    \end{eqnarray}
    Hence, we get 
    \begin{eqnarray*}
        \|x_{k+1} - x_*\|^2 & \overset{\eqref{eq:1symm_strongmonotone_eq1},\eqref{eq:1symm_strongmonotone_eq3}}{\leq} & \left( 1 - \frac{\nu \mu}{L_0 \left( 1 + L_1 \exp{\left( L_1 \|x_0 - x_*\|\right)} \|x_0 - x_* \|\right)} \right) \|x_k - x_*\|^2 \\
        & \leq & \left( 1 - \frac{\nu \mu}{L_0 \left( 1 + L_1 \exp{\left( L_1 \|x_0 - x_*\|\right)} \|x_0 - x_* \|\right)} \right)^{k+1} \|x_0 - x_*\|^2.
    \end{eqnarray*}
\end{proof}

\subsubsection{Proof of Corollary \ref{corollary:1symm_strongmonotone}}
\begin{corollary}
    Suppose $F$ is $\mu$-strongly monotone and $1$-symmetric $(L_0, L_1)$-Lipschitz operator. Then Extragradient with step size $\gamma_k = \frac{\nu}{L_0 + L_1 \|F(x_k)\|}$ satisfy $\|x_{K+1} - x_*\|^2  \leq  \varepsilon$ after
    \begin{equation*}
        K = \frac{2 L_0}{\nu \mu} \log \left( \frac{\|x_0 - x_*\|^2}{\varepsilon}\right) + \frac{1}{\gamma \mu} \log \left( \frac{2L_1 \|x_0 - x_*\|^2}{\gamma^2 L_0} \right)
    \end{equation*}
    many iterations, where $\nu > 0$ satisfy $1 - 4 \nu - 2
    \nu^2 \exp{2\nu} = 0$ and 
    \begin{equation*}
        \gamma \eqdef \frac{\nu}{L_0 \left( 1 + L_1 \exp{\left( L_1 \|x_0 - x_*\|\right)} \|x_0 - x_* \|\right)}.
    \end{equation*}
\end{corollary}

\begin{proof}
    We set $\gamma_k = \frac{\nu}{L_0 + L_1 \|F(x_k)\|}$ and we choose $\nu \in (0, 1)$ such that $1 - 4 \nu - 2
    \nu^2 \exp{2\nu} = 0$. Then we have
    \begin{eqnarray}\label{eq:1symm_strongmonotone_cor_eq1}
        & 1 - 2 \gamma_k \mu - \gamma_k^2 (L_0 + L_1 \|F(x_k)\|)^2 \exp{(2 \gamma_k L_1 \|F(x_k)\|)} \\
        & \geq 1 - \frac{2 \nu \mu}{L_0 + L_1 \|F(x_k)\|} - \nu^2 \exp{2 \nu} \notag \\
        & \geq 1 - 2 \nu - \nu^2 \exp{2 \nu} \notag \\
        & = \frac{1}{2}. 
    \end{eqnarray}
    Therefore, from \eqref{eq:1symm_strongmonotone_eq} and \eqref{eq:1symm_strongmonotone_cor_eq1} we get 
    \begin{eqnarray}\label{eq:1symm_strongmonoton_noexp_eq1}
        \|x_{k+1} - x_*\|^2 & \leq & (1 - \gamma_k \mu) \|x_k - x_*\|^2 - \frac{\gamma_k^2}{2} \|F(x_k)\|^2.
    \end{eqnarray}
    In \eqref{eq:1symm_strongmonotone_eq3}, we found that a lower bound of $\gamma_k$ is 
    \begin{eqnarray*}
        \gamma_k \geq \gamma \eqdef \frac{\nu}{L_0 \left( 1 + L_1 \exp{\left( L_1 \|x_0 - x_*\|\right)} \|x_0 - x_* \|\right)}.
    \end{eqnarray*}
    Then from \eqref{eq:1symm_strongmonoton_noexp_eq1} we get
    \begin{equation}\label{eq:1symm_strongmonoton_noexp_eq2}
        \frac{\gamma^2}{2} \|F(x_k)\|^2 \leq (1 - \gamma \mu)\|x_k - x_*\|^2 \leq (1 - \gamma \mu)^{k+1} \|x_0 - x_*\|^2
    \end{equation}
    This can be rearranged to write
    \begin{equation*}
        \|F(x_k)\|^2 \leq \frac{2(1 - \gamma \mu)^{k+1}}{\gamma^2} \|x_0 - x_*\|^2.
    \end{equation*}
    This implies $\|F(x_k)\|^2 \leq \frac{L_0}{L_1}$ after 
    \begin{equation}
        K' = \frac{1}{\gamma \mu} \log \left( \frac{2L_1 \|x_0 - x_*\|^2}{\gamma^2 L_0} \right)
    \end{equation}
    many iterations. Hence for $k \geq K'$ we have
    \begin{eqnarray*}
        \gamma_k & = & \frac{\nu}{L_0 + L_1 \|F(x_k)\|} \\
        & \geq & \frac{\nu}{2 L_0}.
    \end{eqnarray*}
    In the last inequality we used $ \|F(x_k)\|^2 \leq \frac{L_0}{L_1}$ for $k \geq K'$. Therefore for $k \geq K'$ we obtain
    \begin{eqnarray*}
        \|x_{k+1} - x_*\|^2 & \leq & \left(1 - \frac{\nu \mu}{2 L_0} \right) \|x_k - x_*\|^2 \\
        & \leq & \left(1 - \frac{\nu \mu}{2 L_0} \right)^{k+1 - K'} \|x_{K'} - x_*\|^2 \\
        & \leq & \left(1 - \frac{\nu \mu}{2 L_0} \right)^{k+1 - K'} \|x_0 - x_*\|^2
    \end{eqnarray*}
    Thus we conclude that, $\|x_{K+1} - x_*\|^2 \leq \varepsilon$ after atmost 
    \begin{equation*}
        K = \frac{2 L_0}{\nu \mu} \log \left( \frac{\|x_0 - x_*\|^2}{\varepsilon}\right) + K'
    \end{equation*}
    many iterations.
\end{proof}

\subsubsection{Proof of Theorem \ref{theorem:1symm_strongmonotone_alhpha01}}
\begin{theorem}
    Suppose $F$ is $\mu$-strongly monotone and $\alpha$-symmetric $(L_0, L_1)$-Lipschitz operator with $\alpha \in (0, 1)$. Then Extragradient method with $$\gamma_k = \frac{\nu}{2K_0 + \left( 2 K_1 + 2^{1 - \alpha} K_2^{1 - \alpha} \right) \| F(x_k)\|^{\alpha}}$$ satisfy
    \begin{eqnarray*}
    \textstyle
        \|x_{k+1} - x_*\|^2 \leq \left( 1 - \frac{\nu \mu}{2 K_0 + (2 K_1 + 2^{1 - \alpha} K_2^{1 - \alpha})(K_0 + K_2 \|x_0 - x_*\|^{\nicefrac{\alpha}{1 - \alpha}})^{\alpha} \|x_0 - x_*\|^{\alpha}}\right)^{k+1} \|x_0 - x_*\|^2.
    \end{eqnarray*}
    where $\nu \in (0, 1)$ is a constant such that $1 - \nu - \nu^2 \geq 0$.
\end{theorem}

\begin{proof} 
    Using the update steps of the Extragradient method, we have
    \begin{eqnarray*}
        \|x_{k+1} - x_*\|^2 & = & \| x_k - \gamma_k F(\hx_k) - x_* \|^2 \\
        & = & \| x_k - x_* \|^2 - 2 \gamma_k \la F(\hx_k), x_k - x_* \ra + \gamma_k^2  \| F(\hx_k) \|^2 \\
        & = & \| x_k - x_* \|^2 - 2 \gamma_k \la F(\hx_k), \hx_k - x_* \ra - 2 \gamma_k \la F(\hx_k), x_k - \hx_k \ra \\
        && + \gamma_k^2  \| F(\hx_k) \|^2 \\
        & \overset{\eqref{eq:strong_monotone}}{\leq} & \| x_k - x_* \|^2 - 2 \gamma_k \mu \| \hx_k - x_*\|^2 - 2 \gamma_k \la F(\hx_k), x_k - \hx_k \ra \\
        && + \gamma_k^2  \| F(\hx_k) \|^2 \\
        & \overset{\eqref{eq:inequality1}}{\leq} & \| x_k - x_* \|^2 - \gamma_k \mu \| x_k - x_*\|^2 + 2 \gamma_k \mu \|x_k - \hx_k\|^2 \\
        && - 2 \gamma_k \la F(\hx_k), x_k - \hx_k \ra  + \gamma_k^2  \| F(\hx_k) \|^2 \\
        & = & (1 - \gamma_k \mu) \| x_k - x_* \|^2 + 2 \gamma_k \mu \|x_k - \hx_k\|^2 - 2 \gamma_k \la F(\hx_k), x_k - \hx_k \ra \\
        && + \gamma_k^2  \| F(\hx_k) \|^2 \\
        & = & (1 - \gamma_k \mu) \| x_k - x_* \|^2 + 2 \gamma_k^3 \mu \|F(x_k)\|^2 - 2 \gamma_k^2 \la F(\hx_k), F(x_k) \ra \\
        && + \gamma_k^2  \| F(\hx_k) \|^2 \\
        & \overset{\eqref{eq:inner_prod}}{=} & (1 - \gamma_k \mu) \left\| x_k - x_* \right\|^2 + 2 \gamma_k^3 \mu \|F(x_k)\|^2 \\
        &&- \gamma_k^2 \left(\|F(\hx_k)\|^2 + \|F(x_k)\|^2 - \|F(x_k) - F(\hx_k)\|^2 \right) + \gamma_k^2 \|F(\hx_k)\|^2 \\
        & = & (1 - \gamma_k \mu) \left\| x_k - x_* \right\|^2 -\gamma_k^2(1 - 2 \gamma_k \mu) \|F(x_k)\|^2 \\
        && - \gamma_k^2 \|F(x_k) - F(\hx_k)\|^2 \\
        & \overset{\eqref{eq:alpha(0,1)}}{\leq} & (1 - \gamma_k \mu) \left\| x_k - x_* \right\|^2 -\gamma_k^2(1 - 2 \gamma_k \mu) \|F(x_k)\|^2 \\
        && - \gamma_k^2 \left(K_0 + K_1 \|F(x_k)\|^{\alpha} + K_2 \| x_k - \hx_k \|^{\nicefrac{\alpha}{1 - \alpha}} \right)^2 \|x_k - \hx_k\|^2.
    \end{eqnarray*}
    Thus we have 
    \begin{align*}
    \textstyle
        & \|x_{k+1} - x_*\|^2 \\
        & \leq (1 - \gamma_k \mu) \|x_k - x_*\|^2 \\
        & - \gamma_k^2 \left(1 - 2 \gamma_k \mu - \gamma_k^2 \left(K_0 + K_1 \|F(x_k)\|^{\alpha} + \gamma_k^{\nicefrac{\alpha}{1 - \alpha}} K_2 \| F(x_k) \|^{\nicefrac{\alpha}{1 - \alpha}} \right)^2 \right) \|F(x_k)\|^2.
    \end{align*}
    Here we will choose $\gamma_k > 0$ such that
    \begin{equation*}
        1 - 2 \gamma_k \mu  - \gamma_k^2 \left(K_0 + K_1 \|F(x_k)\|^{\alpha} + \gamma_k^{\nicefrac{\alpha}{1 - \alpha}} K_2 \| F(x_k) \|^{\nicefrac{\alpha}{1 - \alpha}} \right)^2 \geq 0
    \end{equation*} 
    Let us choose $\gamma_k = \frac{\nu}{2K_0 + \left( 2 K_1 + 2^{1 - \alpha} K_2^{1 - \alpha} \right) \| F(x_k)\|^{\alpha}}$ for some $\nu \in (0, 1)$. Then we observe that
    \begin{eqnarray*}
        && \gamma_k \left( K_0 + K_1 \| F(x_k)\|^{\alpha} \right) + \gamma_k^{\nicefrac{1}{1 - \alpha}} K_2 \| F(x_k)\|^{\nicefrac{\alpha}{1 - \alpha}} \\
        & \leq & \frac{\nu \left( K_0 + K_1 \| F(x_k)\|^{\alpha} \right)}{2K_0 + \left( 2 K_1 + 2^{1 - \alpha} K_2^{1 - \alpha} \right) \| F(x_k)\|^{\alpha}} \\
        && + \frac{\nu^{\nicefrac{1}{1 - \alpha}} K_2 \|F(x_k)\|^{\nicefrac{\alpha}{1 - \alpha}}}{\left( 2K_0 + \left( 2K_1 + 2^{1 - \alpha} K_2^{1 - \alpha} \right) \| F(x_k)\|^{\alpha}\right)^{\nicefrac{1}{1 - \alpha}}} \\
        & \leq & \frac{\nu}{2} + \frac{\nu^{\nicefrac{1}{1 - \alpha}}}{2} \\
        & \leq & \nu.        
    \end{eqnarray*}
    The last inequality follows from $\nu \in (0, 1)$. Therefore it is enough to choose $\nu \in (0, 1)$ such that 
    \begin{equation*}
        1 - \frac{2 \nu \mu}{2K_0 + \left( 2 K_1 + 2^{1 - \alpha} K_2^{1 - \alpha} \right) \| F(x_k)\|^{\alpha}} - \nu^2 \geq 0
    \end{equation*}
    However, note that, we always have $\mu \leq K_0$, thus it is enough to choose $\nu \in (0, 1)$ such that 
    \begin{equation*}
        1 - \nu - \nu^2 \geq 0.
    \end{equation*}
    Hence, for this choice of $\nu$ we get
    \begin{eqnarray*}
        \|x_{k+1} - x_*\|^2 & \leq & (1 - \gamma_k \mu) \|x_0 - x_*\|^2.
    \end{eqnarray*}
    Here we lower bound the step size $\gamma_k$ with 
    \begin{eqnarray*}
        \gamma_k \geq \frac{\nu}{2 K_0 + (2 K_1 + 2^{1 - \alpha} K_2^{1 - \alpha})(K_0 + K_2 \|x_0 - x_*\|^{\nicefrac{\alpha}{1 - \alpha}})^{\alpha} \|x_0 - x_*\|^{\alpha}}. 
    \end{eqnarray*}
    Hence we obtain
    \begin{eqnarray*}
    \textstyle
        \|x_{k+1} - x_*\|^2 & \leq \left( 1 - \frac{\nu \mu}{2 K_0 + (2 K_1 + 2^{1 - \alpha} K_2^{1 - \alpha})(K_0 + K_2 \|x_0 - x_*\|^{\nicefrac{\alpha}{1 - \alpha}})^{\alpha} \|x_0 - x_*\|^{\alpha}}\right) \|x_k - x_*\|^2 \\
        & \leq \left( 1 - \frac{\nu \mu}{2 K_0 + (2K_1 + 2^{1 - \alpha} K_2^{1 - \alpha})(K_0 + K_2 \|x_0 - x_*\|^{\nicefrac{\alpha}{1 - \alpha}})^{\alpha} \|x_0 - x_*\|^{\alpha}}\right)^{k+1} \|x_0 - x_*\|^2.
    \end{eqnarray*}
    This completes the proof of the theorem.
\end{proof}

\newpage
\subsection{Convergence Guarantees for Monotone Operators}\label{subsec:monotone_proofs}

\subsubsection{Proof of Theorem \ref{theorem:1symm_monotone}}

\begin{theorem}
    Suppose $F$ is monotone and $1$-symmetric $(L_0, L_1)$-Lipschitz operator. Then \algname{EG} with step size $\gamma_k = \omega_k = \frac{\nu}{L_0 + L_1 \|F(x_k)\|}$ satisfy 
    \begin{equation*}
        \min_{0 \leq k \leq K} \|F(x_k)\|^2 \leq \frac{2L_0^2 \left( 1 + L_1 \exp{\left( L_1 \|x_0 - x_*\|\right) \|x_0 - x_*\|}\right)^2 \|x_0 - x_*\|^2}{\nu^2(K+1)}.
    \end{equation*}
    where $\nu \exp{\nu} = \nicefrac{1}{\sqrt{2}}$.
\end{theorem}
\begin{proof}
    From the update rule of the Extragradient method, we have
    \begin{eqnarray*}
        && \left\| x_{k+1} - x_* \right\|^2 \notag \\
        & = & \| x_k - \gamma_k F(\hx_k) - x_* \|^2 \notag \\
        & = & \| x_k - x_* \|^2 - 2 \gamma_k \la F(\hx_k), x_k - x_* \ra + \gamma_k^2  \| F(\hx_k) \|^2 \notag \\
        & = & \| x_k - x_* \|^2 - 2 \gamma_k \la F(\hx_k), \hx_k - x_* \ra - 2 \gamma_k \la F(\hx_k), x_k - \hx_k \ra \\
        && + \gamma_k^2  \| F(\hx_k) \|^2 \notag \\
        & \overset{\eqref{eq:monotone}}{\leq} & \| x_k - x_* \|^2 - 2 \gamma_k \la F(\hx_k), x_k - \hx_k \ra + \gamma_k^2  \| F(\hx_k) \|^2 \\
        & = & \| x_k - x_* \|^2 - 2 \gamma_k^2 \la F(\hx_k), F(x_k) \ra + \gamma_k^2  \| F(\hx_k) \|^2 \\
        & \overset{\eqref{eq:inner_prod}}{=} & \left\| x_k - x_* \right\|^2 - \gamma_k^2 \|F(\hx_k)\|^2 - \gamma_k^2 \|F(x_k)\|^2 + \gamma_k^2 \| F(x_k) - F(\hx_k)\|^2 \\
        && + \gamma_k^2 \| F(\hx_k)\|^2 \\
        & = & \left\| x_k - x_* \right\|^2 - \gamma_k^2 \|F(x_k)\|^2 + \gamma_k^2 \| F(x_k) - F(\hx_k)\|^2 \\
        & \overset{\eqref{eq:alpha=1}}{\leq} & \left\| x_k - x_* \right\|^2 - \gamma_k^2 \| F(x_k)\|^2 \\
        && + \gamma_k^2 (L_0 + L_1 \|F(x_k)\|)^2 \exp{\left( 2 L_1 \|x_k - \hx_k \| \right) } \|x_k - \hx_k\|^2 \\
        & = & \left\| x_k - x_* \right\|^2 \\
        && - \gamma_k^2 \left(1 - \gamma_k^2 (L_0 + L_1 \|F(x_k)\|)^2 \exp{\left( 2 \gamma_k L_1 \|F(x_k) \| \right)} \right) \| F(x_k)\|^2 \\
        & \leq & \left\| x_k - x_* \right\|^2 \\
        && - \gamma_k^2 \left(1 - \gamma_k^2 (L_0 + L_1 \|F(x_k)\|)^2 \exp{\left( 2 \gamma_k (L_0 + L_1 \|F(x_k) \|) \right)} \right) \| F(x_k)\|^2.
    \end{eqnarray*}
     Here we use $\gamma_k = \frac{\nu}{L_0 + L_1 \|F(x_k)\|}$ for some $\nu \in (0, 1)$ to get
    \begin{eqnarray*}
         \left\| x_{k+1} - x_* \right\|^2 & \leq & \left\| x_k - x_* \right\|^2 - \gamma_k^2 \left(1 - \nu^2  \exp{\left( 2 \nu \right)} \right) \| F(x_k)\|^2
    \end{eqnarray*}
    Then we choose $\nu$ such that $\nu \exp{\nu} = \nicefrac{1}{\sqrt{2}}$ to obtain
    \begin{eqnarray}\label{eq:1symm_monotone_eq1}
        \|x_{k+1} - x_*\|^2 & \leq & \|x_k - x_*\|^2 - \frac{\gamma_k^2}{2} \|F(x_k)\|^2.
    \end{eqnarray}
    In particular the distance of the iterates $x_k$ from $x_*$ are bounded i.e. $\|x_{k+1} - x_*\| \leq \|x_k - x_*\| \leq \|x_0 - x_*\|$. Therefore, using \eqref{eq:alpha=1} with $y = x_*$ and $x = x_k$, we get
    \begin{eqnarray*}
        \|F(x_k)\| & \leq & L_0 \exp{ \left( L_1 \|x_k - x_*\|\right)} \|x_k - x_*\| \\
        & \leq & L_0 \exp{\left(L_1\|x_0 - x_*\| \right)} \|x_0 - x_*\|.
    \end{eqnarray*}
    Then we have the lower bound on step-size given as follows
    \begin{eqnarray}\label{eq:1symm_monotone_eq2}
        \gamma_k = \frac{\nu}{L_0 + L_1 \|F(x_k)\|} \geq \frac{\nu}{L_0 \left( 1 + L_1 \exp{\left(L_1 \|x_0 - x_*\| \right) \|x_0 - x_*\|}\right)}.
    \end{eqnarray}
    Rearranging \eqref{eq:1symm_monotone_eq1} we have
    \begin{eqnarray*}
        \frac{\gamma_k^2}{2} \|F(x_k)\|^2 & \leq & \|x_{k} - x_*\|^2 - \|x_{k+1} - x_*\|^2.
    \end{eqnarray*}
    Summing up the above inequality for $k = 0, 1, \cdots, K$ and dividing by $K+1$ we get
    \begin{equation}\label{eq:1symm_monotone_eq3}
        \frac{1}{K+1} \sum_{k = 0}^K \frac{\gamma_k^2}{2}\|F(x_k)\|^2 \leq \frac{\|x_0 - x_*\|^2 - \|x_{K+1} - x_*\|^2}{K+1} \leq \frac{\|x_0 - x_*\|^2}{K+1}.
    \end{equation}
    Here we will use the lower bound on step-size $\gamma_k$ given in \eqref{eq:1symm_monotone_eq2} to get
    \begin{eqnarray*}
        \frac{1}{K+1} \sum_{k = 0}^K \|F(x_k)\|^2 & \leq & \frac{2L_0^2 \left( 1 + L_1 \exp{\left( L_1 \|x_0 - x_*\|\right) \|x_0 - x_*\|}\right)^2 \|x_0 - x_*\|^2}{\nu^2(K+1)}.
    \end{eqnarray*}
    Finally, note that $\min_{0 \leq k \leq K} \|F(x_k)\|^2 \leq \frac{1}{K+1} \sum_{k = 0}^K \|F(x_k)\|^2$. Therefore we have
    \begin{eqnarray*}
        \min_{0 \leq k \leq K} \|F(x_k)\|^2 & \leq &  \frac{2L_0^2 \left( 1 + L_1 \exp{\left( L_1 \|x_0 - x_*\|\right) \|x_0 - x_*\|}\right)^2 \|x_0 - x_*\|^2}{\nu^2(K+1)}.
    \end{eqnarray*}
    This completes the proof of the Theorem.
\end{proof}

\newpage
\subsubsection{Proof of Theorem \ref{theorem:1symm_monotone_noexp}}
\begin{theorem}
    Suppose $F$ is monotone and $1$-symmetric $(L_0, L_1)$-Lipschitz operator. Then Extragradient method with step size $\gamma_k = \frac{\nu}{L_0 + L_1 \|F(x_k)\|}$ satisfy 
    \begin{eqnarray*}
        \min_{0 \leq k \leq K} \| F(x_k)\| & \leq & \frac{\sqrt{2} L_0 \|x_0 - x_*\|}{\nu \sqrt{K+1} - \sqrt{2} L_1 \|x_0 - x_*\|}
    \end{eqnarray*}
    where $\nu \exp{\nu} = \nicefrac{1}{\sqrt{2}}$ and $K+1 \geq \frac{2L_1^2 \|x_0 - x_*\|^2}{\nu^2}$.
\end{theorem}

\begin{proof}
    From \eqref{eq:1symm_monotone_eq3}, we know steps of Extragradient method satisfy
    \begin{eqnarray*}
        \frac{1}{K+1} \sum_{k = 0}^K \frac{\gamma_k^2}{2} \|F(x_k)\|^2 & \leq & \frac{\|x_0 - x_*\|^2}{K+1}.
    \end{eqnarray*}
    Taking the minimum on the left-hand side we have
    \begin{eqnarray*}
        \min_{0 \leq k \leq K} \gamma_k^2 \|F(x_k)\|^2 & \leq & \frac{2\|x_0 - x_*\|^2}{K+1},
    \end{eqnarray*}
    or equivalently, 
    \begin{eqnarray*}
        \min_{0 \leq k \leq K} \frac{\nu^2 \|F(x_k)\|^2}{(L_0 + L_1 \|F(x_k)\|)^2} & \leq & \frac{2\|x_0 - x_*\|^2}{K+1}.
    \end{eqnarray*}
    Taking the square root on both sides we have
    \begin{eqnarray*}
        \min_{0 \leq k \leq K} \frac{\nu \|F(x_k)\|}{L_0 + L_1 \|F(x_k)\|} & \leq & \frac{\sqrt{2}\|x_0 - x_*\|}{\sqrt{K+1}}.
    \end{eqnarray*}
    Therefore, for some $0 \leq k_0 \leq K$ we have
    \begin{eqnarray*}
        \frac{\|F(x_{k_0})\|}{L_0 + L_1 \|F(x_{k_0})\|} & \leq & \frac{\sqrt{2}\|x_0 - x_*\|}{\nu \sqrt{K+1}}.
    \end{eqnarray*}
    Therefore, rearranging these terms, we get
    \begin{eqnarray*}
       \left( \nu \sqrt{K+1} - \sqrt{2}L_1 \|x_0 - x_*\| \right) \| F(x_{k_0})\| & \leq & \sqrt{2} L_0 \|x_0 - x_*\|.
    \end{eqnarray*}
    When we have $K+1 \geq \frac{2L_1^2 \|x_0 - x_*\|^2}{\nu^2}$ then we can rearrange the terms to obtain
    \begin{eqnarray*}
        \| F(x_{k_0})\| & \leq & \frac{\sqrt{2} L_0 \|x_0 - x_*\|}{\left( \nu \sqrt{K+1} - \sqrt{2}L_1 \|x_0 - x_*\| \right)}.
    \end{eqnarray*}
    for some $0 \leq k_0 \leq K$. Hence, we complete the proof of the theorem
    \begin{eqnarray*}
        \min_{0 \leq k \leq K} \| F(x_k)\| & \leq & \frac{\sqrt{2} L_0 \|x_0 - x_*\|}{ \nu \sqrt{K+1} - \sqrt{2} L_1 \|x_0 - x_*\|}.
    \end{eqnarray*}
\end{proof}

\subsubsection{Proof of Theorem \ref{theorem:alpha01}}
\begin{theorem}
    Suppose $F$ is monotone and $\alpha$-symmetric $(L_0, L_1)$-Lipschitz operator with $\alpha \in (0, 1)$. Then Extragradient method with step size $$\gamma_k = \frac{1}{2 \sqrt{2} K_0 + \left(2 \sqrt{2}K_1 + 2^{\nicefrac{3 (1 - \alpha)}{2}} K_2^{1 - \alpha} \right) \|F(x_k)\|^{\alpha}}$$ satisfy 
    \begin{eqnarray*}
    \textstyle
        \min_{0 \leq k \leq K} \|F(x_k)\|^2 \leq \frac{16 \left( K_0 + \left(K_1 + 2^{\nicefrac{-3}{2}} K_2^{1 - \alpha} \right) (K_0 + K_2 \|x_0 - x_*\|^{\nicefrac{\alpha}{1 - \alpha}})^{\alpha} \|x_0 - x_*\|^{\alpha}\right)^2 \|x_0 - x_*\|^2}{K+1}.
    \end{eqnarray*}
\end{theorem}

\begin{proof}
     Here operator $F$ is $\alpha$-symmetric $(L_0, L_1)$-Lipschitz i.e. it satisfies \eqref{eq:alpha(0,1)}. For the update steps of the Extragradient method, we have
    \begin{eqnarray}\label{eq:alpha01}
        && \left\| x_{k+1} - x_* \right\|^2 \\
        & = & \| x_k - \gamma_k F(\hx_k) - x_* \|^2 \notag \\
        & = & \| x_k - x_* \|^2 - 2 \gamma_k \la F(\hx_k), x_k - x_* \ra + \gamma_k^2  \| F(\hx_k) \|^2 \notag \\
        & = & \| x_k - x_* \|^2 - 2 \gamma_k \la F(\hx_k), \hx_k - x_* \ra - 2 \gamma_k \la F(\hx_k), x_k - \hx_k \ra \notag \\
        && + \gamma_k^2  \| F(\hx_k) \|^2 \notag \\
        & \overset{\eqref{eq:monotone}}{\leq} & \| x_k - x_* \|^2 - 2 \gamma_k \la F(\hx_k), x_k - \hx_k \ra + \gamma_k^2  \| F(\hx_k) \|^2 \notag \\
        & = & \| x_k - x_* \|^2 - 2 \gamma_k^2 \la F(\hx_k), F(x_k) \ra + \gamma_k^2  \| F(\hx_k) \|^2 \notag \\
        & \overset{\eqref{eq:inner_prod}}{=} & \left\| x_k - x_* \right\|^2 - \gamma_k^2 \|F(\hx_k)\|^2 - \gamma_k^2 \|F(x_k)\|^2 + \gamma_k^2 \| F(x_k) - F(\hx_k)\|^2 \notag \\
        && + \gamma_k^2 \| F(\hx_k)\|^2 \notag \\
        & = & \left\| x_k - x_* \right\|^2 - \gamma_k^2 \|F(x_k)\|^2 + \gamma_k^2 \| F(x_k) - F(\hx_k)\|^2 \notag \\
        & \overset{\eqref{eq:alpha(0,1)}}{\leq} & \left\| x_k - x_* \right\|^2 - \gamma_k^2 \| F(x_k)\|^2 \notag \\
        && + \gamma_k^2 \left( K_0 + K_1 \|F(x_k)\|^{\alpha} + K_2 \| x_k - \hx_k \|^{\nicefrac{\alpha}{1 - \alpha}}\right)^2 \|x_k - \hx_k \|^2 \notag\\
        & = & \left\| x_k - x_* \right\|^2 \\
        && - \gamma_k^2 \left(1 - \gamma_k^2 \left( K_0 + K_1 \|F(x_k)\|^{\alpha} +  \gamma_k^{\nicefrac{\alpha}{1 - \alpha}} K_2 \| F(x_k) \|^{\nicefrac{\alpha}{1 - \alpha}}\right)^2 \right) \|F(x_k)\|^2 \notag.
    \end{eqnarray}
    Here we want to choose $\gamma_k$ such that
    \begin{eqnarray*}
        \gamma_k \left( K_0 + K_1 \|F(x_k)\|^{\alpha} \right) + \gamma_k^{\nicefrac{1}{1 - \alpha}} K_2 \| F(x_k) \|^{\nicefrac{\alpha}{1 - \alpha}} & \leq & \frac{1}{\sqrt{2}}.
    \end{eqnarray*}
    For this, it is enough to make sure
    \begin{eqnarray*}
        \gamma_k \left( K_0 + K_1 \| F(x_k)\|^{\alpha} \right) \leq \frac{1}{2 \sqrt{2}} & \text{ and } & \gamma_k^{\nicefrac{1}{1 - \alpha}} K_2 \| F(x_k)\|^{\nicefrac{\alpha}{1 - \alpha}} \leq \frac{1}{2 \sqrt{2}}.
    \end{eqnarray*}
    Therefore, we choose $\gamma_k = \frac{1}{2 \sqrt{2} (K_0 + K_1 \| F(x_k)\|^{\alpha}) + 2^{\nicefrac{3 (1 - \alpha)}{2}} K_2^{1 - \alpha} \|F(x_k)\|^{\alpha}}$ and we get the following from \eqref{eq:alpha01}
    \begin{eqnarray}\label{eq:alpha01_eq3}
        \| x_{k+1} - x_*\|^2 & \leq & \| x_k - x_*\|^2 - \frac{\gamma_k^2}{2} \| F(x_k) \|^2.
    \end{eqnarray}
    Rearranging this inequality, we have
    \begin{eqnarray*}
        \frac{\gamma_k^2}{2} \| F(x_k)\|^2 & \leq & \|x_k - x_*\|^2 - \|x_{k+1} -  x_*\|^2.
    \end{eqnarray*}
    Then we sum up this inequality for $k = 0, 1, \cdots K$ to get
    \begin{eqnarray}\label{eq:alpha01_eq2}
        \frac{1}{K+1}\sum_{k = 0}^K \gamma_k^2 \| F(x_k)\|^2 & \leq & \frac{2\|x_0 - x_*\|^2}{K+1}.
    \end{eqnarray}
    For this step size, we also have $\|x_k - x_0\|^2 \leq \|x_0 - x_*\|^2$ from \eqref{eq:alpha01_eq3}. Now note that from \eqref{eq:alpha(0,1)} we obtain the following bound with $x = x_k$ and $y = x_*$
    \begin{eqnarray*}
        \|F(x_k)\|^{\alpha} & \leq & (K_0 + K_2 \|x_k - x_*\|^{\nicefrac{\alpha}{1 - \alpha}})^{\alpha} \|x_k - x_*\|^{\alpha} \\
        & \overset{\eqref{eq:alpha01_eq3}}{\leq} & (K_0 + K_2 \|x_0 - x_*\|^{\nicefrac{\alpha}{1 - \alpha}})^{\alpha} \|x_0 - x_*\|^{\alpha}.
    \end{eqnarray*}
    We use this to lower bound the step size $\gamma_k$ as follows 
    \begin{eqnarray*}
        \gamma_k & = & \frac{1}{2 \sqrt{2} (K_0 + K_1 \| F(x_k)\|^{\alpha}) + 2^{\nicefrac{3 (1 - \alpha)}{2}} K_2^{1 - \alpha} \|F(x_k)\|^{\alpha}} \\ 
        & \geq & \frac{1}{2 \sqrt{2} K_0 + 2 \sqrt{2} (K_1 + 2^{\nicefrac{-3}{2}} K_2^{1 - \alpha}) (K_0 + K_2 \|x_0 - x_*\|^{\nicefrac{\alpha}{1 - \alpha}})^{\alpha} \|x_0 - x_*\|^{\alpha}}.
    \end{eqnarray*}
    Therefore from \eqref{eq:alpha01_eq2} we obtain
    \begin{equation*}
    \textstyle
        \min_{0 \leq k \leq K} \|F(x_k)\|^2 \leq \frac{16 \left( K_0 + (K_1 + 2^{\nicefrac{-3}{2}} K_2^{1 - \alpha}) (K_0 + K_2 \|x_0 - x_*\|^{\nicefrac{\alpha}{1 - \alpha}})^{\alpha} \|x_0 - x_*\|^{\alpha}\right)^2 \|x_0 - x_*\|^2}{K+1}.
    \end{equation*}
\end{proof}


\newpage
\subsection{Local Convergence Guarantees for Weak Minty Operator}
\subsubsection{Proof of Theorem \ref{theorem:weak_minty_alpha1}}

\begin{theorem}
    Suppose $F$ is weak Minty and $1$-symmetric $(L_0, L_1)$-Lipschitz assumption. Moreover we assume 
    \begin{equation}
        \Delta_1 \eqdef \frac{\nu}{L_0 \left( 1 + L_1 \|x_0 - x_*\| e^{L_1 \|x_0 - x_*\|}\right)} - 4 \rho > 0.
    \end{equation} Then \algname{EG} with step size $\gamma_k = \frac{\nu}{L_0 + L_1 \|F(x_k)\|}$ and $\omega_k = \nicefrac{\gamma_k}{2}$ satisfies
    \begin{eqnarray}
        \min_{0 \leq k \leq K} \|F(\hx_k)\|^2 & \leq & \frac{4 L_0 \left( 1 + L_1 \exp{\left( L_1 \|x_0 - x_*\|\right) \|x_0 - x_*\|}\right) \|x_0 - x_*\|^2}{\nu \Delta_1 (K+1)} 
    \end{eqnarray}
    where $\nu \exp{\nu} = 1$.
\end{theorem}

\begin{proof}
    From the update rule of the Extragradient method, we have
    \begin{eqnarray*}
        &&\left\| x_{k+1} - x_* \right\|^2 \\
        & = & \left\| x_k - \frac{\gamma_k}{2} F(\hx_k) - x_* \right\|^2 \notag \\
        & = & \| x_k - x_* \|^2 - \gamma_k \la F(\hx_k), x_k - x_* \ra + \frac{\gamma_k^2}{4}  \| F(\hx_k) \|^2 \notag \\
        & = & \| x_k - x_* \|^2 - \gamma_k \la F(\hx_k), \hx_k - x_* \ra - \gamma_k \la F(\hx_k), x_k - \hx_k \ra + \frac{\gamma_k^2}{4} \| F(\hx_k) \|^2 \notag \\
        & \overset{\eqref{eq: weak MVI}}{\leq} & \| x_k - x_* \|^2 + \gamma_k \rho \| F(\hx_k) \|^2  - \gamma_k \la F(\hx_k), x_k - \hx_k \ra + \frac{\gamma_k^2}{4} \| F(\hx_k) \|^2 \\
        & = & \| x_k - x_* \|^2 + \gamma_k \rho \| F(\hx_k) \|^2 - \gamma_k^2 \la F(\hx_k), F(x_k) \ra \\
        && + \frac{\gamma_k^2}{4}  \| F(\hx_k) \|^2 \\
        & \overset{\eqref{eq:inner_prod}}{=} & \left\| x_k - x_* \right\|^2 + \gamma_k \rho \| F(\hx_k) \|^2 - \frac{\gamma_k^2}{2} \|F(\hx_k)\|^2 - \frac{\gamma_k^2}{2} \|F(x_k)\|^2  \\
        && + \frac{\gamma_k^2}{2} \| F(x_k) - F(\hx_k)\|^2 + \frac{\gamma_k^2}{4} \| F(\hx_k)\|^2 \\
        & = & \left\| x_k - x_* \right\|^2 + \gamma_k \rho \| F(\hx_k) \|^2 - \frac{\gamma_k^2}{4} \|F(\hx_k)\|^2 - \frac{\gamma_k^2}{2} \|F(x_k)\|^2 \\
        && + \frac{\gamma_k^2}{2} \| F(x_k) - F(\hx_k)\|^2  \\
        & \overset{\eqref{eq:alpha=1}}{\leq} & \left\| x_k - x_* \right\|^2 + \gamma_k \rho \| F(\hx_k) \|^2 - \frac{\gamma_k^2}{4} \|F(\hx_k)\|^2  - \frac{\gamma_k^2}{2} \| F(x_k)\|^2 \\
        && + \frac{\gamma_k^2}{2} (L_0 + L_1 \|F(x_k)\|)^2 \exp{\left( 2 L_1 \|x_k - \hx_k \| \right) } \|x_k - \hx_k\|^2 \\
        & = & \left\| x_k - x_* \right\|^2 - \frac{\gamma_k}{4} \left( \gamma_k - 4 \rho \right) \| F(\hx_k) \|^2 \\
        && - \frac{\gamma_k^2}{2} \left(1 - \gamma_k^2 (L_0 + L_1 \|F(x_k)\|)^2 \exp{\left( 2 \gamma_k L_1 \|F(x_k) \| \right)} \right) \| F(x_k)\|^2 \\
        & \leq & \left\| x_k - x_* \right\|^2 - \frac{\gamma_k}{4} \left( \gamma_k - 4 \rho \right) \| F(\hx_k) \|^2 \\ 
        && - \frac{\gamma_k^2}{2} \left(1 - \gamma_k^2 (L_0 + L_1 \|F(x_k)\|)^2 \exp{\left( 2 \gamma_k (L_0 + L_1 \|F(x_k) \|) \right)} \right) \| F(x_k)\|^2.
    \end{eqnarray*}
    Similar to the proof of Theorem \ref{theorem:1symm_monotone}, we have
    \begin{eqnarray*}
        \|x_{k+1} - x_*\|^2 & \leq & \|x_k - x_*\|^2 - \frac{\gamma_k}{4}(\gamma_k - 4 \rho) \| F(\hx_k)\|^2
    \end{eqnarray*}
    for $\gamma_k = \frac{\nu}{L_0 + L_1 \| F(x_k)\|}$ and $\nu \exp{\nu} = 1$. Again similar to Theorem \ref{theorem:1symm_monotone}, step size $\gamma_k$ is lower bounded with 
    \begin{equation*}
       \gamma_k = \frac{\nu}{L_0 + L_1 \|F(x_k)\|} \geq \frac{\nu}{L_0 \left( 1 + L_1 \exp{\left(L_1 \|x_0 - x_*\| \right) \|x_0 - x_*\|}\right)}.
    \end{equation*}
    Hence from \eqref{eq:restriction_rho} we get
    \begin{eqnarray*}
        \frac{\gamma_k \Delta_1}{4} \|F(\hx_k)\|^2 & \leq & \|x_k - x_*\|^2 - \|x_{k+1} - x_*\|^2
    \end{eqnarray*}
    Then we sum up this inequality for $k = 0, 1, \cdots, K$ to get
    \begin{eqnarray*}
        \frac{1}{K+1}\sum_{k = 0}^K \frac{\gamma_k \Delta_1}{4} \| F(\hx_k)\|^2 & \leq & \frac{\|x_0 - x_*\|^2}{K+1}.
    \end{eqnarray*}
    Therefore, we get
    \begin{equation*}
        \min_{0 \leq k \leq K} \| F(\hx_k)\|^2 \leq \frac{4 \|x_0 - x_*\|^2}{\gamma \Delta_1 (K+1)}.
    \end{equation*}
\end{proof}

\subsubsection{Proof of Theorem \ref{theorem:weak_minty_alpha01}}
\begin{theorem}
    Suppose $F$ is weak Minty and $\alpha$-symmetric $(L_0, L_1)$-Lipschitz with $\alpha \in (0, 1)$. Moreover we assume 
    \begin{equation*}
        \Delta_{\alpha} \eqdef \frac{1}{2 \sqrt{2} K_0 + 2 \sqrt{2} (K_1 + 2^{\nicefrac{-3}{2}} K_2^{1 - \alpha}) (K_0 + K_2 \|x_0 - x_*\|^{\nicefrac{\alpha}{1 - \alpha}})^{\alpha} \|x_0 - x_*\|^{\alpha}} - 4 \rho > 0.
    \end{equation*} 
    Then \algname{EG} with step size $\gamma_k = \frac{1}{2 \sqrt{2} K_0 + \left(2 \sqrt{2}K_1 + 2^{\nicefrac{3 (1 - \alpha)}{2}} K_2^{1 - \alpha} \right) \|F(x_k)\|^{\alpha}}$ and $\omega_k = \frac{\gamma_k}{2}$ satisfy 
    \begin{eqnarray*}
    \textstyle
        \min_{0 \leq k \leq K} \|F(\hx_k)\|^2 \leq \frac{4 \left( K_0 + \left(K_1 + 2^{\nicefrac{-3}{2}} K_2^{1 - \alpha} \right) (K_0 + K_2 \|x_0 - x_*\|^{\nicefrac{\alpha}{1 - \alpha}})^{\alpha} \|x_0 - x_*\|^{\alpha}\right) \|x_0 - x_*\|^2}{\Delta_{\alpha}(K+1)}.
    \end{eqnarray*}
\end{theorem}

\begin{proof}
    From the update rule of the Extragradient method, we have
    \begin{eqnarray*}
        && \left\| x_{k+1} - x_* \right\|^2 \\
        & = & \left\| x_k - \frac{\gamma_k}{2} F(\hx_k) - x_* \right\|^2 \notag \\
        & = & \| x_k - x_* \|^2 - \gamma_k \la F(\hx_k), x_k - x_* \ra + \frac{\gamma_k^2}{4}  \| F(\hx_k) \|^2 \notag \\
        & = & \| x_k - x_* \|^2 - \gamma_k \la F(\hx_k), \hx_k - x_* \ra - \gamma_k \la F(\hx_k), x_k - \hx_k \ra \\
        && + \frac{\gamma_k^2}{4} \| F(\hx_k) \|^2 \notag \\
        & \overset{\eqref{eq: weak MVI}}{\leq} & \| x_k - x_* \|^2 + \gamma_k \rho \| F(\hx_k) \|^2  - \gamma_k \la F(\hx_k), x_k - \hx_k \ra \\
        && + \frac{\gamma_k^2}{4} \| F(\hx_k) \|^2 \\
        & = & \| x_k - x_* \|^2 + \gamma_k \rho \| F(\hx_k) \|^2 - \gamma_k^2 \la F(\hx_k), F(x_k) \ra + \frac{\gamma_k^2}{4}  \| F(\hx_k) \|^2 \\
        & \overset{\eqref{eq:inner_prod}}{=} & \left\| x_k - x_* \right\|^2 + \gamma_k \rho \| F(\hx_k) \|^2 - \frac{\gamma_k^2}{2} \|F(\hx_k)\|^2 - \frac{\gamma_k^2}{2} \|F(x_k)\|^2  \\
        && + \frac{\gamma_k^2}{2} \| F(x_k) - F(\hx_k)\|^2 + \frac{\gamma_k^2}{4} \| F(\hx_k)\|^2 \\
        & = & \left\| x_k - x_* \right\|^2 + \gamma_k \rho \| F(\hx_k) \|^2 - \frac{\gamma_k^2}{4} \|F(\hx_k)\|^2 - \frac{\gamma_k^2}{2} \|F(x_k)\|^2 \\
        && + \frac{\gamma_k^2}{2} \| F(x_k) - F(\hx_k)\|^2  \\
        & \overset{\eqref{eq:alpha=1}}{\leq} & \left\| x_k - x_* \right\|^2 + \gamma_k \rho \| F(\hx_k) \|^2 - \frac{\gamma_k^2}{4} \|F(\hx_k)\|^2  - \frac{\gamma_k^2}{2} \| F(x_k)\|^2 \\
        && + \frac{\gamma_k^2}{2} \left( K_0 + K_1 \|F(x_k)\|^{\alpha} +  K_2 \| x_k - \hx_k \|^{\nicefrac{\alpha}{1 - \alpha}}\right)^2 \|x_k - \hx_k\|^2 \\
        & = & \left\| x_k - x_* \right\|^2 - \frac{\gamma_k}{4} \left( \gamma_k - 4 \rho \right) \| F(\hx_k) \|^2 \\
        && - \frac{\gamma_k^2}{2} \left(1 - \gamma_k^2 \left( K_0 + K_1 \|F(x_k)\|^{\alpha} +  \gamma_k^{\nicefrac{\alpha}{1 - \alpha}} K_2 \| F(x_k) \|^{\nicefrac{\alpha}{1 - \alpha}}\right)^2 \right) \| F(x_k)\|^2 .
    \end{eqnarray*}
     Here we choose $\gamma_k = \frac{1}{2 (K_0 + K_1 \| F(x_k)\|^{\alpha}) + 2^{\nicefrac{3 (1 - \alpha)}{2}} K_2^{1 - \alpha} \|F(x_k)\|^{\alpha}}$ and we get the following from \eqref{eq:alpha01}
    \begin{eqnarray}\label{eq:alpha01_eq3}
        \| x_{k+1} - x_*\|^2 & \leq & \| x_k - x_*\|^2 - \frac{\gamma_k}{4} \left( \gamma_k - 4 \rho \right) \| F(\hx_k) \|^2.
    \end{eqnarray}
    Rearranging this inequality, we have
    \begin{eqnarray*}
       \frac{\gamma_k}{4} \left( \gamma_k - 4 \rho \right) \| F(\hx_k) \|^2 & \leq & \|x_k - x_*\|^2 - \|x_{k+1} -  x_*\|^2.
    \end{eqnarray*}
    Then we sum up this inequality for $k = 0, 1, \cdots K$ to get
    \begin{eqnarray}\label{eq:alpha01_eq2}
        \frac{1}{K+1}\sum_{k = 0}^K \frac{\gamma_k}{4} \left( \gamma_k - 4 \rho \right) \| F(\hx_k) \|^2 & \leq & \frac{2\|x_0 - x_*\|^2}{K+1}.
    \end{eqnarray}
    For this step size, we also have $\|x_k - x_*\|^2 \leq \|x_0 - x_*\|^2$ from \eqref{eq:alpha01_eq3}. Now note that from \eqref{eq:alpha(0,1)} we obtain the following bound with $x = x_k$ and $y = x_*$
    \begin{eqnarray*}
        \|F(x_k)\|^{\alpha} & \leq & (K_0 + K_2 \|x_k - x_*\|^{\nicefrac{\alpha}{1 - \alpha}})^{\alpha} \|x_k - x_*\|^{\alpha} \\
        & \overset{\eqref{eq:alpha01_eq3}}{\leq} & (K_0 + K_2 \|x_0 - x_*\|^{\nicefrac{\alpha}{1 - \alpha}})^{\alpha} \|x_0 - x_*\|^{\alpha}.
    \end{eqnarray*}
    We use this to lower bound the step size $\gamma_k$ as follows 
    \begin{eqnarray*}
        \gamma_k & = & \frac{1}{2 \sqrt{2} (K_0 + K_1 \| F(x_k)\|^{\alpha}) + 2^{\nicefrac{3 (1 - \alpha)}{2}} K_2^{1 - \alpha} \|F(x_k)\|^{\alpha}} \\ 
        & \geq & \frac{1}{2 \sqrt{2} K_0 + 2 \sqrt{2} (K_1 + 2^{\nicefrac{-3}{2}} K_2^{1 - \alpha}) (K_0 + K_2 \|x_0 - x_*\|^{\nicefrac{\alpha}{1 - \alpha}})^{\alpha} \|x_0 - x_*\|^{\alpha}}.
    \end{eqnarray*}
    Therefore from \eqref{eq:alpha01_eq2} we obtain
    \begin{eqnarray*}
        \min_{0 \leq k \leq K} \|F(x_k)\|^2 \leq \frac{4 \|x_0 - x_*\|^2}{\gamma \Delta (K+1)}.
    \end{eqnarray*}
\end{proof}

\newpage
\section{Equivalent Formulation of $\alpha$-Symmetric $(L_0, L_1)$-Lipschitz Assumption}\label{appendix:equiv_formulation}
In this section, we consider the min-max optimization problem $\min_{w_1} \max_{w_2} \mathcal{L}(w_1, w_2)$ and provide an equivalent formulation of $\alpha$-symmetric $(L_0, L_1)$-Lipschitz operator. Next, we provide an example where we use this formulation to compute the constants $\alpha, L_0, L_1$.
\subsection{Proof of Theorem \ref{thm:equiv_formulation}}
\begin{theorem}
    Suppose $F$ is the operator for the problem
    \begin{equation*}
        \min_{w_1} \max_{w_2} \mathcal{L}(w_1, w_2).
    \end{equation*}
    Then $F$ satisfies $\alpha$-symmetric $(L_0, L_1)$-Lipschitz assumption if and only if
    \begin{equation*}
        \|\mathbf{J}(x)\| = \sup_{\|u\| = 1}\|\mathbf{J}(x) u\| \leq L_0 + L_1 \|F(x)\|^{\alpha}
    \end{equation*}
    where 
    \begin{equation*}
        \mathbf{J}(x) = \begin{bmatrix}
        \nabla^2_{w_1 w_1} \mathcal{L}(w_1, w_2) & \nabla^2_{w_2 w_1} \mathcal{L}(w_1, w_2) \\
        -\nabla^2_{w_1 w_2} \mathcal{L}(w_1, w_2) & -\nabla^2_{w_2 w_2} \mathcal{L}(w_1, w_2)
        \end{bmatrix}.
    \end{equation*}
    Here $\| \mathbf{J}(x)\| = \sigma_{\max}(\mathbf{J}(x))$ i.e. maximum singular value of $\mathbf{J}(x)$.
\end{theorem}

\begin{proof}
    Following \eqref{eq:reform_integration}, we have the equivalent characterization of $F$ given by
    \begin{equation*}
        \|F(y) - F(x)\| \leq \left( L_0 + L_1 \int_{0}^1 \left\| F(\theta y + (1 - \theta)x)\right\|^{\alpha} d\theta \right) \|y - x\| \qquad \forall x, y \in \R^d.
    \end{equation*}
    As this inequality holds for any $x, y \in \R^d$, we choose $y = x + \theta' u$ where $\|u\| = 1$ and $\theta' \in (0, 1)$. Then we get
    \begin{equation*}
        \|F(x + \theta' u) - F(x)\| \leq \left( L_0 + L_1 \int_{0}^1 \left\| F(x + \theta' \theta u) \right\|^{\alpha} d\theta \right) \|\theta' u\| \qquad \forall x \in \R^d.
    \end{equation*}
    The right-hand side of this inequality can be rewritten as 
    \begin{eqnarray*}
        \left( L_0 + L_1 \int_{0}^1 \left\| F(x + \theta' \theta u) \right\|^{\alpha} d\theta \right) \|\theta' u\| & = & \theta' \left( L_0 + L_1 \int_{0}^1 \left\| F(x + \theta' \theta u) \right\|^{\alpha} d\theta \right) \\
        & = & L_0 \theta' + L_1 \int_{0}^1 \left\| F(x + \theta' \theta u) \right\|^{\alpha} \theta' d\theta \\
        & = & L_0 \theta' + L_1 \int_{0}^{\theta'} \left\| F(x + \varphi u) \right\|^{\alpha} d\varphi.
    \end{eqnarray*}
    In the last line, we used the change of variable with $\varphi = \theta' \theta$. Therefore, we get
    \begin{equation*}
        \frac{\|F(x + \theta' u) - F(x)\|}{\theta'} \leq L_0 + \frac{L_1}{\theta'} \int_{0}^{\theta'} \left\| F(x + \varphi u) \right\|^{\alpha} d\varphi.
    \end{equation*}
    Then we take $\theta' \to 0$ and use L'Hôpital's rule and Leibniz Integral rule to obtain
    \begin{equation*}
        \lim_{\theta' \to 0} \frac{\|F(x + \theta' u) - F(x)\|}{\theta'} \leq L_0 + L_1 \left\| F(x) \right\|^{\alpha}.
    \end{equation*}
    Moreover, note that the left-hand side is given by $\|\mathbf{J}(x)u\|$ where
    \begin{equation*}
        \mathbf{J}(x) = \begin{bmatrix}
        \nabla^2_{w_1 w_1} \mathcal{L}(w_1, w_2) & \nabla^2_{w_2 w_1} \mathcal{L}(w_1, w_2) \\
        -\nabla^2_{w_1 w_2} \mathcal{L}(w_1, w_2) & -\nabla^2_{w_2 w_2} \mathcal{L}(w_1, w_2)
        \end{bmatrix}.
    \end{equation*}
    Therefore, for any $\|u\| = 1$ we have 
    \begin{equation*}
        \|\mathbf{J}(x) u\| \leq L_0 + L_1 \|F(x)\|^{\alpha}.
    \end{equation*}
    Hence we get 
    \begin{equation*}
        \|\mathbf{J}(x)\| = \sup_{\|u\| = 1}\|\mathbf{J}(x) u\| \leq L_0 + L_1 \|F(x)\|^{\alpha}.
    \end{equation*}
    Now we want to show the other way, i.e. suppose we have $\|\mathbf{J}(x)\| \leq L_0 + L_1 \|F(x)\|^{\alpha}$. For this we define, 
    \begin{equation*}
        q(\theta) \eqdef F(\theta x + (1 - \theta) y).      
    \end{equation*}
    Then $q(1) = F(x)$ and $q(0) = F(y)$ and we have
    \begin{eqnarray*}
        \|F(x) - F(y)\| & = & \|q(1) - q(0)\|\\
        & = & \left\| \int_0^1 \frac{dq(\theta)}{d\theta} d\theta \right\| \\
        & = & \left\| \int_0^1 \frac{dF(\theta x + (1 - \theta) y)}{d\theta} d\theta \right\| \\
        & = & \left\| \int_0^1 \mathbf{J}(\theta x + (1 - \theta) y) (x - y) d\theta \right\| \\
        & \leq  &  \int_0^1 \left\|\mathbf{J}(\theta x + (1 - \theta) y)\right\| \left\|x - y\right\| d\theta \\
        & = &  \left(\int_0^1 \left\|\mathbf{J}(\theta x + (1 - \theta) y)\right\|d\theta \right) \left\|x - y\right\| \\
        & \leq &  \left(\int_0^1L_0 + L_1 \|F(\theta x + (1 - \theta) y)\|^{\alpha} d\theta \right) \left\|x - y\right\| \\
        & = &  \left(L_0 + L_1 \int_0^1 \|F(\theta x + (1 - \theta) y)\|^{\alpha} d\theta \right) \left\|x - y\right\|.
    \end{eqnarray*}
    Then, using Lemma \ref{lemma:reform_integration}, we have the result. 
\end{proof}

\subsection{Computation of $\alpha, L_0, L_1$ for $\mathcal{L}(w_1, w_2)$.}\label{sec:compute_L0L1}

We now revisit the min-max problem defined in~\eqref{eq:min_max_cubic}. Note that, the operator corresponding to this problem is given by 
\begin{equation*}
    F(x) = \begin{bmatrix}
        w_1^2 + w_2 \\
        w_2^2 - w_1
    \end{bmatrix}
\end{equation*}
Then the norm of operator is $\| F(x)\| = \sqrt{\left(w_1^2 + w_2 \right)^2 + \left(w_2^2 - w_1 \right)^2}$. Moreover, the Jacobian matrix is given by
\begin{equation*}
    \mathbf{J}(x) = \begin{bmatrix}
        2 w_1 & 1 \\
        -1 & 2 w_2
    \end{bmatrix}.
\end{equation*}
Then the maximum singular value at any point $x$ is given by
\begin{eqnarray}
    \|\mathbf{J}(x)\| & = & \lambda_{\max} \left( \mathbf{J}(x)^{\top} \mathbf{J}(x) \right) \notag \\
    & = & \lambda_{\max} \left( \begin{bmatrix}
        4 w_1^2 + 1 & 2(w_1 - w_2) \\
        2(w_1 - w_2) & 4 w_2^2 + 1
    \end{bmatrix} \right)  \notag \\
    & \overset{\eqref{eq:eigen_2dmatrix}}{=} & \sqrt{2 (w_1^2 + w_2^2) + 1 + 2 \sqrt{(w_1 - w_2)^2 + (w_1^2 - w_2^2)^2}}
\end{eqnarray}

\begin{figure}
    \centering
     \includegraphics[width=0.6\linewidth]{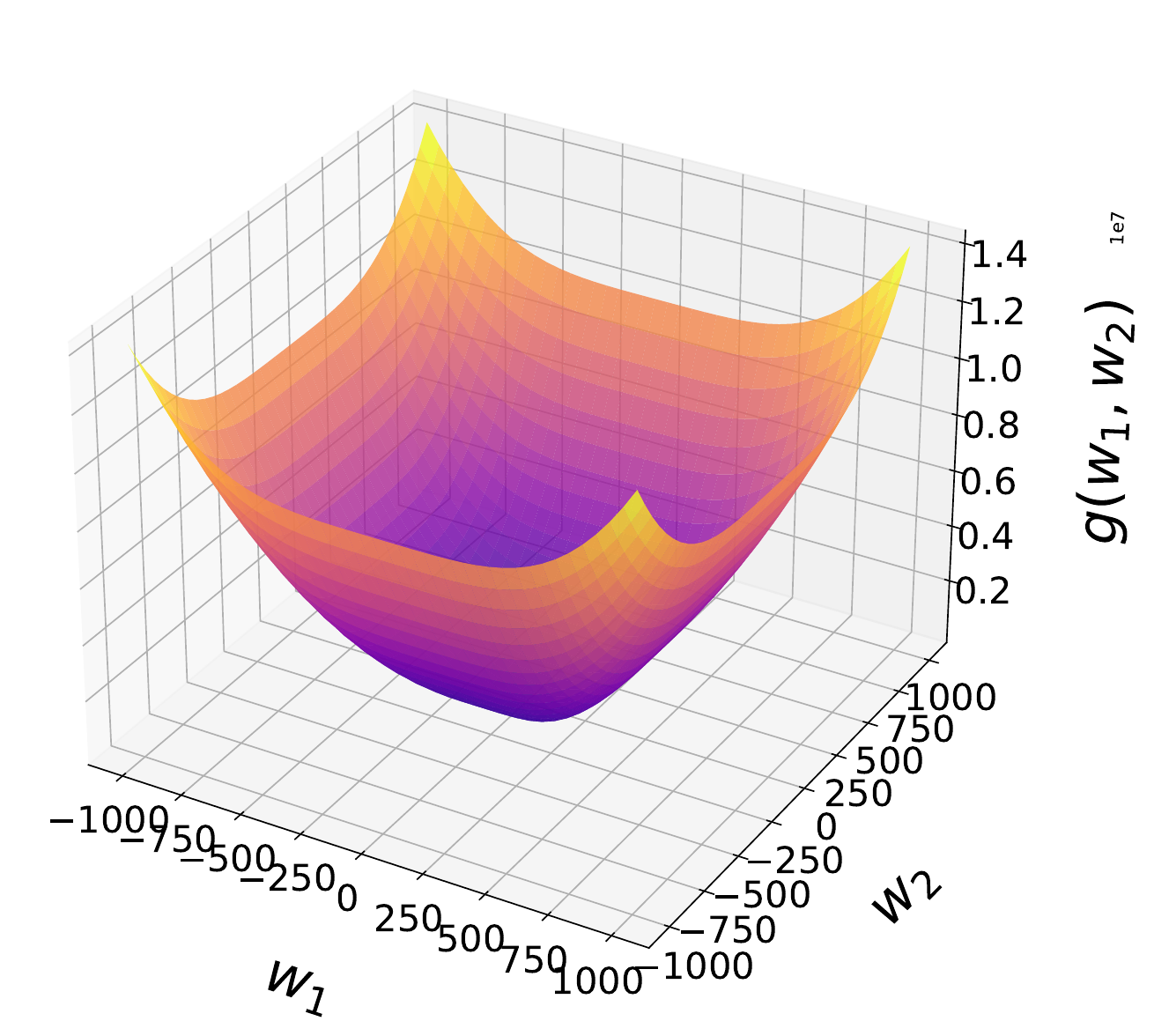}
     \caption{Plot of $g(w_1, w_2)$~\eqref{eq:compute_L0L1}. Here, the $z$-axis is in $10^7$ scale.}\label{fig:compute_L0L1}
\end{figure}

To validate whether the operator $F$ satisfies the condition~\eqref{eq:(L0,L1)-Lipschitz}, we examine whether the following function is non-negative:
\begin{equation}\label{eq:compute_L0L1}
\textstyle
    g(w_1, w_2) = L_0 + L_1 \|F(x)\| - \| \mathbf{J}(x)\|.
\end{equation}
In Figure~\ref{fig:compute_L0L1}, we plot $g(w_1, w_2)$ using $L_0 = 10$ and $L_1 = 10$. We observe that $g(w_1, w_2)$ has no real solution and remains positive for all $w_1, w_2 \in \mathbb{R}$, confirming that the function~\eqref{eq:min_max_cubic} satisfies~\eqref{eq:equiv_formulation} with $(\alpha, L_0, L_1) = (1, 10, 10)$. Thus, the corresponding operator $F$ is $1$-symmetric $(10, 10)$-Lipschitz.

\newpage
\section{Additional Details on Numerical Experiments}\label{appendix:num_exp}

In this section, we provide additional details on the second experiment related to the monotone problem. First, we show that $$\mathcal{L}(w_1, w_2) = \frac{1}{3} \left( w_1^{\top} \A w_1 \right)^{\nicefrac{3}{2}} + w_1^{\top} \B w_2 - \frac{1}{3} \left(w_2^{\top} \C w_2 \right)^{\nicefrac{3}{2}}$$ is convex-concave, and then we find the equilibrium point of $\mathcal{L}$.


\paragraph{Convex-Concave $\mathcal{L}(w_1, w_2)$.} Note that for $\mathcal{L}(w_1, w_2)$ in \eqref{eq:min_max_cubic_Rd} we have
$\nabla_{w_1} \mathcal{L}(w_1, w_2) = \left( w_1^{\top} \A w_1\right)^{\nicefrac{1}{2}} \A w_1 + \B w_2$ and $\nabla_{w_2} \mathcal{L}(w_1, w_2) = \B^{\top} w_1 - \left( w_2^{\top} \C w_2\right)^{\nicefrac{1}{2}} \C w_2$. Then the second-order derivatives are given by
\begin{equation*}
    \nabla_{w_1 w_1}^2 \mathcal{L} (w_1, w_2) = \| \A^{\nicefrac{1}{2}} w_1\| \A + \frac{\A w_1 w_1^{\top} \A^{\top}}{\| \A^{\nicefrac{1}{2}} w_1\|}
\end{equation*}
Here, $\A$ is positive definite and $\A w_1 w_1^{\top} \A^{\top}$ is a positive semidefinite matrix. Hence, $\nabla_{w_1 w_1}^2 \mathcal{L} (w_1, w_2)$ is a positive definite matrix as well and $\mathcal{L}(\cdot, w_2)$ is convex for any $w_2$. Similarly, we show that 
\begin{equation*}
    - \nabla_{w_2 w_2}^2 \mathcal{L} (w_1, w_2) = \| \C^{\nicefrac{1}{2}} w_2\| \C + \frac{\C w_2 w_2^{\top} \C^{\top}}{\| \C^{\nicefrac{1}{2}} w_2\|}
\end{equation*}
and $- \nabla_{w_2 w_2}^2 \mathcal{L} (w_1, w_2)$ is positive definite. Therefore, $\mathcal{L}(w_1, \cdot)$ is concave for any $w_1$. This proves that $\mathcal{L}(w_1, w_2)$ is convex with respect to $w_1$ and concave with respect to $w_2$. Thus, we conclude that the corresponding operator $F$ is monotone. 

\paragraph{Equilibrium of $\mathcal{L}(w_1, w_2)$.} To find the equilibrium points, we solve the set of equations given by $\nabla_{w_1} \mathcal{L}(w_1, w_2) = 0$ and $\nabla_{w_2} \mathcal{L}(w_1, w_2) = 0$, i.e., solve for 
\begin{eqnarray*}
    \left( w_1^{\top} \A w_1\right)^{\nicefrac{1}{2}} \A w_1 + \B w_2 & = & 0, \\
    \B^{\top} w_1 - \left( w_2^{\top} \C w_2\right)^{\nicefrac{1}{2}} \C w_2 & = & 0.
\end{eqnarray*}
Now multiplying the first equation with $w_1^{\top}$ and second one with $w_2^{\top}$, we have 
\begin{eqnarray*}
    \left( w_1^{\top} \A w_1\right)^{\nicefrac{1}{2}} w_1^{\top}\A w_1 + w_1^{\top} \B w_2 & = & 0 \\
    w_2^{\top}\B^{\top} w_1 - \left( w_2^{\top} \C w_2\right)^{\nicefrac{1}{2}} w_2^{\top}\C w_2 & = & 0
\end{eqnarray*}
Combining these two equations, we get
\begin{equation*}
    \left( w_1^{\top} \A w_1\right)^{\nicefrac{1}{2}} w_1^{\top}\A w_1 + \left( w_2^{\top} \C w_2\right)^{\nicefrac{1}{2}} w_2^{\top}\C w_2 = 0
\end{equation*}
which can be equivalently written as 
\begin{equation*}
    \left\| \A^{\nicefrac{1}{2}} w_1 \right\|^3 + \left\| \C^{\nicefrac{1}{2}} w_2 \right\|^3 = 0
\end{equation*}
which implies $w_1 = w_2 = 0$ as both $\A, \C$ are positive definite matrices (hence $\A^{\nicefrac{1}{2}}, \C^{\nicefrac{1}{2}}$ are invertible).

\chapter{Appendices for Chapter \ref{chap:chap-5}} \label{chap:appendix-d}

\section{Further Related Work for Chapter \ref{chap:chap-5}}
\label{apdx:related_work}
The references necessary to motivate chapter \ref{chap:chap-5} and connect it to the most relevant literature are included
in the appropriate sections of the main body of the paper. In this section, we present a broader view of
the literature, including more details on closely related work and more references to papers that are not
directly related to our main results.

\paragraph{Federated Learning.} One of the most popular characteristics of communication-efficient FL algorithms is local updates, where each client/node of the network takes multiple steps of the chosen optimization algorithm locally between the communication rounds. Many algorithms with local updates (e.g., FedAvg, Local GD, Local SGD, etc.) have been proposed and extensively analyzed in FL literature under various settings, including convex and nonconvex problems \cite{mcmahan2017communication,stich2018local, assran2019stochastic, kairouz2021advances,wang2021field,karimireddy2020scaffold,woodworth2020local, koloskova2020unified}.
In local FL algorithms, a big source of variance in the convergence guarantees stems from \emph{inter-client} differences (heterogeneous data). On that end, inspired by the popular variance reduction methods from optimization literature (e.g., SVRG~\cite{johnson2013accelerating} or SAGA~\cite{defazio2014:saga}), researchers start investigating variance reduction mechanisms to handle the variance related to heterogeneous environments (e.g. FedSVRG~\cite{konevcny2016federated}), resulting in the popular SCAFFOLD algorithm\cite{karimireddy2020scaffold}.
Gradient tracking~\cite{di2016next,nedic2017achieving} is another class of methods not affected by data-heterogeneity, but its provable communication complexity scales linearly in the condition number~\cite{koloskova2021improved,alghunaim2023enhanced}, even when combined with local steps \cite{liu2023decentralized}.
Finally, \cite{mishchenko2022proxskip} proposed a new algorithm (ProxSkip) for federated minimization problems that guarantees acceleration of communication complexity under heterogeneous data. As we mentioned in Section~\ref{MainCont}, the framework of our paper includes the algorithm and convergence results of~\cite{mishchenko2022proxskip} as a special case, and it could be seen as an extension of these ideas in the distributed variational inequality setting. 

\paragraph{Minimax Optimization and Variational Inequality Problems.} 
Minimax optimization, and more generally, variational inequality problems (VIPs)~\cite{hartman1966some,facchinei2003finite} appear in various research areas, including but not limited to online learning~\cite{cesa2006prediction}, game theory~\cite{von1947theory}, machine learning~\cite{goodfellow2014generative} and social and economic practices~\cite{facchinei2003finite,parise2019variational}. The gradient descent-ascent (GDA) method is one of the most popular algorithms for solving minimax problems. However, GDA fails to converge for convex-concave minimax problems \cite{mescheder2018training} or bilinear games \cite{gidel2019negative}. To avoid these issues,  \cite{korpelevich1976extragradient} introduced the extragradient method (EG), while \cite{popov1980modification} proposed the optimistic gradient method (OG). Recently, there has been a surge of developments for improving the extra gradient methods for better convergence guarantees. For example, \cite{mokhtari2019convergence,mokhtari2020unified} studied optimistic and extragradient methods as an approximation of the proximal point method to solve bilinear games, while \cite{hsieh2020explore} focused on the stochastic setting and proposed the use of a larger extrapolation step size compared to update step size to prove convergence under error-bound conditions. 

Beyond the convex-concave setting, most machine learning applications expressed as min-max optimization problems involve nonconvex-nonconcave objective functions. Recently, \cite{daskalakis2021complexity} showed that finding local solutions is intractable in the nonconvex-nonconcave regime, which motivates researchers to look for additional structures on problems that can be exploited to prove the convergence of the algorithms. For example, \cite{loizou2021stochastic} provided convergence guarantees of stochastic GDA under expected co-coercivity,
while \cite{yang2020global} studied the convergence stochastic alternating GDA under the PL-PL condition.
\cite{lin2020gradient} showed that solving nonconvex-(strongly)-concave minimax problems using GDA by appropriately maintaining the primal and dual stepsize ratio is possible, and several follow-up works improved the computation complexities~\cite{yang2020catalyst,zhang2021complexity,yang2022faster,zhang2022SAPD}. 
Later \cite{diakonikolas2021efficient} introduced the notion of the weak minty variational inequality (MVI), which captures a large class of nonconvex-nonconcave minimax problems. \cite{gorbunov2022convergence} provided tight convergence guarantees for solving weak MVI using deterministic extragradient and optimistic gradient methods. However, deterministic algorithms can be computationally expensive, encouraging researchers to look for stochastic algorithms. On this end, \cite{diakonikolas2021efficient,pethick2023escaping,bohm2022solving} analyze stochastic extragradient and optimistic gradient methods for solving weak MVI with increasing batch sizes. Recently, \cite{pethick2023solving} used a bias-corrected variant of the extragradient method to solve weak MVI without increasing batch sizes. Lately, variance reduction methods has also proposed for solving min-max optimization problems and VIPs~\cite {alacaoglu2022stochastic,cai2022stochastic}. As we mentioned in the main paper, in this chapter, the proposed convergence guarantees hold for a class of non-monotone problems, i.e., the class of is $\mu$-quasi-strongly monotone and $\ell$-star-cocoercive operators (see Assumption~\ref{assume:main}). 


\paragraph{Minimax Federated Learning.}
Going beyond the more classical centralized setting, recent works study the two-player federated minimax problems \cite{deng2021local,beznosikov2020distributed,sun2022communication,hou2021efficient,sharma2022federated,tarzanagh2022fednest,huang2022adaptive,yang2022sagda}. 
For example, \cite{deng2021local} studied the convergence guarantees of Local SGDA, which is an extension of Local SGD in the minimax setting, under various assumptions; \cite{sharma2022federated} studied furthered and improved the complexity results of Local SGDA in the nonconvex case; \cite{beznosikov2020distributed} studied the convergence rate of Local SGD algorithm with an extra step (we call it Local Stochastic Extragradient or Local SEG), under the (strongly)-convex–(strongly)-concave setting.
Multi-player games, which, as we mentioned in the main paper, can be formulated as a special case of the VIP, are well studied in the context of game theory \cite{rosen1965existence,cai2011minmax,yi2017distributed,chen2023global}. In this chapter, by studying the more general distributed VIPs, our proposed federated learning algorithms can also solve multi-player games. 

Finally, as we mentioned in the main paper, the algorithmic design of our methods is inspired by the proposed algorithm, ProxSkip of~\cite{mishchenko2022proxskip}, for solving composite minimization problems. However, we should highlight that since ~\cite{mishchenko2022proxskip} focuses on solving optimization problems, the function suboptimality $[f(x^k)-f(x_*)]$ is available (a concept that cannot be useful in the VIP setting. Thus, the difference in the analysis between the two papers begins at a deeper conceptual level. In addition, this chapter provides a unified algorithm framework under a general estimator setting (Assumption~\ref{assume:stochastic}) in the VIP regime that captures variance-reduced GDA and the convergence guarantees from \cite{beznosikov2022stochastic} as a special case.  Compared to existing works on federated minimax optimization~\cite{beznosikov2020distributed,deng2021local}, our analysis provides improved communication complexities and avoids the restrictive (uniform) bounded heterogeneity/variance assumptions (see discussion in Section~\ref{asdas}).

\newpage

\section{Further Pseudocodes}\label{apdx:pseudocodes}
Here we present the pseudocodes of the algorithms that due to space limitation did not fit into the main paper. 

We present the \algname{ProxSkip-L-SVRGDA} method in Algorithm~\ref{alg:Stoc-ProxSkip-L-SVRGDA}, which can be seen as a ProxSkip generalization of the L\nobreakdash-SVRGDA algorithm proposed in \cite{beznosikov2022stochastic}. For more details, please refer to Section~\ref{sec:centralized_special_case}. 
Note that the unbiased estimator in this method has the form:
$g_k=F_{j_k}(x_k)-F_{j_k}(w_k)+F(w_k).$

\begin{algorithm}[h]
    \caption{\algname{ProxSkip-L-SVRGDA}}
    \label{alg:Stoc-ProxSkip-L-SVRGDA}
    \begin{algorithmic}[1]
        \REQUIRE Initial point $x_0, h_0$, parameters $\gamma$, probabilities $p,q\in[0,1]$,
        number of iterations $K$
        \STATE Set $w_0=x_0$, compute $F(w_0)$
        \FORALL{$k = 0,1,..., K$}
            \STATE Construct $\colorword{g_k}{red}=F_{j_k}(x_k)-F_{j_k}(\colorword{w_k}{blue})+F(\colorword{w_k}{blue})$, $j_k\in[n]$
            \STATE Update $\colorword{w_{k+1}}{blue}=
            \begin{cases}
                x_k & \text{w.p. } q\\
                \colorword{w_k}{blue} & \text{w.p. } 1-q
            \end{cases}$
            \STATE $\widehat{x}_{k+1}=x_k-\gamma (\colorword{g_k}{red}-h_k)$
            \STATE Flip a coin $\theta_k$, $\theta_k=1$ w.p. $p$, otherwise $0$
            \IF{$\theta_k=1$}
            \STATE $x_{k+1}=\prox_{\frac{\gamma}{p} R}\autopar{\widehat{x}_{k+1}-\frac{\gamma}{p} h_k}$
            \ELSE
            \STATE $x_{k+1}=\widehat{x}_{k+1}$
            \ENDIF
            \STATE $h_{k+1}=h_k+\frac{p}{\gamma}(x_{k+1}-\widehat{x}_{k+1})$ 
        \ENDFOR
        \ENSURE $x_k$
    \end{algorithmic}
\end{algorithm}

Moreover, we present the \algname{ProxSkip-L-SVRGDA-FL} algorithm below in Algorithm~\ref{alg:ProxSkip-L-SVRGDA-FL}. Similar to the relationship between \algname{ProxSkip-SGDA} and \algname{ProxSkip-SGDA-FL}, here this algorithm is the implementation of the \algname{ProxSkip-L-SVRGDA} method mentioned above in the FL regime. For more details, please check Section~\ref{asdas} of the main paper.

\begin{algorithm}[h]
    \caption{\algname{ProxSkip-L-SVRGDA-FL}}
    \label{alg:ProxSkip-L-SVRGDA-FL}
    \begin{algorithmic}[1]
        \REQUIRE Initial points $x_{1,0},\cdots,x_{n,0}\in\mathbb{R}^{d'}$, $\gamma, p\in\mathbb{R}$, initial control variates $h_{1,0}, \cdots, h_{n,0}=0\in\mathbb{R}^{d'}$,
        iteration number $K$
        \STATE Set $\colorword{w_{i,0}}{blue}=x_{i,0}$ for every worker $i\in[n]$
        \FORALL{$k = 0,1,..., K$}
            \STATE \textbf{Server:} 
            Flip two coins $\theta_k$ and $\zeta_k$, where $\theta_k = 1$ w.p. $p$ and $\zeta_k=1$ w.p. $q$, otherwise 0.
            Send $\theta_t$ and $\zeta_k$ to all workers
            \FOR{{\bf each workers $i\in\automedpar{n}$ in parallel}} 
                \STATE $\colorword{g_{i,k}}{red}=F_{i, j_k}(x_{i,k})-F_{i, j_k}(\colorword{w_{i,k}}{blue})+F_i(\colorword{w_{i,k}}{blue})$, where $j_k\sim\text{Unif}([m_i])$
                \STATE Update $\colorword{w_{i, k+1}}{blue}=
                \begin{cases}
                    x_{i,k} & \text{if } \zeta_k=1\\
                    \colorword{w_{i,k}}{blue} & \text{otherwise}
                \end{cases}$
                \STATE $\widehat{x}_{i, k+1}=x_{i, k}-\gamma (\colorword{g_{i,k}}{red}-h_{i,k})$             
                \IF{$\theta_k=1$}
                \STATE Worker: $x_{i, k+1}'=\widehat{x}_{i, k+1}-\frac{\gamma}{p} h_{i,k}$, sends $x_{i, k+1}'$ to the server
                \STATE Server: computes $x_{i, k+1}=\frac{1}{n}\sum_{i=1}^n x_{i, k+1}'$, sends $x_{i, k+1}$ to workers
                \COMMENT{Communication}
                \ELSE
                \STATE $x_{i, k+1}=\widehat{x}_{i, k+1}$
                \COMMENT{Otherwise skip the communication step}
                \ENDIF
                \STATE $h_{i, k+1} = h_{i, k}+\frac{p}{\gamma}(x_{i, k+1}-\widehat{x}_{i, k+1})$ 
            \ENDFOR
        \ENDFOR
        \ENSURE $x_K$
    \end{algorithmic}
\end{algorithm}

\newpage

\section{Useful Lemmas}\label{apdx:useful_lemmas}
\begin{lemma}[Young's Inequality]
    \begin{eqnarray}
        \|a + b\|^2 \leq 2\|a\|^2 + 2 \|b\|^2 \label{eq:YoungsInequality} 
    \end{eqnarray}
\end{lemma}
\label{apdx:useful_lemma}
\begin{lemma}
    For any optimal solution $x_*$ of \eqref{eq:objective_FL_reformulation} and any $\alpha>0$, we have
    \begin{equation}
        x_*=\prox_{\alpha R}\autopar{x_*-\alpha F(x_*)}.
    \end{equation}
\end{lemma}


\begin{lemma}[Firm Nonexpansivity of the Proximal Operator \cite{beck2017first}]
    \label{lm:firm_nonexpansive_proximal}
    Let $f$ be a proper closed and convex function, then for any $x, y\in\mathbb{R}^d$ we have
    \begin{equation}\label{eq:firm_nonexpansive_proximal_eq1}
        \autoprod{x-y, \prox_{f}(x)-\prox_{f}(y)}\geq\autonorm{\prox_{f}(x)-\prox_{f}(y)}^2,
    \end{equation}
    or equivalently,
    \begin{equation}\label{eq:firm_nonexpansive_proximal_eq2}
        \autonorm{\autopar{x-\prox_{f}(x)}-\autopar{y-\prox_{f}(y)}}^2+\autonorm{\prox_{f}(x)-\prox_{f}(y)}^2\leq
        \autonorm{x-y}^2.
    \end{equation}
\end{lemma}
\newpage
\section{Proofs of Lemmas \ref{thm:property_local_estimator} and \ref{lm:ABC_SVRG}}\label{apdx:further_lemmas}
As we mentioned in the main paper, the Lemmas \ref{thm:property_local_estimator} and \ref{lm:ABC_SVRG} have been proved in~\cite{beznosikov2022stochastic}. We include the proofs of these results using our notation for completeness. 


\begin{proof}[Proof of Lemma \ref{thm:property_local_estimator}]
    Note that
    \begin{equation*}
        \begin{split}
            \EE\autonorm{g_k-F(x_*)}^2
            \overset{\eqref{eq:YoungsInequality}}{\leq}\ &
            2\EE\autonorm{g_k-g(x_*)}^2+2\EE\autonorm{g(x_*)-F(x_*)}^2\\
            \leq\ &
            2L_g\autoprod{F(x_k)-F(x_*),x_k-x_*}+2\sigma_*^2,
        \end{split}
    \end{equation*}
    where the second inequality uses Assumption~\ref{assume:ECC}. The statement of Lemma \ref{thm:property_local_estimator} is obtained by comparing the above inequality with the expression of Assumption~\ref{assume:stochastic}.
\end{proof}

\begin{proof}[Proof of Lemma \ref{lm:ABC_SVRG}]
    For the property of $g_k$, it is easy to find that the unbiasedness holds. Then note that
    \begin{eqnarray*}
            &&\mathbb{E}\autonorm{g_k-F(x_*)}^2\\
            &=&
            \mathbb{E}\autonorm{F_{j_k}(x_k)-F_{j_k}(w_k)+F(w_{j_k})-F(x_*)}^2\\
            &=&
            \frac{1}{n}\sum_{i=1}^n\autonorm{F_i(x_k)-F_{j_k}(w_k)+F(w_{j_k})-F(x_*)}^2\\
            &\overset{\eqref{eq:YoungsInequality}}{\leq} &
            \frac{2}{n}\sum_{i=1}^n\autonorm{F_i(x_k)-F_i(x_*)}^2+\autonorm{F_i(x_*)-F_i(w_k)-\autopar{F(x_*)-F(w_k)}}^2\\
            &\leq &
            \frac{2}{n}\sum_{i=1}^n\autonorm{F_i(x_k)-F_i(x_*)}^2+\autonorm{F_i(x_*)-F_i(w_k)}^2\\
            &\leq &
            2\widehat{\ell}\autoprod{F(x_k)-F(x_*), x_k-x_*}+\frac{2}{n}\sum_{i=1}^n\autonorm{F_i(x_*)-F_i(w_k)}^2\\
            &= &
            2\widehat{\ell}\autoprod{F(x_k)-F(x_*), x_k-x_*}+2\sigma_k^2,
    \end{eqnarray*}
    here the second inequality comes from the fact that $\text{Var}(X)\leq\EE (X^2)$, and the third inequality comes from Assumption~\ref{assume:average_star_coco}. Then for the second term above, we have
    \begin{equation*}
        \begin{split}
            \mathbb{E}\automedpar{\sigma_{k+1}^2}
            =\ &
            \frac{1}{n}\sum_{i=1}^n\autonorm{F_i(w_{k+1})-F_i(x_*)}^2\\
            =\ &
            \frac{1}{n}\sum_{i=1}^n\autopar{q\autonorm{F_i(x_k)-F_i(x_*)}^2+(1-q)\autonorm{F_i(w_k)-F_i(x_*)}^2}\\
            \leq\ &
            q\widehat{\ell}\autoprod{F(x_k)-F(x_*), x_k-x_*}+\frac{1-q}{n}\sum_{i=1}^n\autonorm{F_i(w_k)-F_i(x_*)}^2\\
            =\ &
            q\widehat{\ell}\autoprod{F(x_k)-F(x_*), x_k-x_*}+(1-q)\sigma_k^2,
        \end{split}
    \end{equation*}
    here the second equality applies the definition of $w_{k+1}$, and the first inequality is implied by Assumption~\ref{assume:ECC}. 
    The statement of Lemma \ref{lm:ABC_SVRG} is obtained by comparing the above inequalities with the expressions of Assumption~\ref{assume:stochastic}.
\end{proof}

\newpage
\section{Proofs of Main Convergence Results}\label{apdx:convergence_results}
\subsection{Proof of Theorem \ref{thm:convergence_Stoc-ProxSkip-VIP}}
\label{apdx:theorem_Stoc_ProxSkip_VIP}
Here the proof originates from that of ProxSkip~\cite{mishchenko2022proxskip}, while we extend it to the variational inequality setting, and combines it with the more general setting on the stochastic oracle.

\begin{proof}
    Note that following Algorithm \ref{alg:Stoc-ProxSkip-VIP}, we have with probability $p$:
    \begin{equation*}
        \begin{cases}
            x_{k+1}=\prox_{\frac{\gamma}{p} R}\autopar{\widehat{x}_{k+1}-\frac{\gamma}{p} h_k}\\
            h_{k+1}=h_k+\frac{p}{\gamma}\autopar{\prox_{\frac{\gamma}{p} R}\autopar{\widehat{x}_{k+1}-\frac{\gamma}{p} h_k}-\widehat{x}_{k+1}}
        \end{cases}
    \end{equation*}
    and with probability $1-p$:
    \begin{equation*}
        \begin{cases}
            x_{k+1}=\widehat{x}_{k+1}\\
            h_{k+1}=h_k,
        \end{cases}
    \end{equation*}
    and set
    \begin{equation*}
        V_k\triangleq\autonorm{x_k-x_*}^2+\autopar{\frac{\gamma}{p}}^2\autonorm{h_k- F(x_*)}^2+M\gamma^2\sigma_k^2,
    \end{equation*}
    For simplicity, we denote $P(x_k)\triangleq\prox_{\frac{\gamma}{p} R}\autopar{\widehat{x}_{k+1}-\frac{\gamma}{p} h_k}$, so we have
    \begin{eqnarray*}
        \mathbb{E}\automedpar{V_{k+1}} &=&
            p\autopar{\autonorm{P(x_k)-x_*}^2+\autopar{\frac{\gamma}{p}}^2\autonorm{h_k+\frac{p}{\gamma}\autopar{P(x_k)-\widehat{x}_{k+1}}- F(x_*)}^2}\\
            &&\qquad
            + (1-p)\autopar{\autonorm{\widehat{x}_{k+1}-x_*}^2+\autopar{\frac{\gamma}{p}}^2\autonorm{h_k- F(x_*)}^2}+M\gamma^2\sigma_{k+1}^2\\
            &=&
            p\autopar{\autonorm{P(x_k)-x_*}^2+\autonorm{P(x_k)-(\widehat{x}_{k+1}-\frac{\gamma}{p}h_k)- \frac{\gamma}{p}F(x_*)}^2}\\
            &&\qquad + (1-p)\autopar{\autonorm{\widehat{x}_{k+1}-x_*}^2+\autopar{\frac{\gamma}{p}}^2\autonorm{h_k- F(x_*)}^2}+M\gamma^2\sigma_{k+1}^2
    \end{eqnarray*}
    next note that $x_*=\prox_{\frac{\gamma}{p} R}\autopar{x_*-\frac{\gamma}{p} F(x_*)}$, we have
    \begin{eqnarray*}
        &&\autonorm{P(x_k)-(\widehat{x}_{k+1}-\frac{\gamma}{p}h_k)- \frac{\gamma}{p}F(x_*)}^2\\
        &= &\autonorm{P(x_k)-(\widehat{x}_{k+1}-\frac{\gamma}{p}h_k)- \automedpar{\prox_{\frac{\gamma}{p} R}\autopar{x_*-\frac{\gamma}{p} F(x_*)}-(x_*-\frac{\gamma}{p}F(x_*))}}^2
    \end{eqnarray*}

    so by Lemma \ref{lm:firm_nonexpansive_proximal}, we have
    \begin{eqnarray*}
        \mathbb{E}\automedpar{V_{t+1}} &
        \overset{\eqref{eq:firm_nonexpansive_proximal_eq2}}{\leq} & p\autonorm{\widehat{x}_{k+1}-\frac{\gamma}{p}h_k-x_*+\frac{\gamma}{p} F(x_*)}^2 \\
        && + (1-p)\autopar{\autonorm{\widehat{x}_{k+1}-x_*}^2+\autopar{\frac{\gamma}{p}}^2\autonorm{h_k- F(x_*)}^2} +M\gamma^2\sigma_{t+1}^2\\
        &=& \autonorm{\widehat{x}_{k+1}-x_*}^2+\autopar{\frac{\gamma}{p}}^2\autonorm{h_k- F(x_*)}^2-2\frac{\gamma}{p}p\autoprod{\widehat{x}_{k+1}-x_*, h_k-F(x_*)} \\
        && +M\gamma^2\sigma_{t+1}^2,
    \end{eqnarray*}
    let 
    \begin{equation*}
        w_k\triangleq x_k-\gamma g(x_k),\quad
        w^*\triangleq x_*-\gamma F(x_*),
    \end{equation*}
    so we have
    \begin{eqnarray*}
        &&\autonorm{\widehat{x}_{k+1}-x_*}^2-2\frac{\gamma}{p}p\autoprod{\widehat{x}_{k+1}-x_*, h_k-F(x_*)}\\
        &=& \autonorm{w_k-w^*+\gamma\autopar{h_k- F(x_*)}}^2-2\gamma\autoprod{w_k-w^*+\gamma\autopar{h_k- F(x_*)}, h_k-F(x_*)}\\
        &=& \autonorm{w_k-w^*}^2-\gamma^2\autonorm{h_k- F(x_*)}^2\\
        &=& \autonorm{w_k-w^*}^2-p^2\autopar{\frac{\gamma}{p}}^2\autonorm{h_k- F(x_*)}^2,
    \end{eqnarray*}
    so we have
    \begin{equation}
        \label{eq:V-t+1_intermediate}
        \mathbb{E}\automedpar{V_{t+1}}
        \leq
        \autonorm{w_k-w^*}^2+\autopar{1-p^2}\autopar{\frac{\gamma}{p}}^2\autonorm{h_k- F(x_*)}^2+M\gamma^2\sigma_{t+1}^2.
    \end{equation}
    Then by the standard analysis on GDA, we have
    \begin{eqnarray*}
        \autonorm{w_k-w^*}^2 
        &=& 
        \autonorm{x_k-x_*-\gamma \autopar{g(x_k)-F(x_*)}}^2\\
        &=&
        \autonorm{x_k-x_*}^2-2\gamma\autoprod{g(x_k)-F(x_*), x_k-x_*}+\gamma^2 \autonorm{g(x_k)-F(x_*)}^2,
    \end{eqnarray*}
    take the expectation, by Assumption~\ref{assume:stochastic}, we have
    \begin{equation*}
        \begin{split}
            \mathbb{E}\automedpar{\autonorm{w_k-w^*}^2}
            =\ &
            \autonorm{x_k-x_*}^2-2\gamma\autoprod{F(x_k)-F(x_*), x_k-x_*}+\gamma^2 \mathbb{E}\automedpar{\autonorm{g(x_k)-F(x_*)}^2}\\
            \leq\ &
            \autonorm{x_k-x_*}^2-2\gamma\autopar{1-\gamma A}\autoprod{F(x)-F(x_*), x-x_*}+\gamma^2\autopar{B\sigma_k^2+D_1},
        \end{split}
    \end{equation*}
    substitute the above result into \eqref{eq:V-t+1_intermediate}, we have
    \begin{eqnarray*}
    \mathbb{E}\automedpar{V_{t+1}} & \overset{\eqref{eq:V-t+1_intermediate}}{\leq} &
            \mathbb{E}\automedpar{\autonorm{w_k-w^*}^2+\autopar{1-p^2}\autopar{\frac{\gamma}{p}}^2\autonorm{h_k- F(x_*)}^2+M\gamma^2\sigma_{t+1}^2}\\
            &\leq& \mathbb{E}\Big[\autonorm{x_k-x_*}^2-2\gamma\autopar{1-\gamma A}\autoprod{F(x)-F(x_*), x-x_*}+\gamma^2\autopar{B\sigma_k^2+D_1}\\
            && +\autopar{1-p^2}\autopar{\frac{\gamma}{p}}^2\autonorm{h_k- F(x_*)}^2 +M\gamma^2\sigma_{t+1}^2 \Big]\\
            &\leq &
            \mathbb{E}\Big[\autonorm{x_k-x_*}^2-2\gamma\autopar{1-\gamma (A+MC)}\autoprod{F(x)-F(x_*), x-x_*}\\
            && +M\gamma^2(1-\rho)\sigma_k^2 +\gamma^2\autopar{B\sigma_k^2+D_1+MD_2} \\
            && +\autopar{1-p^2}\autopar{\frac{\gamma}{p}}^2\autonorm{h_k- F(x_*)}^2\Big]\\
            &\leq &
            \mathbb{E}\Big[\autopar{1-2\gamma\mu\autopar{1-\gamma (A+MC)}}\autonorm{x_k-x_*}^2 \\
            && +\autopar{1-p^2}\autopar{\frac{\gamma}{p}}^2\autonorm{h_k- F(x_*)}^2 +M\gamma^2\autopar{1-\rho+\frac{B}{M}}\sigma_k^2 \\
            && +\gamma^2\autopar{D_1+MD_2}\Big],
    \end{eqnarray*}
    here the third inequality comes from Assumption~\ref{assume:stochastic} on $\sigma_{t+1}^2$, and the fourth inequality comes from quasi-strong monotonicity in Assumption~\ref{assume:main}.
    Recall that $\gamma\leq\frac{1}{2(A+MC)}$, we have
    \begin{eqnarray*}
        \mathbb{E}\automedpar{V_{t+1}} &\leq& \mathbb{E}\automedpar{\autopar{1-\gamma\mu}\autonorm{x_k-x_*}^2+\autopar{1-p^2}\autopar{\frac{\gamma}{p}}^2\autonorm{h_k- F(x_*)}^2}\\
        && + \mathbb{E}\automedpar{\autopar{1-\rho+\frac{B}{M}}M\gamma^2\sigma_k^2+\gamma^2\autopar{D_1+MD_2}}\\
        &\leq& \mathbb{E}\automedpar{\autopar{1-\min\autobigpar{\gamma\mu, p^2, \rho-\frac{B}{M} }}V_k} \\
        && +\gamma^2\autopar{D_1+MD_2},
    \end{eqnarray*}
    by taking the full expectation and telescoping, we have
    \begin{eqnarray*}
        \mathbb{E}\automedpar{V_K}&\leq&
        \autopar{1-\min\autobigpar{\gamma\mu, p^2, \rho-\frac{B}{M} }}\mathbb{E}\automedpar{V_{T-1}}+\gamma^2\autopar{D_1+MD_2}\\
        &=&
        \autopar{1-\min\autobigpar{\gamma\mu, p^2, \rho-\frac{B}{M} }}^TV_0 \\
        && +\gamma^2\autopar{D_1+MD_2}\sum_{i=0}^{T-1}\autopar{1-\min\autobigpar{\gamma\mu, p^2, \rho-\frac{B}{M} }}^i\\
        &\leq& \autopar{1-\min\autobigpar{\gamma\mu, p^2, \rho-\frac{B}{M} }}^TV_0+\frac{\gamma^2\autopar{D_1+MD_2}}{\min\autobigpar{\gamma\mu, p^2, \rho-\frac{B}{M} }},
    \end{eqnarray*}
    the last inequality comes from the computation of a geometric series. which concludes the proof. Note that here we implicitly require that $\gamma\leq\frac{1}{\mu}$ and $M>\frac{B}{\rho}$. 
\end{proof}

\subsection{Proof of Corollary \ref{cor:Stoc-ProxSkip-VIP-Complexity}}
\label{apdx:cor_ProxSkip_VIP_complexity}

\begin{proof}
    With the above setting, we know that
    \begin{equation*}
        \min\autobigpar{\gamma\mu, p^2, \rho-\frac{B}{M} }
        =
        \gamma\mu,
    \end{equation*}
    and
    \begin{equation*}
        \mathbb{E}\automedpar{V_K}
        \leq
        \autopar{1-\gamma\mu}^TV_0+\frac{\gamma\autopar{D_1+\frac{2B}{\rho} D_2}}{\mu},
    \end{equation*}
    so it is easy to see that by setting
    \begin{equation*}
        T\geq\frac{1}{\gamma\mu}\ln\autopar{\frac{2V_0}{\epsilon}},
        \quad
        \gamma\leq\frac{\mu\epsilon}{2\autopar{D_1+\frac{2B}{\rho} D_2}},
    \end{equation*}
    we have
    \begin{equation*}
        \mathbb{E}\automedpar{V_K}
        \leq
        \epsilon,
    \end{equation*}
    which induces the iteration complexity to be
    \begin{equation*}
        T\geq
        \max\autobigpar{1, \frac{2(A+2BC/\rho)}{\mu}, \frac{2}{\rho}, \frac{2\autopar{D_1+\frac{2B}{\rho} D_2}}{\mu^2\epsilon}}\ln\autopar{\frac{2V_0}{\epsilon}},
    \end{equation*}
    and the corresponding number of calls to the proximal oracle is
    \begin{equation*}
        pT
        \geq
        \sqrt{\max\autobigpar{1, \frac{2(A+2BC/\rho)}{\mu}, \frac{2}{\rho}, \frac{2\autopar{D_1+\frac{2B}{\rho} D_2}}{\mu^2\epsilon}}}\ln\autopar{\frac{2V_0}{\epsilon}},
    \end{equation*}
    which concludes the proof.
\end{proof}

\subsection{Properties of Operators in VIPs}
\label{apdx:thm_FL_Operator_Check}

To make sure that the consensus form Problem~\eqref{eq:objective_FL_reformulation} fits with Assumption~\ref{assume:main}, and the corresponding operator estimators satisfies Assumption~\ref{assume:stochastic}, we provide the following results.

\begin{proposition}
    \label{prop:tranform_equiv_sol}
    If Problem \eqref{eq:FedVIP} attains an unique solution $z^*\in\mathbb{R}^{d'}$, then Problem \eqref{eq:objective_FL_reformulation} attains an unique solution $x_*\triangleq(z^*, z^*, \cdots, z^*)\in\mathbb{R}^d$, and vice versa.
\end{proposition}
\begin{proof}
    Note that each $f_i$ is $\mu$-quasi-strongly monotone and $z^*$ is the unique solution to Problem \eqref{eq:FedVIP}.
    So for the operator $F$ in the reformulation \eqref{eq:objective_FL_reformulation}, first we check the point $x_*=(z^*, z^*, \cdots, z^*)$, note that for any $x=(x_1, x_2, \cdots, x_n)\in\mathbb{R}^d$, we have
    \begin{eqnarray*}
         \autoprod{F(x_*), x-x_*}+R(x)-R(x_*)&=&
            \autoprod{\sum_{i=1}^n F_i(x_*), x-x_*}+R(x)\\
            &=&\sum_{i=1}^n\autoprod{F_i(x_*), x-x_*}+R(x)\\
            &=&\sum_{i=1}^n\autoprod{f_i(z^*), x_i-z^*}+R(x),
    \end{eqnarray*}
    here the first equation incurs the definition of $R$, the third equation is due to the definition that $F_i$.
    Then for any $x\in\mathbb{R}^d$, if $\exists~x_i\neq x_j$ for some $i, j\in[n]$, we have $R(x)=+\infty$, so the RHS above is always positive. Then if $x_1=x_2=\cdots=x_n=x'\in\mathbb{R}^{d'}$, we have
    \begin{equation*}
        \sum_{i=1}^n\autoprod{f_i(z^*), x_i-z^*}+R(x)
        =
        \sum_{i=1}^n\autoprod{f_i(z^*), x'-z^*}
        =
        \autoprod{\sum_{i=1}^nf_i(z^*), x'-z^*}
        \geq
        0,
    \end{equation*}
    where the last inequality comes from the fact that $z^*$ is the solution to Problem \eqref{eq:FedVIP}. So $x_*$ is the solution to Problem \eqref{eq:objective_FL_reformulation}. It is easy to show its uniqueness by contradiction and the uniqueness of $z^*$, which we do not detail here.

    On the opposite side, first it is easy to see that the solution to \eqref{eq:objective_FL_reformulation} must come with the form $x_*=(z^*, z^*, \cdots, z^*)$, then we have for any $x=(x_1, x_2, \cdots, x_n)\in\mathbb{R}^d$
    \begin{equation*}
        \autoprod{F(x_*), x-x_*}+R(x)-R(x_*)
        =
        \sum_{i=1}^n\autoprod{f_i(z^*), x_i-z^*}+R(x),
    \end{equation*}
    then we select $x=(x', x', \cdots, x')$ for any $x'\in\mathbb{R}^{d'}$, so we have
    \begin{equation*}
        \autoprod{F(x_*), x-x_*}+R(x)-R(x_*)
        =
        \sum_{i=1}^n\autoprod{f_i(z^*), x'-z^*}
        =
        \autoprod{F(z^*), x'-z^*}\geq 0,
    \end{equation*}
    which corresponds to the solution of \eqref{eq:FedVIP}, and concludes the proof.
    \end{proof}

    \begin{proposition}
        \label{prop:FL_Operator_Check}
        With Assumption \ref{assume:additional_FL}, the Problem \eqref{eq:objective_FL_reformulation} attains an unique solution $x_*\triangleq(z^*, z^*, \cdots, z^*)\in\mathbb{R}^d$, the operator $F$ is $\mu$-quasi-strongly monotone, and the operator $g(x)\triangleq(g_1(x_1), g_2(x_2), \cdots, g_n(x_n))$ is an unbiased estimator of $F$, and it satisfies $L_g$-expected cocoercivity defined in Assumption \ref{assume:stochastic}.
    \end{proposition}
    \begin{proof}
    The uniqueness result comes from the above proposition. Then note that
    \begin{eqnarray*}
        \autoprod{F(x)-F(x_*), x-x_*} &= &
            \autoprod{\sum_{i=1}^n\autopar{F_i(x)-F_i(x_*)}, x-x_*} \\
            &=&\sum_{i=1}^n\autoprod{F_i(x)-F_i(x_*), x-x_*}\\
            &=& \sum_{i=1}^n\autoprod{f_i(x_i)-f_i(z^*), x_i-z^*}\\
            &\geq&\sum_{i=1}^n\mu\autonorm{x_i-z^*}^2 \\
            &=& \mu\autonorm{x-x_*}^2,
    \end{eqnarray*}
    which verifies the second statement. For the last statement, following the definition of the estimator $g$, 
    \begin{equation*}
        \mathbb{E}g(x_*)
        =
        \EE
        \begin{pmatrix}
            g_1(z^*)\\
            g_2(z^*)\\
            \vdots\\
            g_n(z^*)
        \end{pmatrix}
        =
        \begin{pmatrix}
            f_1(z^*)\\
            f_2(z^*)\\
            \vdots\\
            f_n(z^*)
        \end{pmatrix}
        =
        \sum_{i=1}^n F_i(x_*)=F(x_*),
    \end{equation*}
    which implies that it is an unbiased estimator of $F$.
    Then note that for any $x\in\mathbb{R}^d$,
    \begin{eqnarray*}
        \mathbb{E}\autonorm{g(x)-g(x_*)}^2 &=& \mathbb{E}\sum_{i=1}^n\autonorm{\autopar{g_i(x_i)-g_i(z^*)}}^2 \\
        &=& \sum_{i=1}^n\mathbb{E}\autonorm{g_i(x_i)-g_i(z^*)}^2\\
        &\leq& 
        \sum_{i=1}^nL_g \autoprod{f_i(x_i)-f_i(z^*), x_i-z^*}\\
        &=& 
        L_g\sum_{i=1}^n \autoprod{F_i(x)-F_i(x_*), x-x_*}\\
        &=& 
        L_g\autoprod{F(x)-F(x_*), x-x_*},
    \end{eqnarray*}
    here the inequality comes from the expected cocoercivity of each $g_i$ in Assumption~\ref{assume:additional_FL}, which concludes the proof.
\end{proof}

\subsection{Properties of Operators in Finite-Sum VIPs}
\label{apdx:FL_operator_check}
\begin{proposition}
    \label{prop:FL_Operator_Check_SVRG}
    With Assumption \ref{assume:additional_SVRG_FL}, the problem
    defined above
    attains an unique solution $x_*\triangleq(z^*, z^*, \cdots, z^*)\in\mathbb{R}^d$, the operator $F$ is $\frac{\mu}{n}$-quasi-strongly monotone, and the operator $g_k\triangleq(g_{1,t}, g_{2,t}, \cdots, g_{n,t})$ satisfies Assumption \ref{assume:stochastic} with 
    \begin{equation*}
        A=\widehat{\ell},\ B=2,\ C=\frac{q\widehat{\ell}}{2},\ \rho=q,\ D_1=D_2=0,\ \sigma_k^2=\sum_{i=1}^n\frac{1}{m_i}\sum_{j=1}^{m_i}\autonorm{F_{i,j}(z^*)-F_{i,j}(w_{i,k})}^2.
    \end{equation*}
\end{proposition}
\begin{proof}
    Here the proof is similar to that of Lemma \ref{lm:ABC_SVRG}. The conclusions on the solution $x_*$ and the quasi-strong monotonicity follow the same argument in the proof of Proposition \ref{prop:FL_Operator_Check}. For the property of $g_k$, it is easy to find that the unbiasedness holds. Then note that
    \begin{eqnarray*}
            &&\mathbb{E}\autonorm{g_k-F(x_*)}^2\\
            &=& \sum_{i=1}^n\mathbb{E}\autonorm{g_{i,k}-F_i(x_*)}^2 \\
            &=&
            \sum_{i=1}^n\mathbb{E}\autonorm{F_{i,j_k}(x_{i,k})-F_{i,j_k}(w_{i,k})+F_i(w_{i,j_k})-F_i(z^*)}^2\\
            &=&
            \sum_{i=1}^n\autopar{\frac{1}{m_i}\sum_{j=1}^{m_i}\autonorm{F_{i,j}(x_{i,k})-F_{i,j}(w_{i,k})+F_i(w_{i,k})-F_i(z^*)}^2}\\
            &\overset{\eqref{eq:YoungsInequality}}{\leq} &
            \sum_{i=1}^n\autopar{\frac{2}{m_i}\sum_{j=1}^{m_i}\autonorm{F_{i,j}(x_{i,k})-F_{i,j}(z^*)}^2+\autonorm{F_{i,j}(z^*)-F_{i,j}(w_{i,k})-F_i(z^*)+F_i(w_{i,k})}^2}\\
            &\leq &
            2\sum_{i=1}^n\autopar{\frac{1}{m_i}\sum_{j=1}^{m_i}\autonorm{F_{i,j}(x_{i,k})-F_{i,j}(z^*)}^2+\autonorm{F_{i,j}(z^*)-F_{i,j}(w_{i,k})}^2}\\
            &\leq &
            2\widehat{\ell}\sum_{i=1}^n\autoprod{f_i(x_{i,k})-f_i(z^*), x_{i,k}-z^*}+2\sum_{i=1}^n\frac{1}{m_i}\sum_{j=1}^{m_i}\autonorm{F_{i,j}(z^*)-F_{i,j}(w_{i,k})}^2\\
            &= &
            2\widehat{\ell}\sum_{i=1}^n\autoprod{F_i(x_{i,k})-F_i(z^*), x_k-x_*}+2\sum_{i=1}^n\frac{1}{m_i}\sum_{j=1}^{m_i}\autonorm{F_{i,j}(z^*)-F_{i,j}(w_{i,k})}^2\\
            &= &
            2\widehat{\ell}\autoprod{F(x_k)-F(x_*), x_k-x_*}+2\sigma_k^2,
    \end{eqnarray*}
    the second inequality comes from the fact that $\text{Var}(X)\leq\EE(X^2)$, and the third inequality is implied by Assumption~\ref{assume:additional_SVRG_FL}.
    Then for the second term above, we have
    \begin{equation*}
        \begin{split}
            \mathbb{E}\automedpar{\sigma_{t+1}^2}
            =\ &
            \sum_{i=1}^n\frac{1}{m_i}\sum_{j=1}^{m_i}\mathbb{E}\autonorm{F_{i,j}(w_{i,k+1})-F_{i,j}(z^*)}^2\\
            =\ &
            \sum_{i=1}^n\frac{1}{m_i}\sum_{j=1}^{m_i}\autopar{q\autonorm{F_{i,j}(x_{i,k})-F_{i,j}(z^*)}^2+(1-q)\autonorm{F_{i,j}(w_{i,k})-F_{i,j}(z^*)}^2}\\
            \leq\ &
            \sum_{i=1}^nq\widehat{\ell}\autoprod{f_i(x_{i,k})-f_i(z^*), x_{i,k}-z^*}+(1-q)\sum_{i=1}^n\frac{1}{m_i}\sum_{j=1}^{m_i}\autonorm{F_{i,j}(w_{i,k})-F_{i,j}(z^*)}^2\\
            =\ &
            q\widehat{\ell}\autoprod{F(x_k)-F(x_*), x_k-x_*}+(1-q)\sigma_k^2,
        \end{split}
    \end{equation*}
    here the second equality comes from the definition of $w_{i, k+1}$, and the inequality comes from Assumption~\ref{assume:additional_SVRG_FL}. So we conclude the proof.
\end{proof}


\subsection{Further Comparison of Communication Complexities}
In Table~\ref{table:comparison_v2}, following the discussion (``Comparison with Literature") in Section~\ref{asdas}, we compare our convergence results to closely related works on stochastic local training methods. The Table shows the improvement in terms of communication and iteration complexities of our proposed \algname{ProxSkip-VIP-FL} algorithms over methods like Local SGDA~ \cite{deng2021local}, Local SEG~\cite{beznosikov2020distributed} and FedAvg-S~\cite{hou2021efficient}.

\begin{center}
\begin{table}[htbp]

    \centering
    \small
    \renewcommand{\arraystretch}{2.25}
    \caption{
    Comparison of federated learning algorithms for solving VIPs with strongly monotone and Lipschitz operator. Comparison is in terms of both iteration and communication complexities. }
    \begin{threeparttable}[b]
    \centering
    \resizebox{0.95\textwidth}{!}{
        \begin{tabular}{c |
            >{\centering}p{0.08\textwidth}|
            >{\centering}p{0.38\textwidth}|
            >{\centering\arraybackslash}p{0.37\textwidth}}
            \hline \hline
            \textbf{Algorithm}
            & \small \textbf{Setup}\tnote{1}
            & \textbf{\# Communication}\tnote{2}
            & \textbf{\# Iteration}\\
            \hline \hline
           
            \makecell[c]{
                \textbf{Local SGDA}
                \\
                \cite{deng2021local}
            }
            & SM, LS
            & $\mathcal{O}\autopar{\sqrt{\frac{\kappa^2\sigma_*^2}{\mu\epsilon}}}$
            & $\mathcal{O}\autopar{\frac{\kappa^2\sigma_*^2}{\mu n\epsilon}}$
            \\
            \hline
            \makecell[c]{
                \textbf{Local SEG}
                \vspace{0.25em}
                \\
                \cite{beznosikov2020distributed}
            }            
            & SM, LS
            & $\scriptsize \mathcal{O}\autopar{\max\autopar{\kappa\ln\frac{1}{\epsilon}, \frac{p\Delta^2}{\mu^2n\epsilon}, \frac{\kappa \xi}{\mu\sqrt{\epsilon}}, \frac{\sqrt{p}\kappa \Delta}{\mu\sqrt{\epsilon}}}}$
            & $\mathcal{O}\autopar{\max\autopar{\frac{\kappa}{p}\ln\frac{1}{\epsilon}, \frac{\Delta^2}{\mu^2n\epsilon}, \frac{\kappa \xi}{p\mu\sqrt{\epsilon}}, \frac{\kappa \Delta}{\mu\sqrt{p\epsilon}}}}$
            \\
            \hline
            \makecell[c]{
                \textbf{FedAvg-S}
                \vspace{0.25em}
                \\
                \cite{hou2021efficient}
            }            
            & SM, LS
            & $\Tilde{\mathcal{O}}\autopar{\frac{p\Delta^2}{n\mu^2\epsilon}+\frac{\sqrt{p}\kappa \Delta}{\mu\sqrt{\epsilon}}+\frac{\kappa \xi}{\mu\sqrt{\epsilon}}}$
            & $\Tilde{\mathcal{O}}\autopar{\frac{\Delta^2}{n\mu^2\epsilon}+\frac{\kappa \Delta}{\mu\sqrt{p\epsilon}}+\frac{\kappa \xi}{p\mu\sqrt{\epsilon}}}$
            \\
            \hline
            \makecell[c]{
                \textbf{\algname{ProxSkip-VIP-FL}}
                \vspace{0.25em}
                \\
                (This chapter)
            }            
            & SM, LS\tnote{3}
            & $\mathcal{O}\autopar{\sqrt{\max\autobigpar{\kappa^2, \frac{\sigma_*^2}{\mu^2\epsilon}}}\ln\frac{1}{\eps}}$
            & $\mathcal{O}\autopar{\max\autobigpar{\kappa^2, \frac{\sigma_*^2}{\mu^2\epsilon}}\ln\frac{1}{\eps}}$
            \\
            \hline
            \makecell[c]{
                \textbf{\scriptsize \algname{ProxSkip-L-SVRGDA-FL}}
                \vspace{0.25em}
                \\
                (This chapter)\tnote{4}
            }            
            & SM, LS
            & $\mathcal{O}\autopar{\kappa\ln\frac{1}{\eps}}$
            & $\mathcal{O}\autopar{\kappa^2\ln\frac{1}{\eps}}$
            \\
            \hline \hline
        \end{tabular}
        }
        \begin{tablenotes}
            \footnotesize
            \item[1] 
            SM: strongly monotone, LS: (Lipschitz) smooth. $\kappa\triangleq L/\mu$, $L$ and $\mu$ are the modulus of SM and LS.\\ 
            $\sigma_*^2<+\infty$ is an upper bound of the variance of the stochastic operator at $x_*$. $\Delta$ is an (uniform) \\
            upper bound of the variance of the stochastic operator. 
            $\xi^2$ represents the bounded heterogeneity, i.e., \\
            $g_i(x;\xi_i)$ is an unbiased estimator of $f_i(x)$ for any $i\in\autobigpar{1,\dots,n}$, and \\
            $\xi_i^2(x)\triangleq\sup_{x\in\mathbb{R}^d}\autonorm{f_i(x)-F(x)}^2\leq \xi^2\leq +\infty$.
            \item[2]  $p$ is the probability of synchronization, we can take $p=\mathcal{O}(\sqrt{\epsilon})$, which recovers $\mathcal{O}(1/\sqrt{\epsilon})$  \\
            communication complexity dependence on $\epsilon$ in our result. $\Tilde{\mathcal{O}}(\cdot)$ hides the logarithmic terms.
            \item[3] Our algorithm works for quasi-strongly monotone and star-cocoercive operators, which is more  \\
            general than the SM and LS setting, note that an $L$-LS and $\mu$-SM operator can be shown to be \\
            ($\kappa L$)-star-cocoercive~\cite{loizou2021stochastic}.
            \item[4] When we further consider the finite-sum form problems, we can turn to this algorithm.
        \end{tablenotes}
    \end{threeparttable}
    \label{table:comparison_v2}
\end{table}
\end{center}

\newpage
\section{Details on Numerical Experiments}
\label{apdx:numerical_experiments}
In experiments, we examine the performance of \algname{ProxSkip-VIP-FL} and \algname{ProxSkip-L-SVRGDA-FL}. We compare \algname{ProxSkip-VIP-FL} and \algname{ProxSkip-L-SVRGDA-FL} algorithm with \algname{Local SGDA}~\cite{deng2021local} and \algname{Local SEG}~\cite{beznosikov2020distributed}, and the parameters are chosen according to corresponding theoretical convergence guarantees. Given any function $f(x_1, x_2)$, the $\ell$ co-coercivity parameter of the operator 
\begin{eqnarray*}
    \begin{pmatrix}
        \nabla_{x_1} f(x_1, x_2)\\
        - \nabla_{x_2} f(x_1, x_2)
    \end{pmatrix}
\end{eqnarray*} is given by $\frac{1}{\ell} = \min_{\lambda \in \text{Sp}(J)} \mathcal{R} \left(\frac{1}{\lambda} \right)$ \cite{loizou2021stochastic}. Here, $\text{Sp}$ denotes spectrum of the Jacobian matrix 
\begin{eqnarray*}
    J = \begin{pmatrix}
        \nabla^2_{x_1, x_1} f & \nabla^2_{x_1, x_2} f \\
        -\nabla^2_{x_1, x_2} f & -\nabla^2_{x_2, x_2} f
    \end{pmatrix}.
\end{eqnarray*}
In our experiment, for the min-max problem of the form 
\begin{eqnarray}\label{eq:finite_sum}
    \min_{x_1} \max_{x_2}\frac{1}{n} \sum_{i = 1}^n \frac{1}{m_i} \sum_{j = 1}^{m_i} f_{ij}(x_1, x_2),
\end{eqnarray}
we use stepsizes according to the following Table.

\begin{center}
\begin{threeparttable}[H]
    \renewcommand{\arraystretch}{.8}
    \resizebox{0.95\textwidth}{!}{
    \begin{tabular}{c|c|c}
        \hline \hline
        \textbf{Algorithm}
        & \textbf{Step size $\gamma$}
        & \textbf{Value of $p$}
        \\
        \hline \hline
        \makecell[c]{
            \textbf{\algname{ProxSkip-VIP-FL}}
            \\
            (deterministic)
        }            
        & $\qquad\gamma = \frac{1}{2 \max_{i \in [n]} \ell_i}\qquad$ 
        & $\qquad p = \sqrt{\gamma \mu}\qquad$\\
        \hline
        \makecell[c]{
            \textbf{\algname{ProxSkip-VIP-FL}}
            \\
            (stochastic)
        }            
        & $\gamma = \frac{1}{2 \max_{i, j} \ell_{ij}}$ & $p = \sqrt{\gamma \mu}$\\
        \hline
        \makecell[c]{
            \textbf{$\qquad$ProxSkip-L-SVRGDA-FL$\qquad$}
            \\
            (finite-sum)
        }    
        & $\gamma = \frac{1}{6 \max_{i,j} \ell_{ij}}$ & $p = \sqrt{\gamma \mu}$
        \\
        \hline \hline
   \end{tabular}
   }
   \caption{Parameter settings for each algorithm. Here $\ell_i$ is the co-coercivity parameter corresponding to $\frac{1}{m_i} \sum_{j = 1}^{m_i} f_{ij}(x_1, x_2)$ and $\ell_{ij}$ is the co-coercivity parameter corresponding to $f_{ij}$.}
\label{table:stepsize}
\end{threeparttable}   
\end{center}

The parameters in Table \ref{table:stepsize} are selected based on our theoretical convergence guarantees, presented in the main paper.  In particular, for \algname{ProxSkip-VIP-FL}, we use the stepsizes suggested in Theorem \ref{thm:complexity_FL_ProxSkip}, which follows the setting of Corollary~\ref{thm:convergence_ProxSkip-SGDA}. Note that the full-batch estimator (deterministic setting) of \eqref{eq:finite_sum} satisfies Assumption \ref{assume:additional_FL} when $L_g = \max_{i \in [n]} \ell_i$. Similarly, the stochastic estimators of \eqref{eq:finite_sum} satisfies Assumption \ref{assume:additional_FL} with $L_g = \max_{i,j} \ell_{ij}$. For the variance reduced method, \algname{ProxSkip-L-SVRGDA-FL}, we use the stepsizes as suggested in Theorem~\ref{dnaoao} (which follows the setting in Corollary~\ref{cor:ProxSkip-L-SVRGDA-FL_complexity}) with $\hat{\ell} = \max_{i,j} \ell_{ij}$ for \eqref{eq:finite_sum}. For all methods, the probability of making the proximal update (communication in the FL regime) equals $p = \sqrt{\gamma \mu}$. For the 
\algname{ProxSkip-L-SVRGDA-FL}, following Corollary~\ref{cor:ProxSkip-L-SVRGDA-FL_complexity} we set $q=2\gamma\mu$.

\subsection{Details on Robust least squares}
The objective function of Robust Least Square is given by
\begin{eqnarray*}
    G(\beta , y) = \|{\bf A}\beta - y\|^2 - \lambda \|y - y_0\|^2,
\end{eqnarray*}
for $\lambda > 0$. Note that 
$$
\nabla_{\beta}^2 G(\beta, y) = 2 \textbf{A}^{\top}\textbf{A},
\quad
\nabla^2_y G(\beta, y) = 2 (\lambda - 1)\textbf{I}.
$$
Therefore, for $\lambda > 1$, the objective function $G$ is strongly monotone with parameter $\min \{ 2 \lambda_{\min}(\textbf{A}^{\top}\textbf{A}), 2(\lambda - 1)\}$. Moreover, $G$ can be written as a finite sum problem, similar to \eqref{quadraticgame}, by decomposing the rows of matrix $\textbf{A}$ i.e. 
$$
G(\beta, y) = \sum_{i = 1}^r (A_i^{\top} \beta - y_i )^2 - \lambda (y_i - y_{0i})^2.
$$ 
Here $\mathbf{A}_i^{\top}$ denotes the $i$th row of the matrix $\mathbf{A}$. Now we divide the $r$ rows among $n$ nodes where each node will have $m = \nicefrac{r}{n}$ rows. Then we use the canonical basis vectors $e_i$ (vector with $i$-th entry $1$ and $0$ for other coordinates) to rewrite the above problem as follow
\begin{eqnarray*}
    G(\beta, y) & = & \sum_{i = 1}^n \sum_{j = 1}^m  \left(A_{(i-1)m + j}^{\top}\beta - y_{(i-1)m + j} \right)^2 - \lambda \left(y_{(i-1)m + j} - y_{0((i-1)m + j)} \right)^2 \\
    & = & \sum_{i = 1}^n \sum_{j = 1}^m \beta^{\top} A_{(i-1)m + j} A_{(i-1)m + j}^{\top} \beta - 2 y_{(i-1)m + j} A_{(i-1)m + j}^{\top} \beta + y_{(i-1)m + j}^2\\
    && - \lambda y_{(i-1)m + j}^2  - \lambda y_{0((i-1)m + j)}^2 + 2 \lambda y_{0((i-1)m + j)} y_{((i-1)m + j)}\\
    & = & \sum_{i = 1}^n \sum_{j = 1}^m \beta^{\top} \left( A_{(i-1)m + j} A_{(i-1)m + j}^{\top} \right) \beta  - y^{\top} \left(2 e_{(i-1)m + j} A_{(i-1)m + j}^{\top}\right)\beta \\
    && - y^{\top} \left((\lambda - 1)e_{(i-1)m + j}e_{(i-1)m + j}^{\top} \right) y + \left(2 \lambda y_{0((i-1)m + j)} e_{((i-1)m + j)}^{\top} \right) y\\
    && - \lambda y_{0((i-1)m + j)}^2 \\
    & = & \frac{nm}{n}\sum_{i = 1}^n \frac{1}{m} \sum_{j = 1}^m \beta^{\top} \left( A_{(i-1)m + j} A_{(i-1)m + j}^{\top} \right) \beta  - y^{\top} \left(2 e_{(i-1)m + j} A_{(i-1)m + j}^{\top}\right)\beta \\
    && - y^{\top} \left((\lambda - 1)e_{(i-1)m + j}e_{(i-1)m + j}^{\top} \right) y + \left(2 \lambda y_{0((i-1)m + j)} e_{((i-1)m + j)}^{\top} \right) y \\
    && - \lambda y_{0((i-1)m + j)}^2 .
\end{eqnarray*}
Therefore $G$ is equivalent to \eqref{quadraticgame} with $n$ nodes, $m_i = m = \nicefrac{r}{n}, x_1 = \beta, x_2 = y$ and 
\begin{eqnarray*}
    f_{ij} (x_1, x_2) &=& x_1^{\top} \left( A_{(i-1)m + j} A_{(i-1)m + j}^{\top} \right) x_1  - x_2^{\top} \left(2 e_{(i-1)m + j} A_{(i-1)m + j}^{\top}\right)x_1 \\
    && - x_2^{\top} \left((\lambda - 1)e_{(i-1)m + j}e_{(i-1)m + j}^{\top} \right) x_2 + \left(2 \lambda y_{0((i-1)m + j)} e_{((i-1)m + j)}^{\top} \right) x_2\\
    && - \lambda y_{0((i-1)m + j)}^2.
\end{eqnarray*}
In Figure~\ref{fig: Comparison of ProxSkip-VIP-FL vs Local SGDA vs Local SEG on Heterogeneous Data in the stochastic setting.}, we run our experiment on the "California Housing" dataset from scikit-learn package \cite{pedregosa2011scikit}. This data consists of $8$ attributes of $200$ houses in the California region where the target variable $y_0$ is the price of the house. To implement the algorithms, we divide the data matrix $\textbf{A}$ among $20$ nodes, each node having an equal number of rows of $\textbf{A}$. Similar to the last example, here we also choose our ProxSkip-VIP-FL algorithm, Local SGDA, and Local SEG for comparison in the experiment, also we use $\lambda = 50$, and the theoretical stepsize choice is similar to the previous experiment.

In Figure~\ref{fig:RLSxSynthetic}, we reevaluate the performance of ProxSkip on the Robust Least Square problem with synthetic data. For generating the synthetic dataset, we set $r = 200$ and $s = 20$. Then we sample $\textbf{A} \sim \mathcal{N}(0, 1), \beta_0 \sim \mathcal{N}(0, 0.1), \epsilon \sim \mathcal{N}(0, 0.01)$ and set $y_0 = \mathbf{A}\beta_0 + \epsilon$. In both deterministic (Figure~\ref{fig:RLSxDeterministicxSynthetic}) and stochastic (Figure~\ref{fig:RLSxStochasticxSynthetic}) setting, ProxSkip outperforms Local GDA and Local EG.  

\begin{figure}[H]
\centering
\begin{subfigure}[b]{0.45\textwidth}
    \centering
    \includegraphics[width=\textwidth]{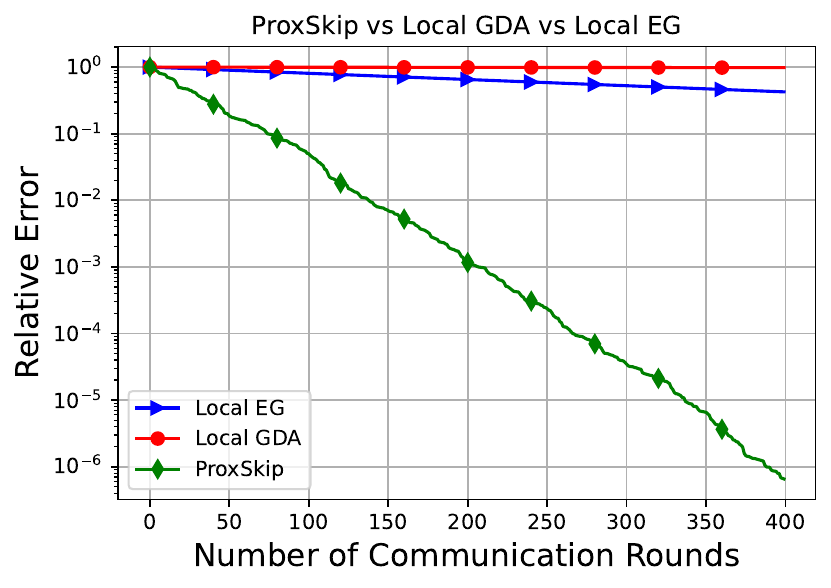}
    \caption{Deterministic Algorithms}
    \label{fig:RLSxDeterministicxSynthetic}
\end{subfigure}
\hspace{1em}
\begin{subfigure}[b]{0.45\textwidth}
    \centering
    \includegraphics[width=\textwidth]{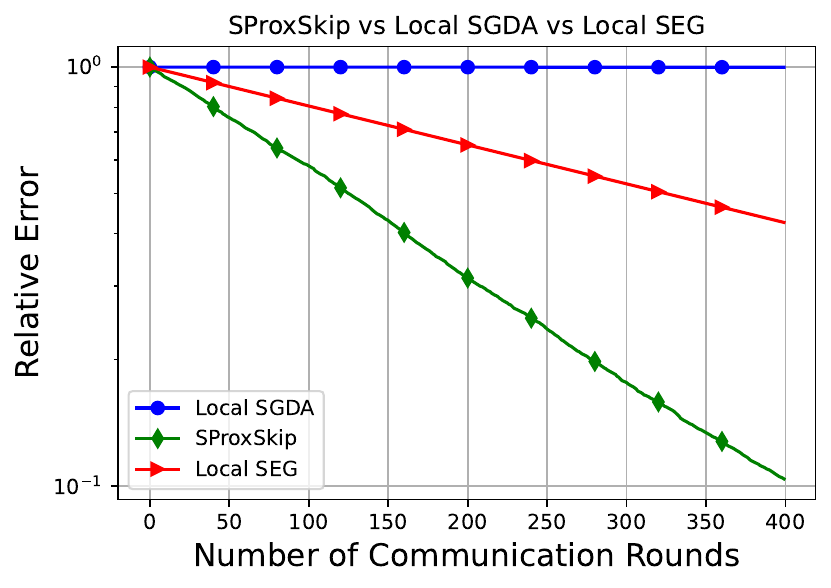}
    \caption{Stochastic Algorithms}
    \label{fig:RLSxStochasticxSynthetic}
\end{subfigure}
\caption{Comparison of algorithms on the Robust Least Square~\eqref{robustleastsquare} using synthetic dataset.}
\label{fig:RLSxSynthetic}
\vspace{-1em}
\end{figure}

\newpage
\section{Additional Experiments}
\label{apdx:more_experiment}

Following the experiments presented in the main paper, we further evaluate the performance of the proposed methods in different settings (problems and stepsize selections). 

\subsection{Fine-Tuned Stepsize}
In Figure~\ref{fig:TunedSteps}, we compare the performance of ProxSkip against that of Local GDA and Local EG on the strongly monotone quadratic game~\eqref{quadraticgame} with heterogeneous data using tuned stepsizes. For tuning the stepsizes, we did a grid search on the set of $\frac{1}{rL}$ where 
$r \in \{1, 2, 4, 8, 16, 64, 128, 256, 512, 1024, 2048\}$ and $L$ is the Lipschitz constant of $F$. ProxSkip outperforms the other two methods in the deterministic setting while it has a comparable performance in the stochastic setting. 

\begin{figure}[H]
\centering
\begin{subfigure}[b]{0.45\textwidth}
    \centering
    \includegraphics[width=\textwidth]{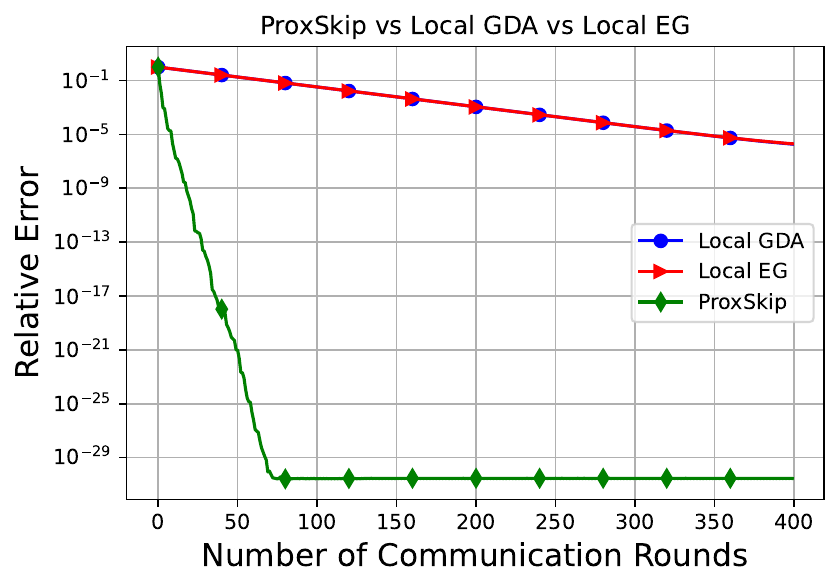}
    \caption{Deterministic Setting.}
\end{subfigure}
\hspace{1em}
\begin{subfigure}[b]{0.45\textwidth}
    \centering
    \includegraphics[width=\textwidth]{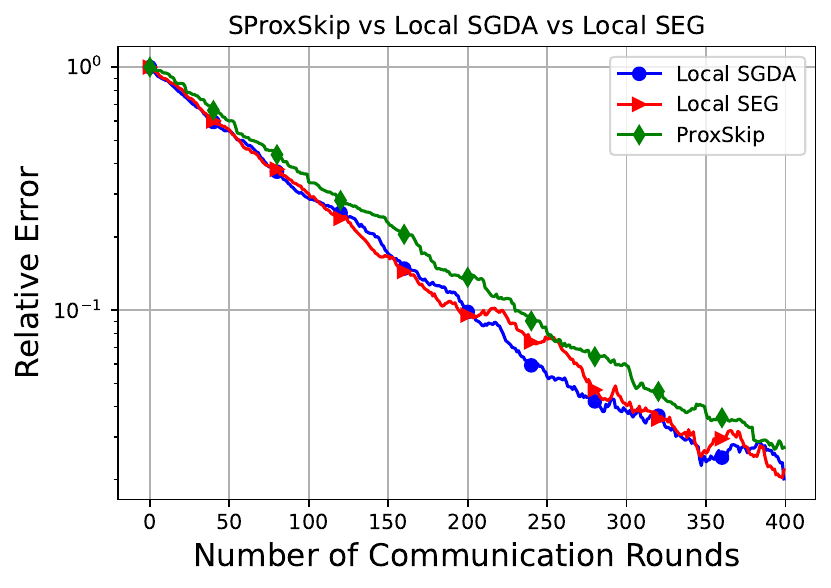}
    \caption{Stochastic Setting.}
\end{subfigure}
\caption{\emph{Comparison of \algname{ProxSkip-VIP-FL} vs \algname{Local SGDA} vs \algname{Local SEG} on Heterogeneous Data with tuned stepsizes.}}
\label{fig:TunedSteps}
\end{figure}

\subsection{\algname{ProxSkip-VIP-FL} vs. \algname{ProxSkip-L-SVRGDA-FL}}
In Figure \ref{fig: Comparison of SProxSkip vs ProxSkip-L-SVRGDA}, we compare the stochastic version of ProxSkip-VIP-FL with \algname{ProxSkip-L-SVRGDA-FL}. In Figure~\ref{fig: Comparison of SProxSkip vs ProxSkip-L-SVRGDA tuned}, we implement the methods with tuned stepsizes while in Figure~\ref{fig: Comparison of SProxSkip vs ProxSkip-L-SVRGDA theoretical} we use the theoretical stepsizes. For the theoretical stepsizes of \algname{ProxSkip-L-SVRGDA-FL}, we use the stepsizes from Corollary \ref{cor:ProxSkip-L-SVRGDA-FL_complexity}. We observe that \algname{ProxSkip-L-SVRGDA-FL} performs better than ProxSkip-VIP-FL with both tuned and theoretical stepsizes.
\begin{figure}[H]
\centering
\begin{subfigure}[b]{0.45\textwidth}
    \centering
    \includegraphics[width=\textwidth]{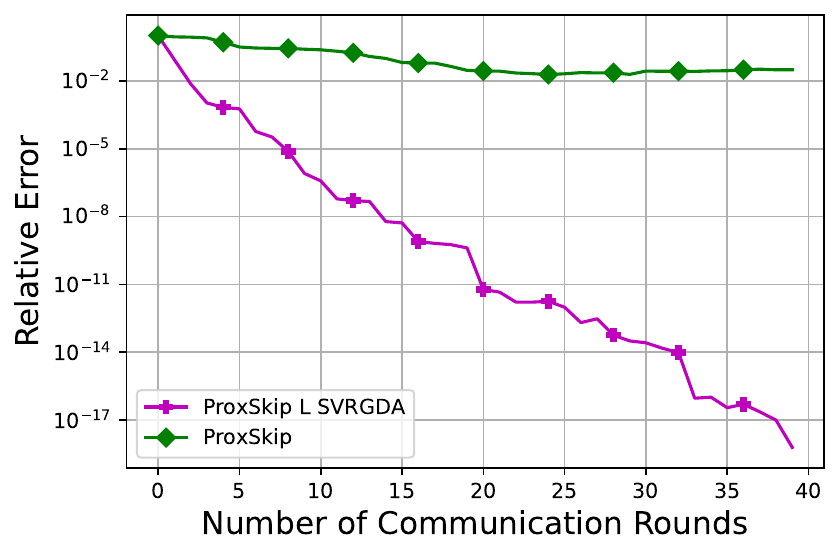
}
    \caption{Tuned stepsize}
    \label{fig: Comparison of SProxSkip vs ProxSkip-L-SVRGDA tuned}
\end{subfigure}
\hspace{1em}
\begin{subfigure}[b]{0.45\textwidth}
    \centering
    \includegraphics[width=\textwidth]{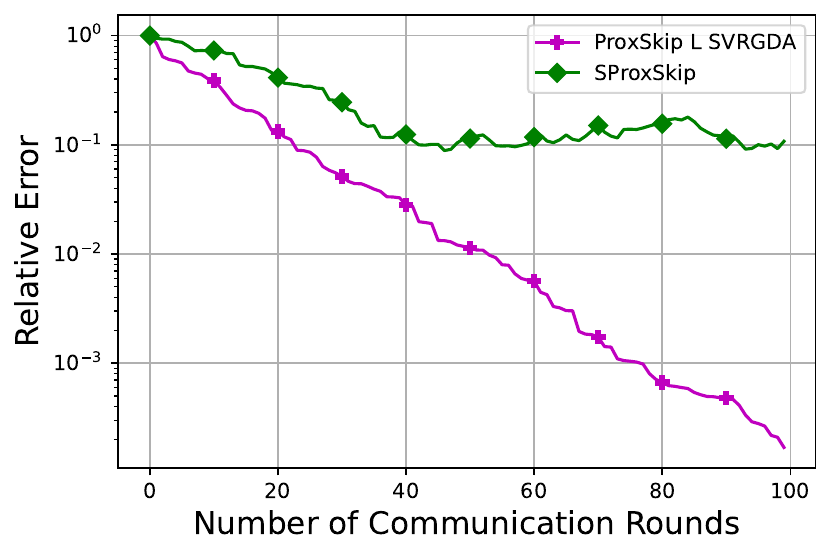}
    \caption{Theoretical stepsize}
    \label{fig: Comparison of SProxSkip vs ProxSkip-L-SVRGDA theoretical}
\end{subfigure}
\caption{\emph{Comparison of \algname{ProxSkip-VIP-FL} and \algname{ProxSkip-L-SVRGDA-FL} using the tuned and theoretical stepsizes.}}
\label{fig: Comparison of SProxSkip vs ProxSkip-L-SVRGDA}
\end{figure}
\subsection{Low vs High Heterogeneity}
\label{Performance on Heterogeneous vs Homogeneous Data with Tuned Stepsize}
We conduct a numerical experiment on a toy example to verify the efficiency of our proposed algorithm. Following the setting in \cite{tarzanagh2022fednest}, we consider the minimax objective function 
$$\min_{x_1 \in \mathbb{R}^d} \max_{x_2 \in \mathbb{R}^d} \frac{1}{n} \sum_{i = 1}^n  f_{i}(x_1, x_2)$$ 
where $f_{i}$ are given by
\begin{equation*}
    f_i(x_1, x_2)=-\automedpar{\frac{1}{2}\autonorm{x_2}^2-b_i^\top x_2+x_2^\top A_i x_1}+\frac{\lambda}{2}\autonorm{x_1}^2
\end{equation*}
Here we set the number of clients $n=100$, $d_1=d_2=20$ and $\lambda=0.1$. We generate $b_i\sim\mathcal{N}(0,s_i^2I_{d_2})$ and $A_i = t_i I_{d_1 \times d_2}$. For Figure~\ref{fig: Comparison of ProxSkip vs Local GDA vs Local EG on Homogeneous vs Heterogeneous Data p1}, we set $s_i = 10$ and $t_i = 1$ for all $i$ while in Figure \ref{fig: Comparison of ProxSkip vs Local GDA vs Local EG on Homogeneous vs Heterogeneous Data p2}, we generate $s_i \sim \text{Unif}(0,20)$ and $t_i \sim \text{Unif}(0, 1)$. 

We implement Local GDA, Local EG, and ProxSkip-VIP-FL with tuned stepsizes (we use grid search to tune stepsizes Appendix \ref{apdx:numerical_experiments}). In Figure \ref{fig: Comparison of ProxSkip vs Local GDA vs Local EG on Homogeneous vs Heterogeneous Data p3}, we observe that Local EG performs better than ProxSkip-VIP-FL in homogeneous data. However, in Figure \ref{fig: Comparison of ProxSkip vs Local GDA vs Local EG on Homogeneous vs Heterogeneous Data p4}, the performance of Local GDA \cite{deng2021local} and Local EG deteriorates for heterogeneous data, and ProxSkip-VIP-FL outperforms both of them in this case. To get stochastic estimates, we add Gaussian noise \cite{beznosikov2022decentralized} (details in Appendix \ref{apdx:numerical_experiments}). We repeat this experiment in the stochastic settings in  Figure \ref{fig: Comparison of ProxSkip vs Local GDA vs Local EG on Homogeneous vs Heterogeneous Data p3} and \ref{fig: Comparison of ProxSkip vs Local GDA vs Local EG on Homogeneous vs Heterogeneous Data p4}. We observe that ProxSkip-VIP-FL has a comparable performance with Local SGDA \cite{deng2021local} and Local SEG in homogeneous data. However, ProxSkip-VIP-FL is faster on heterogeneous data.

\begin{figure}[htbp]
\centering
\begin{subfigure}[b]{0.45\linewidth}
    \centering
    \includegraphics[width=\textwidth]{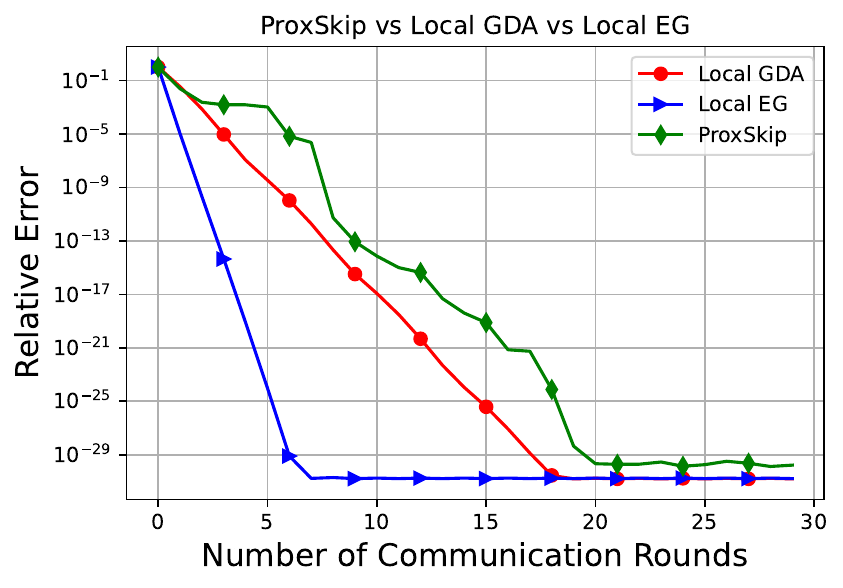
}
    \caption{constant $s_i,t_i$}
    \label{fig: Comparison of ProxSkip vs Local GDA vs Local EG on Homogeneous vs Heterogeneous Data p1}
\end{subfigure}
\hspace{1em}
\begin{subfigure}[b]{0.45\textwidth}
    \centering
    \includegraphics[width=\textwidth]{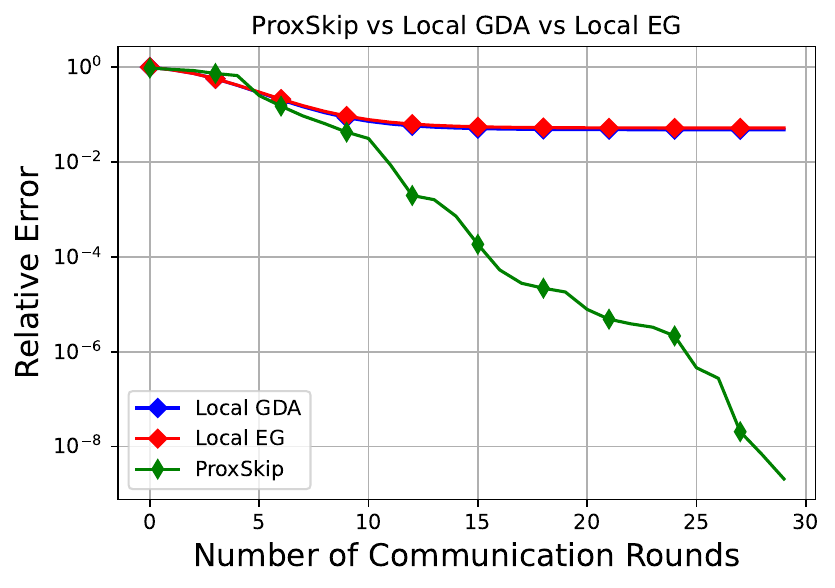}
    \caption{uniformly generated $s_i,t_i$}
    \label{fig: Comparison of ProxSkip vs Local GDA vs Local EG on Homogeneous vs Heterogeneous Data p2}
\end{subfigure}
\begin{subfigure}[b]{0.45\linewidth}
    \centering
    \includegraphics[width=\textwidth]{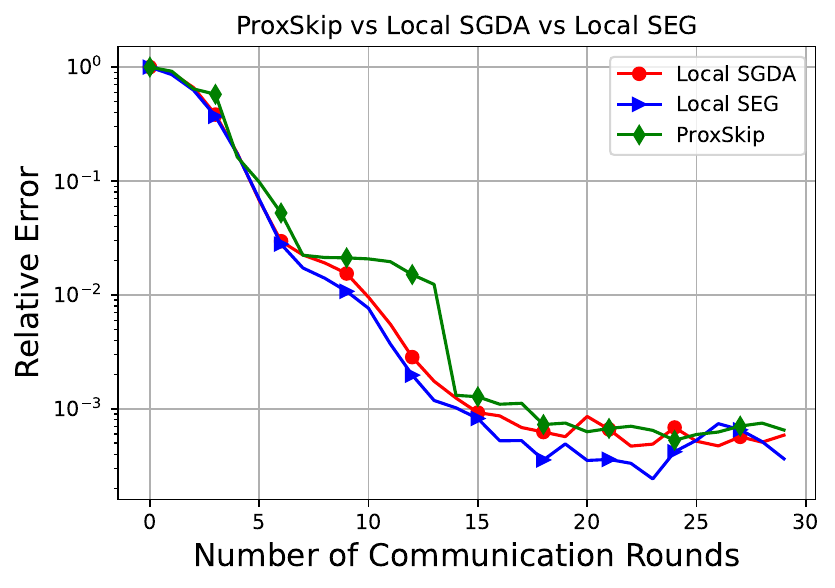
}
    \caption{constant $s_i, t_i$}
    \label{fig: Comparison of ProxSkip vs Local GDA vs Local EG on Homogeneous vs Heterogeneous Data p3}
\end{subfigure}
\hspace{1em}
\begin{subfigure}[b]{0.45\textwidth}
    \centering
    \includegraphics[width=\textwidth]{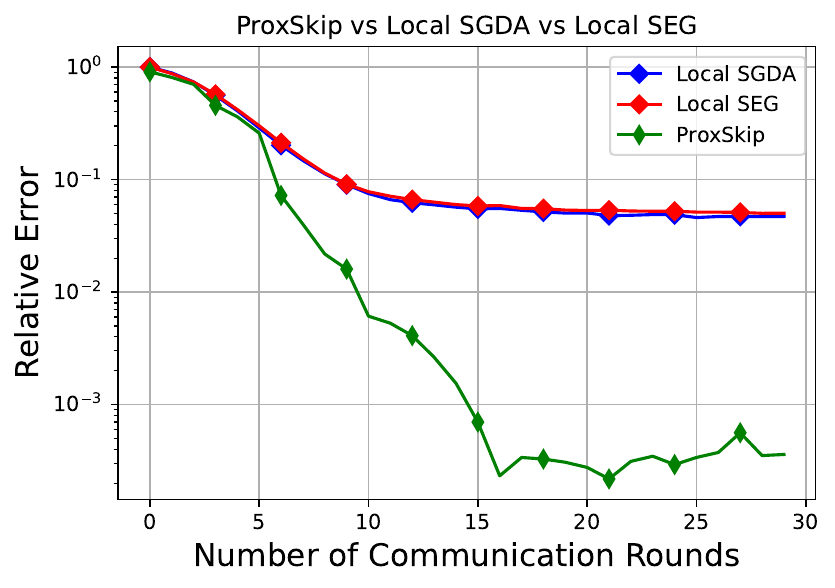}
    \caption{uniformly generated $s_i, t_i$}
    \label{fig: Comparison of ProxSkip vs Local GDA vs Local EG on Homogeneous vs Heterogeneous Data p4}
\end{subfigure}
\caption{\emph{Comparison of \algname{ProxSkip-VIP-FL} vs \algname{Local GDA} vs \algname{Local EG} on Homogeneous vs Heterogeneous Data. In (a) and (b), we run the deterministic algorithms, while in (c) and (d), we run the stochastic versions. For (a) and (c), we set $s_i = 10, t_i = 1$ for all $i\in [n]$ and for (b) and (d), we generate $s_i,t_i$ uniformly from $s_i \sim \text{Unif}(0,20), t_i \sim \text{Unif}(0,1)$.}}
\label{fig: Comparison of ProxSkip vs Local GDA vs Local EG on Homogeneous vs Heterogeneous Data}
\end{figure}

\subsection{Performance on Data with Varying Heterogeneity}
In this experiment, we consider the operator $F$ given by 
$$
F(x) \coloneqq \frac{1}{2} F_1(x) + \frac{1}{2} F_2(x)
$$ 
where 
$$
F_1(x) \coloneqq M(x - x_1^*),\quad F_2(x) \coloneqq M(x - x_2^*)
$$ 
with $M \in \mathbb{R}^{2 \times 2}$ and $x_1^*,x_2^* \in \mathbb{R}^2$. For this experiment we choose 
$$
M \coloneqq I_2,\quad
x_1^* = (\delta, 0),\quad
x_2^* = (0, \delta).
$$ 
Note that, in this case, $x_* = \frac{1}{2}(x_1^* + x_2^*)$. Then the quantity $\max_{i \in [2]} \|F_i(x_*) - F(x_*)\|^2$, which quantifies the amount of heterogeneity in the model, is equal to $\frac{\delta^2}{2}$. Therefore, increasing the value of $\delta$ increases the amount of heterogeneity in the data across the clients. 

We compare the performances of ProxSkip-VIP-FL, Local GDA, and Local EG when $\delta = 0$ and $\delta = 10^6$ in Figure \ref{fig: Comparison of ProxSkip-VIP-FL, Local GDA and Local EG with theoretical stepsizes with different delta}. In either case, ProxSkip-VIP-FL outperforms the other two methods with the theoretical stepsizes.

\begin{figure}[htbp]
\centering
\begin{subfigure}[b]{0.45\textwidth}
    \centering
    \includegraphics[width=\textwidth]{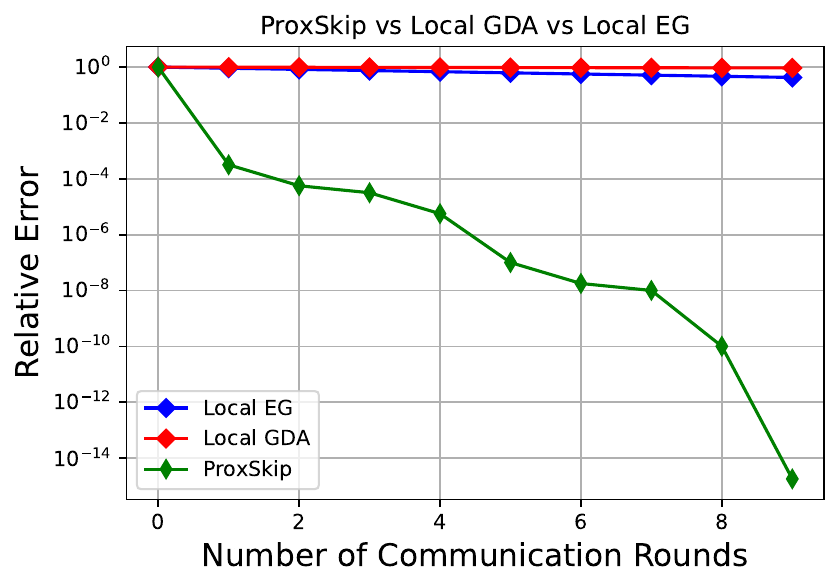}
    \caption{$\delta = 0$}
\end{subfigure}
\hspace{1em}
\begin{subfigure}[b]{0.45\textwidth}
    \centering
    \includegraphics[width=\textwidth]{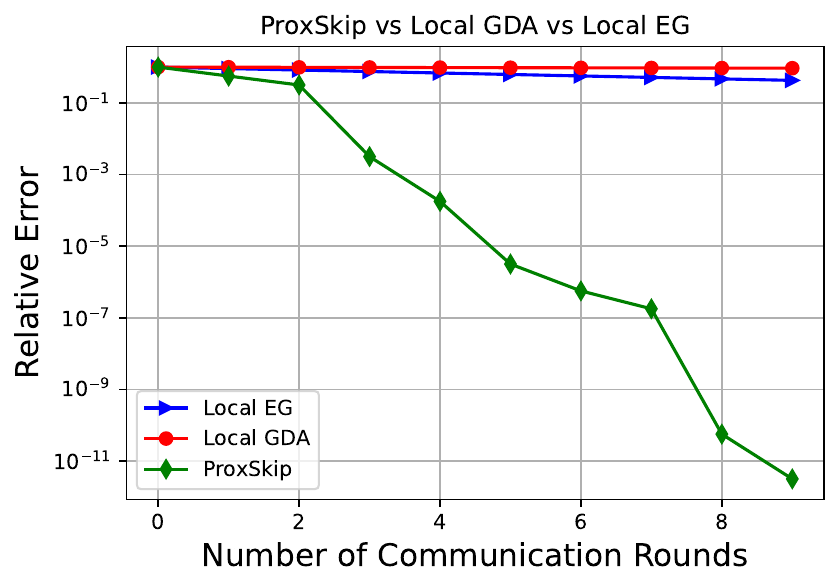}
    \caption{$\delta = 10^6$}
\end{subfigure}
\caption{\emph{Comparison of \algname{ProxSkip-VIP-FL}, \algname{Local GDA} and \algname{Local EG} with theoretical stepsizes for $\delta = 0$ (left) and $\delta = 10^6$ (right).}}
\label{fig: Comparison of ProxSkip-VIP-FL, Local GDA and Local EG with theoretical stepsizes with different delta}
\end{figure}

\subsection{Extra Experiment: Policemen Burglar Problem}
In this experiment, we compare the perforamnce of \algname{ProxSkip-GDA-FL}, \algname{Local GDA} and \algname{Local EG} on a Policemen Burglar Problem~\cite{nemirovski2013mini} (a particular example of matrix game) of the form:
\begin{eqnarray*}
    \min_{x_1 \in \Delta} \max_{x_2 \in \Delta} f(x_1, x_2) = \frac{1}{n} \sum_{i = 1}^n  x_1^{\top}A_ix_2
\end{eqnarray*}
where $\Delta = \left\{ x \in \R^d \mid \mathbf{1}^{\top}x = 1, x \geq 0 \right\}$ is a $(d-1)$ dimensional standard simplex. We generate the $(r, s)$-th element of the matrix $A_i$ as follow
\begin{eqnarray*}
    A_i(r,s) = w_r \left(1 - \exp \left\{- 0.8 |r - s| \right\} \right) \qquad \forall i \in [n] 
\end{eqnarray*}
where $w_r = |w_r'|$ with $w_r' \sim \mathcal{N}(0, 1)$. This matrix game is a constrained monotone problem and we use duality gap to measure the performance of the algorithms. Note that the duality gap for the problem $\min_{x \in \Delta} \max_{y \in \Delta} f(x, y)$ at $(\hat{x}, \hat{y})$ is defined as $\max_{y \in \Delta} f(\hat{x}, y) - \min_{x \in \Delta} f(x, \hat{y})$. 

\begin{figure}[H]
\centering
\includegraphics[width=0.5\textwidth]{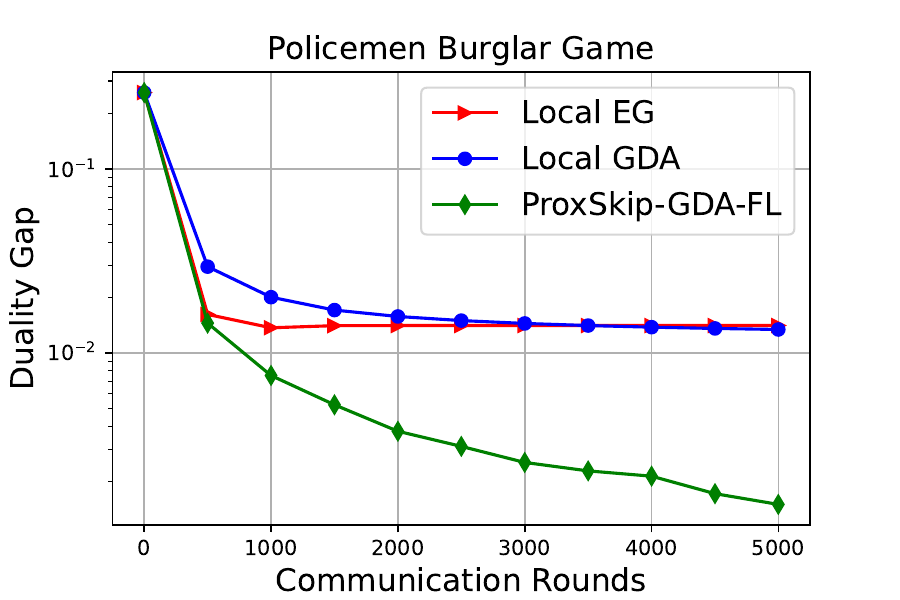}
\caption{Comparison of \algname{ProxSkip-GDA-FL}, \algname{Local EG} and \algname{Local GDA} on Policemen Burglar Problem after $5000$ communication rounds.}\label{fig:matrixgame}
\end{figure}

In Figure \ref{fig:matrixgame}, we plot the duality gap (on the $y$-axis) with respect to the moving average of the iterates, i.e. $\frac{1}{K+1} \sum_{k = 0}^K x_k$ (here $x_k$ is the output after $k$ many communication rounds). As we can observe in Figure \ref{fig:matrixgame}, our proposed algorithm \algname{ProxSkip-GDA-FL} clearly outperforms \algname{Local EG} and \algname{Local GDA} in this experiment.



\end{document}
